\documentclass[a4paper, reqno]{amsart}

\usepackage{tikz} %Used for diagrams
\usetikzlibrary{cd,patterns,arrows,positioning,decorations.pathmorphing}
\usepackage{xparse} %Used for optional arguments in commands
\usepackage{textcomp} %Used for open bullet

\usepackage{etoolbox}
\patchcmd{\smallmatrix}{\thickspace}{\kern.01em}{}{}

\usepackage{amsmath} %Used for introducing notation
\usepackage{hyperref} %Used for making references hyperlink
\usepackage{amssymb} %Used for various predefined commands
\usepackage{longtable} %For tables longer than one page
\usepackage{array} %For controlling alignment in table cells
\usepackage{graphicx} %For fixing the size of bullet
\usepackage{mathtools} %For absolute value
\usepackage{marginnote} %For comments in the margin
\usepackage{soul} %For striking text through
\usepackage{float} %For floating figures

\newcommand{\ZZ}{\mathbb{Z}}
\newcommand{\la}{\lambda}
\newcommand{\La}{\Lambda}
\newcommand{\tn}{\tau_n^{\phantom{0}}}
\newcommand{\tno}{\tau_n^-}
\newcommand{\cA}{\mathcal{A}}
\newcommand{\cB}{\mathcal{B}}
\newcommand{\cC}{\mathcal{C}}

\newcommand{\cF}{\mathcal{F}}
\newcommand{\cG}{\mathcal{G}}
\newcommand{\cI}{\mathcal{I}}
\newcommand{\cJ}{\mathcal{J}}
\newcommand{\cL}{\mathcal{L}}
\newcommand{\cP}{\mathcal{P}}
\newcommand{\cR}{\mathcal{R}}

\newcommand{\cU}{\mathcal{U}}
\newcommand{\cV}{\mathcal{V}}
\newcommand{\ab}{_{\! \scalebox{1.8}{$\cdot$}}}
\newcommand{\vare}{\varepsilon}
\newcommand{\om}{\Omega}
\newcommand{\Corth}{\mathcal{C}{^{\perp_n}}}
\newcommand{\orthC}{{^{\perp_n}}\mathcal{C}}
\newcommand{\Ind}[1]{\overline{#1}}
\newcommand{\PL}{\cP_{\La}}
\newcommand{\IL}{\cI_{\La}}
\newcommand{\ABL}{\cP\strut^{\text{ab}}_{\La}}

\newcommand{\LAB}{\cI\strut^{\text{ab}}_{\La}}

\newcommand{\MABL}{\cP\strut^{\text{mab}}_{\La}}

\newcommand{\MLAB}{\cI\strut^{\text{mab}}_{\La}}

\newcommand{\ABP}{\cP\strut^{\text{ab}}}
\newcommand{\sABP}{\cP\rule[-.45\baselineskip]{0pt}{\baselineskip}^{\text{ab}}}
\newcommand{\IAB}{\cI\strut^{\text{ab}}}
\newcommand{\sIAB}{\cI\rule[-.45\baselineskip]{0pt}{\baselineskip}^{\text{ab}}}
\newcommand{\MABP}{\cP\strut^{\text{mab}}}
\newcommand{\sMABP}{\cP\rule[-.45\baselineskip]{0pt}{\baselineskip}^{\text{mab}}}
\newcommand{\MIAB}{\cI\strut^{\text{mab}}}
\newcommand{\sMIAB}{\cI\rule[-.45\baselineskip]{0pt}{\baselineskip}^{\text{mab}}}
\newcommand{\Lab}{P\strut^{\text{ab}}}
\newcommand{\DLab}{I\strut^{\text{ab}}}
\newcommand{\nospacepunct}[1]{\makebox[0pt][l]{\,#1}}

\newcommand\nct[2]{\fill ({#1},{#2})  circle (0.05cm);} %For vertices in the nct
\newcommand\mct[2]{\draw ({#1},{#2})  circle (0.05cm);} %For vertices not in the nct 

\pgfdeclarepatternformonly{my north west lines}{\pgfqpoint{-6pt}{-6pt}}{\pgfqpoint{4pt}{4pt}}{\pgfqpoint{2pt}{2pt}}%
{
  \pgfsetlinewidth{0.4pt}
  \pgfpathmoveto{\pgfqpoint{0pt}{3pt}}
  \pgfpathlineto{\pgfqpoint{3.1pt}{-0.1pt}}
  \pgfusepath{stroke}
}

\DeclareMathOperator{\im}{Im}
\DeclareMathOperator{\m}{mod}
\DeclareMathOperator{\Hom}{Hom}
\DeclareMathOperator{\End}{End}
\DeclareMathOperator{\Ext}{Ext}
\DeclareMathOperator{\gldim}{gl.dim}
\DeclareMathOperator{\pd}{proj.dim}
\DeclareMathOperator{\id}{inj.dim}
\DeclareMathOperator{\rad}{rad}
\DeclareMathOperator{\topp}{top}
\DeclareMathOperator{\soc}{soc}
\DeclareMathOperator{\add}{add}
\DeclareMathOperator{\lvl}{lvl}
\DeclareMathOperator{\Id}{Id}
\DeclareMathOperator{\Sub}{Sub}
\DeclareMathOperator{\Fac}{Fac}
\DeclareMathOperator{\ind}{ind}
\DeclareMathOperator{\supp}{supp}
\DeclareMathOperator{\rep}{rep}

\DeclarePairedDelimiter\abs{\lvert}{\rvert}%
\DeclarePairedDelimiter\floor{\lfloor}{\rfloor}
\DeclarePairedDelimiter\ceil{\lceil}{\rceil}

\NewDocumentCommand\glue{ O{} O{}}{\mathbin{\raisebox{0.2ex}{${}^{\mbox{\tiny$#1$}}$\rotatebox[origin=c]{90}{$\triangleright$}$^{\mbox{\tiny$#2$}}$}}}
\newcommand{\nsup}[2]{{^{#1}\!{#2}}}
\newcommand{\supn}[2]{{{#1}\!^{#2}}}
\newcommand{\PD}{\nsup{P}{\triangle}}
\newcommand{\cPD}{\cF_P}
\newcommand{\DI}{\supn{\triangle}{I}}
\newcommand{\cDI}{\cG_I}

\newcounter{sarrow}

\NewDocumentCommand\qthree{m O{} O{}}{\begin{smallmatrix} #1 \\ #2 \\ #3 \end{smallmatrix}} %gives a representation of dimension at most 3

\NewDocumentCommand\tinyqthree{m O{} O{}}{\mbox{\tiny$\begin{smallmatrix} #1 \\ #2 \\ #3 \end{smallmatrix}$}} %gives a representation of dimension at most 3 suitable for a power

\theoremstyle{plain}
\newtheorem*{theorem*}{Theorem}
\newtheorem{theorem}{Theorem}[section] % reset theorem numbering for each chapter

\theoremstyle{definition}
\newtheorem{definition}[theorem]{Definition} % definition numbers are dependent on theorem numbers
\newtheorem{example}[theorem]{Example} % same for example numbers 
\newtheorem{corollary}[theorem]{Corollary} % same for corollary numbers 
\newtheorem{lemma}[theorem]{Lemma} % same for lemma numbers 
\newtheorem{proposition}[theorem]{Proposition} % same for proposition numbers
\newtheorem{remark}[theorem]{Remark} % remarks are the same
\numberwithin{equation}{section} % equations are numbered separately, only with respect to their section
\newtheorem{question}{Question} %Questions are numbered only with respect to their own numbers

\begin{document}

\title{Gluing of $n$-cluster tilting subcategories for rep\-re\-sen\-ta\-tion-di\-rect\-ed algebras}

\author{Laertis Vaso}


\maketitle

\begin{abstract}
Given $n\leq d<\infty$, we investigate the existence of algebras of global dimension $d$ which admit an $n$-cluster tilting subcategory. We construct many such examples using rep\-re\-sen\-ta\-tion-di\-rect\-ed algebras. First, given two rep\-re\-sen\-ta\-tion-di\-rect\-ed algebras $A$ and $B$, a projective $A$-module $P$ and an injective $B$-module $I$ satisfying certain conditions, we show how we can construct a new rep\-re\-sen\-ta\-tion-di\-rect\-ed algebra $\La:= B \glue[P][I] A$ in such a way that the representation theory of $\La$ is completely described by the representation theories of $A$ and $B$. Next we introduce $n$-fractured subcategories which generalize $n$-cluster tilting subcategories for rep\-re\-sen\-ta\-tion-di\-rect\-ed algebras. We then show how one can construct an $n$-cluster tilting subcategory for $\La$ by using $n$-fractured subcategories of $A$ and $B$. As an application of our construction, we show that if $n$ is odd and $d\geq n$ then there exists an algebra admitting an $n$-cluster tilting subcategory and having global dimension $d$. We show the same result if $n$ is even and $d$ is odd or $d\geq 2n$.

\end{abstract}

\tableofcontents

\section*{Introduction}

For a rep\-re\-sen\-ta\-tion-fi\-nite algebra $\La$, classical Aus\-lan\-der--Rei\-ten theory gives a complete description of the module category $\m \La$, see for example \cite{ARS}. However in general the whole module category of an algebra is very hard to study. In Osamu Iyama's higher-dimensional Aus\-lan\-der--Rei\-ten theory (\cite{IYA2}, \cite{IYA1}) one moves the focus from $\m\La$ to a suitable subcategory $\cC\subseteq \m\La$ satisfying certain homological properties. Such a subcategory $\cC$ is called an \emph{$n$-cluster tilting subcategory} for some positive integer $n$; if moreover $\cC=\add (M)$ for some $M\in \m\La$, then $M$ is called an \emph{$n$-cluster tilting module}. 

An $n$-cluster tilting subcategory $\cC$ is the setting for a higher-dimensional analogue of the classical Aus\-lan\-der--Rei\-ten theory: it admits an $n$-Aus\-lan\-der--Rei\-ten translation, $n$-almost split sequences and an $n$-Aus\-lan\-der--Rei\-ten duality generalising the classical Aus\-lan\-der--Rei\-ten translation, almost split sequences and Aus\-lan\-der--Rei\-ten duality when $n=1$. However, in general it is not easy to find $n$-cluster tilting subcategories. If we set $d:=\gldim \La$, then the existence of an $n$-cluster tilting subcategory for $n > d$ implies that $\La$ is semisimple. Hence we may restrict to the case $n\leq d$. 

The extreme case $n=d$ is of special interest and has been studied extensively before, for example in \cite{IO} and \cite{HI}. If $\cC$ is given by a $d$-cluster tilting module $M$, it follows that $\cC$ is unique and given by 
\[\cC=\add \{\tau_d^{-i}(\La) \mid i\geq 0\},\]
where $\tau_d^-$ denotes the $d$-Aus\-lan\-der--Rei\-ten translation. In this case $\La$ is called \emph{$d$-rep\-re\-sen\-ta\-tion-fi\-nite} (see \cite{HI2}, \cite{IO2}). It is an open question whether all $d$-cluster tilting subcategories are given by $d$-cluster tilting modules. Nevertheless, if we assume the existence of a $d$-cluster tilting module $M$ we can obtain further results about $\cC=\add(M)$. In particular in this case $\cC$ is directed if and only if $\add(\La)$ is directed. Furthermore it is asked in \cite{HIO} if the mere existence of a $d$-cluster tilting module implies that $\add(\La)$ is directed. 
 
Cases where $n < d$ have also been studied before. For the case where $\La$ is selfinjective, and so $d=\infty$, see for example \cite{EH} and \cite{DI}. Note that in this case $\cC$ is never directed. A class of examples satisfying $n\leq d<\infty$ with $d\in n\ZZ$ first appeared in \cite{JAS} and many more were constructed recently in \cite{JK}. To our knowledge, the only known examples where $n \nmid d$ appear in \cite[Theorem 2]{VAS} for $n$ even and $d=n+2k(n-1)+\floor{\frac{2k}{l}}+\ceil{\frac{2k}{l}}$ where $k\in\ZZ_{\geq 1}$, and in \cite[Example 3.7]{VAS} for $n=2$ and $d=3$. 

Recall that an algebra $\La$ is called \emph{rep\-re\-sen\-ta\-tion-di\-rect\-ed} if there exists no sequence of nonzero nonisomorphisms $f_k: M_{k}\rightarrow M_{k+1}$ between indecomposable modules $M_0,\dots, M_t$ with $M_0\cong M_t$. For rep\-re\-sen\-ta\-tion-di\-rect\-ed algebras, a characterization of $n$-cluster tilting subcategories was given in \cite[Theorem 1]{VAS} (see Theorem \ref{thrm:char}). Using this characterization, it is easy to check the existence of $n$-cluster tilting subcategories. Moreover, in this case $\La$ is rep\-re\-sen\-ta\-tion-fi\-nite and so any $n$-cluster tilting subcategory admits an additive generator. Finally, since $\m\La$ is directed, we have that $\cC$ is also directed. As a consequence, it turns out that there is a unique choice for $\cC$. It follows that one of the simplest cases to consider when trying to find $n$-cluster tilting subcategories is that of $\La$ being rep\-re\-sen\-ta\-tion-di\-rect\-ed.

In this paper we address the general question of whether for a pair of positive integers $(n,d)$ with $n<d$ there exists an algebra $\La$ of global dimension $d$, admitting an $n$-cluster tilting subcategory; we call such an algebra \emph{$(n,d)$-rep\-re\-sen\-ta\-tion-fi\-nite}. We show that for $n$ odd and any $d$ we can find an $(n,d)$-rep\-re\-sen\-ta\-tion-fi\-nite algebra. Moreover, for $n$ even and $d$ odd or $d\geq 2n$ we again answer the question affirmatively. 

To construct $(n,d)$-rep\-re\-sen\-ta\-tion-fi\-nite algebras we first introduce the method of \emph{gluing}. Our method takes as input a rep\-re\-sen\-ta\-tion-di\-rect\-ed algebra $A$ with a certain kind of projective module $P$ and a rep\-re\-sen\-ta\-tion-di\-rect\-ed algebra $B$ with a certain kind of injective module $I$ and returns a new rep\-re\-sen\-ta\-tion-di\-rect\-ed algebra $\La:= B \glue[P][I] A$. The representation theory of $\La$ can be completely described in terms of the representation theories of $A$ and $B$. In particular, the Aus\-lan\-der--Rei\-ten quiver $\Gamma(\La)$ of $\La$ is given as the union of the Aus\-lan\-der--Rei\-ten quivers $\Gamma(A)$ and $\Gamma(B)$ of $A$ and $B$, identified over a common piece. In general there may be several choices of $P$ and $I$, but choosing $P$ and $I$ to be simple modules always works. Similar ideas to this method have appeared before in the literature: in \cite{IPTZ}, the authors are interested in the case where the common piece of $\Gamma(A)$ and $\Gamma(B)$ is a point or a triangle mesh. Our construction generalizes some of their results.

If $A$ and $B$ admit $n$-cluster tilting subcategories, in general it is not true that $B \glue[P][I] A$ admits an $n$-cluster tilting subcategory. To this end we modify the characterization of $n$-cluster tilting subcategories given in \cite[Theorem 1]{VAS} and introduce the more general notion of \emph{$n$-fractured subcategories}. We show that under some compatibility conditions gluing of algebras admitting $n$-fractured subcategories gives rise to an algebra admitting an $n$-fractured subcategory. Moreover, by repeating this process sufficiently many times, one can arrive at an algebra which admits an actual $n$-cluster tilting subcategory, as desired.

Let us call an algebra $\La$ \emph{strongly $(n,d)$-rep\-re\-sen\-ta\-tion-di\-rect\-ed} if $\La$ is rep\-re\-sen\-ta\-tion-di\-rect\-ed and $(n,d)$-rep\-re\-sen\-ta\-tion-fi\-nite. As a corollary of our previous results we show that if $A$ is strongly $(n,d_1)$-rep\-re\-sen\-ta\-tion-di\-rect\-ed, $B$ is strongly $(n,d_2)$-rep\-re\-sen\-ta\-tion-di\-rect\-ed, $P$ is a simple projective $A$-module and $I$ is a simple injective $B$-module then $\La= B \glue[P][I] A$ is strongly $(n,d)$-rep\-re\-sen\-ta\-tion-di\-rect\-ed for some $d$. By iterating this result, many new examples can be constructed. Moreover, while the global dimension $d$ of $\La$ in general is difficult to compute, we show that in some simple cases we have $d=d_1+d_2$.

This paper is divided into four parts. In the first part of the paper we introduce some basic notation and give a motivating example in detail. In the second part, given two rep\-re\-sen\-ta\-tion-di\-rect\-ed algebras $A$ and $B$, we describe our method of gluing of algebras and the associated results. In the third part we introduce $n$-fractured subcategories and describe how they are affected by gluing under certain conditions. In the fourth part of this paper we use these constructions to prove our results about the existence of $(n,d)$-rep\-re\-sen\-ta\-tion-fi\-nite algebras. Most results are proved using standard techniques of representation theory: see for example the books \cite{ARS}, \cite{ASS} as well as the survey article \cite{RIN}. Many examples are given throughout. We also include a list of terminology with reference to their definition in the text as well as an index of symbols at the end of this paper.

\section{Part I: Preliminaries}

\subsection{Conventions}\label{subsection:conventions}

Let us introduce some conventions and notation that we use throughout this paper. Let $K$ be an algebraically closed field and $n\geq 1$ an integer. In this paper by an algebra $\La$ we mean a basic fi\-nite-di\-men\-sion\-al unital associative algebra over $K$ and by a $\La$-module we mean a right $\La$-module. We denote the category of right $\La$-modules by $\m\La$. We write $M_{\La}$ for a module $M\in \m\La$ when the algebra is not clear from the context. 

For a quiver $Q$ we denote by $Q_0$ the set of vertices and by $Q_1$ the set of arrows. For an arrow $\alpha\in Q_1$ we denote by $s(\alpha)$ its source and by $t(\alpha)$ its target. We compose arrows in quivers from the left to the right, that is if $\alpha_i\in Q_1$, $1\leq i \leq n$ are arrows in $Q$, then $\alpha_1\alpha_{2}\cdots \alpha_{n-1}\alpha_{n}$ is a path in $Q$ if $s(\alpha_i)=t(\alpha_{i-1})$. 

Throughout we use quivers with relations and their representations; for details we refer to \cite[Chapter III]{ASS}. Contrary to the notation in \cite{ASS}, we use $\cR$ to denote an admissible ideal of the path algebra $KQ$ of a quiver $Q$. If $K Q/\cR$ is a bound quiver algebra, for a vertex $k\in Q_0$, we denote by $P(k)$ (respectively $I(k)$) the corresponding indecomposable projective (respectively injective) $KQ/\cR$-module.

By a subcategory of an additive category we always mean a full subcategory closed under isomorphisms, direct sums and summands unless specified otherwise. 

Now let $\cA_i\subseteq \m\La$ be subcategories and $M_j\in \m\La$ be modules indexed by some $i\in I$ and $j\in J$. We set
\begin{itemize}
    \item $\Ind{\cA_i}$ --- the set of isomorphism classes of indecomposable modules in $\cA_i$,
    \item $\abs{\cA_i}$ --- the cardinality of $\Ind{\cA_i}$,
    \item $\add\{\cA_i\}_{i \in I}$ --- the subcategory of $\m\La$ containing all direct sums of modules $M$ such that $M \in \cA_i$ for some $i\in I$,
    \item $\add(M_i)$ --- the subcategory of $\m\La$ containing all direct sums of direct summands of $M_i$,
    \item $\add\{\cA_i, M_j\}_{i \in I, j \in J}:=\add\{\cA_i, \add(M_i)\}_{i \in I, j \in J}$,
    \item $\Sub(\cA_i)$ --- the subcategory of $\m\La$ containing all submodules of modules in $\cA_i$,
    \item $\Sub(M_j):=\Sub(\add(M_j))$,
    \item $\Fac(\cA_i)$ --- the subcategory of $\m\La$ containing all factor modules of modules in $\cA_i$,
    \item $\Fac(M_j):=\Fac(\add(M_j))$. 
\end{itemize}

For an algebra $\La$, we denote by $D$ the standard duality $D=\Hom_K(-,K)$ between $\m\La$ and $\m\La^{\text{op}}$. By an ideal of $\La$ we mean a two-sided ideal, unless mentioned otherwise. For $X\in \m\La$ we denote by $\om(X)$ the \emph{syzygy} of $X$, that is the kernel of $P\twoheadrightarrow X$, where $P$ is the (unique, up to isomorphism) minimal projective cover of $X$ and by $\om^- (X)$ the \emph{cosyzygy} of $X$, that is the cokernel of $X\hookrightarrow I$ where $I$ is the (unique, up to isomorphism) minimal injective hull of $X$. Note that $\om (X)$ and $\om^- (X)$ are unique up to isomorphism. We denote by $\tau$ and $\tau^-$ the \emph{Aus\-lan\-der--Rei\-ten translations} and by $\Gamma(\La)$ the \emph{Aus\-lan\-der--Rei\-ten quiver} of $\La$. If $M\in\m\La$ is indecomposable, we denote by $[M]$ the corresponding vertex in the Aus\-lan\-der--Rei\-ten quiver $\Gamma(\La)$. For more details on Aus\-lan\-der--Rei\-ten theory we refer to \cite[Chapter IV]{ASS}. Following \cite{IYA1}, we denote by $\tn$ and $\tno$ the \emph{$n$-Aus\-lan\-der--Rei\-ten translations} defined by $\tn (X) = \tau \om^{n-1} (X)$ and $\tno (X) = \tau^- \om^{-(n-1)}(X)$.

Let $\phi:\La \rightarrow \Gamma$ be an algebra homomorphism. We denote by $\phi_{\ast}: \m \Gamma \rightarrow \m \La$ the \emph{restriction of scalars} functor that turns a $\Gamma$-module $M$ into a $\La$-module via $m\cdot \la=m\cdot \phi(\la)$ for $m\in M$ and $\la \in \La$. We denote by $\phi^{\ast}: \m \La \rightarrow \m \Gamma$ the \emph{induced module} functor, given by $\phi^{\ast}(-)=-\otimes_{\La}\Gamma$. Finally, we denote by $\phi^{!}: \m \La \rightarrow \m \Gamma$ the \emph{coinduced module} functor, given by $\phi^{!}(-) = \Hom_{\La}(\Gamma,-)$. Note that $(\phi^{\ast},\phi_{\ast})$ and $(\phi_{\ast}, \phi^{!})$ form adjoint pairs. 

We denote by $A_h$ the quiver 
\[1\overset{\alpha_1}{\longrightarrow} 2 \overset{\alpha_2}{\longrightarrow} 3 \longrightarrow\cdots\longrightarrow h-1\overset{\alpha_{h-1}}{\longrightarrow} h.\]
Let $\La= KA_m/\cR$ where $\cR$ is an admissible ideal. Then $\La$ is called an \emph{acyclic Nakayama algebra} and its representation theory is well known, see for example \cite[Chapter V]{ASS}.  We also introduce some notation from \cite{VAS}. In particular, recall that the isomorphism classes of the indecomposable $\La$-modules can be described by the representations $M(i,j)$ of the form
\begin{equation}\label{eq:indecomposables}
\underset{\scriptscriptstyle 1}{0} \overset{0}{\longrightarrow} \cdots \overset{0}{\longrightarrow} 0 \overset{0}{\longrightarrow} \underset{\scriptscriptstyle m-(i-1)-(j-1)}{K} \overset{1}{\longrightarrow} \cdots \overset{1}{\longrightarrow} \underset{\scriptscriptstyle m-(i-1)}{K} \overset{0}{\longrightarrow} 0 \overset{0}{\longrightarrow} \cdots \overset{0}{\longrightarrow} \underset{\scriptscriptstyle m}{0} 
\end{equation}
with $M(i,j)I=0$ (\cite[Gabriel's Theorem]{ASS}). We use the convention that $M(i,j)=0$ if the coordinates $(i,j)$ do not define a $\La$-module. Furthermore,  when drawing the Aus\-lan\-der--Rei\-ten quiver of $\La$ we simply write $(i,j)$ for the vertex $[M(i,j)]$. With this notation in mind, the Aus\-lan\-der--Rei\-ten quiver $\Gamma(KA_h)=:\triangle(h)$ of $KA_h$ is
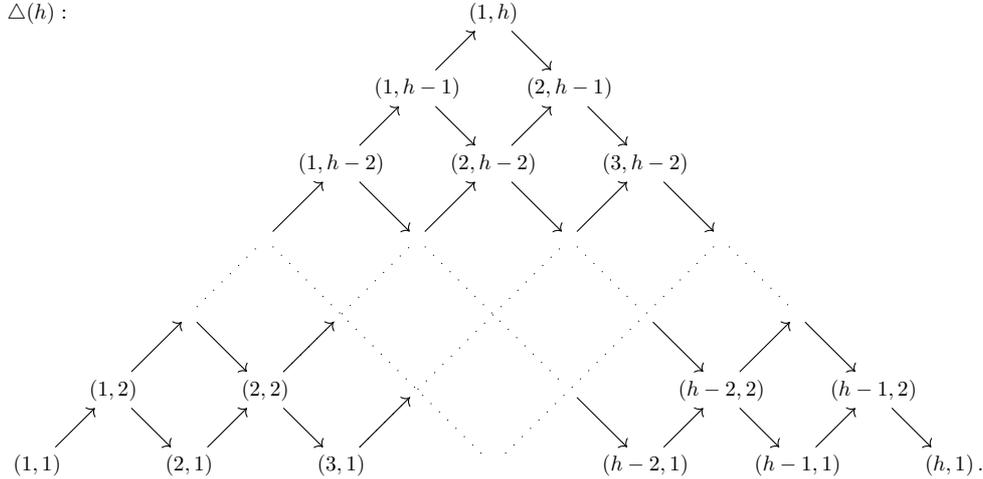
\[\begin{tikzpicture}[scale=0.8, transform shape]
\node (00) at (0,7.5) {$\triangle(h):$};

\node (11) at (0,0) {$(1,1)$};
\node (21) at (2.5,0) {$(2,1)$};
\node (31) at (5,0) {$(3,1)$};
\node (41) at (7.5,0) { };
\node (51) at (10,0) {$(h-2,1)$};
\node (61) at (12.5,0) {$(h-1,1)$};
\node (71) at (15,0) {$(h,1)$\nospacepunct{.}};

\node (12) at (1.25,1.25) {$(1,2)$};
\node (22) at (3.75,1.25) {$(2,2)$};
\node (32) at (6.25,1.25) { };
\node (42) at (8.75,1.25) { };
\node (52) at (11.25,1.25) {$(h-2,2)$};
\node (62) at (13.75,1.25) {$(h-1,2)$};

\node (13) at (2.5,2.5) { };
\node (23) at (5,2.5) { };
\node (33) at (7.5,2.5) { };
\node (43) at (10,2.5) { };
\node (53) at (12.5,2.5) { };

\node (14) at (3.75,3.75) { };
\node (24) at (6.25,3.75) { };
\node (34) at (8.75,3.75) { };
\node (44) at (11.25,3.75) { };

\node (15) at (5,5) {$(1,h-2)$};
\node (25) at (7.5,5) {$(2,h-2)$};
\node (35) at (10,5) {$(3,h-2)$};

\node (16) at (6.25,6.25) {$(1,h-1)$};
\node (26) at (8.75,6.25) {$(2,h-1)$};

\node (17) at (7.5,7.5) {$(1,h)$};

\draw[->] (11) -- (12);
\draw[->] (12) -- (21);
\draw[->] (21) -- (22);
\draw[->] (22) -- (31);

\draw[->] (51) -- (52);
\draw[->] (52) -- (61);
\draw[->] (61) -- (62);
\draw[->] (62) -- (71);

\draw[->] (12) -- (13);
\draw[->] (13) -- (22);
\draw[->] (22) -- (23);

\draw[->] (14) -- (15);
\draw[->] (15) -- (16);
\draw[->] (16) -- (17);

\draw[->] (15) -- (24);
\draw[->] (24) -- (25);
\draw[->] (25) -- (34);
\draw[->] (34) -- (35);
\draw[->] (35) -- (44);

\draw[->] (16) -- (25);
\draw[->] (17) -- (26);
\draw[->] (26) -- (35);
\draw[->] (25) -- (26);

\draw[->] (31) -- (32);
\draw[->] (42) -- (51);
\draw[->] (43) -- (52);
\draw[->] (53) -- (62);

\draw[->] (52) -- (53);

\draw[loosely dotted] (44) -- (41);
\draw[loosely dotted] (14) -- (41);

\draw[loosely dotted] (34) -- (43);
\draw[loosely dotted] (44) -- (53);

\draw[loosely dotted] (24) -- (42);
\draw[loosely dotted] (13) -- (14);
\draw[loosely dotted] (22) -- (24);
\draw[loosely dotted] (31) -- (34);
\end{tikzpicture}\]
The shape of the above quiver will appear throughout our investigation.

For $m\geq h$ we further set $\La_{m,h}:=KA_m/\rad(KA_m)^h$. In particular, for $\La_{m,h}$-modules we have $M(i,j)\neq 0$ if and only if $1\leq i \leq m$, $1\leq j \leq h$ and $2\leq i+j\leq m+1$. With this notation, we also have $KA_h\cong \La_{h,h}$.

Following \cite{MIL}, given two algebra homomorphisms $f:A\rightarrow C$ and $g:B\rightarrow C$, we define the \emph{pullback algebra $(\La,\phi,\psi)$ of $A\overset{f}{\longrightarrow} C \overset{g}{\longleftarrow} B$} to be the subalgebra 
\[\La = \{(a,b)\in A\times B \mid f(a)=g(b)\}\subseteq A\times B,\]
with $\phi:\La\rightarrow A$ and $\psi:\La\rightarrow B$ being induced by the natural projections. When clear from context, we will identify the pullback $(\La,\phi,\psi)$ with the underlying algebra $\La$. Notice that whenever we have $f(a)=g(b)$ for some $a\in A$ and $b\in B$, then there exists a unique $\lambda\in\La$ such that $\phi(\lambda)=a$ and $\psi(\lambda)=b$. It turns out that this is the pullback in the category of $K$-algebras. Note that if $g$ is a surjection, then so is $\phi$ (but the converse is not true). In this case, the diagram 
\[\begin{tikzcd}
\La \arrow[r, "\psi"] \arrow[d, swap, "\phi"] 
& B \arrow[d, "g"] \\
A \arrow[r, swap, "f"] & C
\end{tikzcd}\]
is called a \emph{Milnor square of algebras} (see \cite{MIL}).

\subsection{A motivating example}

Let us first give a motivating example that illustrates the theory that is developed in this paper. 

\begin{example}\label{ex:motivation}
Let $B=\La_{9,4}$. Since $B$ is rep\-re\-sen\-ta\-tion-di\-rect\-ed, the only possible candidate for a $2$-cluster tilting subcategory is
\[\cC_B = \add\left( \bigoplus_{i\geq 0} \tau_2^{-i} (\La) \right).\]
Let us draw the Aus\-lan\-der--Rei\-ten quiver $\Gamma(B)$ of $B$ and encircle the vertices corresponding to indecomposable modules in $\cC_B$:
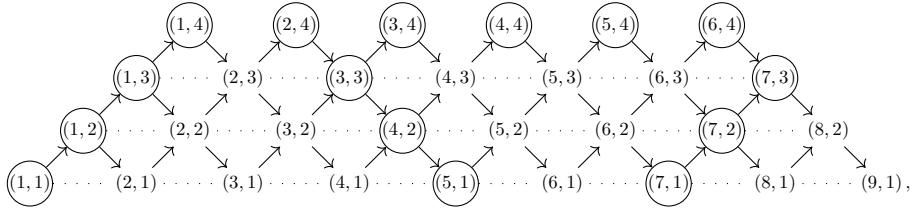
\[\begin{tikzpicture}[scale=0.7, transform shape]
\tikzstyle{nct2}=[circle, minimum width=0.6cm, draw, inner sep=0pt, text centered]

\node[nct2] (11) at (1,1) {$(1,1)$};
\node (21) at (3,1) {$(2,1)$};
\node (31) at (5,1) {$(3,1)$};
\node (41) at (7,1) {$(4,1)$};
\node[nct2] (51) at (9,1) {$(5,1)$};
\node (61) at (11,1) {$(6,1)$};
\node[nct2] (71) at (13,1) {$(7,1)$};
\node (81) at (15,1) {$(8,1)$};
\node (91) at (17,1) {$(9,1)$\nospacepunct{,}};

\node[nct2] (12) at (2,2) {$(1,2)$};
\node (22) at (4,2) {$(2,2)$};
\node (32) at (6,2) {$(3,2)$};
\node[nct2] (42) at (8,2) {$(4,2)$};
\node (52) at (10,2) {$(5,2)$};
\node (62) at (12,2) {$(6,2)$};
\node[nct2] (72) at (14,2) {$(7,2)$};
\node (82) at (16,2) {$(8,2)$};

\node[nct2] (13) at (3,3) {$(1,3)$};
\node (23) at (5,3) {$(2,3)$};
\node[nct2] (33) at (7,3) {$(3,3)$};
\node (43) at (9,3) {$(4,3)$};
\node (53) at (11,3) {$(5,3)$};
\node (63) at (13,3) {$(6,3)$};
\node[nct2] (73) at (15,3) {$(7,3)$};

\node[nct2] (14) at (4,4) {$(1,4)$};
\node[nct2] (24) at (6,4) {$(2,4)$};
\node[nct2] (34) at (8,4) {$(3,4)$};
\node[nct2] (44) at (10,4) {$(4,4)$};
\node[nct2] (54) at (12,4) {$(5,4)$};
\node[nct2] (64) at (14,4) {$(6,4)$};

\draw[->] (11) -- (12);
\draw[->] (21) -- (22);
\draw[->] (31) -- (32);
\draw[->] (41) -- (42);
\draw[->] (51) -- (52);
\draw[->] (61) -- (62);
\draw[->] (71) -- (72);
\draw[->] (81) -- (82);

\draw[->] (12) -- (13);
\draw[->] (22) -- (23);
\draw[->] (32) -- (33);
\draw[->] (42) -- (43);
\draw[->] (52) -- (53);
\draw[->] (62) -- (63);
\draw[->] (72) -- (73);

\draw[->] (13) -- (14);
\draw[->] (23) -- (24);
\draw[->] (33) -- (34);
\draw[->] (43) -- (44);
\draw[->] (53) -- (54);
\draw[->] (63) -- (64);

\draw[->] (12) -- (21);
\draw[->] (22) -- (31);
\draw[->] (32) -- (41);
\draw[->] (42) -- (51);
\draw[->] (52) -- (61);
\draw[->] (62) -- (71);
\draw[->] (72) -- (81);
\draw[->] (82) -- (91);

\draw[->] (13) -- (22);
\draw[->] (23) -- (32);
\draw[->] (33) -- (42);
\draw[->] (43) -- (52);
\draw[->] (53) -- (62);
\draw[->] (63) -- (72);
\draw[->] (73) -- (82);

\draw[->] (14) -- (23);
\draw[->] (24) -- (33);
\draw[->] (34) -- (43);
\draw[->] (44) -- (53);
\draw[->] (54) -- (63);
\draw[->] (64) -- (73);

\draw[loosely dotted] (11.east) -- (21);
\draw[loosely dotted] (21.east) -- (31);
\draw[loosely dotted] (31.east) -- (41);
\draw[loosely dotted] (41.east) -- (51);
\draw[loosely dotted] (51.east) -- (61);
\draw[loosely dotted] (61.east) -- (71);
\draw[loosely dotted] (71.east) -- (81);
\draw[loosely dotted] (81.east) -- (91);

\draw[loosely dotted] (12.east) -- (22);
\draw[loosely dotted] (22.east) -- (32);
\draw[loosely dotted] (32.east) -- (42);
\draw[loosely dotted] (42.east) -- (52);
\draw[loosely dotted] (52.east) -- (62);
\draw[loosely dotted] (62.east) -- (72);
\draw[loosely dotted] (72.east) -- (82);

\draw[loosely dotted] (13.east) -- (23);
\draw[loosely dotted] (23.east) -- (33);
\draw[loosely dotted] (33.east) -- (43);
\draw[loosely dotted] (43.east) -- (53);
\draw[loosely dotted] (53.east) -- (63);
\draw[loosely dotted] (63.east) -- (73);
\end{tikzpicture}\]
where the dotted lines denote the Aus\-lan\-der--Rei\-ten translations. A necessary condition for $\cC_B$ to be a $2$-cluster tilting subcategory is that applying $\tau_2^{-}$ to a noninjective indecomposable module in $\cC_B$ should return a nonprojective indecomposable module in $\cC_B$. Since we have $\tau_2^{-}(M(7,1))=0$ and  $\tau_2^{-}(M(7,2)) =0$, with $M(7,1)$ and $M(7,2)$ both being noninjective, we conclude that $\cC_B$ is not a $2$-cluster tilting subcategory. Let us denote the full subquiver of $\Gamma(B)$ containing the vertices $\{(7,1), (7,2), (7,3), (8,1), (8,2), (9,1)\}$ by $\supn{\triangle}{(7,3)}$. Notice that as quivers we have $\supn{\triangle}{(7,3)} \cong \triangle(3)$.

Next let $A=\La_{6,5}$ and let $\cC_A=\add\left( \bigoplus_{i\geq 1}\tau_2^{-i} (A) \right)$. As before we draw the Aus\-lan\-der--Rei\-ten quiver $\Gamma(A)$ of $A$ and encircle the indecomposable modules in $\cC_A$:
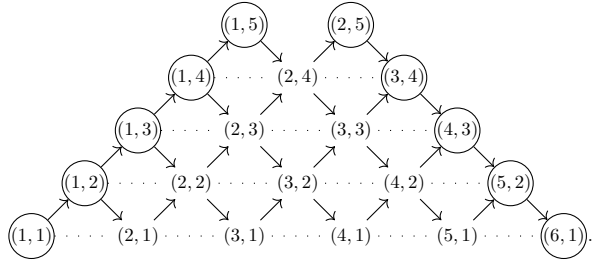
\[\begin{tikzpicture}[scale=0.7, transform shape]
\tikzstyle{nct2}=[circle, minimum width=0.6cm, draw, inner sep=0pt, text centered]

\node[nct2] (11) at (1,1) {$(1,1)$};
\node (21) at (3,1) {$(2,1)$};
\node (31) at (5,1) {$(3,1)$};
\node (41) at (7,1) {$(4,1)$};
\node (51) at (9,1) {$(5,1)$};
\node[nct2] (61) at (11,1) {$(6,1)$\nospacepunct{.}};

\node[nct2] (12) at (2,2) {$(1,2)$};
\node (22) at (4,2) {$(2,2)$};
\node (32) at (6,2) {$(3,2)$};
\node (42) at (8,2) {$(4,2)$};
\node[nct2] (52) at (10,2) {$(5,2)$};

\node[nct2] (13) at (3,3) {$(1,3)$};
\node (23) at (5,3) {$(2,3)$};
\node (33) at (7,3) {$(3,3)$};
\node[nct2] (43) at (9,3) {$(4,3)$};

\node[nct2] (14) at (4,4) {$(1,4)$};
\node (24) at (6,4) {$(2,4)$};
\node[nct2] (34) at (8,4) {$(3,4)$};

\node[nct2] (15) at (5,5) {$(1,5)$};
\node[nct2] (25) at (7,5) {$(2,5)$};

\draw[->] (11) -- (12);
\draw[->] (21) -- (22);
\draw[->] (31) -- (32);
\draw[->] (41) -- (42);
\draw[->] (51) -- (52);

\draw[->] (12) -- (13);
\draw[->] (22) -- (23);
\draw[->] (32) -- (33);
\draw[->] (42) -- (43);

\draw[->] (13) -- (14);
\draw[->] (23) -- (24);
\draw[->] (33) -- (34);

\draw[->] (14) -- (15);
\draw[->] (24) -- (25);

\draw[->] (12) -- (21);
\draw[->] (22) -- (31);
\draw[->] (32) -- (41);
\draw[->] (42) -- (51);
\draw[->] (52) -- (61);

\draw[->] (13) -- (22);
\draw[->] (23) -- (32);
\draw[->] (33) -- (42);
\draw[->] (43) -- (52);

\draw[->] (14) -- (23);
\draw[->] (24) -- (33);
\draw[->] (34) -- (43);

\draw[->] (15) -- (24);
\draw[->] (25) -- (34);

\draw[loosely dotted] (11.east) -- (21);
\draw[loosely dotted] (21.east) -- (31);
\draw[loosely dotted] (31.east) -- (41);
\draw[loosely dotted] (41.east) -- (51);
\draw[loosely dotted] (51.east) -- (61);

\draw[loosely dotted] (12.east) -- (22);
\draw[loosely dotted] (22.east) -- (32);
\draw[loosely dotted] (32.east) -- (42);
\draw[loosely dotted] (42.east) -- (52);

\draw[loosely dotted] (13.east) -- (23);
\draw[loosely dotted] (23.east) -- (33);
\draw[loosely dotted] (33.east) -- (43);

\draw[loosely dotted] (14.east) -- (24);
\draw[loosely dotted] (24.east) -- (34);
\end{tikzpicture}\]
In this case $\cC_A$ is a $2$-cluster tilting subcategory. Let us denote the full subquiver of $\Gamma(A)$ containing the vertices $\{(1,1),(1,2),(1,3),(2,1),(2,2),(3,1)\}$ by $\nsup{(1,3)}{\triangle}$. Notice that again we have $\nsup{(1,3)}{\triangle}\cong \triangle(3)$.

Let us now consider the algebra $\La$ given by the quiver with relations
\[\begin{tikzpicture}[scale=0.9, transform shape]
\node (1) at (1,0) {$1$};
\node (2) at (2,0) {$2$};
\node (3) at (3,0) {$3$};
\node (4) at (4,0) {$4$};
\node (5) at (5,0) {$5$};
\node (6) at (6,0) {$6$};
\node (7) at (7,0) {$7$};
\node (8) at (8,0) {$8$};
\node (9) at (9,0) {$9$};
\node (10) at (10,0) {$10$};
\node (11) at (11,0) {$11$};
\node (12) at (12,0) {$12$,};

\draw[->] (1) to (2);
\draw[->] (2) to (3);
\draw[->] (3) to (4);
\draw[->] (4) to (5);
\draw[->] (5) to (6);
\draw[->] (6) to (7);
\draw[->] (7) to (8);
\draw[->] (8) to (9);
\draw[->] (9) to (10);
\draw[->] (10) to (11);
\draw[->] (11) to (12);

\draw[dotted] (1) to [out=30,in=150] (6);
\draw[dotted] (4) to [out=30,in=150] (8);
\draw[dotted] (5) to [out=30,in=150] (9);
\draw[dotted] (6) to [out=30,in=150] (10);
\draw[dotted] (7) to [out=30,in=150] (11);
\draw[dotted] (8) to [out=30,in=150] (12);

\draw[dotted] (3) to [out=-30,in=-150] (7);

\end{tikzpicture}\]
where the dotted lines indicate zero relations. Then $\La$ can be seen as a certain pullback diagram $A\overset{f}{\longrightarrow} KA_3 \overset{g}{\longleftarrow} B$. Let $\cC_{\La}=\add\left( \bigoplus_{i\geq 1}\tau_2^{-i} (\La) \right)$. As before we draw the Aus\-lan\-der--Rei\-ten quiver $\Gamma(\La)$ of $\La$ and encircle the indecomposable modules in $\cC_{\La}$:
\[\resizebox{\textwidth}{!}{\begin{tikzpicture}
\tikzstyle{nct2}=[circle, minimum width=0.6cm, draw, inner sep=0pt, text centered]

\node[nct2] (11) at (1,1) {$(1,1)$};
\node (21) at (3,1) {$(2,1)$};
\node (31) at (5,1) {$(3,1)$};
\node (41) at (7,1) {$(4,1)$};
\node[nct2] (51) at (9,1) {$(5,1)$};
\node (61) at (11,1) {$(6,1)$};
\node[nct2] (71) at (13,1) {$(7,1)$};
\node (81) at (15,1) {$(8,1)$};
\node (91) at (17,1) {$(9,1)$};
\node (101) at (19,1) {$(10,1)$};
\node (111) at (21,1) {$(11,1)$};
\node[nct2] (121) at (23,1) {$(12,1)$\nospacepunct{.}};

\node[nct2] (12) at (2,2) {$(1,2)$};
\node (22) at (4,2) {$(2,2)$};
\node (32) at (6,2) {$(3,2)$};
\node[nct2] (42) at (8,2) {$(4,2)$};
\node (52) at (10,2) {$(5,2)$};
\node (62) at (12,2) {$(6,2)$};
\node[nct2] (72) at (14,2) {$(7,2)$};
\node (82) at (16,2) {$(8,2)$};
\node (92) at (18,2) {$(9,2)$};
\node (102) at (20,2) {$(10,2)$};
\node[nct2] (112) at (22,2) {$(11,2)$};

\node[nct2] (13) at (3,3) {$(1,3)$};
\node (23) at (5,3) {$(2,3)$};
\node[nct2] (33) at (7,3) {$(3,3)$};
\node (43) at (9,3) {$(4,3)$};
\node (53) at (11,3) {$(5,3)$};
\node (63) at (13,3) {$(6,3)$};
\node[nct2] (73) at (15,3) {$(7,3)$};
\node (83) at (17,3) {$(8,3)$};
\node (93) at (19,3) {$(9,3)$};
\node[nct2] (103) at (21,3) {$(10,3)$};

\node[nct2] (14) at (4,4) {$(1,4)$};
\node[nct2] (24) at (6,4) {$(2,4)$};
\node[nct2] (34) at (8,4) {$(3,4)$};
\node[nct2] (44) at (10,4) {$(4,4)$};
\node[nct2] (54) at (12,4) {$(5,4)$};
\node[nct2] (64) at (14,4) {$(6,4)$};
\node[nct2] (74) at (16,4) {$(7,4)$};
\node (84) at (18,4) {$(8,4)$};
\node[nct2] (94) at (20,4) {$(9,4)$};

\node[nct2] (75) at (17,5) {$(7,5)$};
\node[nct2] (85) at (19,5) {$(8,5)$};

\draw[->] (11) -- (12);
\draw[->] (21) -- (22);
\draw[->] (31) -- (32);
\draw[->] (41) -- (42);
\draw[->] (51) -- (52);
\draw[->] (61) -- (62);
\draw[->] (71) -- (72);
\draw[->] (81) -- (82);
\draw[->] (91) -- (92);
\draw[->] (101) -- (102);
\draw[->] (111) -- (112);

\draw[->] (12) -- (13);
\draw[->] (22) -- (23);
\draw[->] (32) -- (33);
\draw[->] (42) -- (43);
\draw[->] (52) -- (53);
\draw[->] (62) -- (63);
\draw[->] (72) -- (73);
\draw[->] (82) -- (83);
\draw[->] (92) -- (93);
\draw[->] (102) -- (103);

\draw[->] (13) -- (14);
\draw[->] (23) -- (24);
\draw[->] (33) -- (34);
\draw[->] (43) -- (44);
\draw[->] (53) -- (54);
\draw[->] (63) -- (64);
\draw[->] (73) -- (74);
\draw[->] (83) -- (84);
\draw[->] (93) -- (94);

\draw[->] (74) -- (75);
\draw[->] (84) -- (85);

\draw[->] (12) -- (21);
\draw[->] (22) -- (31);
\draw[->] (32) -- (41);
\draw[->] (42) -- (51);
\draw[->] (52) -- (61);
\draw[->] (62) -- (71);
\draw[->] (72) -- (81);
\draw[->] (82) -- (91);
\draw[->] (92) -- (101);
\draw[->] (102) -- (111);
\draw[->] (112) -- (121);

\draw[->] (13) -- (22);
\draw[->] (23) -- (32);
\draw[->] (33) -- (42);
\draw[->] (43) -- (52);
\draw[->] (53) -- (62);
\draw[->] (63) -- (72);
\draw[->] (73) -- (82);
\draw[->] (83) -- (92);
\draw[->] (93) -- (102);
\draw[->] (103) -- (112);

\draw[->] (14) -- (23);
\draw[->] (24) -- (33);
\draw[->] (34) -- (43);
\draw[->] (44) -- (53);
\draw[->] (54) -- (63);
\draw[->] (64) -- (73);
\draw[->] (74) -- (83);
\draw[->] (84) -- (93);
\draw[->] (94) -- (103);

\draw[->] (75) -- (84);
\draw[->] (85) -- (94);

\draw[loosely dotted] (11.east) -- (21);
\draw[loosely dotted] (21.east) -- (31);
\draw[loosely dotted] (31.east) -- (41);
\draw[loosely dotted] (41.east) -- (51);
\draw[loosely dotted] (51.east) -- (61);
\draw[loosely dotted] (61.east) -- (71);
\draw[loosely dotted] (71.east) -- (81);
\draw[loosely dotted] (81.east) -- (91);
\draw[loosely dotted] (91.east) -- (101);
\draw[loosely dotted] (101.east) -- (111);
\draw[loosely dotted] (111.east) -- (121);

\draw[loosely dotted] (12.east) -- (22);
\draw[loosely dotted] (22.east) -- (32);
\draw[loosely dotted] (32.east) -- (42);
\draw[loosely dotted] (42.east) -- (52);
\draw[loosely dotted] (52.east) -- (62);
\draw[loosely dotted] (62.east) -- (72);
\draw[loosely dotted] (72.east) -- (82);
\draw[loosely dotted] (82.east) -- (92);
\draw[loosely dotted] (92.east) -- (102);
\draw[loosely dotted] (102.east) -- (112);

\draw[loosely dotted] (13.east) -- (23);
\draw[loosely dotted] (23.east) -- (33);
\draw[loosely dotted] (33.east) -- (43);
\draw[loosely dotted] (43.east) -- (53);
\draw[loosely dotted] (53.east) -- (63);
\draw[loosely dotted] (63.east) -- (73);
\draw[loosely dotted] (73.east) -- (83);
\draw[loosely dotted] (83.east) -- (93);
\draw[loosely dotted] (93.east) -- (103);

\draw[loosely dotted] (74.east) -- (84);
\draw[loosely dotted] (84.east) -- (94);
\end{tikzpicture}}\]

Notice that we can identify $\Gamma(\La)$ with the amalgamated sum $\Gamma(B)\coprod_{\triangle} \Gamma(A)$, under the identification $\supn{\triangle}{(7,3)} \equiv \triangle(3) \equiv \nsup{(1,3)}{\triangle}$. Under this identification we also see that much of the representation theory of $\La$ is given by the representation theory of $B$ and $A$: indecomposable $\La$-modules correspond to indecomposable $B$-modules or to indecomposable $A$-modules, almost split sequences in $\m\La$ correspond to almost split sequences either in $\m B$ or in $\m A$ and similarly for syzygies and cosyzygies of indecomposable $\La$-modules. Moreover $\cC_{\La}$ is the additive closure of $\cC_A$ and $\cC_B$ viewed inside $\m\La$. Notice that the indecomposable modules in $\cC_A$ and $\cC_B$ corresponding to the identified part match. In this case $\cC_{\La}$ turns out to be a $2$-cluster tilting subcategory. In particular, in $\m\La$ we have $\tau_2^{-}(M(7,1))\cong M(9,4)$ and $\tau_2^{-}\left(M(7,2)\right)\cong M(10,3)$, since these functors can be computed in the subquiver corresponding to $\m A$. 
\end{example}

In Example \ref{ex:motivation} we managed to get a $2$-cluster tilting subcategory by identifying the ``problematic'' piece $\supn{\triangle}{(7,3)}$ of $\Gamma(B)$ with the ``well-behaved'' piece $\nsup{(1,3)}{\triangle}$ of $\Gamma(A)$. In this paper we explain how this process can be defined rigorously and under which conditions it can be applied.

\section{Part II: Gluing}

\subsection{Glued subcategories}

Let us first recall some definitions from \cite{AS2}. Let $\La$ be an algebra and $\cA\subseteq \cL\subseteq \m \La$ be subcategories. Recall that a morphism $g:M\rightarrow N$ in $\cA$ is called \emph{right almost split} if $g$ is not a retraction and any non-retraction $v:V\rightarrow N$ with $V\in\cA$ factors through $g$. Dually we can define \emph{left almost split} morphisms. A short exact sequence $0\rightarrow L \overset{f}{\longrightarrow} M \overset{g}{\longrightarrow} N\rightarrow 0$ in $\cA$ is called an \emph{almost split sequence} if $L$ is indecomposable and $g$ is right almost split or equivalently if $N$ is indecomposable and $f$ is left almost split. 

A module $P$ in $\cA$ is called \emph{$\cA$-projective} if $\Ext^i_{\La}(P,\cA)=0$ for all $i>0$ and a module $I$ in $\cA$ is called \emph{$\cA$-injective} if $\Ext^i_{\La}(\cA,I)=0$ for all $i>0$. We say that \emph{$\cA$ has almost split sequences} if for any non-$\cA$-projective indecomposable module $N\in \cA$ there is an almost split sequence in $\cA$ ending at $N$ and for any non-$\cA$-injective indecomposable module $L\in\cA$ there is an almost split sequence starting at $L$.

Next we introduce the notion of gluing of subcategories. 
 
\begin{definition}\label{def:catglue}
Assume that there exist subcategories $\cA$ and $\cB$ of $\cL$ such that the following are satisfied.
\begin{itemize}
\item[(i)] $\cL=\add\{\cA,\cB\}$,
\item[(ii)] $\cA$ and $\cB$ have almost split sequences,
\item[(iii)] If $M\in \cA\setminus \cB$ and $M$ is indecomposable, then $\Hom_{\La}(M,\cB)=0$,
\item[(iv)] If $N\in \cB$ and $M\in \cA$, then for all $g:N\rightarrow M$, there exists an $X\in \cA\cap \cB$ such that $g=g_1\circ g_2$ for some $g_1:X\rightarrow N$ and $g_2:M\rightarrow X$.
\end{itemize}
In that case $\cL$ is called \emph{the gluing of $\cB$ and $\cA$} and we write $\cL=\cB \glue \cA$. Note that gluing is not a commutative operation.
\end{definition}

\begin{theorem}\label{thrm:general glue}
Assume that $\cL=\cB \glue \cA$. 
\begin{itemize}
\item[(i)] If $N\in\cA$ and $0\rightarrow L \overset{f}{\longrightarrow} M \overset{g}{\longrightarrow} N \rightarrow 0$ is an almost split sequence in $\cA$, then it is also an almost split sequence in $\cL$.
\item[(ii)] If $N\in \cB$ and $0\rightarrow L \overset{f}{\longrightarrow} M \overset{g}{\longrightarrow} N \rightarrow 0$ is an almost split sequence in $\cB$, then it is also an almost split sequence in $\cL$.
\end{itemize}
\end{theorem}

\begin{proof}
\begin{itemize}
\item[(i)] Since $L$ and $N$ are indecomposable in $\cA$, they are also indecomposable in $\cL$. Hence it is enough to show that $g$ is right almost split in $\cL$. Clearly $g$ is not a retraction in $\cL$. Let $v:V\rightarrow N$ be a morphism in $\cL$ which is not a retraction and without loss of generality assume that $V$ is indecomposable. By Definition \ref{def:catglue}(ii), we have that $V\in \cA$ or $V\in \cB$. If $V\in \cA$, then $v$ factors through $g$ because $g$ is right almost split in $\cA$. If $V\in \cB\setminus \cA$, then by Definition \ref{def:catglue}(iv) there exists some $X\in \cA\cap \cB$ such that $v=v_1\circ v_2$ with $v_1:X\rightarrow N$ and $v_2:V\rightarrow X$. Note that $v_1$ is not a retraction since $\Hom_{\La}(N,X)=0$ by Definition \ref{def:catglue}(iii). Hence there exists $h:X\rightarrow M$ such that $v_1=g\circ h$. But then $v=g\circ (h \circ v_2)$ as required.

\item[(ii)] Similarly to (i), let $v:V\rightarrow N$ be a nonzero non-retraction in $\cL$ with $V$ indecomposable. Then $V\in \cA\setminus \cB$ or $V\in\cB$. Since $\Hom_{\La}(V,\cB)\neq 0$, we have that $V\not\in \cA\setminus \cB$ by Definition \ref{def:catglue}(iii). Hence $V\in \cB$ and $v$ factors through $g$ since $g$ is right almost split in $\cB$.\qedhere
\end{itemize}
\end{proof}

\subsection{Glued rep\-re\-sen\-ta\-tion-di\-rect\-ed algebras} 

Throughout this subsection we only consider rep\-re\-sen\-ta\-tion-di\-rect\-ed algebras. Our aim is to construct algebras $A$, $B$ and $\La$ such that $\m \La = \m B \glue \m A$. Our construction gives $\La$ as a certain pullback of $A$ and $B$ over $KA_h$ for a specific $h$, called \emph{gluing}. This gluing is based on the existence of certain modules which we call \emph{left abutments} and \emph{right abutments}. We then show how we can describe completely the representation theory of $\La$ using the representation theories of $A$ and $B$.

Pullbacks similar to gluing have been considered before in \cite[Section 3]{IPTZ}. In particular, the authors of that paper assume the existence of such a pullback and are interested in the cases $h=1$ and $h=2$. After finishing this paper, it has come to the author's attention that such pullbacks have been also considered in \cite{LEV,BCW}, citing the work of \cite{IPTZ}. To make this article self-contained as well as to establish notation, we include full proofs of some results which can also be found in \cite{IPTZ,LEV,BCW}. Moreover, we include many additional properties that we will use later.

\subsubsection{Abutments}\label{sect:abutments}

If $M$ is a $\La$-module, then $M$ is said to be \emph{uniserial} if it has a unique composition series. In this case $M$ has simple top and socle and hence $M$ is indecomposable. Being uniserial is equivalent to the \emph{radical series} 
\[0\subseteq \dots \subseteq \rad^2(M)\subseteq \rad(M)\subseteq M \]
being a composition series of $M$ \cite[Lemma V.2.2.]{ASS}. Clearly submodules of uniserial modules are uniserial. 

\begin{definition}\label{def:abutment}
Let $\La$ be a rep\-re\-sen\-ta\-tion-di\-rect\-ed algebra. We call a uniserial projective $\La$-module $P$ a \emph{left abutment} if every submodule of $P$ is projective and for any indecomposable projective $\La$-module $P'$ not isomorphic to a submodule of $P$, we have that all morphisms $U \rightarrow P'$ with $U\subseteq P$ factor through $P$. 

We call an indecomposable injective $\La$-module $I$ a \emph{right abutment} if $D(I)$ is a left abutment as a $\La^{\text{op}}$-module.
\end{definition}

Let $P$ be a left abutment with composition series
\[0\subseteq P_h \subseteq \cdots \subseteq P_2 \subseteq P_1=P.\]
Then the modules $P_i$ are also uniserial and so indecomposable. Hence there exist primitive orthogonal idempotents $e_1,\dots,e_h$ such that $P_i\cong e_i\La$ and hence the composition series of $P$ corresponds to a diagram
\[0\overset{f_h}{\longrightarrow} e_h\La \overset{f_{h-1}}{\longrightarrow} \cdots \overset{f_2}{\longrightarrow} e_2\La \overset{f_1}{\longrightarrow} e_1\La,\]
where $f_i\in \Hom_{\La}(e_{i+1}\La,e_i\La)=e_{i}\La e_{i+1}$. We call such a choice of $(e_i,f_i)_{i=1}^h$ a \emph{realization of the left abutment $P$} and we denote $e\ab=\sum_{i=1}^he_i$. Note that $h$ is the length $l(P)$ of $P$ and that $f_h=0$. We will call $h$ the \emph{height of the left abutment $P$}. 

For a right abutment $I$ such that $D(I)$ has a realization $(e_{h-i+1},f_{h-i})_{i=1}^h$, we call $(e_i, D(f_{i-1}))_{i=1}^h=:(e_i,g_{i-1})_{i=1}^h$ a \emph{realization of the right abutment $I$} and $h$ the \emph{height of the right abutment $I$}. Diagrammatically, we have a sequence of factor modules
\[D(\La e_h)\overset{g_{h-1}}{\longrightarrow}D(\La e_{h-1}) \overset{g_{h-2}}{\longrightarrow}\cdots   \overset{g_1}{\longrightarrow}D(\La e_1) \overset{g_0}{\longrightarrow}0,\]
where $g_0=0$.

Note that simple projective modules are the same as left abutments of height $1$ and simple injective modules are the same as right abutments of height $1$. Note also that since $\La$ is rep\-re\-sen\-ta\-tion-di\-rect\-ed, there exists at least one simple projective module and one simple injective module.

The following lemma will be used to characterize algebras admitting abutments in terms of a quiver with relations.

\begin{lemma}\label{lemma:homsofabutments}
Let $\La$ be a rep\-re\-sen\-ta\-tion-di\-rect\-ed algebra. 
\begin{itemize}
\item[(a)] If $P$ is a left abutment realized by $(e_i,f_i)_{i=1}^h$, then 
\[\dim_K\left(\Hom_\La(e_i\La,e_j\La)\right)=\begin{cases} 1 &\mbox{ if $i\geq j$,} \\ 0 &\mbox{ if $i<j$.}\end{cases}\]
In particular, $\left\{f_{i}\circ\cdots\circ f_{i+k}\right\}$ is a basis of $\Hom_\La(e_{i+k+1}\La, e_i\La)$.
\item[(b)] If $I$ is a right abutment realized by $(e_i,g_{i-1})_{i=1}^h$, then 
\[\dim_K\left(\Hom_\La(D(\La e_i),D(\La e_j))\right)=\begin{cases} 1 &\mbox{ if $i\geq j$,} \\ 0 &\mbox{ if $i< j$.}\end{cases}\]
In particular, $\left\{g_{i}\circ\cdots\circ g_{i+k}\right\}$ is a basis of $\Hom_\La(D(\La e_{i+k+1}), D(\La e_{i}))$.
\end{itemize}
\end{lemma}

\begin{proof}
We only prove (a); then (b) follows from the definition and (a). If $i=j$ then by \cite[Proposition IX.1.4]{ASS} we have $\End_\La(e_i\La, e_i\La)\cong K$. Notice that by definition, if $i>j$, we have $\Hom_\La(e_i\La, e_j\La)\neq 0$. Since $\La$ is rep\-re\-sen\-ta\-tion-di\-rect\-ed, it follows that for $i<j$ we have $\Hom_\La(e_i\La,e_j\La)=0$. It remains to show that for $i>j$, we have $\Hom_\La(e_i\La, e_j\La)\cong K$. Since the morphism $f_i: e_{i+1}\La\rightarrow e_i\La$ corresponds to the radical inclusion $\rad(e_i\La) \subseteq e_i\La$, it follows that any homomorphism $g_i:e_{i+1}\La \rightarrow e_i\La$ factors through $f_i$. Since $\End_\La(e_i\La,e_i\La)\cong K$, it follows that $g_i=\lambda f_i$ for some $\lambda\in K$ and so $\Hom_{\La}(e_{i+1}\La,e_i\La)\cong K$. The result follows by a simple induction.
\end{proof}

\begin{remark}\label{rem:strong abutment}
The requirement of $P$ being a left abutment in Lemma \ref{lemma:homsofabutments}(a) is stronger than what is used in the proof. Specifically, Lemma \ref{lemma:homsofabutments}(a) holds for any uniserial projective module such that all submodules are projective and dually for Lemma \ref{lemma:homsofabutments}(b).
\end{remark}

Recall that a \emph{presentation} of an algebra $\La$ is an isomorphism $\Phi:KQ/\cR\overset{\sim}{\longrightarrow} \La$ where $Q$ is a quiver and $\cR$ is an admissible ideal of $KQ$. Recall also that $Q$ is unique (up to isomorphism of quivers) but $\cR$ is in general not unique. In particular, if $\{e_1,\dots e_k\}$ is a complete set of primitive orthogonal idempotents of $\La$ then the quiver $Q=Q_{\La}$ is the quiver with $(Q_{\La})_0=\{1,\dots,k\}$ and with arrows $i\to j$ being in bijection with a basis of the $K$-vector space $e_i(\rad\La/\rad^2\La)e_j$. For more details we refer to \cite[Chapter II.3]{ASS}. The following proposition describes abutments in terms of quivers with relations.

\begin{proposition} 
\label{prop:abutmentsquiver}
Let $\La$ be a rep\-re\-sen\-ta\-tion-di\-rect\-ed algebra.
\begin{itemize}
\item[(a)] $P$ is a left abutment realized by $(e_i,f_i)_{i=1}^h$
if and only if there exists a presentation $\Phi:KQ_{\La}/\cR\overset{\sim}{\longrightarrow}\La$ of $\La$ where $Q_{\La}$ is of the form
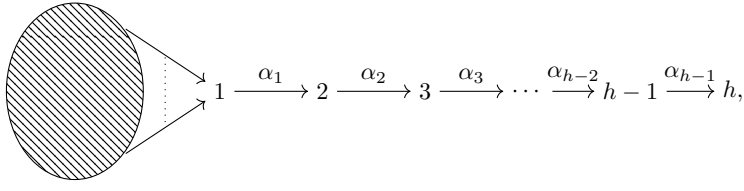
\[\begin{tikzpicture}[scale=0.9, transform shape]
\node(A) at (0,0) {$1$};
\node(B) at (1.5,0) {$2$};
\node(C) at (3,0) {$3$};
\node(D) at (4.5,0) {$\cdots$};
\node(E) at (6,0) {$h-1$};
\node(F) at (7.5,0) {$h$,};
\node(G) at (-1.5,1) {$ $};
\node(H) at (-1.5,-1) {$ $};

\draw[->] (A) -- node[above] {$\alpha_1$} (B);
\draw[->] (B) -- node[above] {$\alpha_2$} (C);
\draw[->] (C) -- node[above] {$\alpha_3$} (D);
\draw[->] (D) -- node[above] {$\alpha_{h-2}$} (E);
\draw[->] (E) -- node[above] {$\alpha_{h-1}$} (F);

\draw[->] (G) -- (A);
\draw[->] (H) -- (A);

\draw[dotted] (-0.8, 0.5) -- (-0.8, -0.5);

\draw[pattern=north west lines] (-2.12,0) ellipse (1cm and 1.3cm);
\end{tikzpicture}\]
and such that no path of the form $\alpha_i\cdots\alpha_{i+k}$ is in $\cR$, and for $1\leq i \leq h$ we have $\Phi(a_i)=e_i$ and, moreover, we have $\Phi(\alpha_i)=f_i$, where $a_i$ is the idempotent of $KQ_\La/\cR$ corresponding to the vertex $i$.

\item[(b)] $I$ is a right abutment realized by $(e_i,g_{i-1})_{i=1}^h$
if and only if there exists a presentation $\Phi:KQ_{\La}/\cR\overset{\sim}{\longrightarrow}\La$ of $\La$ where $Q_{\La}$ is of the form
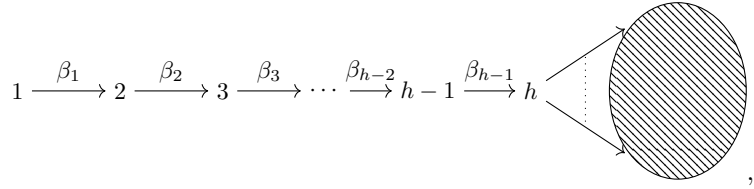
\[\begin{tikzpicture}[scale=0.9, transform shape]
\node(A) at (0,0) {$1$};
\node(B) at (1.5,0) {$2$};
\node(C) at (3,0) {$3$};
\node(D) at (4.5,0) {$\cdots$};
\node(E) at (6,0) {$h-1$};
\node(F) at (7.5,0) {$h$};
\node(G) at (9,1) {$ $};
\node(H) at (9,-1) {$ $};

\draw[->] (A) -- node[above] {$\beta_1$} (B);
\draw[->] (B) -- node[above] {$\beta_2$} (C);
\draw[->] (C) -- node[above] {$\beta_3$} (D);
\draw[->] (D) -- node[above] {$\beta_{h-2}$} (E);
\draw[->] (E) -- node[above] {$\beta_{h-1}$} (F);

\draw[->] (F) -- (G);
\draw[->] (F) -- (H);

\draw[dotted] (8.3, 0.5) -- (8.3, -0.5);

\draw[pattern=north west lines] (9.65,0) ellipse (1cm and 1.3cm);
\end{tikzpicture},\]
and such that no path of the form $\beta_i\cdots\beta_{i+k}$ is in $\cR$, and for $1\leq i \leq h$ we have $\Phi(b_i)=e_i$, and, moreover, for $1\leq i \leq h- 1$ we have $\Phi(\beta_i)=g_{i}$, where $b_i$ is the idempotent of $KQ_\La/\cR$ corresponding to the vertex $i$.
\end{itemize}

\end{proposition}

\begin{proof}
\begin{itemize}
\item[(a)] Throughout this proof let $e'=1_\La-e\ab$ and identify $\Hom_{\La}(e_i\La,e_j\La)$ with $e_j\La e_i$. In particular, we have that  $\La=e\ab\La\oplus e'\La$.

Assume first that $P$ is a left abutment realized by $(e_i,f_i)_{i=1}^h$. Notice that by the uniqueness of the composition series of $e_1\La$ it follows that $\rad(e_i\La) \cong e_{i+1}\La$ (under the convention $e_{h+1}=0$). In particular, we have that $e_i\La / e_{i+1}\La \cong \topp(e_i\La)= S_{i}$.

Let us first show that the quiver $Q_{\La}$ has the required form. We extend the set $\{e_1,\dots,e_h\}$ to a complete set of primitive orthogonal idempotents of $\La$. Then the idempotents $\{e_1,\dots,e_h\}$ correspond to vertices $Q_{e\ab}=\{1,\dots,h\}$ in the quiver $Q_{\La}$. We set $Q_{e'}:=(Q_\La)_0\setminus Q_{e\ab}$. For $1\leq i \leq h$ and for $x\in (Q_{\La})_0$ we have
\[ e_i(\rad \La / \rad^2\La)e_{x}  =\left(\rad(e_i\La) / \rad^2(e_i\La)\right)e_{x}\cong S_{i+1}e_x = \begin{cases} S_{i+1} &\mbox{ if $x=i+1$,} \\0 &\mbox{ otherwise,} \end{cases}\]
under the convention $S_{h+1}=0$. In particular, we have $\dim_K(e_i(\rad \La / \rad^2\La)e_{x})=\delta_{i+1,x}$ since $\La$ is basic. It follows that there is no arrow from the vertices $Q_{e\ab}$ to the vertices $Q_{e'}$ and that for $1\leq i,j \leq h$ there is an arrow $\alpha_i:i\to j$ if and only if $j=i+1$, and in that case it is the only arrow.

Let $2\leq i \leq h$. It remains to show that there are no arrows from $Q_{e'}$ to $e_i$. Let $x\in Q_{e'}$ and let $e_xae_i+\rad^2\La\in e_x(\rad\La /\rad^2\La)e_i$ for some $a\in \rad\La$. It is enough to show that $e_xae_i\in \rad^2\La$. Since $e_xae_i\in e_x(\rad\La) e_i = \Hom_{\La}(e_i\La,e_x\rad\La)$, we have that $e_xae_i$ corresponds to a morphism $e_i\La\to e_x\rad\La=\rad (e_x\La)$. Composing with the inclusion $\rad (e_x\La)\to e_x\La$, we obtain a morphism $e_i\La \to \rad (e_x\La) \to e_x\La$. By the factorization property in the definition of an abutment, this morphism factors through $e_{i-1}\La$. In particular, we get a commutative diagram
\[\begin{tikzcd}
e_i\La \arrow[r, "e_xae_i"] \arrow[dr, swap, "u_1"] 
& \rad (e_x\La) \arrow[r] & e_x\La  \\
& e_{i-1}\La \arrow[u, "u_2"] \arrow[ur] &
\end{tikzcd}\]
where the arrow $u_2$ exists because the morphism $e_{i-1}\La\to e_x\La$ factors through $\rad(e_x\La)$. We now study the morphisms $u_1$ and $u_2$. For $u_1$ notice that $u_1\in \Hom_{\La}(e_i\La,e_{i-1}\La)=e_{i-1}\La e_i$ and we claim that $e_{i-1}\La e_{i}=e_{i-1}(\rad\La) e_i$. Indeed, we have $e_{i-1}(\rad\La) e_i \subseteq e_{i-1}\La e_i$ and both vector spaces have dimension equal to $1$ since by Lemma \ref{lemma:homsofabutments}(a) we have
\begin{align*}
    \dim_K(e_{i-1}(\rad\La)e_i) &=\dim_K\left(\Hom_{\La}(e_i\La,e_{i-1}\rad\La)\right)\\
    &=\dim_K\left(\Hom_{\La}(e_i\La,\rad(e_{i-1}\La))\right) \\
    &=\dim_K\left(\Hom_{\La}(e_i\La,e_i\La)\right) \\
    &=1 \\
    &=\dim_K\left(\Hom_{\La}(e_i\La,e_{i-1}\La)\right)\\
    &=\dim_K(e_{i-1}\La e_i).
\end{align*}
Hence $u_1\in e_{i-1}(\rad\La) e_i\subseteq \rad\La$ since $\rad\La$ is a two-sided ideal. Moreover, we have \[u_2\in\Hom_{\La}(e_1\La,e_x\rad\La)=e_x(\rad\La) e_i\subseteq \rad\La,\]
and so $u_2\in\rad\La$ as well. Hence $e_xae_i=u_2u_1\in\rad^2\La$ as required.

Next define an algebra homomorphism $F:KQ_{\La}\to \La$ in the following way. For each pair of vertices $x,y\in (Q_{\La})_0$ we pick a basis $\{z_{\alpha}+\rad^2\La \mid \alpha:x\to y\}$ of $e_x(\rad \La /\rad^2\La)e_y$. In particular, for $\alpha_i:i\to i+1$ we pick $z_{\alpha_i}=f_i+\rad^2\La$. For $x\in (Q_{\La})_0$ we define $F(a_x)= e_x$ and for $\alpha\in(Q_{\La})_0$ we define $F(\alpha)=z_{\alpha}$. We define $F$ on paths in $KQ_{\La}$ to be the multiplication of the corresponding elements in $\La$ and extend by $K$-linearity. By \cite[Theorem 3.7]{ASS} this induces an isomorphism $\Phi:KQ_{\La}/\ker F\overset{\sim}{\longrightarrow} \La$ where $\ker F=:\cR$ is an admissible ideal of $KQ_{\La}$. Clearly $\Phi(a_i)=e_i$ for $1\leq i\leq h$ and $\Phi(\alpha_i)=f_i$ for $1\leq i\leq h-1$ by construction. Moreover 
\[F(\alpha_i\cdots\alpha_{i+k}) = f_{i}\circ\cdots \circ f_{i+k} \neq 0,\]
and so $\alpha_i\cdots\alpha_{i+k}\not\in \cR$, which proves this direction.

For the other direction, we may identify $KQ_{\La}/\cR$-modules with representations of $Q_{\La}$ bound by $\cR$. We then have by a direct computation that $P(h)$ is a simple projective module and for $1\leq i \leq h-1$ we have $\rad(P(i)) \cong P(i+1)$. Therefore the element $\alpha_i\in a_i(KQ_{\La}/\cR) a_{i+1}=\Hom_{KQ_{\La}/\cR}(P(i+1), P(i))$ corresponds to the inclusion $\rad(P(i)) \subseteq P(i)$. Hence the radical series of $P(1)$ is its composition series and so $P(1)$ is uniserial. Moreover this composition series corresponds to the diagram
\[0\overset{0}{\longrightarrow} P(h) \overset{\alpha_{h-1}}{\longrightarrow} \cdots \overset{\alpha_2}{\longrightarrow} P(2) \overset{\alpha_1}{\longrightarrow} P(1).\]
Since there are no other arrows with target $j$ for $2\leq j \leq h$, then for $k\not\in\{1,\dots,h\}$ we have
\begin{align*}
    \Hom_{KQ_{\La}/\cR}(P(j), P(k)) &= a_k(KQ_{\La}/\cR) a_j \\
    &= a_k(KQ_{\La}/\cR) a_1 \alpha_1\cdots \alpha_{j-1} a_j\\
    &=\Hom_{KQ_{\La}/\cR}(P(1),P(k))\alpha_1\cdots \alpha_{j-1}. 
\end{align*}
It follows that $P(1)$ is a left abutment with realization $(a_i,\alpha_i)_{i=1}^h$. By the assumptions on $\Phi$ it follows that $\Phi(P(1))\cong e_1\La$ is a left abutment with realization $(e_i,f_i)_{i=1}^h$.
\item[(b)] This follows immediately from the definition and (a), since $Q_{\La^{\text{op}}}=Q^{\text{op}}_{\La}$.\qedhere
\end{itemize}
\end{proof}

Proposition \ref{prop:abutmentsquiver} shows that abutments are linearly oriented \emph{arms} in the sense of Ringel \cite{RIN}.   

\begin{remark}\label{rem:heredity}
It follows from Proposition \ref{prop:abutmentsquiver} that if $(e_i,f_i)_{i=1}^h$ is a realization of a left abutment of height $h$, then $(e_i,f_i)_{i=k}^h$ is a realization of a left abutment of height $h-k+1$, for any $1\leq k \leq h$. In particular, a submodule of a left abutment is also a left abutment.

Similarly, if $(e_i, g_{i-1})_{i=1}^h$ is a realization of a right abutment of height $h$, then $(e_i, g_{i-1})_{i=1}^k$ is a realization of a right abutment of height $k$, for any $1\leq k\leq h$. In particular, a quotient module of a right abutment is also a right abutment.
\end{remark}

If $\La$ is given by a quiver with relations, it is easy to find all abutments using Proposition \ref{prop:abutmentsquiver}, as the following examples show.

\begin{example}
\label{ex:twogluesquiver}
Let $B$ be given by the quiver with relations
\[\begin{tikzpicture}[scale=0.9, transform shape]
\node (1) at (0,1) {$1$};
\node (2) at (1,1) {$2$};
\node (3) at (2,1) {$3$};
\node (4) at (3,1) {$4$};
\node (5) at (4,0.5) {$5$};
\node (6) at (5,0.5) {$6$};
\node (7) at (6,0.5) {$7$\nospacepunct{.}};
\node (1') at (1,0) {$1'$};
\node (2') at (2,0) {$2'$};
\node (3') at (3,0) {$3'$};

\draw[->] (1) to (2);
\draw[->] (2) to (3);
\draw[->] (3) to (4);
\draw[->] (4) to (5);
\draw[->] (5) to (6);
\draw[->] (6) to (7);
\draw[->] (1') to (2');
\draw[->] (2') to (3');
\draw[->] (3') to (5);

\draw[dotted] (2) to [out=30,in=150] (4);
\draw[dotted] (3) to [out=-30,in=0] (5.west);
\draw[dotted] (4) to [out=0,in=160] (7);
\draw[dotted] (2') to [out=30,in=0] (5.west);
\draw[dotted] (3') to [out=0,in=210] (6);
\end{tikzpicture}\]
Then by Proposition \ref{prop:abutmentsquiver} the left abutments are $P(5)$, $P(6)$ and $P(7)$ with heights $3$, $2$ and $1$ respectively, and the right abutments are $I(3)$, $I(2)$, $I(1)$, $I(3')$, $I(2')$ and $I(1')$ with heights $3$, $2$, $1$, $3$, $2$ and $1$ respectively.
\end{example}

\begin{example}\label{ex:KAabutments}
It follows from Proposition \ref{prop:abutmentsquiver}(a) that the algebra $KA_h$ has exactly $h$ left abutments, namely $\{t_i KA_h\}_{i=1}^h$, and that the height of $t_i KA_h$ is $h-i+1$. By Proposition \ref{prop:abutmentsquiver}(b) the algebra $KA_h$ has exactly $h$ right abutments, namely $\{D(KA_h t_i)\}_{i=1}^h$ and the height of $D(KA_h t_i)$ is $i$. In particular, $t_1 KA_h \cong D(KA_h t_h)$ is both a left and a right abutment. 

By the same proposition it follows that if an algebra $\La$ admits a module $M$ that is both a left and a right abutment, then $\La\cong KA_h$ and $M$ is the unique indecomposable pro\-jec\-tive-in\-jec\-tive $KA_h$-module. In particular, $M$ has the same height $h$ as a left and a right abutment.
\end{example}

We have the following important Corollary.

\begin{corollary}\label{cor:forfooting}
Let $U$ be a left abutment realized by $(e_i,f_i)_{i=1}^{h}$ (respectively a right abutment realized by $(e_i,g_{i-1})_{i=1}^h$) and let $\Phi:KQ_\La/\cR\overset{\sim}{\longrightarrow}\La$ be as in Proposition \ref{prop:abutmentsquiver}(a) (respectively as in Proposition \ref{prop:abutmentsquiver}(b)). Let $\pi$ be the epimorphism $\pi:KQ_\La/\cR\longrightarrow KA_h$ given by identifying the full subquiver with vertices $Q_{e\ab}$ with $A_h$. Then the morphism $\pi\circ\Phi^{-1}$ is independent of the choice of $\Phi$ and it satisfies $\pi\circ\Phi^{-1}(e_i)=t_i$ for $1\leq i\leq h$ and $\pi\circ\Phi^{-1}(f_i)=\alpha_i$ (respectively $\pi\circ\Phi^{-1}(g_i)=\alpha_i$) for $1\leq i \leq h-1$. 
\end{corollary}

\begin{proof}
Let us assume that $U$ is a left abutment and $\Phi$ is as in Proposition \ref{prop:abutmentsquiver}(a); the other case is similar. Notice that we have a short exact sequence 
\[0\longrightarrow \left(1-\sum_{i=1}^h a_i\right) \longrightarrow KQ_\La/\cR \overset{\pi}{\longrightarrow} KA_h \longrightarrow 0,\]
In particular, $\pi\left(\sum_{i=1}^h a_i\right)=1_{KA_h}$. Let $\Phi,\Psi:KQ_\La/\cR\rightarrow \La$ be isomorphisms satisfying $\Phi(a_i)=e_i=\Psi(a_i)$ and $\Phi(\alpha_i)=f_i=\Psi(\alpha_i)$. By the description in Proposition \ref{prop:abutmentsquiver}(a) we have that $\Phi^{-1}(e\ab a)=\Psi^{-1}(e\ab a)$ for all $a\in \La$. It follows that 
\begin{align*}
\pi\circ \Phi^{-1}(a) &= \pi(\Phi^{-1}(a)) =\pi\left(\sum_{i=1}^h a_i\right)\pi(\Phi^{-1}(a))=\pi\left(\left(\sum_{i=1}^h a_i\right) \Phi^{-1}(a)\right) \\
&=\pi\left(\Phi^{-1}(e\ab a)\right) 
=\pi\left(\Psi^{-1}(e\ab a)\right) 
=\pi\circ\Psi^{-1}(a),
\end{align*}
which proves that $\pi\circ\Phi^{-1}=\pi\circ\Psi^{-1}$. The equalities $\pi\circ\Phi^{-1}(e_i)=t_i$ and $\pi\circ\Phi^{-1}(f_i)=\alpha_i$ are immediate by definition.
\end{proof}

Corollary \ref{cor:forfooting} justifies the following definition.

\begin{definition}\label{def:footing}
For a left abutment $P$ realized by $(e_i,f_i)_{i=1}^h$ (respectively a right abutment $I$ realized by $(e_i,g_{i-1})_{i=1}^h$) we denote the epimorphism $\pi\circ\Phi^{-1}:\La\twoheadrightarrow KA_h$ by $f_P$ (respectively $g_I$) and we call it the \emph{footing at $P$} (respectively \emph{$I$}). 
\end{definition}

An easy consequence of Definition \ref{def:footing} is the following.

\begin{corollary}\label{cor:footingi}
Let $\La$ be a rep\-re\-sen\-ta\-tion-di\-rect\-ed algebra.
\begin{itemize}
\item[(a)] If $P$ is a left abutment realized by $(e_i,f_i)_{i=1}^h$, then $f_P(e\ab\lambda)=0$ implies $e\ab\lambda=0$.
\item[(b)] If $I$ is a right abutment realized by $(e_i, g_{i-1})_{i=1}^h$, then $g_I(\lambda e\ab)=0$ implies $\lambda e\ab=0$.
\end{itemize}
\end{corollary}

\begin{proof}
We only prove (a); (b) is similar. We prove the contrapositive statement instead. Assume that $e\ab\lambda \neq 0$. Since $\Phi^{-1}$ is an isomorphism, it follows that $\Phi^{-1}(e\ab\lambda)\neq 0$. By Proposition \ref{prop:abutmentsquiver}(a), we have that 
\[\Phi^{-1}(e\ab\lambda)=\left(\sum_{i=1}^h a_i\right) \Phi^{-1}(\lambda)\neq 0.\] 
By the same proposition and the definition of $\pi$ it follows that 
\[\pi \left(\left(\sum_{i=1}^h a_i\right) \Phi^{-1}(\lambda)\right)\neq 0.\] 
But then $f_P(e\ab\lambda)=\pi\circ \Phi^{-1}(e\ab\lambda)\neq 0$, as required.
\end{proof}

The following Lemma describes abutments in terms of the Aus\-lan\-der--Rei\-ten quiver $\Gamma(\La)$ of $\La$.

\begin{proposition}\label{prop:triangles} Let $\La$ be a rep\-re\-sen\-ta\-tion-di\-rect\-ed algebra and $\Gamma(\La)$ be its Aus\-lan\-der--Rei\-ten quiver.
\begin{itemize} 
\item[(a)] $P=e_1\La$ is a left abutment realized by $(e_i,f_i)_{i=1}^h$, if and only if
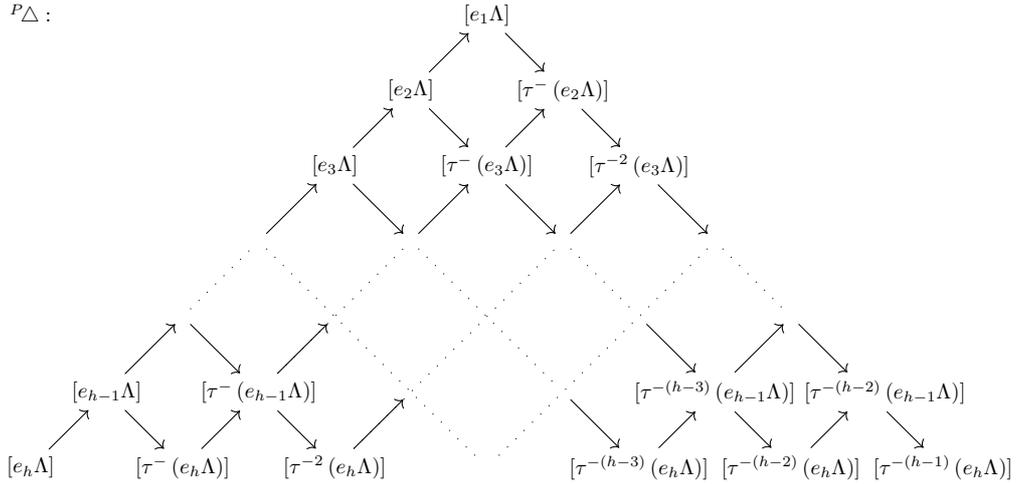
\[\begin{tikzpicture}[scale=0.8, transform shape]
\node (00) at (0,7.5) {$\PD:$};

\node (11) at (0,0) {$[e_h\La]$};
\node (21) at (2.5,0) {$[\tau^-\left(e_h\La\right)]$};
\node (31) at (5,0) {$[\tau^{-2}\left(e_h\La\right)]$};
\node (41) at (7.5,0) { };
\node (51) at (10,0) {$[\tau^{-(h-3)}\left(e_h\La\right)]$};
\node (61) at (12.5,0) {$[\tau^{-(h-2)}\left(e_h\La\right)]$};
\node (71) at (15,0) {$[\tau^{-(h-1)}\left(e_h\La\right)]$};

\node (12) at (1.25,1.25) {$[e_{h-1}\La]$};
\node (22) at (3.75,1.25) {$[\tau^-\left(e_{h-1}\La\right)]$};
\node (32) at (6.25,1.25) { };
\node (42) at (8.75,1.25) { };
\node (52) at (11.25,1.25) {$[\tau^{-(h-3)}\left(e_{h-1}\La\right)]$};
\node (62) at (13.75,1.25) {\;\;\;\;\;\;$[\tau^{-(h-2)}\left(e_{h-1}\La\right)]$};

\node (13) at (2.5,2.5) { };
\node (23) at (5,2.5) { };
\node (33) at (7.5,2.5) { };
\node (43) at (10,2.5) { };
\node (53) at (12.5,2.5) { };

\node (14) at (3.75,3.75) { };
\node (24) at (6.25,3.75) { };
\node (34) at (8.75,3.75) { };
\node (44) at (11.25,3.75) { };

\node (15) at (5,5) {$[e_3\La]$};
\node (25) at (7.5,5) {$[\tau^-\left(e_3\La\right)]$};
\node (35) at (10,5) {$[\tau^{-2}\left(e_3\La\right)]$};

\node (16) at (6.25,6.25) {$[e_2\La]$};
\node (26) at (8.75,6.25) {$[\tau^-\left(e_2\La\right)]$};

\node (17) at (7.5,7.5) {$[e_1\La]$};

\draw[->] (11) -- (12);
\draw[->] (12) -- (21);
\draw[->] (21) -- (22);
\draw[->] (22) -- (31);

\draw[->] (51) -- (52);
\draw[->] (52) -- (61);
\draw[->] (61) -- (62);
\draw[->] (62) -- (71);

\draw[->] (12) -- (13);
\draw[->] (13) -- (22);
\draw[->] (22) -- (23);

\draw[->] (14) -- (15);
\draw[->] (15) -- (16);
\draw[->] (16) -- (17);

\draw[->] (15) -- (24);
\draw[->] (24) -- (25);
\draw[->] (25) -- (34);
\draw[->] (34) -- (35);
\draw[->] (35) -- (44);

\draw[->] (16) -- (25);
\draw[->] (17) -- (26);
\draw[->] (26) -- (35);
\draw[->] (25) -- (26);

\draw[->] (31) -- (32);
\draw[->] (42) -- (51);
\draw[->] (43) -- (52);
\draw[->] (53) -- (62);

\draw[->] (52) -- (53);

\draw[loosely dotted] (44) -- (41);
\draw[loosely dotted] (14) -- (41);

\draw[loosely dotted] (34) -- (43);
\draw[loosely dotted] (44) -- (53);

\draw[loosely dotted] (24) -- (42);
\draw[loosely dotted] (13) -- (14);
\draw[loosely dotted] (22) -- (24);
\draw[loosely dotted] (31) -- (34);
\end{tikzpicture}\]
is a full subquiver of $\Gamma(\La)$, there are no other arrows in $\Gamma(\La)$ going into $\PD$ and, moreover, all northeast arrows are monomorphisms, all southeast arrows are epimorphisms and all modules in the same row have the same dimension. In particular, $\tau^{-i}\left(e_h \La\right)$ is the simple top of $e_{h-i}\La$ for $1\leq i \leq h-1$. We call $\PD$ \emph{the foundation of $P$}.

\item[(b)] $I=D(\La e_h)$ is a right abutment realized by $(e_i,g_{i-1})_{i=1}^h$, if and only if
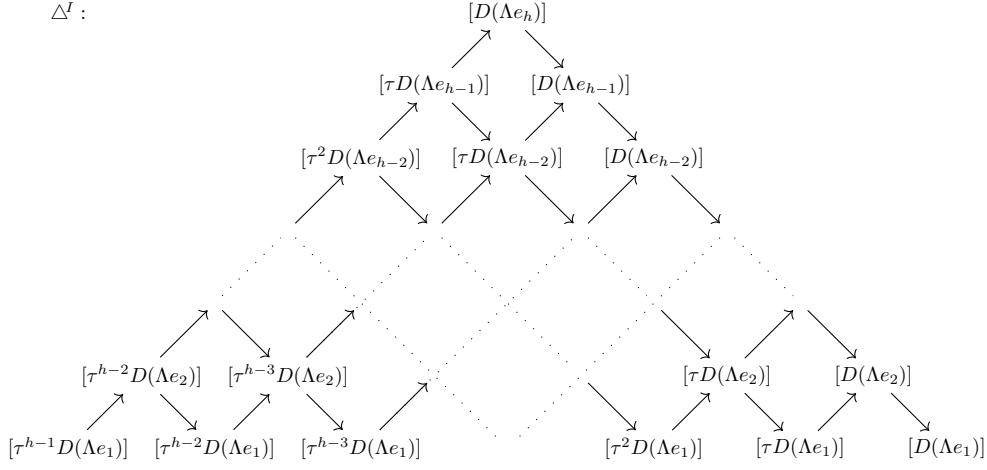
\[\begin{tikzpicture}[scale=0.77, transform shape]
\node (00) at (0,7.5) {$\DI:$};

\node (11) at (0,0) {$[\tau^{h-1} D(\La e_{1})]$};
\node (21) at (2.5,0) {$[\tau^{h-2} D(\La e_{1})]$};
\node (31) at (5,0) {$[\tau^{h-3} D(\La e_{1})]$};
\node (41) at (7.5,0) { };
\node (51) at (10,0) {$[\tau^2 D(\La e_{1})]$};
\node (61) at (12.5,0) {$[\tau D(\La e_{1})]$};
\node (71) at (15,0) {$[D(\La e_{1})]$};

\node (12) at (1.25,1.25) {$[\tau^{h-2}D(\La e_2)]$};
\node (22) at (3.75,1.25) {$[\tau^{h-3}D(\La e_2)]$};
\node (32) at (6.25,1.25) { };
\node (42) at (8.75,1.25) { };
\node (52) at (11.25,1.25) {$[\tau D(\La e_{2})]$};
\node (62) at (13.75,1.25) {$[D(\La e_{2})]$};

\node (13) at (2.5,2.5) { };
\node (23) at (5,2.5) { };
\node (33) at (7.5,2.5) { };
\node (43) at (10,2.5) { };
\node (53) at (12.5,2.5) { };

\node (14) at (3.75,3.75) { };
\node (24) at (6.25,3.75) { };
\node (34) at (8.75,3.75) { };
\node (44) at (11.25,3.75) { };

\node (15) at (5,5) {$[\tau^2 D(\La e_{h-2})]$};
\node (25) at (7.5,5) {$[\tau D(\La e_{h-2})]$};
\node (35) at (10,5) {$[D(\La e_{h-2})]$};

\node (16) at (6.25,6.25) {$[\tau D(\La e_{h-1})]$};
\node (26) at (8.75,6.25) {$[D(\La e_{h-1})]$};

\node (17) at (7.5,7.5) {$[D(\La e_h)]$};

\draw[->] (11) -- (12);
\draw[->] (12) -- (21);
\draw[->] (21) -- (22);
\draw[->] (22) -- (31);

\draw[->] (51) -- (52);
\draw[->] (52) -- (61);
\draw[->] (61) -- (62);
\draw[->] (62) -- (71);

\draw[->] (12) -- (13);
\draw[->] (13) -- (22);
\draw[->] (22) -- (23);

\draw[->] (14) -- (15);
\draw[->] (15) -- (16);
\draw[->] (16) -- (17);

\draw[->] (15) -- (24);
\draw[->] (24) -- (25);
\draw[->] (25) -- (34);
\draw[->] (34) -- (35);
\draw[->] (35) -- (44);

\draw[->] (16) -- (25);
\draw[->] (17) -- (26);
\draw[->] (26) -- (35);
\draw[->] (25) -- (26);

\draw[->] (31) -- (32);
\draw[->] (42) -- (51);
\draw[->] (43) -- (52);
\draw[->] (53) -- (62);

\draw[->] (52) -- (53);

\draw[loosely dotted] (44) -- (41);
\draw[loosely dotted] (14) -- (41);

\draw[loosely dotted] (34) -- (43);
\draw[loosely dotted] (44) -- (53);

\draw[loosely dotted] (24) -- (42);
\draw[loosely dotted] (13) -- (14);
\draw[loosely dotted] (22) -- (24);
\draw[loosely dotted] (31) -- (34);
\end{tikzpicture}\]
is a full subquiver of $\Gamma(\La)$, there are no other arrows in $\Gamma(\La)$ leaving $\DI$ and, moreover, all northeast arrows are monomorphisms, all southeast arrows are epimorphisms and all modules in the same row have the same dimension. In particular, $\tau^{i}D(\La e_1)  $ is the simple socle of $D(\La e_{i+1})$ for $1\leq i \leq h-1$. We call $\DI$ \emph{the foundation of $I$}.
\end{itemize}
\end{proposition}

\begin{proof}
We only prove (a); (b) is similar. Assume first that $\PD$ is a full subquiver of $\Gamma(\La)$ satisfying the required properties. Since all northeast arrows are monomorphisms and there are no other arrows going into $\PD$, it follows that $e_1\La$ is uniserial. For the factorization property, let $P'$ be an indecomposable projective module such that there exists a nonzero homomorphism $\phi:e_i\La \rightarrow P'$and $P'\not\subseteq e_1\La$ for some $1\leq i \leq h$. Set $J_i=\tau^{-(h-i)}\left(e_{h-i+1}\La\right)$ and $\cJ:=\add \{J_i\}_{i=1}^h$, that is $\cJ$ is the additive closure of all indecomposable modules $X$ such that $[X]$ appears in the rightmost southeast diagonal of $\PD$. Since the only indecomposable projective modules $Y$ with $[Y]\in\PD$ are isomorphic to submodules of $e_1\La$, it follows that $[P']\not\in\PD$. Since there are no arrows going into $\PD$, the only arrows going out of $\PD$ have one of the vertices $\{[J_1],\dots,[J_h]\}$ as a source. Hence $\phi$ factors through $\cJ$ so that $\phi=g_1\circ g_2$ with $g_2:e_i\La\rightarrow N$ and $g_1:N\rightarrow P'$ for some $N\in \cJ$. Moreover, since there are no other arrows going out of $\PD\setminus \{[J_1],\dots,[J_h]\}$, all squares in $\PD$ correspond to almost split sequences and hence they are commutative. It follows that any morphism from $e_i\La$ to $\cJ$ factors through $e_1\La$. Hence the morphism $g_2$ factors through $e_1\La$ which shows that $\phi:e_i\La \rightarrow P'$ factors through $e_1\La$, as required.

For the other direction we use induction on $h\geq 1$. If $h=1$ then $e_1\La$ is a simple projective module and so there are no irreducible morphisms in $\Gamma(\La)$ into $e_1\La$. We assume the result is true for $h=k$ and will prove it for $h=k+1$. By induction hypothesis, and since by Remark \ref{rem:heredity} we have that $e_2\La$ is also a left abutment of height $h-1$, it follows that
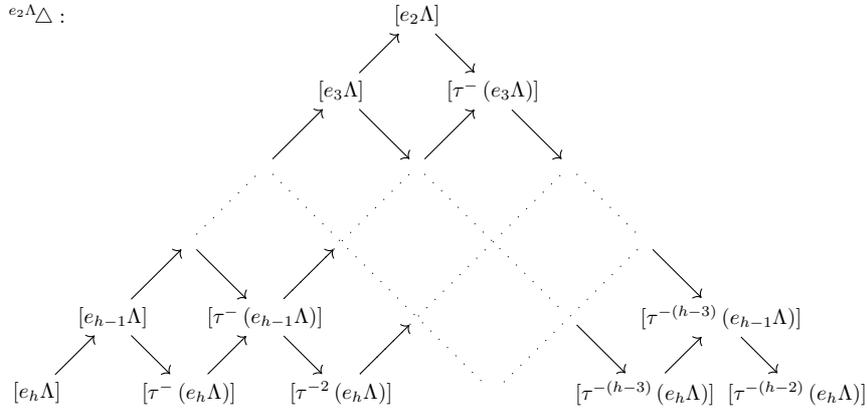
\[\begin{tikzpicture}[scale=0.8, transform shape]

\node (00) at (0,6.25) {$\nsup{e_2\La}{\triangle}:$};
\node (11) at (0,0) {$[e_h\La]$};
\node (21) at (2.5,0) {$[\tau^-\left(e_h\La\right)]$};
\node (31) at (5,0) {$[\tau^{-2}\left(e_h\La\right)]$};
\node (41) at (7.5,0) { };
\node (51) at (10,0) {$[\tau^{-(h-3)}\left(e_h\La\right)]$};
\node (61) at (12.5,0) {$[\tau^{-(h-2)}\left(e_h\La\right)]$};
\node (12) at (1.25,1.25) {$[e_{h-1}\La]$};
\node (22) at (3.75,1.25) {$[\tau^-\left(e_{h-1}\La\right)]$};
\node (32) at (6.25,1.25) { };
\node (42) at (8.75,1.25) { };
\node (52) at (11.25,1.25) {$[\tau^{-(h-3)}\left(e_{h-1}\La\right)]$};
\node (13) at (2.5,2.5) { };
\node (23) at (5,2.5) { };
\node (33) at (7.5,2.5) { };
\node (43) at (10,2.5) { };
\node (14) at (3.75,3.75) { };
\node (24) at (6.25,3.75) { };
\node (34) at (8.75,3.75) { };
\node (15) at (5,5) {$[e_3\La]$};
\node (25) at (7.5,5) {$[\tau^-\left(e_3\La\right)]$};
\node (16) at (6.25,6.25) {$[e_2\La]$};

\draw[->] (11) -- (12);
\draw[->] (12) -- (21);
\draw[->] (21) -- (22);
\draw[->] (22) -- (31);
\draw[->] (51) -- (52);
\draw[->] (52) -- (61);
\draw[->] (12) -- (13);
\draw[->] (13) -- (22);
\draw[->] (22) -- (23);
\draw[->] (14) -- (15);
\draw[->] (15) -- (16);
\draw[->] (15) -- (24);
\draw[->] (24) -- (25);
\draw[->] (25) -- (34);
\draw[->] (16) -- (25);
\draw[->] (31) -- (32);
\draw[->] (42) -- (51);
\draw[->] (43) -- (52);

\draw[loosely dotted] (43) -- (41);
\draw[loosely dotted] (14) -- (41);
\draw[loosely dotted] (34) -- (43);
\draw[loosely dotted] (24) -- (42);
\draw[loosely dotted] (13) -- (14);
\draw[loosely dotted] (22) -- (24);
\draw[loosely dotted] (31) -- (34);
\end{tikzpicture}\]
is also a full subquiver of $\Gamma(\La)$ and there are no other arrows in $\Gamma(\La)$ going into $\nsup{e_2\La}{\triangle}$. Since $e_1\La$ is uniserial, we have $e_2\La \cong \rad(e_1\La)$ and so there is an arrow $[e_2\La] \rightarrow [e_1\La]$ in $\Gamma(\La)$. We claim that this and the arrow $[e_2\La] \rightarrow [\tau^-\left(e_3\La\right)]$ are the only arrows in $\Gamma(\La)$ starting from $[e_2\La]$. To see this, note that any other arrow starting from $[e_2\La]$ corresponds to the inclusion of $e_2\La$ into an indecomposable projective module $P'$, since there are no other arrows going into $[e_2\La]$. But then this would correspond to some irreducible homomorphism that would not factor through $e_1\La$, contradicting the fact that $e_1\La$ is a left abutment. Hence there is an almost split sequence 
\[0\longrightarrow e_2\La \longrightarrow e_1\La \oplus \tau^-\left(e_3\La\right) \longrightarrow \tau^- \left(e_2\La\right)\longrightarrow 0.\]
Then a similar argument shows that there are exactly two arrows from $[\tau^{-(j-2)}\left(e_{j}\La\right)]$ for $3\leq j \leq h$, exactly as required. 

Since $e_1\La$ is uniserial, we know that $\dim_K(e_{h-i} \La)=i+1$. Since almost split sequences are exact sequences, it easily follows from simple dimension arguments that northeast arrows are monomorphisms, southeast arrows are epimorphisms and along the same row the dimensions remain the same. In particular, the last row has only simple modules, and since there is always an epimorphism $e_{h-i}\La\twoheadrightarrow \tau^{-i}\left(e_h\La\right)$ in $\PD$, the result follows.
\end{proof}

If $P$ is a left abutment of $\La$ we set
\[\cPD := \add\{ X \in \m\La \mid \text{$X$ is indecomposable and $[X]\in \PD$}\}.\]
Similarly, if $I$ is a right abutment of $\La$ we set
\[\cDI := \add\{ X \in \m\La \mid \text{$X$ is indecomposable and $[X]\in \DI$}\}.\]
Using Proposition \ref{prop:triangles} it can be shown that
\[\cPD=\Fac(\Sub(P)),\;\; \cDI= \Sub(\Fac(I)).\]

The following corollary shows that every abutment gives rise to an example of glued subcategories.

\begin{corollary} Let $\La$ be a rep\-re\-sen\-ta\-tion-di\-rect\-ed algebra.
\begin{itemize}
\item[(a)] Let $P$ be a left abutment of $\La$. Then $\m \La = \cPD \glue (\m \La)$.
\item[(b)] Let $I$ be a right abutment of $\La$. Then $\m\La = (\m\La) \glue \cDI$.
\end{itemize}
\end{corollary}

\begin{proof}
Follows immediately by Proposition \ref{prop:triangles}.
\end{proof}

\begin{example}\label{ex:twogluesARquiver}
Let B be as in Example \ref{ex:twogluesquiver}. Then the Aus\-lan\-der--Rei\-ten quiver $\Gamma(B)$ of $B$ is
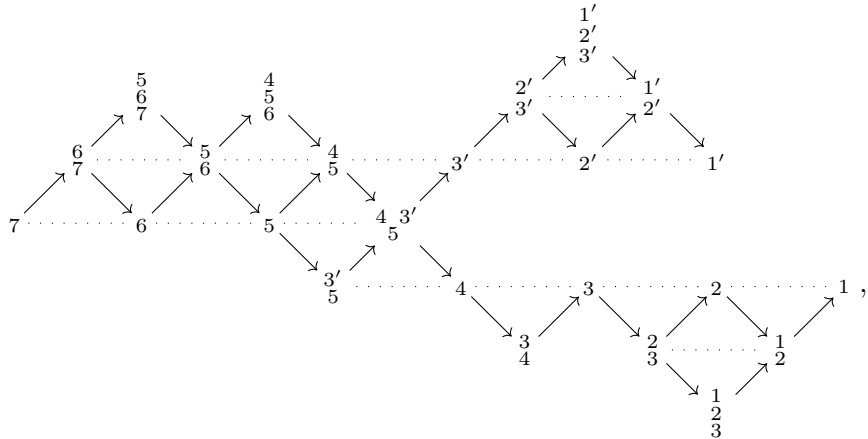
\[\begin{tikzpicture}[scale=1.2, transform shape]
\tikzstyle{nct}=[shape= rectangle, minimum width=6pt, minimum height=7.5, draw, inner sep=0pt]
\tikzstyle{nct2}=[circle, minimum width=6pt, draw, inner sep=0pt]
\tikzstyle{nct3}=[circle, minimum width=6pt, draw=white, inner sep=0pt, scale=0.9]

\node[nct3] (A) at (0,0) {$\qthree{}[7][]$};
\node[nct3] (B) at (0.7,0.7) {$\qthree{6}[7]$};
\node[nct3] (C) at (1.4,1.4) {$\qthree{5}[6][7]$};
\node[nct3] (D) at (1.4,0) {$\qthree{}[6][]$};
\node[nct3] (E) at (2.1,0.7) {$\qthree{5}[6]$};
\node[nct3] (F) at (2.8,1.4) {$\qthree{4}[5][6]$};
\node[nct3] (G) at (2.8,0) {$\qthree{}[5][]$};
\node[nct3] (H) at (3.5,0.7) {$\qthree{4}[5]$};
\node[nct3] (I) at (3.5,-0.7) {$\qthree{3'}[5]$};
\node[nct3] (J) at (4.2,0) {$\begin{smallmatrix} 4 && 3' \\ & 5 &
\end{smallmatrix}$};
\node[nct3] (K) at (4.9,0.7) {$\qthree{}[3'][]$};
\node[nct3] (L) at (4.9,-0.7) {$\qthree{}[4][]$};
\node[nct3] (M) at (5.6,1.4) {$\qthree{2'}[3']$};
\node[nct3] (N) at (5.6,-1.4) {$\qthree{3}[4]$};
\node[nct3] (O) at (6.3,2.1) {$\qthree{1'}[2'][3']$};
\node[nct3] (P) at (6.3,0.7) {$\qthree{}[2'][]$};
\node[nct3] (Q) at (6.3,-0.7) {$\qthree{}[3][]$};
\node[nct3] (R) at (7,1.4) {$\qthree{1'}[2']$};
\node[nct3] (S) at (7,-1.4) {$\qthree{2}[3]$};
\node[nct3] (T) at (7.7,0.7) {$\qthree{}[1'][]$};
\node[nct3] (U) at (7.7,-0.7) {$\qthree{}[2][]$};
\node[nct3] (V) at (7.7,-2.1) {$\qthree{1}[2][3]$};
\node[nct3] (W) at (8.4,-1.4) {$\qthree{1}[2]$};
\node[nct3] (X) at (9.1,-0.7) {$\qthree{}[1][]$\nospacepunct{,}};

\draw[->] (A) to (B);
\draw[->] (B) to (C);
\draw[->] (B) to (D);
\draw[->] (C) to (E);
\draw[->] (D) to (E);
\draw[->] (E) to (F);
\draw[->] (E) to (G);
\draw[->] (F) to (H);
\draw[->] (G) to (H);
\draw[->] (G) to (I);
\draw[->] (H) to (J);
\draw[->] (I) to (J);
\draw[->] (J) to (K);
\draw[->] (J) to (L);
\draw[->] (K) to (M);
\draw[->] (L) to (N);
\draw[->] (M) to (O);
\draw[->] (M) to (P);
\draw[->] (N) to (Q);
\draw[->] (O) to (R);
\draw[->] (P) to (R);
\draw[->] (Q) to (S);
\draw[->] (R) to (T);
\draw[->] (S) to (U);
\draw[->] (S) to (V);
\draw[->] (U) to (W);
\draw[->] (V) to (W);
\draw[->] (W) to (X);

\draw[loosely dotted] (A.east) -- (D);
\draw[loosely dotted] (B.east) -- (E);
\draw[loosely dotted] (D.east) -- (G);
\draw[loosely dotted] (E.east) -- (H);
\draw[loosely dotted] (G.east) -- (J);
\draw[loosely dotted] (H.east) -- (K);
\draw[loosely dotted] (I.east) -- (L);
\draw[loosely dotted] (K.east) -- (P);
\draw[loosely dotted] (M.east) -- (R);
\draw[loosely dotted] (L.east) -- (Q);
\draw[loosely dotted] (Q.east) -- (U);
\draw[loosely dotted] (P.east) -- (T);
\draw[loosely dotted] (S.east) -- (W);
\draw[loosely dotted] (U.east) -- (X);
\end{tikzpicture}\]
where the labels indicate the composition series of the corresponding indecomposable modules. We can see the foundations $\nsup{\tinyqthree{7}}{\triangle}$, $\nsup{\tinyqthree{6}[7][\;]}{\triangle}$ and $\nsup{\tinyqthree{5}[6][7]}{\triangle}$ of the left abutments and the foundations $\supn{\triangle}{\tinyqthree{1}[2][3]}$, $\supn{\triangle}{\tinyqthree{1}[2][\;]}$, $\supn{\triangle}{\tinyqthree{1}}$, $\supn{\triangle}{\tinyqthree{1'}[2'][3']}$, $\supn{\triangle}{\tinyqthree{1'}[2'][\;]}$ and $\supn{\triangle}{\tinyqthree{1'}}$ of the right abutments that were computed in Example \ref{ex:twogluesquiver}.
\end{example}

The following corollaries will be used later.

\begin{corollary}\label{cor:diminabutment}
Let $\La$ be a rep\-re\-sen\-ta\-tion-di\-rect\-ed algebra.
\begin{itemize}
\item[(a)] Let $P$ be a left abutment of $\La$ and $M\in\cPD$. Then $\pd(M)\leq 1$.
\item[(b)] Let $I$ be a right abutment of $\La$ and $N\in\cDI$. Then $\id(N)\leq 1$.
\end{itemize}
\end{corollary}

\begin{proof}
We only prove (a); (b) is similar. Let $(e_i,f_i)_{i=1}^h$ be a realization of $P$. Without loss of generality, we may assume that $M$ is indecomposable. Then $M$ corresponds to one of the vertices in $\PD$, where $\PD$ is as in Proposition \ref{prop:triangles}(a). Since no arrows of $\Gamma(\La)$ have target in $\PD$ but source outside of $\PD$, it follows that the projective cover of $M$ is $e_i\La$ for some $i\in\{1,\dots,h\}$. Since $\om(M)$ is a submodule of $e_i\La$, it follows that $\om(M)=e_j\La$ for some $j\geq i$ or $\om(M)=0$. In both cases we have $\pd(M)\leq 1$ as required. 
\end{proof}

\begin{corollary}\label{cor:pathsupport}
Let $\La$ be a rep\-re\-sen\-ta\-tion-di\-rect\-ed algebra.
\begin{itemize}
\item[(a)] Let $P$ be a left abutment realized by $(e_i,f_i)_{i=1}^h$. Then for every $\la \in \La$ we have $e\ab\la = e\ab\la e\ab$.
\item[(b)] Let $I$ be a right abutment realized by $(e_i,g_{i-1})_{i=1}^h$. Then for every $\la \in \La$ we have $\la e\ab = e\ab\la e\ab$.
\end{itemize}
\end{corollary}

\begin{proof}
We only prove (a); (b) is similar. Rewriting $e\ab\la=e\ab\la e\ab$ as $e\ab\la(1-e\ab)=0$, we see that it is enough to show that $e\ab\La(1-e\ab)=0$. We have
\[ e\ab \La (1-e\ab) \cong \bigoplus_{i\not\in \{1,\dots,h\}} e\ab \La e_i \cong \bigoplus_{i\not\in \{1,\dots,h\}} \Hom_{\La}(e_i \La, e\ab\La) = 0,\]
where the last equality follows from Proposition \ref{prop:triangles} since there is no arrow going into $\PD$ in $\Gamma(\La)$.
\end{proof}

\begin{corollary}\label{cor:intersupport}
Let $\La$ be a rep\-re\-sen\-ta\-tion-di\-rect\-ed algebra.
\begin{itemize}
\item[(a)] Let $P$ be a left abutment realized by $(e_i,f_i)_{i=1}^h$ and let $M\in\cPD$. Then for every $m\in M$ we have $me\ab=m$.
\item[(b)] Let $I$ be a right abutment realized by $(e_i,g_{i-1})_{i=1}^h$ and let $N\in\cDI$. Then for every $n\in N$ we have $ne\ab=n$.
\end{itemize}
\end{corollary}

\begin{proof}
We only prove (a); (b) is similar. If for a module $X$ we have that $xe\ab=x$ holds for all $x\in X$ then it clearly holds for all submodules and epimorphic images of $X$, so by Proposition \ref{prop:triangles} it is enough to show (a) for $M=e_1\La$. But this follows immediately by Corollary \ref{cor:pathsupport}.
\end{proof}

\subsubsection{Gluing via pullbacks}\label{subsubsection:gluing via pullbacks}

The following definition is the the main concept of this section.

\begin{definition}\label{def:gluing}
We call the pullback $\La$ of $A\overset{f_P}{\longrightarrow} KA_h \overset{g_I}{\longleftarrow} B$ the \emph{gluing of $A$ and $B$ along $P$ and $I$} and we denote it by $\La:=B \glue[P][I] A$.
\end{definition}

In the following and when $I$ and $P$ are clear from context we will simply call $\La$ the \emph{gluing of $A$ and $B$} and denote it by $\La:=B \glue A$. Notice that since pullbacks are associative, the operation of gluing is associative too. Moreover, the indecomposable projective and the indecomposable injective modules over $\La$ can be related to the indecomposable projective and the indecomposable injective modules over $A$ and $B$ by considering certain idempotents. This is discussed in \cite[Construction 3.2]{IPTZ}. We include the details here, adapted to our conventions, for the convenience of the reader. 

We start with a simple example of gluing.

\begin{example}\label{ex:KAglued}
Let $A$ be a rep\-re\-sen\-ta\-tion-di\-rect\-ed algebra and $P$ be a left abutment of $A$ of height $h$. Let $B=KA_h$ and let $I=I(h)$ be the unique indecomposable injective-projective $B$-module. By Example \ref{ex:KAabutments} we have that $I$ is a right abutment of $KA_h$ of height $h$. The identity map $\Id_{KA_h}:KA_h \rightarrow KA_h$ is the unique $K$-algebra morphism that satisfies the conditions of Corollary \ref{cor:forfooting} and so the footing at $I$ is $g_I=\Id_{KA_h}$. It is easy to see that the pullback of $A \overset{f_P}{\longrightarrow} KA_h \overset{\Id_{KA_h}}{\longleftarrow} KA_h$ is $(A,\Id_{A},f_P)$ and so $A = KA_h \glue[P][I(h)] A$. Similarly, if $I$ is a right abutment of $A$ of height $h$ and $P(h)$ is the unique left abutment of $KA_h$ of height $h$, we have $A= A \glue[P(h)][I] KA_h$. 
\end{example}

The following lemma describes gluing of algebras given by quivers with relations.

\begin{lemma}\label{lemma:gluingquivers}
Let $A=KQ_A/\cR_A$ be a rep\-re\-sen\-ta\-tion-di\-rect\-ed algebra given by a quiver with relations of the form
\begin{equation}\label{eq:the quiver description of A}
\begin{tikzpicture}[baseline={(current bounding box.center)}, scale=0.9, transform shape]
\node(A) at (0,0) {$1$};
\node(B) at (1.5,0) {$2$};
\node(C) at (3,0) {$3$};
\node(D) at (4.5,0) {$\cdots$};
\node(E) at (6,0) {$h-1$};
\node(F) at (7.5,0) {$h$,};
\node(G) at (-1.5,1) {$ $};
\node(H) at (-1.5,-1) {$ $};

\draw[->] (A) -- node[above] {$\alpha_1$} (B);
\draw[->] (B) -- node[above] {$\alpha_2$} (C);
\draw[->] (C) -- node[above] {$\alpha_3$} (D);
\draw[->] (D) -- node[above] {$\alpha_{h-2}$} (E);
\draw[->] (E) -- node[above] {$\alpha_{h-1}$} (F);

\draw[->] (G) -- (A);
\draw[->] (H) -- (A);

\draw[dotted] (-0.8, 0.5) -- (-0.8, -0.5);

\draw[pattern=north west lines] (-2.12,0) ellipse (1cm and 1.3cm);
\node (Y) at (-2.12,0) {$\mathbf{Q'_A}$};
\end{tikzpicture}
\end{equation}
where no path of the form $\alpha_i\cdots\alpha_{i+k}$ is in $\cR_A$, and $B=KQ_B/\cR_B$ be a rep\-re\-sen\-ta\-tion-di\-rect\-ed algebra given by a quiver with relations of the form
\begin{equation}\label{eq:the quiver description of B}
\begin{tikzpicture}[baseline={(current bounding box.center)}, scale=0.9, transform shape]
\node(A) at (0,0) {$1$};
\node(B) at (1.5,0) {$2$};
\node(C) at (3,0) {$3$};
\node(D) at (4.5,0) {$\cdots$};
\node(E) at (6,0) {$h-1$};
\node(F) at (7.5,0) {$h$};
\node(G) at (9,1) {$ $};
\node(H) at (9,-1) {$ $};

\draw[->] (A) -- node[above] {$\beta_1$} (B);
\draw[->] (B) -- node[above] {$\beta_2$} (C);
\draw[->] (C) -- node[above] {$\beta_3$} (D);
\draw[->] (D) -- node[above] {$\beta_{h-2}$} (E);
\draw[->] (E) -- node[above] {$\beta_{h-1}$} (F);

\draw[->] (F) -- (G);
\draw[->] (F) -- (H);

\draw[dotted] (8.3, 0.5) -- (8.3, -0.5);

\draw[pattern=north west lines] (9.65,0) ellipse (1cm and 1.3cm);
\node (Z) at (9.65,0) {$\mathbf{Q'_B}$};
\end{tikzpicture},
\end{equation}
where no path of the form $\beta_i\cdots\beta_{i+k}$ is in $\cR_B$. Then $P=P_A(1)$ is a left abutment of height $h$, $I=I_B(h)$ is a right abutment of height $h$. Then the gluing of $A$ and $B$ over $P$ and $I$ is given by the bound quiver algebra $\La=KQ_{\La}/\cR_{\La}$ where $Q_{\La}$ is the quiver
\begin{equation}\label{eq:the quiver description of Lambda}
\begin{tikzpicture}[baseline={(current bounding box.center)}, scale=0.9, transform shape]
\node(A) at (0,0) {$1$};
\node(B) at (1.5,0) {$2$};
\node(C) at (3,0) {$3$};
\node(D) at (4.5,0) {$\cdots$};
\node(E) at (6,0) {$h-1$};
\node(F) at (7.5,0) {$h$,};
\node(G) at (-1.5,1) {$ $};
\node(H) at (-1.5,-1) {$ $};
\node(G2) at (9,1) {$ $};
\node(H2) at (9,-1) {$ $};

\draw[->] (A) -- node[above] {$\lambda_1$} (B);
\draw[->] (B) -- node[above] {$\lambda_2$} (C);
\draw[->] (C) -- node[above] {$\lambda_3$} (D);
\draw[->] (D) -- node[above] {$\lambda_{h-2}$} (E);
\draw[->] (E) -- node[above] {$\lambda_{h-1}$} (F);

\draw[->] (G) -- (A);
\draw[->] (H) -- (A);
\draw[->] (F) -- (G2);
\draw[->] (F) -- (H2);

\draw[dotted] (-0.8, 0.5) -- (-0.8, -0.5);
\draw[dotted] (8.3, 0.5) -- (8.3, -0.5);

\draw[pattern=north west lines] (9.65,0) ellipse (1cm and 1.3cm);
\draw[pattern=north west lines] (-2.12,0) ellipse (1cm and 1.3cm);

\node (Y) at (-2.12,0) {$\mathbf{Q'_A}$};
\node (Z) at (9.65,0) {$\mathbf{Q'_B}$};
\end{tikzpicture},
\end{equation}
and $\cR_\La$ is generated by all elements in $\cR_A$ and $\cR_B$ and all paths starting from $\mathbf{Q'_A}$ and ending in $\mathbf{Q'_B}$, under the identifications $\alpha_i=\beta_i=\lambda_i$.
\end{lemma}

\begin{proof}
That $P$ and $I$ are left and right abutments of height $h$ follows by Proposition \ref{prop:abutmentsquiver}. The description of $\La$ as a quiver with relations is a straightforward calculation which is discussed after Lemma 3.4 in \cite{IPTZ}.
\end{proof}

Let us fix our setting for the rest of this section. First we fix two rep\-re\-sen\-ta\-tion-di\-rect\-ed algebras $A$ and $B$, such that $A$ admits a left abutment $P$ realized by $(e_i,f_i)_{i=1}^h$ and $B$ admits a right abutment $I$ realized by $(\epsilon_i,g_{i-1})_{i=1}^h$. Notice that $P$ and $I$ have the same height. Accordingly, we have footing maps $f_P: A \rightarrow KA_h$ and $g_I: B\rightarrow KA_h$. With this setting, the gluing $\La\coloneqq B\glue[P][I] A$ is defined. That is, we have the following pullback diagram
\begin{equation}\label{eq:pullback for Lambda}
    \begin{tikzpicture}[baseline={(current bounding box.center)}]
        \node (A) at (0,1.5) {$\La$};
        \node (B) at (1.5,1.5) {$B$};
        \node (C) at (0,0) {$A$};
        \node (D) at (1.5,0) {$KA_h$\nospacepunct{.}};

        \draw[->>] (A) -- node[auto] {$\psi$} (B);
        \draw[->>] (A) -- node[left] {$\phi$} (C);
        \draw[->>] (B) -- node[auto] {$g_I$} (D);
        \draw[->>] (C) -- node[below] {$f_P$} (D);
    \end{tikzpicture}
\end{equation}

Suggestively for what follows, we write
\begin{alignat*}{2}
1_A &=e_{h-l+1}+\cdots+e_0+ &&e_1+\cdots+e_h  \\
1_B &= &&\epsilon_1+\cdots+\epsilon_h+\epsilon_{h+1}+\cdots+\epsilon_m 
\end{alignat*}
where all $e_i$'s and $\epsilon_i$'s are primitive orthogonal idempotents. Furthermore, when clear from context, we will use the notation $1_C$ for both $e\ab=\sum_{i=1}^h e_i$ and $\epsilon\ab=\sum_{i=1}^h\epsilon_i$. 

We also fix presentations of $A$ and $B$ as in Proposition \ref{prop:abutmentsquiver}. Specifically we have an isomorphism $\Phi_A:KQ_A/\cR_A\overset{\sim}{\longrightarrow} A$ where $Q_A$ is as in (\ref{eq:the quiver description of A}) and such that no path of the form $\alpha_i\cdots\alpha_{i+k}$ is in $\cR_A$. Without loss of generality, and by Proposition \ref{prop:abutmentsquiver}(a), we denote the vertices of $Q_A$ by $\{h-l+1,\dots,0,1,\dots,h\}$ and the idempotent of $KQ_A/\cR_A$ corresponding to the vertex $i\in (Q_A)_0$ by $a_i$ so that $\Phi_A(a_i)=e_i$ and $\Phi_A(\alpha_i)=f_i$. In particular, the $KQ_A/\cR_A$-module $P(1)$ is a left abutment of height $h$ with footing map $f_P\circ\Phi_A:KQ_A/\cR_A\to KA_h$.

Similarly we have an isomorphism $\Phi_B:KQ_B/\cR_B\overset{\sim}{\longrightarrow} B$ where $Q_B$ is as in (\ref{eq:the quiver description of B}) and such that no path of the form $\beta_i\cdots\beta_{i+k}$ is in $\cR_B$. Without loss of generality, and by Proposition \ref{prop:abutmentsquiver}(b), we denote the vertices of $Q_B$ by $\{1,\dots,h,h+1,\dots,m\}$ and the idempotent of $KQ_B/\cR_B$ corresponding to the vertex $i\in (Q_B)_0$ by $b_i$ so that $\Phi_B(b_i)=\epsilon_i$ and $\Phi_B(\beta_i)=g_i$. In particular, the $KQ_B/\cR_B$-module $I(h)$ is a right abutment of height $h$ with footing map $g_I\circ\Phi_B:KQ_B/\cR_B\to KA_h$.

With this setting, the gluing $KQ_B/\cR_B\glue[P(1)][I(h)] KQ_A/\cR_A$ is defined. That is, we have the following pullback diagram
\begin{equation}\label{eq:pullback for KQ Lambda}
        \begin{tikzpicture}[baseline={(current bounding box.center)}]
        \node (A) at (0,1.5) {$KQ_{\La}/\cR_{\La}$};
        \node (B) at (3,1.5) {$KQ_B/\cR_B$};
        \node (C) at (0,0) {$KQ_A/\cR_A$};
        \node (D) at (3,0) {$KA_h$\nospacepunct{.}};

        \draw[->>] (A) -- node[auto] {$\psi'$} (B);
        \draw[->>] (A) -- node[left] {$\phi'$} (C);
        \draw[->>] (B) -- node[auto] {$g_I\circ \Phi_B$} (D);
        \draw[->>] (C) -- node[below] {$f_P\circ\Phi_A$} (D);
    \end{tikzpicture}
\end{equation}
where $Q_{\La}$ and $\cR_{\La}$ are as in Lemma \ref{lemma:gluingquivers}. In particular, the vertices of $Q_{\La}$ are 
\[\{h-l+1,\dots,0,1,\dots,h,h+1,\dots,m\}\]
and we denote by $l_i$ the idempotent of $KQ_{\La}/\cR_{\La}$ corresponding to the vertex $i\in(Q_{\La})_0$. 

\begin{proposition}\label{prop:presentation of Lambda}
There exists a presentation $\Phi_{\La}:KQ_{\La}/\cR_{\La}\overset{\sim}{\longrightarrow} \La$ such that $\psi\circ\Phi_{\La}=\Phi_B\circ\psi'$, and $\phi\circ\Phi_{\La}=\Phi_A\circ\phi'$, and, moreover
\[\Phi_{\La}(l_i) = \begin{cases} (e_i,0) &\mbox{ for $h-l+1 \leq i \leq 0$,} \\ (e_i,\epsilon_i) &\mbox{ for $1\leq i \leq h$,} \\ (0,\epsilon_i) &\mbox{ for $h+1\leq i \leq m$.}\end{cases}\]
\end{proposition}

\begin{proof}
By the pullback diagrams (\ref{eq:pullback for Lambda}) and (\ref{eq:pullback for KQ Lambda})  we get a commutative diagram
\begin{equation}\label{diagram:isomorphism}\begin{tikzcd}
KQ_{\La}/\cR_{\La} \arrow[rr, two heads, "\psi'"] \arrow[dd, two heads, swap, "\phi'"] \arrow[dr, "\rotatebox{0}{\(\sim\)}", "\Phi_{\La}"'] & & KQ_{B}/\cR_{B} \arrow[d, "\rotatebox{90}{\(\sim\)}\;\Phi_B"] \\ 
& \La \arrow[r, two heads, "\psi"] \arrow[d, two heads, "\phi"] & B \arrow[d, two heads, "g_I"] \\
KQ_A/\cR_A \arrow[r, "\sim", "\Phi_A"'] & A \arrow[r, two heads, swap, "f_P"] & KA_h,
\end{tikzcd}
\end{equation}
where $\Phi_{\La}$ is the morphism induced by the universal property of the pullback. That $\Phi_{\La}$ is an isomorphism follows by the uniqueness of the pullback up to isomorphism. It remains to compute $\Phi_{\La}(l_i)$. If $l_i$ is the idempotent of $KQ_{\La}/\cR_{\La}$ corresponding to the vertex $i$ for some $h-l+1\leq i \leq m$, then by Lemma \ref{lemma:gluingquivers} and our assumptions on $\Phi_A$ and $\Phi_B$ we have
\[\phi\circ\Phi_{\La}(l_i) = \Phi_A\circ\phi'(l_i) = \begin{cases} \Phi_A(a_i) &\mbox{ for $h-l+1 \leq i \leq h$,} \\ \Phi_A(0) &\mbox{ for $h+1\leq i \leq m$,}\end{cases}=\begin{cases} e_i &\mbox{ for $h-l+1 \leq i \leq h$,} \\ 0 &\mbox{ for $h+1\leq i \leq m$,}\end{cases}\]
\[\psi\circ\Phi_{\La}(l_i) = \Phi_B\circ\psi'(l_i) =  \begin{cases} \Phi_B(0) &\mbox{ for $h-l+1 \leq i \leq 0$,} \\ \Phi_B(b_i) &\mbox{ for $1\leq i \leq m$.} \end{cases}=\begin{cases} 0 &\mbox{ for $h-l+1 \leq i \leq 0$,} \\ \epsilon_i &\mbox{ for $1\leq i \leq m$.} \end{cases}\]
By the fact that $\La$ is a pullback, it follows that 
\[\Phi_{\La}(l_i) = \begin{cases} (e_i,0) &\mbox{ for $h-l+1 \leq i \leq 0$,} \\ (e_i,\epsilon_i) &\mbox{ for $1\leq i \leq h$,} \\ (0,\epsilon_i) &\mbox{ for $h+1\leq i \leq m$.}\end{cases}\]
\end{proof}

In the following we set $\vare_i \coloneqq \Phi_{\La}(l_i)$. 

\begin{lemma}\label{lemma:primitiveidempotents}
The set $\{\vare_i\}_{i=h-l+1}^m$ is a complete set of primitive orthogonal idempotents of $\La$.
\end{lemma}

\begin{proof}
Immediate since $\Phi_{\La}$ is an isomorphism and $\{l_i\}_{i=h-l+1}^m$ is a complete set of primitive orthogonal idempotents of $KQ_{\La}/\cR_{\La}$.
\end{proof}

We further set
\[\vare_A = \vare_{h-l+1} + \cdots + \vare_h,\;\; \vare_B=\vare_1+\cdots+\vare_m,\;\; \vare_{C}=\vare_1+\cdots+\vare_h,\]
\[ \vare_{A'}=\vare_A-\vare_C, \;\; \vare_{B'}=\vare_B-\vare_C,\]
and we denote $l_X\coloneqq \Phi_{\Lambda}^{-1}(\vare_X)$ for $X\in\{A,B,C,A',B'\}$. We also introduce some notation to simplify expressions later in this section. We will denote $[A',B']:=\{A',1,\dots,h,B'\}$ and we order the set $[A',B']$ by $A'<1<\cdots < h < B'$. In particular, we have 
\[\vare_{\La}=\sum_{i\in [A',B']} \vare_i. \]

For convenience, let us recall the functors defined by $\phi$ and $\psi$ on the corresponding module categories. The epimorphism $\phi:\La\to A$ induces the restriction of scalars functor $\phi_{\ast}:\m A\to \m \La$, which has a left adjoint $\phi^{\ast}(-)=-\otimes_{\La}A$ and a right adjoint $\phi^{!}(-)=\Hom_{\La}(A,-)$; similarly for $\psi$. Collectively, we have the functors
\[\begin{tikzpicture}
\node (A) at (0,0) {$\m\La$};
\node (B) at (3,0) {$\m A$};
\node (C) at (3.55,-0.135) {$,$};

\draw[->] (B) -- node[above] {$\phi_{\ast}$} (A); 
\draw[->] (A) to [out=30,in=150]  node[auto]{$\phi^{\ast}$} (B);
\draw[->] (A) to [out=-30,in=-150]  node[below]{$\phi^{!}$} (B);
\end{tikzpicture} 
\qquad
\begin{tikzpicture}
\node (A) at (0,0) {$\m\La$};
\node (B) at (3,0) {$\m B$.};

\draw[->] (B) -- node[above] {$\psi_{\ast}$} (A); 
\draw[->] (A) to [out=30,in=150]  node[auto]{$\psi^{\ast}$} (B);
\draw[->] (A) to [out=-30,in=-150]  node[below]{$\psi^{!}$} (B);
\end{tikzpicture}\]

Since $\phi$ and $\psi$ are epimorphisms, it follows that $\phi_\ast$ and $\psi_\ast$ are full and faithful. Our aim in this section is to show that if $\La = B \glue A$ then $\m\La = \left(\m B\right) \glue \left(\m A\right)$, where we identify $\m A$ and $\m B$ with their images under $\phi_\ast$ and $\psi_\ast$ respectively. To this end, we need to verify that $\m A$ and $\m B$ satisfy the conditions of Definition \ref{def:catglue}.

\begin{definition}\label{def:supported}
Let $X\in\m\La$. We will say that $X$ \emph{is supported in $A$} if $X\cong \phi_{\ast}(M)$ for some $M\in\m A$ and that $X$ \emph{is supported in $B$} if $X\cong\psi_{\ast}(N)$ for some $N\in\m B$.
\end{definition}

Recall that the module category of a bound quiver algebra $\m (KQ/\cR)$ is equivalent to the category $\rep(Q,\cR)$ of fi\-nite-di\-men\-sion\-al representations of $Q$ bound by $\cR$. We consider this equivalence as an identification. The \emph{support} of a representation $M=(M_i,\phi_{\alpha})$ of $Q$ bound by $\cR$ is the set
\[\supp(M)\coloneqq \{i\in Q_0 \mid M_i\neq 0\}.\]
We have the following equivalent characterizations of a module being supported in $A$ or $B$.

\begin{lemma}\label{lemma:supported}
Let $X\in \m \La$. The following are equivalent. 
\begin{itemize}
    \item[(a)] $X$ is supported in $A$ (respectively $B$).
    \item[(b)] $\phi_{\ast}\phi^!(X)\cong X$ (respectively $\psi_{\ast}\psi^{\ast}(X)\cong X$).
    \item[(c)] $\supp(\Phi_{\La\ast}(X))\subseteq (Q_A)_{0}$ (respectively $\supp(\Phi_{\La\ast}(X))\subseteq (Q_B)_{0}$).
    \item[(d)] For all $x\in X$ we have $x\vare_A=x$ (respectively $x\vare_B=x)$.
\end{itemize}
\end{lemma}

\begin{proof}
Let us prove the result for $A$; the result for $B$ is similar. Assume that (a) holds and we will show that (b) holds. Then $X\cong \phi_{\ast}(M)$ for some $M\in\m A$. Since $(\phi_\ast,\phi^!)$ form an adjoint pair where the left adjoint is full and faithful, the unit of the adjunction is a natural isomorphism (see \cite[page 90]{CWM}). Hence
\[\phi_{\ast}\phi^!(X)\cong \phi_{\ast}\phi^!\phi_{\ast}(M) \cong \phi_{\ast}(M) \cong X,\]
as required.

Assume that (b) holds and we will show that (c) holds. Using diagram (\ref{diagram:isomorphism}) we have
\begin{align*}
    \Phi_{\La\ast}(X) &\cong \Phi_{\La\ast} \circ \phi_{\ast}\circ\phi^!(X) \\
    &=\left(\phi\circ\Phi_{\La}\right)_{\ast}\circ\phi^{!}(X) \\
    &=\left(\Phi_A\circ\phi'\right)_{\ast}\circ\phi^!(X) \\
    &=\phi'_{\ast}\circ\Phi_{A\ast}\circ\phi^!_{\ast}(X).
\end{align*} 
Here $\Phi_{A\ast}\circ \phi^!_{\ast}(X)$ is a $KQ_A/\cR_A$-module and it is well known that, viewed as a representation of $Q_A$ bound by $\cR_A$, the restriction of scalars functor $\phi_{\ast}':\m (KQ_A/\cR_A)\to \m(KQ_{\La}/\cR_{\La})$ maps it to the representation obtained by putting zero in all vertices and arrows not in $Q_A$. Hence
\[\supp(\Phi_{\La\ast}(X)) = \supp\left(\phi'_{\ast}(\Phi_{A\ast}\circ\phi^!_{\ast}(X))\right)\subseteq (Q_A)_{0},\]
as required.

Assume that (c) holds and we will show that (d) holds. For every $x\in X$ we have by definition of the restriction of scalars functor and since $\Phi_A$ is an isomorphism that
\[x \vare_A = x\Phi_A^{-1}(\vare_A) = xl_A=x,\]
where the last equality follows since by the assumption we have that the representation $\Phi_{\La\ast}(X)$ of $Q_{\La}$ bound by $\cR_{\La}$ is supported in $(Q_A)_{0}$.

Assume that (d) holds and we will show that (a) holds. By assumption we have that $x\vare_A=x$ for all $x\in X$. Since $\Phi_{\La\ast}:\m\La\to \m(KQ_{\La}/\cR_{\La})$ is an equivalence of categories, it follows that for every $x\in\Phi_{\La\ast}(X)$ we have $xl_A=x$. By Lemma \ref{lemma:gluingquivers} it follows that this defines a representation $M$ of $Q_{A}$ bound by $\cR_A$ by setting $M_i=\left(\Phi_{\La\ast}(X)\right)_i$ for $i\in (Q_A)_{0}$ and $M_{\alpha}=\left(\Phi_{\La\ast}(X)\right)_{\alpha}$ for $\alpha\in (Q_A)_1$. By construction, we have $\phi'_{\ast}(M)\cong \Phi_{\La\ast}(X)$. Then (a) follows by the commutativity of the left square in (\ref{diagram:isomorphism}) and the fact that $\Phi_{\La}$ and $\Phi_A$ are isomorphisms. 
\end{proof}

We have the following immediate corollary.

\begin{corollary} \label{cor:relatives}
Let $M\in \m\La$ and suppose that $M$ is supported in $A$ (respectively $B$).
\begin{itemize}
\item[(a)] All submodules of $M$ are supported in $A$ (respectively $B$). In particular, $\rad M$ and $\soc M$ are supported in $A$ (respectively $B$).

\item[(b)] All quotient modules of $M$ are supported in $A$ (respectively $B$). In particular, $\topp M$ is supported in $A$ (respectively $B$).
\end{itemize}
\end{corollary}

\begin{proof}
Immediate by Lemma \ref{lemma:supported}(d).
\end{proof}

We can now identify the indecomposable projective and injective $\La$-modules.

\begin{proposition} \label{prop:projinj}
\begin{itemize}
\item[(a)] $\vare_i\La$ is supported in $A$ for $h-l+1\leq i \leq 0$ and is supported in $B$ for $1 \leq i \leq m$.
\item[(b)] $D(\La \vare_i)$ is supported in $A$ for $h-l+1 \leq i \leq h$ and is supported in $B$ for $h+1 \leq i \leq m$.
\end{itemize}
\end{proposition}

\begin{proof}
We only prove (a); (b) is similar. It is enough to show that the support of $\Phi_{\La\ast}(\epsilon_i\La)=P_{KQ_{\La}/\cR_{\La}}(i)$ is included in $(Q_A)_0$ for $h-l+1\leq i\leq 0$ and is included in $(Q_B)_0$ for $1\leq i \leq m$, which follows immediately by Lemma \ref{lemma:gluingquivers}.
\end{proof}

\begin{corollary}\label{cor:numberofprojandinj}
Let $s_A$ be the number of simple projective $A$-modules up to isomorphism and $t_A$ be the number of simple injective $A$-modules up to isomorphism. Similarly define $s_B$, $t_B$, $s_{\La}$ and $t_{\La}$. Then 
\[ (s_{\La},t_{\La}) = (s_A+s_B-1,t_A+t_B-1). \]
\end{corollary}

\begin{proof}
Let us show that $s_{\La}=s_A+s_B-1$; that $t_{\La}=t_A+t_B-1$ is proved similarly. Let $h-l+1\leq i \leq m$. If $1\leq i \leq m$, then $\vare_i \La$ is a simple projective $\La$-module if and only if $\epsilon_i B$ is a simple projective $B$-module by Proposition \ref{prop:projinj}(a). Since $\left\{\epsilon_i\right\}_{1\leq i \leq m}$ is a complete set of primitive orthogonal idempotents for $B$, it follows that there are exactly $s_B$ simple projective $\La$-modules $\vare_i\La$ for $1\leq i \leq m$.

Similarly, if $h-l+1\leq i \leq 0$ then $\vare_i\La$ is a simple projective $\La$-module if and only if $e_i A$ is a simple projective $A$-module. By Proposition \ref{prop:triangles}(a) it follows that if $1\leq i \leq h$, then $e_i A$ is simple if and only if $i=h$. Since $\{e_i \}_{h-l+1\leq 1 \leq h}$ is a complete set of primitive orthogonal idempotents for $A$, it follows that there are exactly $s_A-1$ simple projective $\La$-modules $\vare_i\La$ for $h-l+1\leq i \leq 0$.

Finally, since $\{\vare_i\}_{h-l+1\leq i \leq m}$ is a complete set of primitive orthogonal idempotents for $\La$ by Lemma \ref{lemma:primitiveidempotents}, it follows that there are exactly $s_B+(s_A-1)$ simple projective $\La$-modules, as required.
\end{proof}

\begin{corollary} \label{cor:radical}
For every $m\in\rad(\vare_h \La)$ we have $m\vare_{B'}=m$.
\end{corollary}

\begin{proof}
By Lemma \ref{lemma:gluingquivers} it follows that for every $m'\in\Phi_{\La\ast}(\rad(\vare_h\La)=\rad\left( P_{KQ_{\La}/\cR_{\La}}(h)\right)$ we have $m'l_{B'}=m'$. Since $\Phi_{\La\ast}$ is an equivalence of categories and $\Phi_{\La}(l_{B'})=\vare_{B'}$, the result follows.
\end{proof}

The following lemma contains important information about the directedness that is required to prove $\m\La=\left(\m B\right) \glue \left( \m A\right)$.

\begin{lemma} \label{lemma:techiii}
Let $i,j\in [A',B']$.
\begin{itemize}
\item[(i)] $\vare_{A'} \La \vare_{B'} = 0$ and $\vare_{B'} \La \vare_{A'}=0$.
\item[(ii)] If $1\leq  i,j \leq h$, then $\dim_K(\vare_i \La \vare_j) = \begin{cases} 1 & \mbox{ if $i\leq j$,} \\ 0 &\mbox{ if $i>j$.} \end{cases}$
\item[(iii)] If $i < j$, then $\vare_j\La \vare_i = 0$.
\item[(iv)] $\La \cong \bigoplus\limits_{A'\leq i \leq j \leq B'} \vare_i \La \vare_j$.
\item[(v)] If $1\leq j\leq k \leq i \leq h$, then $\vare_j \La \vare_i \cong \vare_j \La \vare_k \La \vare_i$. 
\item[(vi)] If $1\leq i \leq h$, then $\vare_i \La \vare_{B'}\cong \vare_i \La \vare_h \La \vare_{B'}$.
\end{itemize}
\end{lemma}

\begin{proof} Immediate from Lemma \ref{lemma:gluingquivers} by using the isomorphism $\Phi_{\La}$ to transfer computations to the bound quiver algebra $KQ_{\La}/\cR_{\La}$ and using the fact that $\Phi^{-1}_{\La}(\vare_x)=l_x$ for $x\in \left(Q_{\La}\right)_0\cup\{A,B,A',B'\}$.
\end{proof}

The following proposition is the most important step in showing the main result in this section.

\begin{proposition}\label{prop:decomposition}
Let $M\in \m\La$. Then $M\cong X \oplus Y$ for some $X,Y\in \m\La$ where $X$ is supported in $A$ and $Y$ is supported in $B$. 
\end{proposition}

\begin{proof}
Let us pick a sequence of nonzero morphisms
\[\vare_1 \La \overset{p_1}{\longleftarrow} \cdots \overset{p_{h-1}}{\longleftarrow} \vare_h \La \overset{p_h}{\longleftarrow\joinrel\rhook} \rad(\vare_h\La),\]
where $p_h$ corresponds to the inclusion of the radical of $\vare_h\La$. By applying $\Hom_\La(-,M)$, and since 
\[\Hom_{\La}(\vare_i \La, M) \cong M\vare_i,\]
we get the commutative diagram
\[\begin{tikzpicture}
\node (A1) at (0,1.5) {$\Hom_{\La}(\vare_1\La,M)$};
\node (A2) at (2.5,1.5) {$\cdots$};
\node (A3) at (5.5,1.5) {$\Hom_{\La}(\vare_h\La,M)$};
\node (A4) at (9.5,1.5) {$\Hom_{\La}(\rad(\vare_h\La),M)$\nospacepunct{,}};

\node (B1) at (0,0) {$M\vare_1$};
\node (B2) at (2.5,0) {$\cdots$};
\node (B3) at (5.5,0) {$M\vare_h$};

\draw[->] (A1) -- node[above] {$-\circ p_1$} (A2);
\draw[->] (A2) -- node[above] {$-\circ p_{h-1}$} (A3);
\draw[->] (A3) -- node[above] {$-\circ p_{h}$} (A4);

\draw[->] (B1) -- node[above] {$ q_1$} (B2);
\draw[->] (B2) -- node[above] {$ q_{h-1}$} (B3);

\draw[->] (B1) -- node[left] {$s_1$} (A1);
\draw[->] (B3) -- node[left] {$s_h$} (A3);

\draw[->] (B3) -- node[below] {$q_h$} (A4);

\end{tikzpicture}\]
where 
\begin{itemize}
\item[$\bullet$] $s_i(m\vare_i) = p_{m}$ with $p_m(\vare_i \la)=m\vare_i\la$,
\item[$\bullet$] $s_i^{-1}(\chi: \vare_i\La \rightarrow M) = \chi(\vare_i)=\chi(\vare_i)\vare_i$, 
\item[$\bullet$] $q_i(m\vare_i)=mp_i(\vare_{i+1})\vare_{i+1}$ for $1\leq i \leq h-1$, 
\item[$\bullet$] $[q_h(m\vare_h)](\vare_h\la)=m\vare_h\la$ for $\la\in \rad(\La)$.
\end{itemize}

Let $U_h = \ker q_h$ and for $1\leq i \leq h-1$ define $U_{i}=q_{i}^{-1}(U_{i+1})$. Moreover for $1\leq i \leq h$ let $V_i$ be such that $q_{i-1}(V_i)\subseteq V_i$ and $M\vare_i = U_i\oplus V_i$. Set $U=\bigoplus_{i=1}^h U_i$, $V=\bigoplus_{i=1}^h V_i$, $X=M\vare_{A'}\oplus U$ and $Y= V \oplus M\vare_{B'}$. Then clearly $M\cong X\oplus Y$ as vector spaces and it remains to show that both $X$ and $Y$ are submodules of $M$ since by construction it is clear that $X$ is supported in $A$ and $Y$ is supported in $B$.

Let us start by showing that $X=M\vare_{A'}\oplus U$ is a submodule of $M$. First let $m\vare_{A'} \in M\vare_{A'}$ and $\la \in \La$. Then by Lemma \ref{lemma:techiii}(iv) we have $\la = \vare_x \la \vare_y$ with $A'\leq x \leq y \leq B'$. We need to show that $m\vare_{A'}\la=m\vare_{A'}\vare_x\la\vare_y \in X$ for all $A'\leq x \leq y \leq B'$. If $A'<x$ then $m(\vare_{A'}\vare_x)\la\vare_y=m0\la\vare_y=0$, so we assume that $x=A'$. If $y=A'$, then $m\vare_{A'}\la\vare_{A'} \in M\vare_{A'}$, while if $y=B'$, then $m\vare_{A'}\la\vare_{B'}=0$ by Lemma \ref{lemma:techiii}(i). It remains to check the case $1\leq y \leq h$. Since $m\vare_{A'}\la\vare_y \in M\vare_y$, it is enough to show that $q_h\circ \cdots \circ q_y (m\vare_{A'}\vare_y)=0$ since then $m\vare_{A'}\la\vare_y\in U_y$. We have $s_y(m\vare_{A'}\la\vare_y)=p_{m\vare_{A'}\la}$ and for any $n\in \rad(\vare_h\La)$ we have $n=n\vare_{B'}$ by Corollary \ref{cor:radical}. Hence for any $n\in \rad(\vare_h\La)$ we have
\begin{align*} [q_h\circ \cdots \circ q_y  (m\vare_{A'}\la\vare_y)](n) &= s_y(m\vare_{A'}\la\vare_y)\circ p_y \circ \dots \circ p_h (n \vare_{B'}) \\
&=p_{m\vare_{A'}\la}\circ p_y \circ \dots \circ p_h (n \vare_{B'}) \\
&= m\vare_{A'} \la p_y\circ \dots \circ p_h(n)\vare_{B'} \\
&=0,
\end{align*}
where the last equality comes from Lemma \ref{lemma:techiii}(i). 

Next let $u\vare_i\in U_i$ and $\lambda\in \La$. Again by Lemma \ref{lemma:techiii}(iv) we have $\la=\vare_x\la \vare_y$ with $x\leq y \leq B'$ and it is enough to show that $u\vare_i \la=u\vare_i\vare_x\la\vare_y \in X$. We can assume that $x=i$ since otherwise $u(\vare_i\vare_x)\la\vare_y=u0\la\vare_y=0$. If $y<B'$, then $u\vare_i \la \vare_y \in M\vare_y$ and it is enough to show $u\vare_i\la\vare_y \in U_y$. Since $\vare_i \La \vare_y \cong \Hom_{\La}(\vare_y \La, \vare_i \La) \cong K$, there exists some $k\in K$ such that
\[p_{i}\circ \cdots \circ p_{y-1}(\vare_y)=k\vare_i\la\vare_y.\]
Then
\begin{align*}
q_{y-1}\circ\cdots\circ q_i(u\vare_i) &= s_y^{-1}(s_i(u\vare_i)\circ p_i \circ \cdots \circ p_{y-1}) \\
&= s_y^{-1}(p_u\circ p_i \circ\cdots \circ p_{y-1}) \\
&= p_u\circ p_i \circ\cdots \circ p_{y-1}(\vare_y) \\
&=u(p_i \circ\cdots \circ p_{y-1}(\vare_y)) \\
&=u (k\vare_i \la \vare_ y) \\
&=k (u\vare_i \la \vare_y).
\end{align*}
Since the left hand side is in $U_y$ by construction, we have $k(u\vare_i\la\vare_y)\in U_y$ as required. If $y=B'$ then by Lemma \ref{lemma:techiii}(vi) we have $u\vare_i\lambda\vare_{B'}=u\vare_i\lambda_1\vare_h\lambda_2\vare_{B'}$. Using  the same argument as before, we can show that $u\vare_i\lambda_1\vare_h\in U_h$. We claim that then $u\vare_i\lambda_1\vare_h\lambda_2\vare_{B'}=0\in X$, which is enough to show that $X$ is a submodule of $M$. To show this claim it is enough to show that $U_h$ satisfies $U_h\La\vare_{B'}=0$. To show this, let $u\vare_h\in U_h$ and $\la\in \La$. If $\la\in \rad(\La)$ then by construction we have $u\vare_h \la = [q_h(u\vare_h)](\vare_h\la)=0$. In general, let $\la=\sum_{i=h-l+1}^mc_i\vare_{i}+\mu$ with $\mu\in \rad(\La)$, since $\vare_i$ are the only elements in $\La$ that act nontrivially on simple $\La$-modules. Then 
\[u\vare_h\la \vare_{B'} = u\vare_h\left(\sum_{i=h-l+1}^mc_i\vare_{i}\right)\vare_{B'}+u\vare_h\mu\vare_{B'} = u0+0=0,\]
since $\vare_h\left(\sum_{i=h-l+1}^mc_i\vare_{i}\right)\vare_{B'}=c_h\vare_h\vare_{B'}=0$ and $\mu\in\rad(\La)$. Hence, the claim is proved and $X$ is a submodule of $M$.

It remains to show that $Y= V \oplus M\vare_{B'}$ is a submodule of $M$. First, let $v\vare_i\in V_i$ and $\lambda\in \La$. As in the previous cases, we can assume that $\lambda=\vare_i \la \vare_y$ with $i\leq y \leq B'$. If $y=B'$ then $v\vare_i\la\vare_{B'}\in M\vare_{B'}$ and so $v\vare_i\la\in Y$. If $y\leq h$, using the same argument as in the previous case we can show that 
\[q_{y-1}\circ \cdots \circ q_i(v\vare_i)= kv\vare_i\la\vare_y\]
for some $k\in K$. Since by construction we have $q_{y-1}\circ \cdots \circ q_i(V_i)\subseteq V_y$, it follows that $v\vare_i\la\vare_y\in V_y$ and so again $v\vare_i\la\in Y$.

Finally, if $m\vare_{B'}\in M\vare_{B'}$ and $\lambda\in \La$ we have $m\vare_{B'}\lambda=0$ unless $\lambda=\vare_{B'}\la\vare_{B'}$, in which case $m\vare_{B'}\la\in M{\vare_{B'}}$. This shows that $Y$ is a submodule of $M$ and concludes the proof. 
\end{proof}

\begin{corollary}\label{cor:choices}
If $M\in \m\La$ is indecomposable, then $M$ is supported in $A$ or $M$ is supported in $B$.
\end{corollary}

\begin{proof}
Immediate by Proposition \ref{prop:decomposition}.
\end{proof}

\begin{lemma}\label{lemma:techproj} Let $M\in \m\La$.
\begin{itemize} \item[(a)] Assume that $\Hom_{\La}(\vare_i\La, M) \neq 0$. If $M$ is supported in $A$, then $h-l+1\leq i \leq h$. If $M$ is supported in $B$, then $1\leq i \leq m$.
\item[(b)] Assume that $\Hom_{\La}(M,D(\La\vare_i))\neq 0$. If $M$ is supported in $A$, then $h-l+1\leq i \leq h$. If $M$ is supported in $B$, then $1\leq i \leq m$.
\end{itemize}
\end{lemma}

\begin{proof}
We only prove (a); (b) is similar. Since $\Hom_{\La}(\vare_i\La,M)\cong M\vare_i \cong \Phi_{\La\ast}(M)l_i$, the result follows immediately by Lemma \ref{lemma:supported} by noting that $\Phi_{\La\ast}(M)l_i=(\Phi_{\La\ast}(M))_i$ when viewing $M$ as a representation of $Q_{\La}$ bound by $\cR_{\La}$.
\end{proof}

\begin{lemma}\label{lemma:intersection}
Let $M\in \m\La$ be indecomposable. Then the following are equivalent.
\begin{itemize}
\item[(a)]$M$ is supported in both $A$ and $B$, 

\item[(b)] $M\in \phi_{\ast}(\cPD)$,

\item[(c)] $M\in \psi_{\ast}(\cDI)$.
\end{itemize}

\end{lemma}

\begin{proof} 
By Proposition \ref{prop:triangles} it easily follows that for an $A$-module (respectively $B$-module) $X$ we have $\supp(X)\subseteq \{1,\dots,h\}$ if and only if $X\in \cF_P$ (respectively $X\in\cG_I$). Then by Lemma \ref{lemma:supported} we have that $M$ is supported in both $A$ and $B$ if and only if $\supp(\Phi_{\La\ast}(M))\subseteq (Q_A)_0\cap (Q_B)_0 = \{1,\dots,h\}$, which holds if and only if $\Phi_{\La\ast}(M) \cong \phi'_{\ast}(X)$ for some $X\in\m (KQ_{A}/\cR_A)$ and $\Phi_{\La\ast}(M)\cong \psi'_{\ast}(Y)$ for some $Y\in\m (KQ_B/\cR_B)$. The result follows by the commutativity of the diagram (\ref{diagram:isomorphism}) and the fact that $\Phi_A$, $\Phi_B$ and $\Phi_{\La}$ are isomorphisms.
\end{proof}

To simplify notation in the rest of this section, let us denote the subcategories $\phi_\ast(\m A)\subseteq \m\La$ and $\psi_\ast(\m B)\subseteq \m\La$ by $(\m A)_\ast$ and $(\m B)_\ast$ respectively. Now we are ready to show the main result for this section. 

\begin{proposition}\label{prop:theyareglued}
If $\La=B \glue A$, then $\m \La=\left(\m B\right)_\ast \glue \left(\m A\right)_\ast$.
\end{proposition}

\begin{proof}
We need to check conditions (i)--(iv) of Definition \ref{def:catglue} with $\cA=(\m A)_\ast$ and $\cB=(\m B)_\ast$. Condition (i) is immediate. Condition (ii) follows from Proposition \ref{prop:decomposition}, since if $M\in\m\La$ is indecomposable, then either $M$ is supported in $A$ or $M$ is supported in $B$ by Corollary \ref{cor:choices}. 

For condition (iii) let $M\in (\m A)_\ast \setminus (\m B)_\ast$ be indecomposable and assume by way of contradiction that for some $N\in (\m B)_\ast$ there exists a nonzero morphism $g:M\rightarrow N$. In particular, we have that $M\twoheadrightarrow \im g\hookrightarrow N$ and so $\im g\in \left(\m A \right)_{\ast} \cap \left(\m B\right)_{\ast}=\phi_\ast(\cPD)$, where the last equality follows by Lemma \ref{lemma:intersection}. Since both $M$ and $\im g$ are in the image of $\phi_{\ast}$ and since $\phi_{\ast}$ is full and reflects epimorphisms (as $\phi_{\ast}$ is faithful), it follows that $M\twoheadrightarrow \im g$ is the image of an epimorphism in $\m A$. By the (functorial) isomorphism of Lemma \ref{lemma:supported}(b), applying $\phi^{!}$ to $M\twoheadrightarrow \im g$ recovers this epimorphism. Hence there exists an epimorphism $\phi^{!}(M) \twoheadrightarrow \phi^{!}(\im g)$ in $\m A$ with $\phi^{!}\left(\im g\right)\in \cPD$. By Proposition \ref{prop:triangles}, this means that $\phi^{!}(M)$ is in $\cPD$. But by Lemma \ref{lemma:intersection} this implies that $M\in \phi_\ast(\cPD)= \left(\m A\right)_{\ast} \cap \left(\m B\right)_{\ast}$, which contradicts $M\in (\m A)_{\ast} \setminus (\m B)_{\ast}$.

For condition (iv) notice that any $g:N\rightarrow M$ with $N\in (\m B)_{\ast}$ and $M\in (\m A)_{\ast}$ factors as $N\twoheadrightarrow \im g \hookrightarrow M$ and $\im g$ is in $\left(\m A\right)_{\ast} \cap \left(\m B\right)_{\ast}$ by Corollary \ref{cor:relatives}.  
\end{proof}

The following corollaries describe the representation theory of $\La$ in terms of the representation theory of $A$ and $B$ and will be particularly useful in the following section.

\begin{corollary}\label{cor:Laisdirected}
$\La$ is rep\-re\-sen\-ta\-tion-di\-rect\-ed.
\end{corollary}

\begin{proof}
Let $f_0:Y_0\rightarrow Y_1$ be a nonzero morphism between indecomposable modules $Y_0,Y_1\in \m\La$. We need to show that there exists no chain of nonzero nonisomorphisms $f_i:Y_i\rightarrow Y_{i+1}$, $1\leq i \leq k$ with $Y_{k+1}\cong Y_0$. 

Assume by way of contradiction that such a chain exists. If all $Y_i$ are supported in $B$, then this gives rise to a chain of indecomposable $B$-module nonzero nonisomorphisms $\psi^{\ast}(f_i):\psi^{\ast}(Y_i)\rightarrow \psi^{\ast}(Y_{i+1})$ for $0\leq i\leq k$ such that $\psi^{\ast}(Y_{k+1})\cong  \psi^{\ast}(Y_0)$, which contradicts the fact that $B$ is rep\-re\-sen\-ta\-tion-di\-rect\-ed. 

Hence there exists some minimal $j$ such that $Y_{j}$ is not supported in $B$. Then $Y_j$ is supported in $A$ by Corollary \ref{cor:choices}. Since $Y_j$ is supported in $A$ and not in $B$, and since $\m\La = \left(\m B\right)_{\ast} \glue \left(\m A\right)_{\ast}$, it follows that $Y_i$ is supported in $A$ and not in $B$ for all $i\geq j$ by Definition \ref{def:catglue}(iii). Since $Y_{k+1}\cong Y_0$ and $j$ was minimal, it follows that $j=0$ and that all $Y_i$ are supported in $A$. Then this gives rise to a chain of indecomposable $A$-module nonzero nonisomorphisms $\phi^{!}(f_i):\phi^{!}(Y_i)\rightarrow \phi^{!}(Y_{i+1})$ for $0\leq i \leq k$ such that $\phi^{!}(Y_{k+1})\cong \phi^{!}(Y_0)$, which contradicts the fact that $A$ is rep\-re\-sen\-ta\-tion-di\-rect\-ed.
\end{proof}

\begin{corollary}\label{cor:compute}
Let $M\in \m \La$ be indecomposable. 
\begin{itemize}
\item[(a1)] If $M\in (\m A)_{\ast} \setminus (\m B)_{\ast}$, then $\tau (M) \cong \phi_{\ast} \tau \phi^{!} (M)$ and $\om (M) \cong \phi_{\ast} \om \phi^{!} (M)$. 
\item[(a2)] If $M\in (\m B)_{\ast}$, then $\tau (M) \cong \psi_{\ast}\tau \psi^{\ast} (M)$ and $\om (M) \cong \psi_{\ast} \om \psi^{\ast} (M)$.
\item[(b1)] If $M\in (\m B)_{\ast} \setminus (\m A)_{\ast}$, then $\tau^{-} (M) \cong \psi_{\ast} \tau^{-} \psi^{\ast} (M)$ and $\om^{-} (M) \cong \psi_{\ast} \om^{-} \psi^{\ast} (M)$.
\item[(b2)] If $M\in (\m A)_{\ast}$, then $\tau^{-} (M) \cong \phi_{\ast} \tau^{-}\phi^{!} (M)$ and $\om^{-} (M) \cong \phi_{\ast} \om^- \phi^{!} (M)$.
\end{itemize}
\end{corollary}

\begin{proof}
We only prove (a1) and (a2); (b1) and (b2) are similar. The claims about $\tau$ follow immediately by Theorem \ref{thrm:general glue} and so we only show the claims about the syzygy.

If $M\in (\m A)_{\ast} \setminus (\m B)_{\ast}$ and $\Hom_{\La}(\vare_i \La , M)\neq 0$, then $h-l+1\leq i \leq h$ by Lemma \ref{lemma:techproj}. Moreover, since $M\in (\m A)_{\ast} \setminus (\m B)_{\ast}$, we have that $M\not\in \cPD$ by Corollary \ref{lemma:intersection}. In particular, $\Hom_{\La}(M,S(i))=0$ for $1\leq i \leq h$ by Proposition \ref{prop:triangles}. Therefore, if $\vare_{i}\La$ is a summand of the projective cover of $M$, then $h-l+1\leq i \leq 0$. But then $\vare_i\La$ is supported in $A$ by Proposition \ref{prop:projinj} and so the projective cover of $M$ is supported in $A$. In particular, we can compute the syzygy of $M$ by viewing $M$ as an $A$-module instead.

If $M\in (\m B)_{\ast}$ then again by Lemma \ref{lemma:techproj} the projective cover of $M$ is supported in $B$ and the result follows as in the previous case.
\end{proof}

\begin{corollary}\label{cor:globaldimensionglued}
Let $\gldim(\La)=d$, $\gldim(A)=d_1$ and $\gldim(B)=d_2$. Then 
\[\max\{d_1,d_2\} \leq d \leq d_1+d_2.\]
\end{corollary}

\begin{proof}
Let $M\in\m\La$ and set $\cU^k=\phi_{\ast}\left(\om^k(\m A)\right)$. Since $\gldim(A)=d_1$, we have that $\cU^k=0$ for $k>d_1$. We claim that 
\begin{equation}\label{eq:syzygypower}
\om^j (M) \subseteq \add(\cU^j,(\m B)_\ast) \text{ for any $j\geq 0$.}
\end{equation} 
We prove (\ref{eq:syzygypower}) by induction. The base case $j=0$ follows immediately by Proposition \ref{prop:theyareglued}. For the induction step, assume that (\ref{eq:syzygypower}) holds for $j=k$ and we will show that it holds for $j=k+1$. Then we have that
\[ \om^k(M) \cong X \oplus Y, \]
with $X\in \cU^k$ and $Y\in (\m B)_{\ast}$. By Proposition \ref{prop:theyareglued} we can write $X\cong X_1 \oplus X_2$ with $X_1\in (\m A)_{\ast}\setminus (\m B)_{\ast}$ and $X_2\in (\m B)_\ast$. Then by Corollary \ref{cor:compute}(a1) we have that $\om (X_1) \in \cU^{k+1}$ and by Corollary \ref{cor:compute}(a2) we have that $\om (X_2)$, $\om (Y) \in (\m B)_{\ast}$. Hence 
\[\om^{k+1} (M) \cong (\om (X_1)) \oplus \left(\om(X_2 \oplus Y)\right) \in \add(\cU^{k+1}, (\m B)_{\ast}),\]
and the induction step is proved.

Let us now show that $d\leq d_1+d_2$. If we have $\om^{d_1}(M)=0$, then $\pd(M)\leq d_1\leq d_1+d_2$. If $\om^{d_1}(M)\neq 0$, then by (\ref{eq:syzygypower}) we can write $\om^{d_1}(M)\cong U \oplus V$ with  $U\in\cU^{d_1}\setminus (\m B)_{\ast}$ and $V\in(\m B)_\ast)$. Then
\[ \pd(M) = d_1 + \max\left\{ \pd(U),\pd(V) \right\}. \]
Since $\phi^!(U)\in \om^{d_1}(\m A)$, it follows that $\phi^!(U)$ is a projective $A$-module and hence $\om \phi^!(U)=0$. Since $U\not\in (\m B)_{\ast}$, Corollary \ref{cor:compute}(a1) gives $\om (U) = 0$ and so $\pd(U)=0$. Since $V\in (\m B)_{\ast}$, we can compute the projective resolution of $V$ in $\m B$ by Corollary \ref{cor:compute}(a2). In particular we have that $\pd(V)\leq \gldim(B)=d_2$. It follows that
\[ \pd(M) = d_1+\pd(V) \leq d_1+d_2,\]
and since $M$ was arbitrary, we conclude that $d\leq d_1+d_2$.

Next, let us now show that $d_2\leq d$. Let $N$ be a $B$-module with $\pd(N)=d_2$. Then $\psi_{\ast}(N)$ is a $\La$-module and by Corollary \ref{cor:compute}(a2) we have that $\pd(\psi_{\ast}(N))=d_2$. Hence $d_2\leq d$. Finally, let us show that $d_1\leq d$. Similarly to before, let $L$ be an $A$-module with $\id(L)=d_1$. Then $\phi_{\ast}(L)$ is a $\La$-module and by Corollary \ref{cor:compute}(b2) we have that $\id(\phi_{\ast}(L))=d_1$, which completes the proof.
\end{proof}

\begin{corollary}\label{cor:AR glued}
For the Aus\-lan\-der--Rei\-ten quiver of $\La$ we have, as quivers, $\Gamma(\La)=\Gamma(B)\coprod_{\triangle} \Gamma(A)$, where the right-hand side denotes the amalgamated sum under the identification $\triangle=\phi_\ast(\PD)=\psi_\ast(\DI)$. Moreover, in this identification, the vertex $[M]\in \Gamma(\La)$ corresponds to the vertex $[\phi^{!}(M)]$ in $\Gamma(A)$ if $M$ is supported in $A$ and to the vertex $[\psi^{\ast}(M)]$ in $\Gamma(B)$ if $M$ is supported in $B$.
\end{corollary}

\begin{proof}
Immediate by Proposition \ref{prop:theyareglued} and Theorem \ref{thrm:general glue}, since almost split sequences in $\Gamma(\La)$ correspond to almost split sequences in either $\Gamma(A)$ or $\Gamma(B)$. The vertex identification follows from Proposition \ref{prop:decomposition}.
\end{proof}

\begin{example}\label{ex:first glue}
Let $B$ be as in Example \ref{ex:twogluesquiver} and let $A$ be given by the quiver with relations
\[\begin{tikzpicture}[scale=0.9, transform shape]
\node (1) at (0,1) {$1$};
\node (2) at (1,1) {$2$};
\node (3) at (2,1) {$3$\nospacepunct{.}};

\node (0) at (-1,1) {$0$};
\draw[->] (0) to (1);
\draw[dotted] (0) to [out=30,in=150] (3);

\draw[->] (1) to (2);
\draw[->] (2) to (3);

\end{tikzpicture}\]
Let $I=I_B(3)=\qthree{1}[2][3]\in\m B$ be the indecomposable injective $B$-module corresponding to the vertex $3$ of $Q_B$ and $P=P_A(1)=\qthree{1}[2][3]\in\m A$ be the indecomposable projective $A$-module corresponding to the vertex $1$ of $Q_A$. Then $I$ is a right abutment of $B$ and $P$ is a left abutment of $A$, both of height $3$. Hence the gluing $\La = B \glue[P][I] A$ is defined and by Lemma \ref{lemma:gluingquivers} we have that $\La$ is given by the quiver with relations
\[\begin{tikzpicture}[scale=0.9, transform shape]
\node (1) at (0,1) {$1$};
\node (2) at (1,1) {$2$};
\node (3) at (2,1) {$3$};
\node (4) at (3,1) {$4$};
\node (5) at (4,0.5) {$5$};
\node (6) at (5,0.5) {$6$};
\node (7) at (6,0.5) {$7$\nospacepunct{.}};
\node (1') at (1,0) {$1'$};
\node (2') at (2,0) {$2'$};
\node (3') at (3,0) {$3'$};

\node (0) at (-1,1) {$0$};
\draw[->] (0) to (1);
\draw[dotted] (0) to [out=30,in=150] (3);

\draw[->] (1) to (2);
\draw[->] (2) to (3);
\draw[->] (3) to (4);
\draw[->] (4) to (5);
\draw[->] (5) to (6);
\draw[->] (6) to (7);
\draw[->] (1') to (2');
\draw[->] (2') to (3');
\draw[->] (3') to (5);

\draw[dotted] (2) to [out=30,in=150] (4);
\draw[dotted] (3) to [out=-30,in=0] (5.west);
\draw[dotted] (4) to [out=0,in=160] (7);
\draw[dotted] (2') to [out=30,in=0] (5.west);
\draw[dotted] (3') to [out=0,in=210] (6);
\end{tikzpicture}\]
The Aus\-lan\-der--Rei\-ten quiver $\Gamma(B)$ of $B$ was computed in Example \ref{ex:twogluesARquiver}. The Aus\-lan\-der--Rei\-ten quiver $\Gamma(A)$ of $A$ is
\[\begin{tikzpicture}[scale=1.2, transform shape]
\tikzstyle{nct3}=[circle, minimum width=6pt, draw=white, inner sep=0pt, scale=0.9]

\node[nct3] (A) at (0,0) {$\qthree{}[3][]$};
\node[nct3] (B) at (0.7,0.7) {$\qthree{2}[3]$};
\node[nct3] (C) at (1.4,1.4) {$\qthree{1}[2][3]$};
\node[nct3] (D) at (1.4,0) {$\qthree{}[2][]$};
\node[nct3] (E) at (2.1,0.7) {$\qthree{1}[2]$};
\node[nct3] (F) at (2.8,1.4) {$\qthree{0}[1][2]$};
\node[nct3] (G) at (2.8,0) {$\qthree{}[1][]$};
\node[nct3] (H) at (3.5,0.7) {$\qthree{0}[1]$};
\node[nct3] (J) at (4.2,0) {$\qthree{}[0][]$\nospacepunct{.}};

\draw[->] (A) to (B);
\draw[->] (B) to (C);
\draw[->] (B) to (D);
\draw[->] (C) to (E);
\draw[->] (D) to (E);
\draw[->] (E) to (F);
\draw[->] (E) to (G);
\draw[->] (F) to (H);
\draw[->] (G) to (H);
\draw[->] (H) to (J);

\draw[loosely dotted] (A.east) -- (D);
\draw[loosely dotted] (B.east) -- (E);
\draw[loosely dotted] (D.east) -- (G);
\draw[loosely dotted] (E.east) -- (H);
\draw[loosely dotted] (G.east) -- (J);
\end{tikzpicture}\]

Using Corollary \ref{cor:AR glued} we conclude that the Aus\-lan\-der--Rei\-ten quiver $\Gamma(\La)$ of $\La$ is $\Gamma(B)\coprod_{\triangle} \Gamma(A)$ or
\[\begin{tikzpicture}[scale=1.2, transform shape]
\tikzstyle{nct3}=[circle, minimum width=6pt, draw=white, inner sep=0pt, scale=0.9]
\node[nct3] (A) at (0,0) {$\qthree{}[7][]$};
\node[nct3] (B) at (0.7,0.7) {$\qthree{6}[7]$};
\node[nct3] (C) at (1.4,1.4) {$\qthree{5}[6][7]$};
\node[nct3] (D) at (1.4,0) {$\qthree{}[6][]$};
\node[nct3] (E) at (2.1,0.7) {$\qthree{5}[6]$};
\node[nct3] (F) at (2.8,1.4) {$\qthree{4}[5][6]$};
\node[nct3] (G) at (2.8,0) {$\qthree{}[5][]$};
\node[nct3] (H) at (3.5,0.7) {$\qthree{4}[5]$};
\node[nct3] (I) at (3.5,-0.7) {$\qthree{3'}[5]$};
\node[nct3] (J) at (4.2,0) {$\begin{smallmatrix} 4 && 3' \\ & 5 &
\end{smallmatrix}$};
\node[nct3] (K) at (4.9,0.7) {$\qthree{}[3'][]$};
\node[nct3] (L) at (4.9,-0.7) {$\qthree{}[4][]$};
\node[nct3] (M) at (5.6,1.4) {$\qthree{2'}[3']$};
\node[nct3] (N) at (5.6,-1.4) {$\qthree{3}[4]$};
\node[nct3] (O) at (6.3,2.1) {$\qthree{1'}[2'][3']$};
\node[nct3] (P) at (6.3,0.7) {$\qthree{}[2'][]$};
\node[nct3] (Q) at (6.3,-0.7) {$\qthree{}[3][]$};
\node[nct3] (R) at (7,1.4) {$\qthree{1'}[2']$};
\node[nct3] (S) at (7,-1.4) {$\qthree{2}[3]$};
\node[nct3] (T) at (7.7,0.7) {$\qthree{}[1'][]$};
\node[nct3] (U) at (7.7,-0.7) {$\qthree{}[2][]$};
\node[nct3] (V) at (7.7,-2.1) {$\qthree{1}[2][3]$};
\node[nct3] (W) at (8.4,-1.4) {$\qthree{1}[2]$};
\node[nct3] (X) at (9.1,-0.7) {$\qthree{}[1][]$};
\node[nct3] (Y) at (9.1,-2.1) {$\qthree{0}[1][2]$};
\node[nct3] (Z) at (9.8,-1.4) {$\qthree{0}[1]$};
\node[nct3] (AA) at (10.4,-0.7) {$\qthree{}[0][]$\nospacepunct{,}};

\draw[->] (A) to (B);
\draw[->] (B) to (C);
\draw[->] (B) to (D);
\draw[->] (C) to (E);
\draw[->] (D) to (E);
\draw[->] (E) to (F);
\draw[->] (E) to (G);
\draw[->] (F) to (H);
\draw[->] (G) to (H);
\draw[->] (G) to (I);
\draw[->] (H) to (J);
\draw[->] (I) to (J);
\draw[->] (J) to (K);
\draw[->] (J) to (L);
\draw[->] (K) to (M);
\draw[->] (L) to (N);
\draw[->] (M) to (O);
\draw[->] (M) to (P);
\draw[->] (N) to (Q);
\draw[->] (O) to (R);
\draw[->] (P) to (R);
\draw[->] (Q) to (S);
\draw[->] (R) to (T);
\draw[->] (S) to (U);
\draw[->] (S) to (V);
\draw[->] (U) to (W);
\draw[->] (V) to (W);
\draw[->] (W) to (X);
\draw[->] (X) to (Z);
\draw[->] (Z) to (AA);
\draw[->] (W) to (Y);
\draw[->] (Y) to (Z);

\draw[loosely dotted] (A.east) -- (D);
\draw[loosely dotted] (B.east) -- (E);
\draw[loosely dotted] (D.east) -- (G);
\draw[loosely dotted] (E.east) -- (H);
\draw[loosely dotted] (G.east) -- (J);
\draw[loosely dotted] (H.east) -- (K);
\draw[loosely dotted] (I.east) -- (L);
\draw[loosely dotted] (K.east) -- (P);
\draw[loosely dotted] (M.east) -- (R);
\draw[loosely dotted] (L.east) -- (Q);
\draw[loosely dotted] (Q.east) -- (U);
\draw[loosely dotted] (P.east) -- (T);
\draw[loosely dotted] (S.east) -- (W);
\draw[loosely dotted] (U.east) -- (X);
\draw[loosely dotted] (W.east) -- (Z);
\draw[loosely dotted] (X.east) -- (AA);
\end{tikzpicture}\]
where the intersection of $\phi_{\ast}(\Gamma(A))$ and $\psi_{\ast}(\Gamma(B))$ is exactly $\triangle(3)$.
\end{example}

\begin{example}\label{ex:naming}
Let $A$ be a rep\-re\-sen\-ta\-tion-di\-rect\-ed algebra with a left abutment $P$ of height $h$. Then, by Example \ref{ex:KAglued}, we have $A = KA_h \glue[P][I(h)] A$ and so by Corollary \ref{cor:AR glued} we have $\Gamma(\La) = \Gamma(A) \coprod_{\triangle} \triangle(h)$ under the identification $(\Id_A)_{\ast} (\PD) = (f_P)_{\ast}(\triangle (h))$. Hence we can view any $A$-module $T\in \cPD$ as a $KA_h$-module via the functor $f_P^!$. Similarly, if $I$ is a right abutment of $A$ of height $h$, any $A$-module $X\in \cDI$ can be viewed as a $KA_h$-module through the identification $A = A \glue[P(h)][I] KA_h$ and the corresponding functor $g_I^{\ast}$.
\end{example}

We finish this section with a corollary that describes the connection between abutments of $\La$ and abutments of $A$ and $B$.
\begin{corollary}\label{cor:remaining} 
Let $\La = B \glue A$ and let $h-l+1\leq i \leq m$.
\begin{itemize}
\item[(a1)] If $h-l+1\leq i \leq h$ and $D( A e_i)$ is a right abutment of $A$, then $D(\La \vare_i)$ is a right abutment of $\La$.
\item[(a2)] If $h-l+1\leq i \leq 0$ and $e_i A$ is a left abutment of $A$ such that $\cF_{e_i A} \cap \cPD = 0$, then $\vare_i \La$ is a left abutment of $\La$.
\item[(b1)] If $1\leq i \leq m$ and $\epsilon_i B$ is a left abutment of $B$, then $\vare_i\La$ is a left abutment of $\La$.
\item[(b2)] If $h+1\leq i \leq m$, and $D(B \epsilon_i)$ is a right abutment of $B$ such that $\cG_{D(B \epsilon_i)} \cap \cDI = 0$, then $D(\La \vare_i)$ is a right abutment of $\La$.
\item[(c1)] If $\vare_i \La$ is a left abutment of $\La$, then $e_i A$ is a left abutment of $A$ if $h-l+1\leq i \leq 0$ and $\epsilon_i B$ is a left abutment of $B$ if $1\leq i \leq m$.
\item[(c2)] If $D(\La \vare_i)$ is a right abutment of $\La$, then $D(A e_i)$ is a right abutment of $A$ if $h-l+1\leq i \leq h$ and $D(B e_i)$ is a right abutment of $B$ if $h+1\leq i \leq m$.
\end{itemize}
\end{corollary}

\begin{proof}
Let us indicatively show (a2) and (c1); the rest are similar. For (a2) notice that since $h-l+1\leq i \leq 0$, we have that $\vare_i \La$ is supported in $A$ by Proposition \ref{prop:projinj}. In particular, by the definition of $\phi$, we have $\phi^!(\vare_i \La)=e_i A$. Since $\cF_{e_i A} \cap \cPD = 0$, it follows from Proposition \ref{prop:triangles} that the two subquivers ${\nsup{e_i A}{\triangle}}$ and $\PD$ of $\Gamma(A)$ are disjoint. Hence by Corollary \ref{cor:AR glued} it follows that $\phi_\ast\left({\nsup{e_i A}{\triangle}}\right)$ is of the form $\nsup{\vare_i\La}{\triangle}$ as in Proposition \ref{prop:triangles} and hence $\vare_i\La$ is a left abutment of $\La$.

For (c1) notice that $\phi^!(\vare_i\La)\cong e_i\La$ for $h-l+1\leq i \leq 0$ and $\psi^\ast(\vare_i\La) \cong \epsilon_i B$ for $1\leq i \leq m$ again by Proposition \ref{prop:projinj}. Then by Proposition \ref{cor:relatives}, the whole foundation $\nsup{\vare_i\La}{\triangle}$ is supported either in $A$ or in $B$, respectively. Then by Corollary \ref{cor:AR glued} the image of $\nsup{\vare_i\La}{\triangle}$ under $\phi^!$ (respectively $\psi^\ast$) is of the form $\nsup{e_i A}{\triangle}$ (respectively $\nsup{\epsilon_i B}{\triangle}$) and hence a left abutment of $A$ (respectively $B$) by Proposition \ref{prop:triangles}.
\end{proof}

\section{Part III: Fractures}

In this section we will show how to use gluing to construct many examples of rep\-re\-sen\-ta\-tion-di\-rect\-ed algebras admitting $n$-cluster tilting subcategories. In subsection \ref{subsec:fractured subcategories} we introduce the building blocks of our construction. In subsection \ref{sect:construction} we show how the construction works. In subsection \ref{sect:slices} we are interested in a special case of our construction which we can describe completely.

\subsection{Fractured subcategories}\label{subsec:fractured subcategories}
First, let us introduce some notation. Let $\La$ be a rep\-re\-sen\-ta\-tion-di\-rect\-ed algebra. We set
\[\PL=\cP:=\add(\La),\;\; \ABL=\ABP:=\add\left\{ P\in \cP \mid \text{$P$ is a left abutment of $\La$}\right\},\] 
\[\IL=\cI:=\add(D(\La)),\;\; \LAB=\IAB:=\add\left\{I\in \cI \mid \text{$I$ is a right abutment of $\La$}\right\}.\]

By Proposition \ref{prop:triangles} it follows that for $[P],[Q]\in\Ind{\ABP}$ the relation 
\[\text{$[P]\leq [Q]$ if and only if $\cPD \subseteq \cF_{Q}$}\]
is a partial order. Similarly we define $\leq$ on $\Ind{\IAB}$. We will refer to elements of those sets as \emph{maximal} or \emph{minimal} with respect to these partial orders. We set
\[\MABL=\MABP:=\add\left\{ P\in \ABP \mid \text{$P$ is maximal}\right\},\;\;\MLAB=\MIAB:=\add\left\{ I\in \IAB \mid \text{$I$ is maximal}\right\}.\]

The following important definition is due to Iyama (\cite{IYA1}, \cite{IYA2}). 

\begin{definition}\label{def:n-ct}
We call a subcategory $\mathcal{C}$ of $\m\La$ an \emph{$n$-cluster tilting subcategory} if 
\begin{align*}
\cC&=\Corth=\orthC,
\end{align*}
where
\[\Corth:=\{X\in\m\La \mid \Ext^i(\cC,X)=0 \text{ for all $0<i<n$}\},\]
\[\orthC:=\{X\in\m\La \mid \Ext^i(X,\cC)=0 \text{ for all $0<i<n$}\}.\]
\end{definition}

It is clear from the definition that $\m\La$ is the unique $1$-cluster tilting subcategory of $\La$. Observe that since $\La$ is rep\-re\-sen\-ta\-tion-fi\-nite, then any additive subcategory of $\m\La$ is of the form $\add (M)$ for some $M\in\m\La$. When $\add(M)$ is $n$-cluster tilting we call $M$ an \emph{$n$-cluster tilting module}. 

Note that $n$-cluster tilting subcategories are usually defined in more general settings by adding the requirement of functorial finiteness, but since $\add(M)$ is always functorially finite we can use the above definition.

Before we proceed, let us introduce one more piece of notation. Let $\cC$, $\cV$ be subcategories of $\m\La$. We set $\cC_{\setminus \cV}$ to be the additive closure of all indecomposable modules $X\in \cC$ such that $X\not\in \cV$. With this in mind we recall the following characterization of $n$-cluster tilting subcategories for rep\-re\-sen\-ta\-tion-di\-rect\-ed algebras.

\begin{theorem}\emph{\cite[Theorem 1]{VAS}}
\label{thrm:char}
Assume that $\La$ is a rep\-re\-sen\-ta\-tion-di\-rect\-ed algebra and let $\cC$ be a subcategory of $\m\La$. Then $\cC$ is an $n$-cluster tilting subcategory if and only if the following conditions hold:
\begin{itemize}
\item[(1)] $\cP \subseteq \cC$,

\item[(2)] $\tn$ and $\tno$ induce mutually inverse bijections
\[\begin{tikzpicture}
\node (0) at (0,0) {$\Ind{\cC_{\setminus\cP}}$};
\node (1) at (2,0) {$\Ind{\cC_{\setminus\cI}}$,};

\draw[-latex] (0) to [bend left=20] node [above] {$\tn$} (1);
\draw[-latex] (1) to [bend left=20] node [below] {$\tno$} (0);
\end{tikzpicture}\]

\item[(3)] $\om^i (M)$ is indecomposable for all indecomposable $M\in \cC_{\setminus\cP}$ and $0<i<n$,

\item[(4)] $\om^{-i}(N)$ is indecomposable for all indecomposable $N\in \cC_{\setminus\cI}$ and $0<i<n$.
\end{itemize}
\end{theorem}

Let $\{P_1,\dots,P_k\}$ be a complete collection of non-isomorphic representatives of maximal left abutment of $\La$. For $1\leq j \leq k$, let
\[ 0 \subseteq P_{jh_j} \subseteq \cdots \subseteq P_{j2} \subseteq P_{j1}=P_j\]
be the composition series of $P_j$. Set $P^{(j)}\coloneqq P_{j1} \oplus \cdots \oplus P_{jh_j}$. Then $\{P^{(j)}\}_{j=1}^k$ are all projective $\La$-modules with no isomorphic summands. Hence there exists a projective $\La$-module $Q$ such that 
\[\La \cong P^{(1)}\oplus\cdots\oplus P^{(k)}\oplus Q.\]
Then $\cP=\add\{P^{(1)},\dots,P^{(k)},Q\}$. To generalize the definition of an $n$-cluster tilting subcategory for rep\-re\-sen\-ta\-tion-di\-rect\-ed algebras using Theorem \ref{thrm:char}, we first replace the basic module $P^{(j)}$ by a suitable basic module 
\[ T^{(j)} = T_{j1} \oplus \cdots \oplus T_{jh_j} \in \cF_{P_j}.\]
Then we set $\cP^L\coloneqq\add\{T^{(1)},\dots,T^{(k)},Q\}$. Dually, we replace right abutments and define a subcategory $\cI^R$ in a similar manner. Then we replace all instances of $\cP$ and $\cI$ in Theorem \ref{thrm:char} with $\cP^L$ and $\cI^R$ respectively. Since we want to generalize the definition of $n$-cluster tilting, we also want to have $\Ext_{\La}^i(T^{(j)},T^{(j)})=0$ for $0<i<n$ and any $1\leq j \leq k$. Since by Corollary \ref{cor:diminabutment} we have that $\pd(T)\leq 1$, this simplifies to $\Ext_{\La}^1(T^{(j)},T^{(j)})=0$. Since $T^{(j)}\in \cF_{P_j}$, if we view $T$ as a $KA_h$-module via $f^!_{P_j}$, we conclude that $f^!_{P_j}(T^{(j)})$ should be a tilting $KA_h$-module. Tilting modules of $KA_h$ were classified in \cite{HR}. The following Proposition asserts that a basic tilting module of $KA_h$ has the correct number of indecomposable summands, which is necessary for our construction to work.

\begin{proposition}\cite[paragraph (4.1)]{HR}\label{prop:KAtilting}
Let $T$ be a basic tilting module of $KA_h$. Then $T$ has exactly $h$ indecomposable summands.
\end{proposition}

Let $P$ be a left abutment of $\La$. Recall that by Example \ref{ex:naming} we can view a $\La$-module $T$ in $\cPD$ as a $KA_h$-module via $f_P^{!}$ and dually for right abutments. With this in mind, we give the following definition.

\begin{definition}\label{def:fracture}
Let $\La$ be a rep\-re\-sen\-ta\-tion-di\-rect\-ed algebra.
\begin{itemize}
\item[(a)] Let $P$ be a maximal left abutment of $\La$ realized by $(e_i,f_i)_{i=1}^h$. A \emph{fracture of $P$} is a module $T \in \cPD$ such that $f_P^{!}(T):=T^{!}$ is a basic tilting $KA_h$-module. The \emph{level of $T$}, denoted $\lvl(T)$, is defined to be the number
\[\lvl(T):=\max\left(\left\{i\in \{1,\dots,h\} \mid e_{h-i+1}\La \not\in \add(T)\right\}\cup\{0\}\right)+1.\]
\item[(b)] Let $I$ be a maximal right abutment of $\La$ realized by $(e_i,g_{i-1})_{i=1}^h$. A \emph{fracture of $I$} is a module $T \in \cDI$ such that $g_I^{\ast}(T):=T^{\ast}$ is a basic tilting $KA_h$-module. The \emph{level of $T$}, denoted $\lvl(T)$, is defined to be the number
\[\lvl(T):=\max\left(\left\{i\in \{1,\dots,h\} \mid D(\La e_i) \not\in \add(T)\right\}\cup\{0\}\right)+1.\]
\end{itemize}
\end{definition}

Notice that in particular a fracture is a basic module. The following lemma collects some basic information about fractures.

\begin{lemma}\label{lemma:fracturebasic}
Let $\La$ be a rep\-re\-sen\-ta\-tion-di\-rect\-ed algebra.
\begin{itemize}
    \item[(a)] Let $T$ be a fracture of a maximal left abutment $P$, realized by $(e_i,f_i)_{i=1}^h$. Then $T$ has $h$ indecomposable summands and $\pd(T)\leq 1$.
    \item[(b)] Let $T$ be a fracture of a maximal right abutment $I$, realized by $(e_i,g_{i-1})_{i=1}^h$. Then $T$ has $h$ indecomposable summands and $\id(T)\leq 1$.
\end{itemize}
\end{lemma}

\begin{proof}
Follows immediately by Corollary \ref{cor:diminabutment} and Proposition \ref{prop:KAtilting}.
\end{proof}

\begin{example}\label{ex:projfracture}
For a maximal left abutment $P$ realized by $(e_i,f_i)_{i=1}^h$, there exists a unique (up to isomorphism) fracture of $P$ that is projective, namely $T=\bigoplus_{i=1}^h e_i\La$. To see that this is a fracture, notice that
\[T^{!}\cong\bigoplus_{i=1}^h \phi^{!}(e_i\La)\cong\bigoplus_{i=1}^h t_iKA_h\cong KA_h\] 
is a basic tilting $K A_h$-module. The fact that $T$ is the unique projective fracture of $P$ follows by Lemma \ref{lemma:fracturebasic}. Similarly, if $I$ is a right abutment realized by $(e_i,g_{i-1})_{i=1}^h$, then $T=\bigoplus_{i=1}^h D(\La e_i)$ is the unique fracture of $I$ that is injective.
\end{example}

\begin{definition}\label{def:fracturing}
Let $\La$ be a rep\-re\-sen\-ta\-tion-di\-rect\-ed algebra.
\begin{itemize}
\item[(a)] A \emph{left fracturing} $T^L$ of $\La$ is a module 
\[T^L=\bigoplus\limits_{[P] \in \Ind{\sMABP}}T^{(P)},\] 
where $T^{(P)}$ is a fracture of $P$. We set $\cP^L := \add \left\{\cP_{\setminus\sABP}, T^L\right\}$.
\item[(b)] A \emph{right fracturing} $T^R$ of $\La$ is a module 
\[T^R=\bigoplus\limits_{[I] \in \Ind{\sMIAB}}T^{(I)},\]
 where $T^{(I)}$ is a fracture of $I$. We set $\cI^R:=\add\left\{ \cI_{\setminus\sIAB}, T^R \right\}$.
\item[(c)] A \emph{fracturing} of $\La$ is a pair $(T^L,T^R)$ where $T^L$ is a left fracturing of $\La$ and $T^R$ is a right fracturing of $\La$.
\end{itemize}
\end{definition}

It follows easily by the definition of a fracturing and of maximal left and right abutments that a fracturing is a basic module. Hence if $(T^L, T^R)$ is a fracturing of $\La$, then we have by Proposition \ref{prop:KAtilting} that $\abs{\cP^L}=\abs{\cP}$ and $\abs{\cI^R}=\abs{\cI}$. In particular, we always have $\abs{\cP^L}=\abs{\cI^R}$. 

\begin{lemma}\label{lemma:uniqfract}
Let $(T^L, T^R)$ be a fracturing of $\La$. Then
\begin{itemize}
\item[(a)] The following are equivalent
\begin{itemize}
\item[(a1)] $\cP^L=\cP$,
\item[(a2)] $T^L$ is projective,
\item[(a3)] $T^L \cong \bigoplus\limits_{[P] \in \Ind{\sABP}}P$.
\end{itemize}
\item[(b)] The following are equivalent
\begin{itemize}
\item[(b1)] $\cI^R=\cI$,
\item[(b2)] $T^R$ is injective,
\item[(b3)] $T^R \cong \bigoplus\limits_{[I] \in \Ind{\sIAB}}I$.
\end{itemize}
\end{itemize}
\end{lemma}

\begin{proof}
We only prove (a); (b) is similar. First we show (a1) implies (a2). If $\cP^L=\cP$ then every module in $\cP^L$ is projective. In particular, $T^L$ is projective. To see that (a2) implies (a3), first notice that if $T^L=\bigoplus\limits_{[Q] \in \Ind{\sMABP}}T^{(Q)}$ is projective then $T^{(Q)}$ is projective for every maximal left abutment $Q$ of $\La$. By Example \ref{ex:projfracture} this implies that every indecomposable submodule of $Q$ is isomorphic to a summand of $T^{(Q)}$. Since an abutment is either maximal or isomorphic to a submodule of a maximal abutment, it follows that a representative of each isomorphism class of each abutment $P$ of $\La$ appears exactly once as a direct summand of $T^L$. Finally, (a3) implies (a1) immediately from the definition.
\end{proof}

If $\La$ is an algebra, we will denote by $\Lab$ a left fracturing of $\La$ which is projective as a module and by $\DLab$ a right fracturing of $\La$ which is injective as a module. By Lemma \ref{lemma:uniqfract}, it follows that $\Lab$ and $\DLab$ are unique up to isomorphism. 

\begin{example}\label{ex:fracturing}
Let $B$ be as in Example \ref{ex:twogluesquiver}. The unique maximal left abutment is $P(5)$ and the maximal right abutments are $I(3)$ and $I(3')$. Consider the modules 
\[T^{(P(5))}=\qthree{}[7][]\oplus\qthree{6}[7]\oplus\qthree{5}[6][7],\;\; T^{(I(3))}=\qthree{}[3][]\oplus\qthree{2}[3]\oplus\qthree{1}[2][3]\;\; \text{ and } \;T^{(I(3'))}=\qthree{}[2'][]\oplus\qthree{2'}[3']\oplus\qthree{1'}[2'][3'].\] 
By construction, $T^{(P(5))}$ is the unique (up to isomorphism) projective fracture of $P(5)$ and the modules $T^{(I(3))}$ and $T^{(I(3'))}$ are fractures of $I(3)$ and $I(3')$ respectively. Then $(T^L_B,T^R_B)$ is a fracturing of $B$, where $T^L_B=T^{(P(5))}$ and $T_B^R=T^{(I(3))}\oplus T^{(I(3'))}$. Since $T^{(P(5))}$ is projective, we have $\cP^L=\add (B)$ by Lemma \ref{lemma:uniqfract}. Following the definition, we also have 
\[\cI^R = \add \left(\qthree{5}[6][7]\oplus \qthree{4}[5][6] \oplus \begin{smallmatrix} 4&&3'\\&5& \end{smallmatrix}\oplus \qthree{3}[4] \oplus T^{(I(3))} \oplus T^{(I(3'))}\right).\]
\end{example}

\begin{definition}\label{def:nctfract}
Let $n\geq 2$. Assume that $\La$ is a rep\-re\-sen\-ta\-tion-di\-rect\-ed algebra with a fracturing $(T^L, T^R)$ and let $\cC$ be a subcategory of $\m\La$. Then $\cC$ is called a \emph{$(T^L,T^R,n)$-fractured subcategory} if
\begin{itemize}
\item[(1)] $\cP^L\subseteq \cC$,

\item[(2)] $\tn$ and $\tno$ induce mutually inverse bijections
\[\begin{tikzpicture}
\node (0) at (0,0) {$\Ind{\cC_{\setminus\cP^L}}$};
\node (1) at (2,0) {$\Ind{\cC_{\setminus\cI^R}}$,};

\draw[-latex] (0) to [bend left=20] node [above] {$\tn$} (1);
\draw[-latex] (1) to [bend left=20] node [below] {$\tno$} (0);
\end{tikzpicture}\]

\item[(3)] $\om^i (M)$ is indecomposable for all indecomposable $M\in \cC_{\setminus\cP^L}$ and $0<i<n$,

\item[(4)] $\om^{-i} (N)$ is indecomposable for all indecomposable $N\in \cC_{\setminus\cI^R}$ and $0<i<n$.
\end{itemize}
\end{definition}

Let us first note that Definition \ref{def:nctfract} makes sense for $n=1$ as well. However, the case $\La=KA_h$ behaves in a special way in that case for our purposes. Since the unique $1$-cluster tilting subcategory for a rep\-re\-sen\-ta\-tion-di\-rect\-ed algebra $\La$ is $\m\La$ itself, and to avoid needlessly complicating the results and proofs of this section we opt to assume that $n\geq 2$ when considering $n$-fractured subcategories.

Notice that conditions (1) and (2) in Definition \ref{def:nctfract} imply that $\cI^R\subseteq \cC$, since $\abs{\cP^L}=\abs{\cI^R}$. In particular, we have that $T^L\in \cC$ and $T^R\in \cC$. This definition generalizes the definition of an $n$-cluster tilting subcategory for a rep\-re\-sen\-ta\-tion-di\-rect\-ed algebra in the sense of the following proposition.

\begin{proposition}\label{prop:fracluster}
Let $n\geq 2$. Let $\La$ be a rep\-re\-sen\-ta\-tion-di\-rect\-ed algebra and $(T^L,T^R)$ be a fracturing of $\La$. Let $\cC$ be a $(T^L, T^R, n)$-fractured subcategory of $\m\La$ for $n\geq 2$. Then $\cC$ is an $n$-cluster tilting subcategory if and only if $T^L\cong\Lab$  and $T^R\cong\DLab$.
\end{proposition}

\begin{proof}
If $T^L\cong \Lab$ and $T^R \cong \DLab$, then Lemma \ref{lemma:uniqfract} implies that $\cP^L=\cP$ and $\cI^R=\cI$. Then Theorem \ref{thrm:char} implies that $\cC$ is an $n$-cluster tilting subcategory. 

Assume now that $\cC$ is an $n$-cluster tilting subcategory of $\La$. We will show that $T^L$ is projective (the proof that $T^R$ is injective is similar). Assume by way of contradiction that $T^L$ is not projective. Then $\pd(T^L) = 1$ by Lemma \ref{lemma:fracturebasic}. Hence $\Ext_{\La}^1(T^L,\La)\neq 0$ which contradicts $T^L\in \cC$.
\end{proof}

Proposition \ref{prop:fracluster} motivates the following definition.

\begin{definition}\label{def:left and right nfract}
Let $n\geq 2$. Let $\La$ be a rep\-re\-sen\-ta\-tion-di\-rect\-ed algebra with a fracturing $(T^L,T^R)$ and $\cC$ be a $(T^L,T^R,n)$-fractured subcategory. Then $\cC$ will be called a \emph{left $n$-cluster tilting subcategory} if $T^L\cong \Lab$ and a \emph{right $n$-cluster tilting subcategory} if $T^R\cong \DLab$.
\end{definition}

\begin{example}\label{ex:KAhfracturing}
Let $n\geq 2$ and $\La=KA_h$. Let $(T^L,T^R)$ be a fracturing of $\La$. Since an indecomposable projective $\La$-module is a submodule of the maximal left abutment $P(1)$, it follows that $\cP_{\setminus\sABP}$ is empty and so $\cP^L=\add(T^L)$. Similarly we have $\cI^R=\add(T^R)$. Since $\tau_n(M)=0=\tau_n^-(M)$ for any $M\in\m\La$, it is immediate from the definition that $\cC\subseteq \m\La$ is a $(T^L,T^R,n)$-fractured subcategory of $\La$ if and only if $T^L=T^R$ and $\cC=\add(T^L)$.
\end{example}

\begin{example}\label{ex:fracturedsubcategory}
Let $B$ and $(T_B^L,T_B^R)$ be as in Example \ref{ex:fracturing}. Let 
\[ \cC_B=\add \left(\bigoplus_{i\geq 0} \tau_2^{-i}(B)\right).\]
By computing $\tau_2^-$, we find that
\[\cC_B=\add\left(B \oplus \qthree{4}[5] \oplus \begin{smallmatrix} 4&&3' \\ &5& \end{smallmatrix} \oplus \qthree{}[2'][] \oplus \qthree{}[3][]\right).\] 
A simple calculation verifies that $\cC_B$ is a $(T_B^L, T_B^R,2)$-fractured subcategory. Since $T_B^{L}$ is projective, it is a left $n$-cluster tilting subcategory. For the convenience of the reader who might want to verify those claims, we give the Aus\-lan\-der--Rei\-ten quiver of $B$ where we encircle the indecomposable modules which are in $\cC_B$:
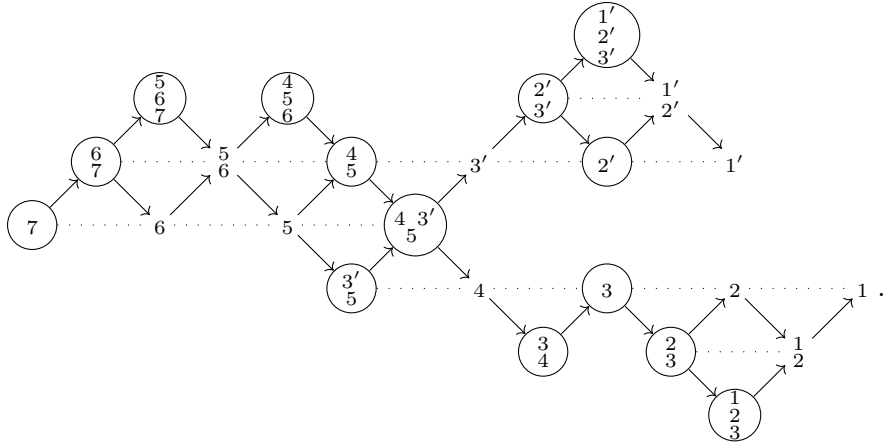
\[\begin{tikzpicture}[scale=1.2, transform shape]
\tikzstyle{nct2}=[circle, minimum width=0.6cm, draw, inner sep=0pt, text centered, scale=0.9]
\tikzstyle{nct22}=[circle, minimum width=0.6cm, draw, inner sep=0pt, text centered, scale=0.9]
\tikzstyle{nct3}=[circle, minimum width=6pt, draw=white, inner sep=0pt, scale=0.9]

\node[nct2] (A) at (0,0) {$\qthree{}[7][]$};
\node[nct2] (B) at (0.7,0.7) {$\qthree{6}[7]$};
\node[nct2] (C) at (1.4,1.4) {$\qthree{5}[6][7]$};
\node[nct3] (D) at (1.4,0) {$\qthree{}[6][]$};
\node[nct3] (E) at (2.1,0.7) {$\qthree{5}[6]$};
\node[nct2] (F) at (2.8,1.4) {$\qthree{4}[5][6]$};
\node[nct3] (G) at (2.8,0) {$\qthree{}[5][]$};
\node[nct2] (H) at (3.5,0.7) {$\qthree{4}[5]$};
\node[nct2] (I) at (3.5,-0.7) {$\qthree{3'}[5]$};
\node[nct2] (J) at (4.2,0) {$\begin{smallmatrix} 4 && 3' \\ & 5 &
\end{smallmatrix}$};
\node[nct3] (K) at (4.9,0.7) {$\qthree{}[3'][]$};
\node[nct3] (L) at (4.9,-0.7) {$\qthree{}[4][]$};
\node[nct2] (M) at (5.6,1.4) {$\qthree{2'}[3']$};
\node[nct2] (N) at (5.6,-1.4) {$\qthree{3}[4]$};
\node[nct22] (O) at (6.3,2.1) {$\qthree{1'}[2'][3']$};
\node[nct2] (P) at (6.3,0.7) {$\qthree{}[2'][]$};
\node[nct2] (Q) at (6.3,-0.7) {$\qthree{}[3][]$};
\node[nct3] (R) at (7,1.4) {$\qthree{1'}[2']$};
\node[nct2] (S) at (7,-1.4) {$\qthree{2}[3]$};
\node[nct3] (T) at (7.7,0.7) {$\qthree{}[1'][]$};
\node[nct3] (U) at (7.7,-0.7) {$\qthree{}[2][]$};
\node[nct2] (V) at (7.7,-2.1) {$\qthree{1}[2][3]$};
\node[nct3] (W) at (8.4,-1.4) {$\qthree{1}[2]$};
\node[nct3] (X) at (9.1,-0.7) {$\qthree{}[1][]$\nospacepunct{.}};

\draw[->] (A) to (B);
\draw[->] (B) to (C);
\draw[->] (B) to (D);
\draw[->] (C) to (E);
\draw[->] (D) to (E);
\draw[->] (E) to (F);
\draw[->] (E) to (G);
\draw[->] (F) to (H);
\draw[->] (G) to (H);
\draw[->] (G) to (I);
\draw[->] (H) to (J);
\draw[->] (I) to (J);
\draw[->] (J) to (K);
\draw[->] (J) to (L);
\draw[->] (K) to (M);
\draw[->] (L) to (N);
\draw[->] (M) to (O);
\draw[->] (M) to (P);
\draw[->] (N) to (Q);
\draw[->] (O) to (R);
\draw[->] (P) to (R);
\draw[->] (Q) to (S);
\draw[->] (R) to (T);
\draw[->] (S) to (U);
\draw[->] (S) to (V);
\draw[->] (U) to (W);
\draw[->] (V) to (W);
\draw[->] (W) to (X);

\draw[loosely dotted] (A.east) -- (D);
\draw[loosely dotted] (B.east) -- (E);
\draw[loosely dotted] (D.east) -- (G);
\draw[loosely dotted] (E.east) -- (H);
\draw[loosely dotted] (G.east) -- (J);
\draw[loosely dotted] (H.east) -- (K);
\draw[loosely dotted] (I.east) -- (L);
\draw[loosely dotted] (K.east) -- (P);
\draw[loosely dotted] (M.east) -- (R);
\draw[loosely dotted] (L.east) -- (Q);
\draw[loosely dotted] (Q.east) -- (U);
\draw[loosely dotted] (P.east) -- (T);
\draw[loosely dotted] (S.east) -- (W);
\draw[loosely dotted] (U.east) -- (X);
\end{tikzpicture}\]
\end{example}

\subsection{Main construction}\label{sect:construction}

Our aim is to glue algebras admitting fractured subcategories in such a way that the resulting algebra also admits a corresponding fractured subcategory. To this end we need to first describe how to glue algebras with a fracturing. So let us fix two rep\-re\-sen\-ta\-tion-di\-rect\-ed algebras $A$ and $B$ with fracturings $(T^L_A, T^R_A)$, respectively $(T^L_B, T^R_B)$, and set $\La \coloneqq B \glue[Q][J] A$ where $Q$ is a left abutment of $A$ and $J$ is a right abutment of $B$, both of the same height $h$.

Let $P$ be a maximal left abutment of $\La$. Then either $P$ is supported in $B$ or $P$ is not supported in $B$, in which case it is supported in $A$ by Proposition \ref{prop:decomposition}. By Corollary \ref{cor:remaining}(c1), in the first case $\psi^{\ast}(P)$ is a left abutment of $B$ and in the second case $\phi^{!}(P)$ is a left abutment of $A$. Moreover, in either case it is a maximal left abutment by Corollary \ref{cor:AR glued}. Set
\[T^{(P)}_\ast \coloneqq \begin{cases} \psi_\ast\left(T^{(\psi^{\ast}(P))}_B\right) &\mbox{ if $P$ is supported in $B$,} \\ \phi_\ast\left(T^{(\phi^{!}(P))}_A\right) &\mbox{ otherwise.} \end{cases}\]
and
\[T^L_{\La} \coloneqq \bigoplus_{[P]\in \Ind{\sMABP_{\La}}}T^{(P)}_\ast.\]
Observe that if $P$ is supported in $B$, then by construction the composition $\m\La \overset{(-)^\ast}{\longrightarrow} \m B \overset{(-)^\ast}{\longrightarrow} \m KA_h$ maps $T_{\ast}^{(P)}$ to a basic tilting $KA_h$-module and similarly if $P$ is supported in $A$.

Dually, maximal right abutments of $\La$ correspond to maximal right abutments of $A$ or of $B$ and so we set
\[ T^{(I)}_\ast \coloneqq \begin{cases} \phi_\ast\left(T^{(\phi^!(I))}_A\right) &\mbox{ if $I$ is supported in $A$,} \\ \psi_\ast\left(T^{(\psi^{\ast}(I))}_B\right) &\mbox{ otherwise,} \end{cases}\;\;\text{  and  }\;\; T^R_{\La} \coloneqq \bigoplus_{[I]\in \Ind{\sMIAB_{\La}}}T^{(I)}_\ast.\]
Then, by the above considerations it follows that $(T_\La^L, T_\La^R)$ is a fracturing of $\La$, which we call the \emph{gluing of the fracturings $(T_A^L, T_A^R)$ and $(T_B^L, T_B^R)$ at $Q$ and $J$} and we denote it by $(T_\La^L, T_\La^R)=(T_B^L, T_B^R) \glue[Q][J] (T_A^L, T_A^R)$. 

Although trivial gluing is a very special case of gluing which is not of much interest in this paper, it behaves somewhat unexpectedly with respect to gluing of fracturings. We illustrate the situation with the following example.

\begin{example}\label{ex:KAhfracturedglue}
Let $A$ be a rep\-re\-sen\-ta\-tion-di\-rect\-ed algebra, let $Q$ be a maximal left abutment of $A$ of height $h_Q$ and let $P\in \cF_Q$ be a left abutment of height $h\leq h_Q$. Let $(T_A^L,T_A^R)$ be a fracturing of $A$ and write $T_A^L\cong T_1\oplus T_A^{(Q)}$. Let also $B=KA_h$ and let $(T_B^R,T_B^L)=(T,T')$ be a fracturing of $B$, where $T$ and $T'$ are tilting $B$-modules. Then by Example \ref{ex:KAglued} we have that if $\La = B \glue[P][I_B(h)] A$, then $\La\cong A$. Clearly, all maximal right abutments of $\La$ are supported in $A$. Hence for any maximal right abutment $I$ of $\La$ we have that $T_{\ast}^{(I)}=T_A^{(I)}$ and so $T_{\La}^R=T_{A}^R$. Moreover, by Corollary \ref{cor:AR glued} it follows that no maximal left abutment of $\La$ is supported in $B$ except for $Q$ when $h=h_Q$. We consider the cases $h<h_Q$ and $h=h_Q$ separately. 

Assume that $h<h_Q$. Then $Q$ is a maximal left abutment of $\La$ not supported in $B$. Hence $T_{\ast}^{(Q)}=T_{A}^{(Q)}$ and so
\[(T_{\La}^L,T_{\La}^R) = (T_1\oplus T_{\ast}^{(Q)}, T_A^R)=(T_1\oplus T_{A}^{(Q)}, T_A^R)=(T_A^L, T_A^R).\]

Assume now that $h=h_Q$. Then $Q$ is supported in $B$. Hence $T_{\ast}^{(Q)}=T_{B}^{(P_B(1))}=T$ and so
\[(T_{\La}^L,T_{\La}^R) =(T_1\oplus T_{\ast}^{(Q)}, T_A^R)=(T_1\oplus T,T_A^R),\]
which in general is different from $(T_A^L,T_A^R)$. 
\end{example}

The following theorem shows how to use a $(T_A^L,T_A^R,n)$-fractured subcategory and a $(T_B^L,T_B^R,n)$-fractured subcategory to construct a $(T_{\La}^L,T_{\La}^R,n)$-fractured subcategory under a compatibility condition.

\begin{theorem}\label{thrm:fractsubcat}
Let $n\geq 2$. Let $A$ be a rep\-re\-sen\-ta\-tion-di\-rect\-ed algebra with a fracturing $(T^L_A, T^R_A)$ and let $\cC_A$ be a $(T^L_A, T^R_A,n)$-fractured subcategory. Let $Q\in\m A$ be a maximal left abutment and let $P \in \cF_Q$ be a left abutment of height $h$. Moreover, let $B$ be a rep\-re\-sen\-ta\-tion-di\-rect\-ed algebra with a fracturing $(T^L_B, T^R_B)$ and let $\cC_B$ be a $(T^L_B, T^R_B,n)$-fractured subcategory. Let $J\in\m B$ be a maximal right abutment and let $I \in \cG_J$ be a right abutment of height $h$.

Assume that $h\geq \lvl(T_A^{(Q)})$ and $h\geq \lvl(T_B^{(J)})$ and that
\begin{equation} \label{eq:compatibility}
f_P^{!}\left(\add\left(T_A^{(Q)}\right) \cap \cF_P\right) = g_I^{\ast}\left(\add\left( T_B^{(J)}\right) \cap \cG_I\right).
\end{equation}
Then $(T_\La^L, T_\La^R)=(T_B^L, T_B^R) \glue[P][I] (T_A^L, T_A^R)$ is a fracturing of $\La = B \glue[P][I] A$ and  $\cC_{\La}=\add\left\{\phi_\ast(\cC_A), \psi_\ast(\cC_B)\right\}$ is a $(T_\La^L, T_\La^R,n)$-fractured subcategory.
\end{theorem}

\begin{proof}
Assume that $B=KA_h$. Then $T_B^L=T_B^R$ by Example \ref{ex:KAhfracturing}, since $B$ admits a $(T_B^L,T_B^R,n)$-fractured subcategory by assumption. By Example \ref{ex:KAhfracturedglue} and condition (\ref{eq:compatibility}) it then follows that $\La=A$, $(T_{\La}^L,T_{\La}^R)=(T_A^L,T_A^R)$ and $\cC_{\La}=\cC_A$ and so the result holds. Similarly if $A=KA_h$. Hence we may assume that $A \not\cong KA_h$ and $B\not\cong KA_h$.

We need to prove conditions (1)--(4) of Definition \ref{def:nctfract}. We pick idempotents of $A,B$ and $\La$ as in Section \ref{subsubsection:gluing via pullbacks}. For condition (1) we need to show that $\cP_{\La}^L\subseteq \cC_{\La}$, or equivalently
\[\add\left\{\left(\cP_{\La}\right)_{\setminus\cP_{\La}^{\text{ab}}}, T_{\La}^L\right\}\subseteq \add\left\{\phi_{\ast}(\cC_{A}),\psi_{\ast}(\cC_B)\right\}.\]

By the construction of $T_{\La}^L$, and since $T_A^L\in \add\{\cC_A\}$ and $T_B^L\in\add\{\cC_B\}$, it follows that $T_{\La}^L\in \cC_{\La}$. Then, if $\vare_i\La$ is an indecomposable projective $\La$-module, it is enough to show that if $\vare_i\La$ is not a left abutment, then  $\vare_i\La\in \cC_{\La}$. If $1\leq i \leq m$, then $\vare_i\La$ is supported in $B$ by Proposition \ref{prop:projinj} and so $\psi^{\ast}(\vare_i\La)=\epsilon_i B$. By Corollary \ref{cor:remaining}(b1) it follows that $\epsilon_i B$ is not a left abutment of $B$ and so $\epsilon_i B\in \cC_B$. Hence \[\psi_{\ast}(\epsilon_i B)=\vare_i \La \in \psi_{\ast}(\cC_B)\subseteq \cC_{\La},\]
as required. 

If $h-l+1\leq i \leq 0$, then $\vare_i \La$ is supported in $A$ and $\phi^{!}(\vare_i\La)=e_i A$. If $e_i A$ is not a left abutment of $A$, then a similar argument as before shows that $\vare_i\La\in\cC_{\La}$. If $e_i A$ is a left abutment of $A$, then we must have that $\phi_{\ast}(\cF_{e_i A} \cap \cPD ) \neq 0$, since this intersection being zero implies via Corollary \ref{cor:remaining}(a2) that $\vare_i\La$ is a left abutment, contradicting our assumption. In particular, we have $\cF_{e_i A} \cap \cF_{e_1A} \neq 0$, since by assumption $P\cong e_1A$. Since $i<1$, by Proposition \ref{prop:triangles} we have that $\cF_{e_1A} \subsetneq \cF_{e_iA}$ and that $\nsup{e_1A}{\triangle}$ is a full subquiver of $\nsup{e_iA}{\triangle}$. In particular, $e_1A$ and $e_iA$ are both abutments appearing in the same radical series of the maximal abutment $Q$ and the height of $e_iA$ is greater than $h$. Since $h\geq \lvl(T_A^{(Q)})$, it follows from the definition of the level that $e_iA$ is an indecomposable summand of $T_A^{(Q)}$. Since $e_iA\in \add(T_A^{(Q)})\subseteq \cC_A$, it follows that $\vare_i\La\in \cC_{\La}$, as required. 

Hence, condition (1) is satisfied. Conditions (3) and (4) follow immediately by Corollary \ref{cor:compute} and the corresponding conditions being true for $\cC_A$ and $\cC_B$. 

It remains to show that condition (2) holds for $\cC_{\La}$. Let $M\in \cC_{\setminus\cP_{\La}^L}$. We will show that $\tn(M)\in \cC_{\setminus\cI^R_{\La}}$ and $\tno\tn(M)\cong M$; the dual fact that if $N\in \cC_{\setminus\cI^R_{\La}}$ then $\tno(N) \in \cC_{\setminus\cP^L_{\La}}$ and $\tn\tno(N)\cong N$ can be shown similarly.

Since $M$ is indecomposable, it follows that $M\in\add(\phi_{\ast}(\cC_A))$ or $M\in \add(\psi_{\ast}(\cC_B))$. Therefore, $\phi^{!}(M)\in \cC_A$ or $\psi^{\ast}(M)\in \cC_B$ and both functors preserve indecomposability.

Assume first that $\psi^{\ast}(M)\in\cC_B$. We claim that $\psi^{\ast}(M)\in(\cC_B)_{\setminus\cP_B^L}$. Assume by way of contradiction that $\psi^{\ast}(M)\in \cP_B^L=\add\{(\cP_B)_{\setminus\cP^{\text{ab}}_B},T_B^L\}$ instead. If $\psi^{\ast}(M)$ is projective but not a left abutment, then we reach a contradiction since in that case $M$ would also be projective but not a left abutment by Proposition \ref{prop:projinj} and Corollary \ref{cor:remaining}. Hence $M\in \add(T_B^L)$ and so $M\in \add(T_B^{(Z)})$ for some maximal left abutment $Z$. It follows from Corollary \ref{cor:AR glued} and Proposition \ref{prop:triangles} that $\psi_{\ast}(Z)$ is a maximal left abutment of $\La$ unless $B\cong KA_h$ and $Z\cong I \not\cong J$. But by our assumption $B\not\cong KA_h$ and so $\psi_{\ast}(Z)$ is indeed a maximal left abutment of $\La$. It follows that
\[M\cong \psi_{\ast}\psi^{\ast}(M)\in \add\left(\psi_{\ast}(T^{(Z)}_B)\right)\subseteq \add\{T_{\La}^L\},\]
contradicting $M\not\in \cC_{\setminus\cP^L_{\La}}$.

Hence we have $\psi^{\ast}(M)\in (\cC_{B})_{\setminus\cP_B^L}$. Since $\cC_B$ is a $(T_B^L, T_B^{R},n)$-fractured subcategory, it follows that $\tn \psi^{\ast}(M)\in 
(\cC_B)_{\setminus\cI^{R}_{B}}$. A similar argument as before shows that $\psi_{\ast}((\cC_B)_{\setminus\cI^R_B}) \subseteq (\cC_{\La})_{\setminus\cI_{\La}^R}$. Moreover, by Corollary \ref{cor:compute}, it follows that $\tn (M)\cong \psi_{\ast}\tn\psi^{\ast}(M)$ and so $\tn(M)\in (\cC_{\La})_{\setminus\cI_{\La}^R}$. The previous argument shows also that we can compute 
\[\tno \tn M \cong \psi_{\ast} \tno \psi^{\ast} \psi_{\ast} \tn \psi^{\ast} (M)\cong \psi_{\ast}\tno\tn \psi^{\ast}(M)\cong \psi_{\ast}\psi^{\ast}(M)\cong M,\]
as required.

Finally, it remains to check the case $\phi^{!}(M)\in \cC_A$. As before we can easily show that $\phi^{!}(M)\in (\cC_A)_{\setminus\cP_A^L}$. We note in particular that by the condition $h\geq \lvl(T_A^{(Q)})$ it follows that if $\phi^{!}(M)\in \cF_Q$ then $\phi^{!}(M)\in \add(T_A^{(Q)}) \cap \cF_{P}$. Then we distinguish two cases. If $\phi^{!}(M)\not\in\add(T_A^{(Q)})\cap \cPD$, then $\tn$ can be computed inside $\m A$ as per Corollary \ref{cor:compute} and the previous case. On the other hand, if $\phi^{!}(M)\in \add(T_A^{(Q)})\cap \cPD$, it follows from Corollary \ref{cor:AR glued} that $M$ is supported both in $A$ and in $B$. In particular, viewing $M$ as a $KA_h$-module via the compositions $\m \La \overset{\phi^{!}}{\longrightarrow} \m A \overset{f_P^{!}}{\longrightarrow} \m KA_h$ and $\m\La \overset{\psi^{\ast}}{\longrightarrow} \m B \overset{g_I^{\ast}}{\longrightarrow} \m KA_h$ produces the same module by Corollary \ref{cor:AR glued}. Hence the compatibility condition (\ref{eq:compatibility}) implies that $\phi^{\ast}(M)\in \add(T_B^{(J)})\cap \cG_I$. In particular $\psi^{\ast}(M)\in\cC_B$, in which case we showed that condition (2) is satisfied. This completes the proof.
\end{proof}

The following corollary of Theorem \ref{thrm:fractsubcat} is of particular interest.

\begin{corollary}\label{cor:glueatsimple}
Let $A$ be a strongly $(n,d_1)$-rep\-re\-sen\-ta\-tion-di\-rect\-ed algebra and $B$ be a strongly $(n,d_2)$-rep\-re\-sen\-ta\-tion-di\-rect\-ed algebra. Let $P$ be a simple projective $A$-module and $I$ be a simple injective $B$-module. Then $\La=B \glue[P][I] A$ is a strongly $(n,d)$-rep\-re\-sen\-ta\-tion-di\-rect\-ed algebra for some $d$ with $\max\{d_1,d_2\}\leq d\leq d_1+d_2$.
\end{corollary}

\begin{proof}
First we have that $\La$ is rep\-re\-sen\-ta\-tion-di\-rect\-ed by Corollary \ref{cor:Laisdirected} and that $\max\{d_1,d_2\}\leq d\leq d_1+d_2$ by Corollary \ref{cor:globaldimensionglued}. It remains to show that $\La$ admits an $n$-cluster tilting subcategory $\cC_{\La}$. If $n=1$, then $\cC_{\La}=\m\La$. Assume that $n\geq 2$. By Proposition \ref{prop:fracluster}, we have that there exists a $(\Lab_A, \DLab_A,n)$-fractured subcategory $\cC_A$ of $\m A$ and a $(\Lab_B, \DLab_B, n)$-fractured subcategory $\cC_{B}$ of $B$. If $Q$ is a maximal left abutment of $A$, then the corresponding tilting module $T^{(Q)}_A$ for the fracturing $(\Lab_A, \DLab_B)$ is projective, since $\Lab_A$ is projective. It follows that $\lvl(T^{(Q)}_A)=1$. Similarly if $J$ is a maximal right abutment of $B$, we have that $\lvl(T^{(J)}_B)=1$.

Since $P$ and $I$ are simple, we have that both $P$ and $I$ have height $1$. In particular, if $Q$ is a maximal left $A$-abutment  with $\cPD\subseteq \cF_Q$ and $J$ is a maximal right $B$-abutment with $\cDI \subseteq \cG_J$, it follows that 
\[f_P^!(\add(T_A^{(Q)})\cap \cPD) \cong \add(S(1)) \cong g_I^{\ast}(\add(T_B^{(J)})\cap \cDI),\] 
where $S(1)$ is the unique simple $KA_1$-module. Hence, it follows by Theorem \ref{thrm:fractsubcat} that $(T_\La^L, T_\La^R)=(\Lab_B, \DLab_B) \glue[P][I] (\Lab_A, \DLab_A)$ is a fracturing of $\La$ and $\cC_{\La}=\add\left\{\phi_\ast(\cC_A), \psi_\ast(\cC_B)\right\}$ is a $(T_\La^L, T_\La^R,n)$-fractured subcategory. It remains to show that $T_{\La}^L\cong \Lab_{\La}$ and $T_{\La}^R\cong \DLab_{\La}$. Let us only show the first isomorphism; the other follows by similar arguments.

We have 
\[T^L_{\La} = \bigoplus_{[R]\in \MABP_{\La}}T^{(R)}_\ast\]
where
\[ T^{(R)}_\ast = \begin{cases} \psi_\ast\left(T^{(\psi^{\ast}(R))}_B\right) &\mbox{ if $R$ is supported in $B$,} \\ \phi_\ast\left(T^{(\phi^{!}(R))}_A\right) &\mbox{ otherwise,} \end{cases}\]
so it is enough to show that $T_{\ast}^{(R)}$ is projective. If $R$ is supported in $B$, then $T_B^{(\psi^{\ast}(R))}$ is a projective $B$-module by assumption, and so its image under $\psi_{\ast}$ is a projective $\La$-module by Proposition \ref{prop:projinj}. If $R$ is supported in $A$, then there is no arrow going into the triangle  $^{R}\!\triangle$. Then by Corollary \ref{cor:AR glued} there is no arrow going into the triangle $^{\phi^!(R)}\!\triangle$. Let $h'$ be the height of $R$. Since $T^{(\phi^!(R))}_A$ is projective by assumption, it is the unique projective fracture corresponding to the triangle $^{\phi^!(R)}\!\triangle$. That is, $T^{\phi^!(R)}_A$ is  isomorphic to a direct sum $\bigoplus_{i=1}^{h'}T_i$, where $\{[T_i]\}_{i=1}^{h'}$ are all the different leftmost vertices in the triangle $^{\phi^!(R)}\!\triangle$. Hence lifting this through $\phi_{\ast}$, we again get a direct sum corresponding to the leftmost vertices of the triangle $^{R}\!\triangle$, which is a projective $\La$-module, completing the proof.
\end{proof}

Since rep\-re\-sen\-ta\-tion-di\-rect\-ed algebras always have simple projective and injective modules, Corollary \ref{cor:glueatsimple} can be used to construct arbitrarily many $n$-cluster tilting subcategories from known $n$-cluster tilting subcategories of rep\-re\-sen\-ta\-tion-di\-rect\-ed algebras. 

We describe the next simplest case of using Theorem \ref{thrm:fractsubcat}. First we need to have $T_A^R=\cI_A$ and $T_A^L$ to have exactly one nonprojective fracture $T_A^{(Q)}$ corresponding to a maximal left abutment $Q$. Similarly we need to have $T_B^L=\cP_L$ and $T_B^R$ to have exactly one noninjective fracture $T_B^{(J)}$ corresponding to a maximal right abutment $J$. Then, after gluing the resulting subcategory will be a $(\cP_{\La}, \cI_{\La},n)$-fractured subcategory or equivalently an $n$-cluster tilting subcategory.

Even if there are more nonprojective fractures chosen for the left fracturing of $A$ (or similarly noninjective fractures chosen for the right fracturing of $B$), by the construction of gluing one can glue at each fracture independently. Say we have an algebra $\La$ and that at each nonprojective fracture we glue by a left $n$-cluster tilting subcategory, while at each noninjective fracture we glue by a right $n$-cluster tilting subcategory and each gluing is compatible as per the requirements of Theorem \ref{thrm:fractsubcat}. Then the result will be an algebra such that the gluing of all the fractured subcategories is an $n$-cluster tilting subcategory. We illustrate with a detailed example.

\begin{example}\label{ex:nctwithtwoglues}
Let $B$, $(T_B^L, T_B^R)$ and $\cC_B$ be as in Example \ref{ex:fracturedsubcategory}. Recall that $\cC_B$ is a left $n$-cluster tilting subcategory obtained by repeatedly applying $\tau_2^-$ starting from $B$ and there are two noninjective fractures in $T_B^R$, namely 
\[T^{\left({I(3)}_B\right)}=\qthree{}[3][]\raisebox{-1ex}{$\scriptstyle B$}\oplus\qthree{2}[3]\raisebox{-1.4ex}{$\scriptstyle B$}\oplus\qthree{1}[2][3]\raisebox{-1.8ex}{$\scriptstyle B$}\;\; \text{ and } \;T^{\left({I(3')}_B\right)}=\qthree{}[2'][]\raisebox{-1ex}{$\scriptstyle B$}\oplus\qthree{2'}[3']\raisebox{-1.4ex}{$\scriptstyle B$}\oplus\qthree{1'}[2'][3']\raisebox{-2ex}{$\scriptstyle B$},\]

We want to glue two appropriate algebras with $B$, one alongside $\qthree{1}[2][3]\raisebox{-1.8ex}{$\scriptstyle B$}$ and one alongside $\qthree{1'}[2'][3']\raisebox{-2ex}{$\scriptstyle B$}$. Let us start with $\qthree{1}[2][3]\raisebox{-1.8ex}{$\scriptstyle B$}$. Consider the algebra $A$ as in Example \ref{ex:first glue}. It is easy to see that $A$ admits a $2$-cluster tilting subcategory, given by $\cC_A=\add(A\oplus D(A))$. Then by Proposition \ref{prop:fracluster} we have that $\cC_A$ is $(T_A^L, T_A^R,2)$-fractured subcategory where 
\[T^{\left({P(1)}_A\right)}=\qthree{}[3][]\raisebox{-1ex}{$\scriptstyle A$}\oplus\qthree{2}[3]\raisebox{-1.4ex}{$\scriptstyle A$}\oplus\qthree{1}[2][3]\raisebox{-1.8ex}{$\scriptstyle A$}.\]
Hence, viewing $T^{\left({I(3)}_B\right)}$ and $T^{\left({P(1)}_A\right)}$ as $KA_3$-modules via the respective functors, we have that they coincide since
\[g_{{I(3)}_B}^{\ast} \left( \qthree{}[3][]\raisebox{-1ex}{$\scriptstyle B$}\oplus\qthree{2}[3]\raisebox{-1.4ex}{$\scriptstyle B$}\oplus\qthree{1}[2][3]\raisebox{-1.8ex}{$\scriptstyle B$}\right) \cong \qthree{}[3][]\raisebox{-1ex}{$\scriptstyle KA_3$}\oplus\qthree{2}[3]\raisebox{-1.4ex}{$\scriptstyle KA_3$}\oplus\qthree{1}[2][3]\raisebox{-1.8ex}{$\scriptstyle KA_3$} \cong f_{{P(1)}_A}^{!}\left(\qthree{}[3][]\raisebox{-1ex}{$\scriptstyle A$}\oplus\qthree{2}[3]\raisebox{-1.4ex}{$\scriptstyle A$}\oplus\qthree{1}[2][3]\raisebox{-1.8ex}{$\scriptstyle A$}\right).\]
In particular, by Theorem \ref{thrm:fractsubcat}, the algebra $\La_1=B \glue[{P(1)}_A][{I(3)}_B] A$ admits a $(T_{\La_1}^L, T_{\La_1}^R,2)$-fractured subcategory $\cC_{\La_1}$ where $(T_{\La_1}^L, T_{\La_1}^R)=(T_{B}^L, T_{B}^R) \glue[{P(1)}_A][{I(3)}_B] (T_{A}^L, T_{A}^R)$. Viewing the Aus\-lan\-der--Rei\-ten quivers of $B$ and $A$ embedded in the Aus\-lan\-der--Rei\-ten quiver of $\La_1$ we can find an additive generator of $\cC_{\La_1}$. If we denote the indecomposable modules in the $2$-fractured subcategories by encircling the corresponding vertices we have:
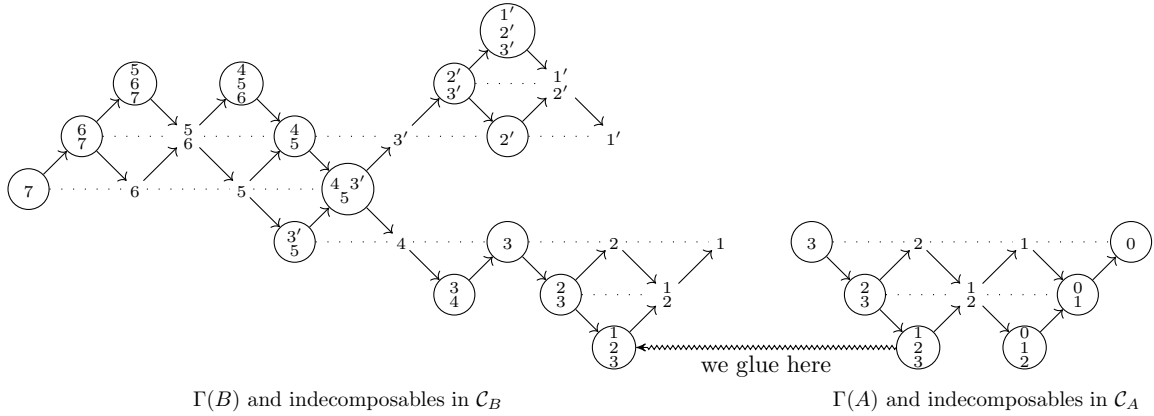
\[\begin{tikzpicture}
\tikzstyle{nct2}=[circle, minimum width=0.6cm, draw, inner sep=0pt, text centered, scale=0.9]
\tikzstyle{nct22}=[circle, minimum width=0.6cm, draw, inner sep=0pt, text centered, scale=0.9]
\tikzstyle{nct3}=[circle, minimum width=6pt, draw=white, inner sep=0pt, scale=0.9]
\node[scale=0.8] (Name) at (4.2,-2.8) {$\Gamma(B)$ and indecomposables in $\cC_B$};
\node[scale=0.8] (Name2) at (12.6,-2.8) {$\Gamma(A)$ and indecomposables in $\cC_A$};

\node[nct2] (A) at (0,0) {$\qthree{}[7][]$};
\node[nct2] (B) at (0.7,0.7) {$\qthree{6}[7]$};
\node[nct2] (C) at (1.4,1.4) {$\qthree{5}[6][7]$};
\node[nct3] (D) at (1.4,0) {$\qthree{}[6][]$};
\node[nct3] (E) at (2.1,0.7) {$\qthree{5}[6]$};
\node[nct2] (F) at (2.8,1.4) {$\qthree{4}[5][6]$};
\node[nct3] (G) at (2.8,0) {$\qthree{}[5][]$};
\node[nct2] (H) at (3.5,0.7) {$\qthree{4}[5]$};
\node[nct2] (I) at (3.5,-0.7) {$\qthree{3'}[5]$};
\node[nct22] (J) at (4.2,0) {$\begin{smallmatrix} 4 && 3' \\ & 5 &
\end{smallmatrix}$};
\node[nct3] (K) at (4.9,0.7) {$\qthree{}[3'][]$};
\node[nct3] (L) at (4.9,-0.7) {$\qthree{}[4][]$};
\node[nct2] (M) at (5.6,1.4) {$\qthree{2'}[3']$};
\node[nct2] (N) at (5.6,-1.4) {$\qthree{3}[4]$};
\node[nct2] (O) at (6.3,2.1) {$\qthree{1'}[2'][3']$};
\node[nct2] (P) at (6.3,0.7) {$\qthree{}[2'][]$};
\node[nct2] (Q) at (6.3,-0.7) {$\qthree{}[3][]$};
\node[nct3] (R) at (7,1.4) {$\qthree{1'}[2']$};
\node[nct2] (S) at (7,-1.4) {$\qthree{2}[3]$};
\node[nct3] (T) at (7.7,0.7) {$\qthree{}[1'][]$};
\node[nct3] (U) at (7.7,-0.7) {$\qthree{}[2][]$};
\node[nct2] (V) at (7.7,-2.1) {$\qthree{1}[2][3]$};
\node[nct3] (W) at (8.4,-1.4) {$\qthree{1}[2]$};
\node[nct3] (X) at (9.1,-0.7) {$\qthree{}[1][]$};

\draw[->] (A) to (B);
\draw[->] (B) to (C);
\draw[->] (B) to (D);
\draw[->] (C) to (E);
\draw[->] (D) to (E);
\draw[->] (E) to (F);
\draw[->] (E) to (G);
\draw[->] (F) to (H);
\draw[->] (G) to (H);
\draw[->] (G) to (I);
\draw[->] (H) to (J);
\draw[->] (I) to (J);
\draw[->] (J) to (K);
\draw[->] (J) to (L);
\draw[->] (K) to (M);
\draw[->] (L) to (N);
\draw[->] (M) to (O);
\draw[->] (M) to (P);
\draw[->] (N) to (Q);
\draw[->] (O) to (R);
\draw[->] (P) to (R);
\draw[->] (Q) to (S);
\draw[->] (R) to (T);
\draw[->] (S) to (U);
\draw[->] (S) to (V);
\draw[->] (U) to (W);
\draw[->] (V) to (W);
\draw[->] (W) to (X);

\node[nct2] (AA) at (10.3,-0.7) {$\qthree{}[3][]$};
\node[nct2] (BB) at (11,-1.4) {$\qthree{2}[3][]$}; 
\node[nct3] (CC) at (11.7,-0.7) {$\qthree{}[2][]$}; 
\node[nct2] (DD) at (11.7,-2.1) {$\qthree{1}[2][3]$}; 
\node[nct3] (EE) at (12.4,-1.4) {$\qthree{1}[2][]$}; 
\node[nct3] (FF) at (13.1,-0.7) {$\qthree{}[1][]$}; 
\node[nct2] (GG) at (13.1,-2.1) {$\qthree{0}[1][2]$}; 
\node[nct2] (HH) at (13.8,-1.4) {$\qthree{0}[1][]$}; 
\node[nct2] (II) at (14.5,-0.7) {$\qthree{}[0][]$}; 

\draw[->] (AA) to (BB);
\draw[->] (BB) to (CC);
\draw[->] (BB) to (DD);
\draw[->] (CC) to (EE);
\draw[->] (DD) to (EE);
\draw[->] (EE) to (FF);
\draw[->] (EE) to (GG);
\draw[->] (FF) to (HH);
\draw[->] (GG) to (HH);
\draw[->] (HH) to (II);

\draw[decorate, decoration={zigzag, segment length=+2pt, amplitude=+.75pt,post length=+4pt}, -stealth'] (DD.west) -- (V.east) node[below, midway, scale=0.9] {we glue here};

\draw[loosely dotted] (A.east) -- (D);
\draw[loosely dotted] (B.east) -- (E);
\draw[loosely dotted] (D.east) -- (G);
\draw[loosely dotted] (E.east) -- (H);
\draw[loosely dotted] (G.east) -- (J);
\draw[loosely dotted] (H.east) -- (K);
\draw[loosely dotted] (I.east) -- (L);
\draw[loosely dotted] (K.east) -- (P);
\draw[loosely dotted] (M.east) -- (R);
\draw[loosely dotted] (L.east) -- (Q);
\draw[loosely dotted] (Q.east) -- (U);
\draw[loosely dotted] (P.east) -- (T);
\draw[loosely dotted] (S.east) -- (W);
\draw[loosely dotted] (U.east) -- (X);

\draw[loosely dotted] (AA.east) -- (CC);
\draw[loosely dotted] (BB.east) -- (EE);
\draw[loosely dotted] (CC.east) -- (FF);
\draw[loosely dotted] (EE.east) -- (HH);
\draw[loosely dotted] (FF.east) -- (II);
\end{tikzpicture}\]
and after gluing we get
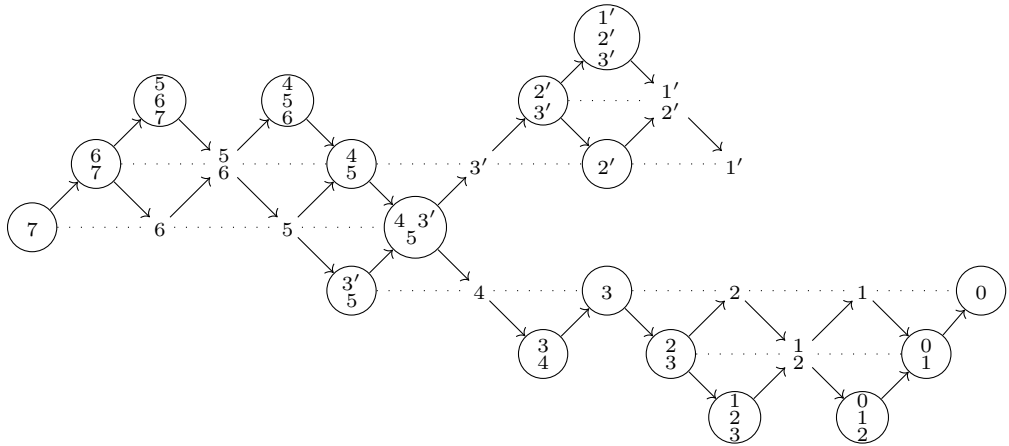
\[\begin{tikzpicture}[scale=1.2, transform shape]
\tikzstyle{nct2}=[circle, minimum width=0.6cm, draw, inner sep=0pt, text centered, scale=0.9]
\tikzstyle{nct22}=[circle, minimum width=0.6cm, draw, inner sep=0pt, text centered, scale=0.9]
\tikzstyle{nct3}=[circle, minimum width=6pt, draw=white, inner sep=0pt, scale=0.9]
\node[scale=0.8] (Name) at (5,-2.8) {$\Gamma(\La_1)$ and indecomposables in $\cC_{\La_1}$\nospacepunct{.}};

\node[nct2] (A) at (0,0) {$\qthree{}[7][]$};
\node[nct2] (B) at (0.7,0.7) {$\qthree{6}[7]$};
\node[nct2] (C) at (1.4,1.4) {$\qthree{5}[6][7]$};
\node[nct3] (D) at (1.4,0) {$\qthree{}[6][]$};
\node[nct3] (E) at (2.1,0.7) {$\qthree{5}[6]$};
\node[nct2] (F) at (2.8,1.4) {$\qthree{4}[5][6]$};
\node[nct3] (G) at (2.8,0) {$\qthree{}[5][]$};
\node[nct2] (H) at (3.5,0.7) {$\qthree{4}[5]$};
\node[nct2] (I) at (3.5,-0.7) {$\qthree{3'}[5]$};
\node[nct22] (J) at (4.2,0) {$\begin{smallmatrix} 4 && 3' \\ & 5 &
\end{smallmatrix}$};
\node[nct3] (K) at (4.9,0.7) {$\qthree{}[3'][]$};
\node[nct3] (L) at (4.9,-0.7) {$\qthree{}[4][]$};
\node[nct2] (M) at (5.6,1.4) {$\qthree{2'}[3']$};
\node[nct2] (N) at (5.6,-1.4) {$\qthree{3}[4]$};
\node[nct2] (O) at (6.3,2.1) {$\qthree{1'}[2'][3']$};
\node[nct2] (P) at (6.3,0.7) {$\qthree{}[2'][]$};
\node[nct2] (Q) at (6.3,-0.7) {$\qthree{}[3][]$};
\node[nct3] (R) at (7,1.4) {$\qthree{1'}[2']$};
\node[nct2] (S) at (7,-1.4) {$\qthree{2}[3]$};
\node[nct3] (T) at (7.7,0.7) {$\qthree{}[1'][]$};
\node[nct3] (U) at (7.7,-0.7) {$\qthree{}[2][]$};
\node[nct2] (V) at (7.7,-2.1) {$\qthree{1}[2][3]$};
\node[nct3] (W) at (8.4,-1.4) {$\qthree{1}[2]$};
\node[nct3] (X) at (9.1,-0.7) {$\qthree{}[1][]$};
\node[nct2] (Y) at (9.1,-2.1) {$\qthree{0}[1][2]$};
\node[nct2] (Z) at (9.8,-1.4) {$\qthree{0}[1]$};
\node[nct2] (AA) at (10.4,-0.7) {$\qthree{}[0][]$};

\draw[->] (A) to (B);
\draw[->] (B) to (C);
\draw[->] (B) to (D);
\draw[->] (C) to (E);
\draw[->] (D) to (E);
\draw[->] (E) to (F);
\draw[->] (E) to (G);
\draw[->] (F) to (H);
\draw[->] (G) to (H);
\draw[->] (G) to (I);
\draw[->] (H) to (J);
\draw[->] (I) to (J);
\draw[->] (J) to (K);
\draw[->] (J) to (L);
\draw[->] (K) to (M);
\draw[->] (L) to (N);
\draw[->] (M) to (O);
\draw[->] (M) to (P);
\draw[->] (N) to (Q);
\draw[->] (O) to (R);
\draw[->] (P) to (R);
\draw[->] (Q) to (S);
\draw[->] (R) to (T);
\draw[->] (S) to (U);
\draw[->] (S) to (V);
\draw[->] (U) to (W);
\draw[->] (V) to (W);
\draw[->] (W) to (X);
\draw[->] (X) to (Z);
\draw[->] (Z) to (AA);
\draw[->] (W) to (Y);
\draw[->] (Y) to (Z);

\draw[loosely dotted] (A.east) -- (D);
\draw[loosely dotted] (B.east) -- (E);
\draw[loosely dotted] (D.east) -- (G);
\draw[loosely dotted] (E.east) -- (H);
\draw[loosely dotted] (G.east) -- (J);
\draw[loosely dotted] (H.east) -- (K);
\draw[loosely dotted] (I.east) -- (L);
\draw[loosely dotted] (K.east) -- (P);
\draw[loosely dotted] (M.east) -- (R);
\draw[loosely dotted] (L.east) -- (Q);
\draw[loosely dotted] (Q.east) -- (U);
\draw[loosely dotted] (P.east) -- (T);
\draw[loosely dotted] (S.east) -- (W);
\draw[loosely dotted] (U.east) -- (X);
\draw[loosely dotted] (W.east) -- (Z);
\draw[loosely dotted] (X.east) -- (AA);
\end{tikzpicture}\]
In particular, $\cC_{\La_1}$ is a $2$-left cluster tilting subcategory, as expected. Moreover, $\La_1$ has two maximal right abutments, namely ${I(2)}_{\La_1}$ and ${I(3')}_{\La_1}$. The fracture corresponding to the first one is injective, while the fracture of the second one is
\[T^{\left({I(3')}_{\La_1}\right)}=\qthree{}[2'][]\raisebox{-1ex}{$\scriptstyle \La_1$}\oplus\qthree{2'}[3']\raisebox{-1.4ex}{$\scriptstyle \La_1$}\oplus\qthree{1'}[2'][3']\raisebox{-2ex}{$\scriptstyle \La_1$},\]
which is noninjective. Hence we want to glue at ${I(3')}_{\La_1}$. Let $C$ be the algebra given by the quiver with relations 
\[\begin{tikzpicture}[scale=0.9, transform shape]
\node (-2) at (-3.4,1) {$-2'$};
\node (-1) at (-2.2,1) {$-1'$};
\node (0) at (-1,1) {$0'$};
\node (1) at (0,1) {$1'$};
\node (2) at (1,1) {$2'$};
\node (3) at (2,1) {$3'$\nospacepunct{.}};

\draw[dotted] (0) to [out=30,in=150] (3);
\draw[dotted] (-1) to [out=30,in=150] (2);
\draw[dotted] (-2) to [out=30,in=150] (0);

\draw[->] (-2) to (-1);
\draw[->] (-1) to (0);
\draw[->] (0) to (1);
\draw[->] (1) to (2);
\draw[->] (2) to (3);
\end{tikzpicture}\]
Then the Aus\-lan\-der--Rei\-ten quiver $\Gamma(C)$ of $C$ is 
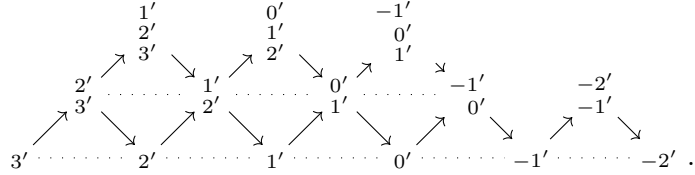
\[\begin{tikzpicture}[scale=1.2, transform shape]
\tikzstyle{nct3}=[circle, minimum width=6pt, draw=white, inner sep=0pt, scale=0.9]
\node[nct3] (A) at (0,0) {$\qthree{}[3'][]$};
\node[nct3] (B) at (0.7,0.7) {$\qthree{2'}[3']$};
\node[nct3] (C) at (1.4,1.4) {$\qthree{1'}[2'][3']$};
\node[nct3] (D) at (1.4,0) {$\qthree{}[2'][]$};
\node[nct3] (E) at (2.1,0.7) {$\qthree{1'}[2']$};
\node[nct3] (F) at (2.8,1.4) {$\qthree{0'}[1'][2']$};
\node[nct3] (G) at (2.8,0) {$\qthree{}[1'][]$};
\node[nct3] (H) at (3.5,0.7) {$\qthree{0'}[1']$};
\node[nct3] (I) at (4.2,0) {$\qthree{}[0'][]$};
\node[nct3] (J) at (4.2,1.4) {$\qthree{-1'\;\;\;}[0'][1']$};
\node[nct3] (K) at (4.9,0.7) {$\qthree{-1'}[\;\;\;0']$};
\node[nct3] (L) at (5.6,0) {$\qthree{}[-1'][]$};
\node[nct3] (M) at (6.3,0.7) {$\qthree{-2'}[-1']$};
\node[nct3] (N) at (7,0) {$\qthree{}[-2'][]$\nospacepunct{.}};

\draw[->] (A) to (B);
\draw[->] (B) to (C);
\draw[->] (B) to (D);
\draw[->] (C) to (E);
\draw[->] (D) to (E);
\draw[->] (E) to (F);
\draw[->] (E) to (G);
\draw[->] (F) to (H);
\draw[->] (G) to (H);
\draw[->] (H) to (I);
\draw[->] (H) to (J);
\draw[->] (J) to (K);
\draw[->] (I) to (K);
\draw[->] (K) to (L);
\draw[->] (L) to (M);
\draw[->] (M) to (N);

\draw[loosely dotted] (A.east) -- (D);
\draw[loosely dotted] (B.east) -- (E);
\draw[loosely dotted] (D.east) -- (G);
\draw[loosely dotted] (E.east) -- (H);
\draw[loosely dotted] (G.east) -- (I);
\draw[loosely dotted] (H.east) -- (K);
\draw[loosely dotted] (I.east) -- (L);
\draw[loosely dotted] (L.east) -- (N);
\end{tikzpicture}\]
Hence by Proposition \ref{prop:triangles} there is a unique maximal left abutment, namely ${P(1')}_C=\qthree{1'}[2'][3']\raisebox{-2ex}{$\scriptstyle C$}$ and a unique maximal right abutment, namely $\qthree{-2'}[-1']\raisebox{-1.4ex}{$\scriptstyle C$}$. It follows that $(T_C^L, T_C^R)$ is a fracturing of $C$, where 
\[T_C^L=\qthree{}[2'][]\raisebox{-1ex}{$\scriptstyle C$}\oplus\qthree{2'}[3']\raisebox{-1.4ex}{$\scriptstyle C$}\oplus\qthree{1'}[2'][3']\raisebox{-2ex}{$\scriptstyle C$} \text{\; and \;} T_C^R=\qthree{-2'}[-1']\raisebox{-1.4ex}{$\scriptstyle C$}\oplus \qthree{}[-2'][]\raisebox{-1ex}{$\scriptstyle C$}.\]
It is easy to see that $C$ has a $(T_C^L, T_C^R,2)$-fractured subcategory such that the gluing 
\[\La_2=\La_1 \glue[{P(1')}_C][{I(3')}_{\La_1}] C\] 
is compatible according to Theorem \ref{thrm:fractsubcat}. Hence the gluing of the subcategories $\cC_{\La_1}$ and $\cC_C$ is a $2$-cluster tilting subcategory. Concretely, the Aus\-lan\-der--Rei\-ten quivers of $\La_1$ and $C$ along with their $2$-fractured subcategories are
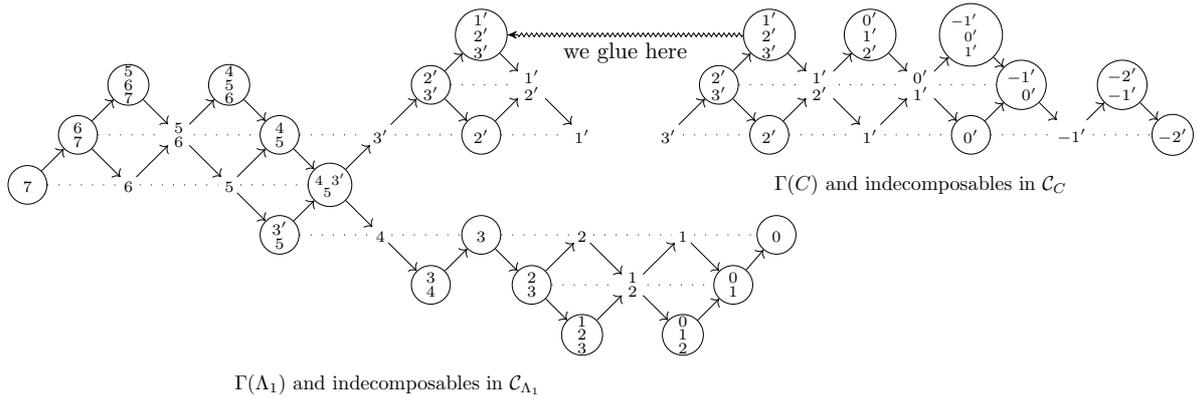
\[\resizebox{\textwidth}{!}{\begin{tikzpicture}
\tikzstyle{nct}=[shape= rectangle, minimum width=6pt, minimum height=7.5, draw, inner sep=0pt]
\tikzstyle{nct2}=[circle, minimum width=0.6cm, draw, inner sep=0pt, text centered, scale=0.9]
\tikzstyle{nct22}=[circle, minimum width=0.6cm, draw, inner sep=0pt, text centered, scale=0.8]
\tikzstyle{nct3}=[circle, minimum width=6pt, draw=white, inner sep=0pt, scale=0.9]
\node[scale=0.8] (Name) at (5,-2.8) {$\Gamma(\La_1)$ and indecomposables in $\cC_{\La_1}$};
\node[scale=0.8] (Name) at (12.4,0) {$\Gamma(C)$ and indecomposables in $\cC_{C}$};

\node[nct2] (A) at (0,0) {$\qthree{}[7][]$};
\node[nct2] (B) at (0.7,0.7) {$\qthree{6}[7]$};
\node[nct2] (C) at (1.4,1.4) {$\qthree{5}[6][7]$};
\node[nct3] (D) at (1.4,0) {$\qthree{}[6][]$};
\node[nct3] (E) at (2.1,0.7) {$\qthree{5}[6]$};
\node[nct2] (F) at (2.8,1.4) {$\qthree{4}[5][6]$};
\node[nct3] (G) at (2.8,0) {$\qthree{}[5][]$};
\node[nct2] (H) at (3.5,0.7) {$\qthree{4}[5]$};
\node[nct2] (I) at (3.5,-0.7) {$\qthree{3'}[5]$};
\node[nct22] (J) at (4.2,0) {$\begin{smallmatrix} 4 && 3' \\ & 5 &
\end{smallmatrix}$};
\node[nct3] (K) at (4.9,0.7) {$\qthree{}[3'][]$};
\node[nct3] (L) at (4.9,-0.7) {$\qthree{}[4][]$};
\node[nct2] (M) at (5.6,1.4) {$\qthree{2'}[3']$};
\node[nct2] (N) at (5.6,-1.4) {$\qthree{3}[4]$};
\node[nct2] (O) at (6.3,2.1) {$\qthree{1'}[2'][3']$};
\node[nct2] (P) at (6.3,0.7) {$\qthree{}[2'][]$};
\node[nct2] (Q) at (6.3,-0.7) {$\qthree{}[3][]$};
\node[nct3] (R) at (7,1.4) {$\qthree{1'}[2']$};
\node[nct2] (S) at (7,-1.4) {$\qthree{2}[3]$};
\node[nct3] (T) at (7.7,0.7) {$\qthree{}[1'][]$};
\node[nct3] (U) at (7.7,-0.7) {$\qthree{}[2][]$};
\node[nct2] (V) at (7.7,-2.1) {$\qthree{1}[2][3]$};
\node[nct3] (W) at (8.4,-1.4) {$\qthree{1}[2]$};
\node[nct3] (X) at (9.1,-0.7) {$\qthree{}[1][]$};
\node[nct2] (Y) at (9.1,-2.1) {$\qthree{0}[1][2]$};
\node[nct2] (Z) at (9.8,-1.4) {$\qthree{0}[1]$};
\node[nct2] (AA) at (10.4,-0.7) {$\qthree{}[0][]$};

\draw[->] (A) to (B);
\draw[->] (B) to (C);
\draw[->] (B) to (D);
\draw[->] (C) to (E);
\draw[->] (D) to (E);
\draw[->] (E) to (F);
\draw[->] (E) to (G);
\draw[->] (F) to (H);
\draw[->] (G) to (H);
\draw[->] (G) to (I);
\draw[->] (H) to (J);
\draw[->] (I) to (J);
\draw[->] (J) to (K);
\draw[->] (J) to (L);
\draw[->] (K) to (M);
\draw[->] (L) to (N);
\draw[->] (M) to (O);
\draw[->] (M) to (P);
\draw[->] (N) to (Q);
\draw[->] (O) to (R);
\draw[->] (P) to (R);
\draw[->] (Q) to (S);
\draw[->] (R) to (T);
\draw[->] (S) to (U);
\draw[->] (S) to (V);
\draw[->] (U) to (W);
\draw[->] (V) to (W);
\draw[->] (W) to (X);
\draw[->] (X) to (Z);
\draw[->] (Z) to (AA);
\draw[->] (W) to (Y);
\draw[->] (Y) to (Z);

\node[nct3] (A2) at (8.9,0.7) {$\qthree{}[3'][]$};
\node[nct2] (B2) at (9.6,1.4) {$\qthree{2'}[3']$};
\node[nct2] (C2) at (10.3,2.1) {$\qthree{1'}[2'][3']$};
\node[nct2] (D2) at (10.3,0.7) {$\qthree{}[2'][]$};
\node[nct3] (E2) at (11,1.4) {$\qthree{1'}[2']$};
\node[nct2] (F2) at (11.7,2.1) {$\qthree{0'}[1'][2']$};
\node[nct3] (G2) at (11.7,0.7) {$\qthree{}[1'][]$};
\node[nct3] (H2) at (12.4,1.4) {$\qthree{0'}[1']$};
\node[nct2] (I2) at (13.1,0.7) {$\qthree{}[0'][]$};
\node[nct22] (J2) at (13.1,2.1) {$\qthree{-1'\;\;\;}[0'][1']$};
\node[nct2] (K2) at (13.8,1.4) {$\qthree{-1'}[\;\;\;0']$};
\node[nct3] (L2) at (14.5,0.7) {$\qthree{}[-1'][]$};
\node[nct2] (M2) at (15.2,1.4) {$\qthree{-2'}[-1']$};
\node[nct2] (N2) at (15.9,0.7) {$\qthree{}[-2'][]$};

\draw[->] (A2) to (B2);
\draw[->] (B2) to (C2);
\draw[->] (B2) to (D2);
\draw[->] (C2) to (E2);
\draw[->] (D2) to (E2);
\draw[->] (E2) to (F2);
\draw[->] (E2) to (G2);
\draw[->] (F2) to (H2);
\draw[->] (G2) to (H2);
\draw[->] (H2) to (I2);
\draw[->] (H2) to (J2);
\draw[->] (J2) to (K2);
\draw[->] (I2) to (K2);
\draw[->] (K2) to (L2);
\draw[->] (L2) to (M2);
\draw[->] (M2) to (N2);

\draw[decorate, decoration={zigzag, segment length=+2pt, amplitude=+.75pt,post length=+4pt}, -stealth'] (C2.west) -- (O.east) node[below, midway, scale=0.9] {we glue here};

\draw[loosely dotted] (A.east) -- (D);
\draw[loosely dotted] (B.east) -- (E);
\draw[loosely dotted] (D.east) -- (G);
\draw[loosely dotted] (E.east) -- (H);
\draw[loosely dotted] (G.east) -- (J);
\draw[loosely dotted] (H.east) -- (K);
\draw[loosely dotted] (I.east) -- (L);
\draw[loosely dotted] (K.east) -- (P);
\draw[loosely dotted] (M.east) -- (R);
\draw[loosely dotted] (L.east) -- (Q);
\draw[loosely dotted] (Q.east) -- (U);
\draw[loosely dotted] (P.east) -- (T);
\draw[loosely dotted] (S.east) -- (W);
\draw[loosely dotted] (U.east) -- (X);
\draw[loosely dotted] (W.east) -- (Z);
\draw[loosely dotted] (X.east) -- (AA);

\draw[loosely dotted] (A2.east) -- (D2);
\draw[loosely dotted] (B2.east) -- (E2);
\draw[loosely dotted] (D2.east) -- (G2);
\draw[loosely dotted] (E2.east) -- (H2);
\draw[loosely dotted] (G2.east) -- (I2);
\draw[loosely dotted] (H2.east) -- (K2);
\draw[loosely dotted] (I2.east) -- (L2);
\draw[loosely dotted] (L2.east) -- (N2);
\end{tikzpicture}}\]
the algebra $\La_2$ is given by the quiver with relations
\[\begin{tikzpicture}[scale=0.9, transform shape]
\node (1) at (0,1) {$1$};
\node (2) at (1,1) {$2$};
\node (3) at (2,1) {$3$};
\node (4) at (3,1) {$4$};
\node (5) at (4,0.5) {$5$};
\node (6) at (5,0.5) {$6$};
\node (7) at (6,0.5) {$7$\nospacepunct{,}};
\node (1') at (1,0) {$1'$};
\node (2') at (2,0) {$2'$};
\node (3') at (3,0) {$3'$};

\node (0) at (-1,1) {$0$};
\draw[->] (0) to (1);
\draw[dotted] (0) to [out=30,in=150] (3);

\draw[->] (1) to (2);
\draw[->] (2) to (3);
\draw[->] (3) to (4);
\draw[->] (4) to (5);
\draw[->] (5) to (6);
\draw[->] (6) to (7);
\draw[->] (1') to (2');
\draw[->] (2') to (3');
\draw[->] (3') to (5);

\draw[dotted] (2) to [out=30,in=150] (4);
\draw[dotted] (3) to [out=-30,in=0] (5.west);
\draw[dotted] (4) to [out=0,in=160] (7);
\draw[dotted] (2') to [out=30,in=0] (5.west);
\draw[dotted] (3') to [out=0,in=210] (6);

\node (-2') at (-2.4,0) {$-2'$};
\node (-1') at (-1.2,0) {$-1'$};
\node (0') at (0,0) {$0'$};

\draw[dotted] (0') to [out=30,in=150] (3');
\draw[dotted] (-1') to [out=30,in=150] (2');
\draw[dotted] (-2') to [out=30,in=150] (0');

\draw[->] (-2') to (-1');
\draw[->] (-1') to (0');
\draw[->] (0') to (1');
\end{tikzpicture}\]
and the Aus\-lan\-der--Rei\-ten quiver $\Gamma(\La_2)$ of $\La_2$ with the $2$-cluster tilting subcategory $\cC_{\La_2}$ is
\[\begin{tikzpicture}[scale=1.2, transform shape]
\tikzstyle{nct2}=[circle, minimum width=0.6cm, draw, inner sep=0pt, text centered, scale=0.9]
\tikzstyle{nct22}=[circle, minimum width=0.6cm, draw, inner sep=0pt, text centered, scale=0.8]
\tikzstyle{nct3}=[circle, minimum width=6pt, draw=white, inner sep=0pt, scale=0.9]
\node[scale=0.8] (Name) at (6,-2.8) {$\Gamma(\La_2)$ and indecomposables in $\cC_{\La_2}$\nospacepunct{.}};

\node[nct2] (A) at (0,0) {$\qthree{}[7][]$};
\node[nct2] (B) at (0.7,0.7) {$\qthree{6}[7]$};
\node[nct2] (C) at (1.4,1.4) {$\qthree{5}[6][7]$};
\node[nct3] (D) at (1.4,0) {$\qthree{}[6][]$};
\node[nct3] (E) at (2.1,0.7) {$\qthree{5}[6]$};
\node[nct2] (F) at (2.8,1.4) {$\qthree{4}[5][6]$};
\node[nct3] (G) at (2.8,0) {$\qthree{}[5][]$};
\node[nct2] (H) at (3.5,0.7) {$\qthree{4}[5]$};
\node[nct2] (I) at (3.5,-0.7) {$\qthree{3'}[5]$};
\node[nct22] (J) at (4.2,0) {$\begin{smallmatrix} 4 && 3' \\ & 5 &
\end{smallmatrix}$};
\node[nct3] (L) at (4.9,-0.7) {$\qthree{}[4][]$};
\node[nct2] (N) at (5.6,-1.4) {$\qthree{3}[4]$};
\node[nct2] (Q) at (6.3,-0.7) {$\qthree{}[3][]$};
\node[nct2] (S) at (7,-1.4) {$\qthree{2}[3]$};
\node[nct3] (U) at (7.7,-0.7) {$\qthree{}[2][]$};
\node[nct2] (V) at (7.7,-2.1) {$\qthree{1}[2][3]$};
\node[nct3] (W) at (8.4,-1.4) {$\qthree{1}[2]$};
\node[nct3] (X) at (9.1,-0.7) {$\qthree{}[1][]$};
\node[nct2] (Y) at (9.1,-2.1) {$\qthree{0}[1][2]$};
\node[nct2] (Z) at (9.8,-1.4) {$\qthree{0}[1]$};
\node[nct2] (AA) at (10.4,-0.7) {$\qthree{}[0][]$};

\draw[->] (A) to (B);
\draw[->] (B) to (C);
\draw[->] (B) to (D);
\draw[->] (C) to (E);
\draw[->] (D) to (E);
\draw[->] (E) to (F);
\draw[->] (E) to (G);
\draw[->] (F) to (H);
\draw[->] (G) to (H);
\draw[->] (G) to (I);
\draw[->] (H) to (J);
\draw[->] (I) to (J);
\draw[->] (J) to (L);
\draw[->] (L) to (N);
\draw[->] (N) to (Q);
\draw[->] (Q) to (S);
\draw[->] (S) to (U);
\draw[->] (S) to (V);
\draw[->] (U) to (W);
\draw[->] (V) to (W);
\draw[->] (W) to (X);
\draw[->] (X) to (Z);
\draw[->] (Z) to (AA);
\draw[->] (W) to (Y);
\draw[->] (Y) to (Z);

\node[nct3] (A2) at (4.9,0.7) {$\qthree{}[3'][]$};
\node[nct2] (B2) at (5.6,1.4) {$\qthree{2'}[3']$};
\node[nct2] (C2) at (6.3,2.1) {$\qthree{1'}[2'][3']$};
\node[nct2] (D2) at (6.3,0.7) {$\qthree{}[2'][]$};
\node[nct3] (E2) at (7,1.4) {$\qthree{1'}[2']$};
\node[nct2] (F2) at (7.7,2.1) {$\qthree{0'}[1'][2']$};
\node[nct3] (G2) at (7.7,0.7) {$\qthree{}[1'][]$};
\node[nct3] (H2) at (8.4,1.4) {$\qthree{0'}[1']$};
\node[nct2] (I2) at (9.1,0.7) {$\qthree{}[0'][]$};
\node[nct22] (J2) at (9.1,2.1) {$\qthree{-1'\;\;\;}[0'][1']$};
\node[nct2] (K2) at (9.8,1.4) {$\qthree{-1'}[\;\;\;0']$};
\node[nct3] (L2) at (10.5,0.7) {$\qthree{}[-1'][]$};
\node[nct2] (M2) at (11.2,1.4) {$\qthree{-2'}[-1']$};
\node[nct2] (N2) at (11.9,0.7) {$\qthree{}[-2'][]$};

\draw[->] (J) to (A2);

\draw[->] (A2) to (B2);
\draw[->] (B2) to (C2);
\draw[->] (B2) to (D2);
\draw[->] (C2) to (E2);
\draw[->] (D2) to (E2);
\draw[->] (E2) to (F2);
\draw[->] (E2) to (G2);
\draw[->] (F2) to (H2);
\draw[->] (G2) to (H2);
\draw[->] (H2) to (I2);
\draw[->] (H2) to (J2);
\draw[->] (J2) to (K2);
\draw[->] (I2) to (K2);
\draw[->] (K2) to (L2);
\draw[->] (L2) to (M2);
\draw[->] (M2) to (N2);

\draw[loosely dotted] (A.east) -- (D);
\draw[loosely dotted] (B.east) -- (E);
\draw[loosely dotted] (D.east) -- (G);
\draw[loosely dotted] (E.east) -- (H);
\draw[loosely dotted] (G.east) -- (J);
\draw[loosely dotted] (H.east) -- (A2);
\draw[loosely dotted] (I.east) -- (L);
\draw[loosely dotted] (L.east) -- (Q);
\draw[loosely dotted] (Q.east) -- (U);
\draw[loosely dotted] (S.east) -- (W);
\draw[loosely dotted] (U.east) -- (X);
\draw[loosely dotted] (W.east) -- (Z);
\draw[loosely dotted] (X.east) -- (AA);

\draw[loosely dotted] (A2.east) -- (D2);
\draw[loosely dotted] (B2.east) -- (E2);
\draw[loosely dotted] (D2.east) -- (G2);
\draw[loosely dotted] (E2.east) -- (H2);
\draw[loosely dotted] (G2.east) -- (I2);
\draw[loosely dotted] (H2.east) -- (K2);
\draw[loosely dotted] (I2.east) -- (L2);
\draw[loosely dotted] (L2.east) -- (N2);
\end{tikzpicture}\]
\end{example}

\begin{remark}
The algebras of Example \ref{ex:nctwithtwoglues} and Corollary \ref{cor:glueatsimple} give rise to algebras with many interesting properties. For instance let us consider the number of sinks and sources in the quiver of an algebra. Following the notation of Corollary \ref{cor:numberofprojandinj}, and since the number of sinks (respectively sources) in the quiver of an algebra is equal to the number of simple projective (respectively injective) modules, we denote for an algebra $\La$ by $s_{\La}$ the number of sources in its quiver and by $t_{\La}$ the number of sinks in its quiver. Then let $A$ be a strongly $(2,d_A)$-rep\-re\-sen\-ta\-tion-di\-rect\-ed algebra with $s_A=2$ and $t_A=1$ (for example, we may take $A$ to be the algebra $\La_2$ as in Example \ref{ex:nctwithtwoglues}). Let $B$ be a $(2,d_B)$-rep\-re\-sen\-ta\-tion-di\-rect\-ed algebra which admits a $2$-cluster tilting subcategory. By gluing at the simple projective $A$-module and any simple injective $B$-module we get the algebra $B^{(1)} = B \glue A$. By Corollary \ref{cor:glueatsimple}, we have that $B^{(1)}$ admits a $2$-cluster tilting subcategory. By Corollary \ref{cor:numberofprojandinj} we have that $(s_{B^{(1)}},t_{B^{(1)}})=(s_B,t_B+1)$. Continuing inductively, let $B^{(i)}$ be a sequence of algebras defined by $B^{(i)}=B^{(i-1)} \glue A $ where the gluing is done over any simple projective $A$-module and any simple injective $B^{(i-1)}$-module. Then we get that $B^{(i)}$ admits a $2$-cluster tilting subcategory and $(s_{B^{(i)}}, t_{B^{(i)}})=(s_B,t_B+i)$.

A similar argument shows that if we let $B_{(j)}$ be a sequence of algebras defined by $B_{(1)}=A^{\text{op}} \glue B$ and $B_{(j)} = A^{\text{op}} \glue B_{(j-1)}$, where all gluings are done over simple modules, then again $B_{(j)}$ admits a $2$-cluster tilting subcategory and $(s_{B_{(j)}}, t_{B_{(j)}})=(s_B+j,t_B)$. More generally, we have that
\[\Big(s_{B^{(i)}_{(j)}},t_{B^{(i)}_{(j)}}\Big) = (s_B+j,t_B+i).\]
In particular, by choosing $(s_B,t_B)=(1,1)$ (for example, we may take $B=KA_h/\rad(KA_h)^{h-1}$ for some $h\geq 3$), we have that for any pair $(s,t)$ with $s,t\geq 1$ there exists an algebra $\La$ such that $\La$ admits a $2$-cluster tilting subcategory and $(s_{\La}, t_{\La})=(s,t)$. It follows that for any given pair of numbers $(s,t)$, there exists a quiver $Q$ with $s$ sinks and $t$ sources and a bound quiver algebra $\La=KQ/\cR$ such that $\La$ admits a $2$-cluster tilting subcategory. Note that by construction the number of vertices of the quiver of $\La$ is of the order of $s+t$ but can be made arbitrarily large.
\end{remark}

In Example \ref{ex:nctwithtwoglues} it was not clear how one should find the algebras $A$ and $C$. They depended on the type of fractures that the algebra $B$ had and clearly they are not unique since we can always glue at simple modules via Corollary \ref{cor:glueatsimple}. The fractures in this example corresponded to slice modules of $KA_3$ and we will see in section \ref{sect:slices} how we can find appropriate algebras to glue in this case. More generally we have the following question.

\begin{question} \label{Question: can we glue?} Let $n\geq 2$ and let $T=\bigoplus_{i=1}^hT_i$ be a basic tilting module of $KA_h$.
\begin{itemize}
\item[(a)] Can we find a rep\-re\-sen\-ta\-tion-di\-rect\-ed algebra $B$ with a right fracturing $T^R_B = \bigoplus_{[J]\in \MIAB}T^{(J)}$ such that the conditions 
\begin{enumerate}
    \item[(1)] there exists a maximal right abutment $I$ of $B$ such that $g_I^{\ast}(T^{(I)})\cong T$,
    \item[(2)] for every maximal right abutment $J$ with $J\not\cong I$ we have that $T^{(J)}$ is injective, and
    \item[(3)] there exists a $(\Lab_B, T_B^R,n)$-fractured subcategory
\end{enumerate}
are satisfied?

\item[(b)] Can we find a rep\-re\-sen\-ta\-tion-di\-rect\-ed algebra $A$ with a left fracturing $T^L_A = \bigoplus_{[Q]\in \MABP}T^{(Q)}$ such that the conditions
\begin{enumerate}
    \item[(1)] there exists a maximal left abutment $P$ of $A$ such that $f_P^{!}(T^{(P)})\cong T$,
    \item[(2)] for every maximal left abutment $Q$ with $Q\not\cong P$ we have that $T^{(Q)}$ is projective, and
    \item[(3)] there exists a $(T_A^L,\DLab_A,n)$-fractured subcategory
\end{enumerate}
are satisfied?
\end{itemize}
\end{question}

If we can answer Question \ref{Question: can we glue?}(a) (respectively \ref{Question: can we glue?}(b)) affirmatively we will say that \emph{we can complete $T$ on the left (respectively right)}. Notice that by symmetry we can complete $T$ on the left if and only if we can complete $D(T)$ on the right by taking $A=B^{\text{op}}$. 

As a special case, let $T=\bigoplus_{k=1}^h M(i_k,j_k)$ be such that $T\cong\bigoplus_{k=1}^{h}M(h+2-j_k-i_k,j_k)$. This condition is equivalent to saying that viewing the indecomposable summands of $T$ as vertices in $\triangle(h)$, they are symmetric along the perpendicular bisector of the bottom line of the triangle. It follows that if we can answer Question \ref{Question: can we glue?}(a) affirmatively in this case, then the algebra $B \glue[D(I)][I] B^{\text{op}}$ admits an $n$-cluster tilting subcategory by Theorem \ref{thrm:fractsubcat}. A similar result holds if we can answer Question \ref{Question: can we glue?}(b) affirmatively. We illustrate this situation with an example.

\begin{example}\label{ex:selfdual}
Let $A$ be given by the quiver with relations
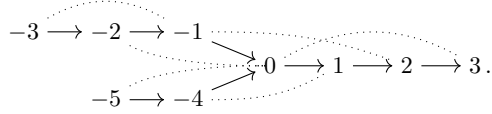
\[\begin{tikzpicture}[scale=0.9, transform shape]
\node (-5) at (-2.4,0) {$-5$};
\node (-4) at (-1.2,0) {$-4$};
\node (-3) at (-3.6,1) {$-3$};
\node (-2) at (-2.4,1) {$-2$};
\node (-1) at (-1.2,1) {$-1$};
\node (0) at (0,0.5) {$0$};
\node (1) at (1,0.5) {$1$};
\node (2) at (2,0.5) {$2$};
\node (3) at (3,0.5) {$3$\nospacepunct{.}};

\draw[->] (-5) to (-4);
\draw[->] (-3) to (-2);
\draw[->] (-2) to (-1);
\draw[->] (-1) to (0);
\draw[->] (-4) to (0);
\draw[->] (0) to (1);
\draw[->] (1) to (2);
\draw[->] (2) to (3);

\draw[dotted] (-2) to [out=-30,in=0] (0.west);
\draw[dotted] (-5) to [out=30,in=0] (0.west);
\draw[dotted] (0) to [out=30,in=150] (3);
\draw[dotted] (-3) to [out=30,in=150] (-1);
\draw[dotted] (-1) to [out=0,in=160] (2);
\draw[dotted] (-4) to [out=0,in=210] (1);
\end{tikzpicture}\]

By Proposition \ref{prop:abutmentsquiver} there is a unique maximal left abutment, namely ${P(1)}_A=\qthree{1}[2][3]$. If we set $T^R_A=I^{\text{ab}}$ to be an injective right fracturing of $A$ and
\[T^L_A = T^{({(P(1)}_A)} = \qthree{}[3][]\raisebox{-1ex}{$\scriptstyle A$}\oplus \qthree{1}[2][3]\raisebox{-1.8ex}{$\scriptstyle A$}\oplus \qthree{}[1]{}\raisebox{-1ex}{$\scriptstyle A$},\]
then $(T^L_A, T^R_A)$ is a fracturing of $A$ and there exists a $(T^L_A, T^R_A,3)$-fractured subcategory $\cC_A$. The Aus\-lan\-der--Rei\-ten quiver $\Gamma(A)$ as well as the indecomposable modules in $\cC_A$ are
\[\begin{tikzpicture}[scale=1.2, transform shape]
\tikzstyle{nct2}=[circle, minimum width=0.6cm, draw, inner sep=0pt, text centered, scale=0.9]
\tikzstyle{nct22}=[circle, minimum width=0.6cm, draw, inner sep=0pt, text centered, scale=0.8]
\tikzstyle{nct3}=[circle, minimum width=6pt, draw=white, inner sep=0pt, scale=0.9]
\node[scale=0.8] (Name) at (3,-2) {$\Gamma(A)$ and indecomposables in $\cC_{A}$\nospacepunct{.}};

\node[nct2] (Z) at (0,1.4) {$\qthree{1}[2][3]$};
\node[nct3] (W) at (-0.7,0.7) {$\qthree{2}[3]$};
\node[nct2] (V) at (-1.4,0) {$\qthree{}[3][]$};
\node[nct3] (A) at (0,0) {$\qthree{}[2][]$};
\node[nct3] (B) at (0.7,0.7) {$\qthree{1}[2]$};
\node[nct2] (C) at (1.4,1.4) {$\qthree{0}[1][2]$};
\node[nct2] (D) at (1.4,0) {$\qthree{}[1][]$};
\node[nct3] (E) at (2.1,0.7) {$\qthree{0}[1]$};
\node[nct22] (F) at (2.8,1.4) {$\qthree{-1\;\;\;}[0][1]$};
\node[nct3] (G) at (2.8,0) {$\qthree{}[0][]$};
\node[nct3] (H) at (3.5,0.7) {$\qthree{-1\;\;\;}[0]$};
\node[nct2] (I) at (3.5,-0.7) {$\qthree{-4\;\;\;}[0]$};
\node[nct22] (J) at (4.2,0) {$\begin{smallmatrix} -1 & -4 \\ & 0\;\;\;
\end{smallmatrix}$};
\node[nct3] (K) at (4.9,0.7) {$\qthree{}[-4][]$};
\node[nct3] (L) at (4.9,-0.7) {$\qthree{}[-1][]$};
\node[nct2] (M) at (5.6,1.4) {$\qthree{-5}[-4]$};
\node[nct2] (N) at (5.6,-1.4) {$\qthree{-2}[-1]$};
\node[nct2] (P) at (6.3,0.7) {$\qthree{}[-5][]$};
\node[nct3] (Q) at (6.3,-0.7) {$\qthree{}[-2][]$};
\node[nct2] (S) at (7,-1.4) {$\qthree{-3}[-2]$};
\node[nct2] (U) at (7.7,-0.7) {$\qthree{}[-3][]$};

\draw[->] (V) to (W);
\draw[->] (W) to (A);
\draw[->] (W) to (Z);
\draw[->] (Z) to (B);
\draw[->] (A) to (B);
\draw[->] (B) to (C);
\draw[->] (B) to (D);
\draw[->] (C) to (E);
\draw[->] (D) to (E);
\draw[->] (E) to (F);
\draw[->] (E) to (G);
\draw[->] (F) to (H);
\draw[->] (G) to (H);
\draw[->] (G) to (I);
\draw[->] (H) to (J);
\draw[->] (I) to (J);
\draw[->] (J) to (K);
\draw[->] (J) to (L);
\draw[->] (K) to (M);
\draw[->] (L) to (N);
\draw[->] (M) to (P);
\draw[->] (N) to (Q);
\draw[->] (Q) to (S);
\draw[->] (S) to (U);

\draw[loosely dotted] (V.east) -- (A);
\draw[loosely dotted] (W.east) -- (B);
\draw[loosely dotted] (A.east) -- (D);
\draw[loosely dotted] (B.east) -- (E);
\draw[loosely dotted] (D.east) -- (G);
\draw[loosely dotted] (E.east) -- (H);
\draw[loosely dotted] (G.east) -- (J);
\draw[loosely dotted] (H.east) -- (K);
\draw[loosely dotted] (I.east) -- (L);
\draw[loosely dotted] (K.east) -- (P);
\draw[loosely dotted] (L.east) -- (Q);
\draw[loosely dotted] (Q.east) -- (U);
\end{tikzpicture}\]
In particular, $\cC_A$ is a right $3$-cluster tilting subcategory. If we view the fracture appearing in the foundation of ${P(1)}_A$ as a $KA_3$-module, we have
\[f^{!}_{P(1)_A}\left(T_A^{\left({P(1)}_A\right)}\right) \cong M(1,1)\oplus M(1,3) \oplus M(3,1)\]
and
\begin{align*} 
M(3+2-1-1,1)\oplus M(3+2-1-3,3) \oplus M(3+2-3-1,1) &= M(3,1)\oplus M(1,3)\oplus M(1,1)\\
&\cong f^{!}_{P(1)_A}\left(T_A^{\left({P(1)}_A\right)}\right).
\end{align*}
Then the algebra $B=A^{\text{op}}$ is given by the quiver with relations
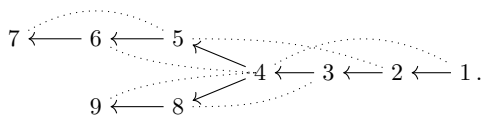
\[\begin{tikzpicture}[scale=0.9, transform shape]
\node (9) at (-2.4,0) {$9$};
\node (8) at (-1.2,0) {$8$};
\node (7) at (-3.6,1) {$7$};
\node (6) at (-2.4,1) {$6$};
\node (5) at (-1.2,1) {$5$};
\node (4) at (0,0.5) {$4$};
\node (3) at (1,0.5) {$3$};
\node (2) at (2,0.5) {$2$};
\node (1) at (3,0.5) {$1$\nospacepunct{.}};

\draw[->] (1) to (2);
\draw[->] (2) to (3);
\draw[->] (3) to (4);
\draw[->] (4) to (5);
\draw[->] (4) to (8);
\draw[->] (5) to (6);
\draw[->] (6) to (7);
\draw[->] (8) to (9);

\draw[dotted] (6) to [out=-30,in=0] (4.west);
\draw[dotted] (9) to [out=30,in=0] (4.west);
\draw[dotted] (4) to [out=30,in=150] (1);
\draw[dotted] (7) to [out=30,in=150] (5);
\draw[dotted] (5) to [out=0,in=160] (2);
\draw[dotted] (8) to [out=0,in=210] (3);
\end{tikzpicture}\]
and there exists a unique maximal right abutment of $B$, namely ${I(3)}_B$. Then, for the choice of fracturing $\left(P_B^{\text{ab}},T_B^{\left({I(3)}_B\right)}\right)$ with 
\[g_{{I(3)}_B}^{\ast}\left(T_B^{\left({I(3)}_B\right)}\right) \cong f_{{P(1)}_A}^{!} \left(T_A^{\left({P(1)}_A\right)}\right)\]
there exists a left $3$-cluster tilting subcategory $\cC_B=D(\cC_A)$. Hence we can apply Theorem \ref{thrm:fractsubcat}. The algebra $\La=B \glue[{P(1)}_A][{I(3)}_B] A$ is given by the quiver with relations
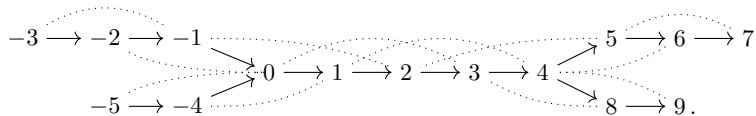
\[\begin{tikzpicture}[scale=0.9, transform shape]
\node (-5) at (-2.4,0) {$-5$};
\node (-4) at (-1.2,0) {$-4$};
\node (-3) at (-3.6,1) {$-3$};
\node (-2) at (-2.4,1) {$-2$};
\node (-1) at (-1.2,1) {$-1$};
\node (0) at (0,0.5) {$0$};
\node (1) at (1,0.5) {$1$};
\node (2) at (2,0.5) {$2$};
\node (3) at (3,0.5) {$3$};

\draw[->] (-5) to (-4);
\draw[->] (-3) to (-2);
\draw[->] (-2) to (-1);
\draw[->] (-1) to (0);
\draw[->] (-4) to (0);
\draw[->] (0) to (1);
\draw[->] (1) to (2);
\draw[->] (2) to (3);

\draw[dotted] (-2) to [out=-30,in=0] (0.west);
\draw[dotted] (-5) to [out=30,in=0] (0.west);
\draw[dotted] (0) to [out=30,in=150] (3);
\draw[dotted] (-3) to [out=30,in=150] (-1);
\draw[dotted] (-1) to [out=0,in=160] (2);
\draw[dotted] (-4) to [out=0,in=210] (1);

\node (9) at (6,0) {$9$\nospacepunct{.}};
\node (8) at (5,0) {$8$};
\node (7) at (7,1) {$7$};
\node (6) at (6,1) {$6$};
\node (5) at (5,1) {$5$};
\node (4) at (4,0.5) {$4$};

\draw[->] (3) to (4);
\draw[->] (4) to (5);
\draw[->] (4) to (8);
\draw[->] (5) to (6);
\draw[->] (6) to (7);
\draw[->] (8) to (9);

\draw[dotted] (6) to [out=-150,in=0] (4.east);
\draw[dotted] (9) to [out=150,in=0] (4.east);
\draw[dotted] (1) to [out=30,in=150] (4);
\draw[dotted] (5) to [out=30,in=150] (7);
\draw[dotted] (5) to [out=180,in=20] (2);
\draw[dotted] (8) to [out=180,in=-30] (3);
\end{tikzpicture}\]
Then the Aus\-lan\-der--Rei\-ten quivers of $A$ and $B$ along with their $3$-fractured subcategories are
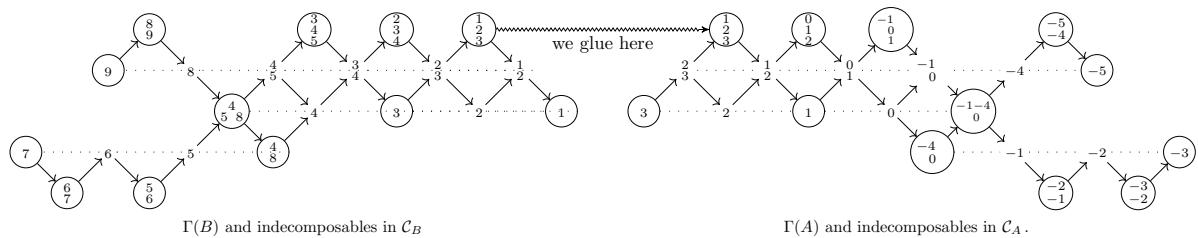
\[\resizebox{\textwidth}{!}{\begin{tikzpicture}
\tikzstyle{nct}=[shape= rectangle, minimum width=6pt, minimum height=7.5, draw, inner sep=0pt]
\tikzstyle{nct2}=[circle, minimum width=0.6cm, draw, inner sep=0pt, text centered, scale=0.9]
\tikzstyle{nct22}=[circle, minimum width=0.6cm, draw, inner sep=0pt, text centered, scale=0.8]
\tikzstyle{nct3}=[circle, minimum width=6pt, draw=white, inner sep=0pt, scale=0.9]
\node[scale=0.8] (Name) at (3,-2) {$\Gamma(A)$ and indecomposables in $\cC_{A}$\nospacepunct{.}};
\node[scale=0.8] (Name) at (-7.2,-2) {$\Gamma(B)$ and indecomposables in $\cC_{B}$};

\node[nct2] (Z) at (0,1.4) {$\qthree{1}[2][3]$};
\node[nct3] (W) at (-0.7,0.7) {$\qthree{2}[3]$};
\node[nct2] (V) at (-1.4,0) {$\qthree{}[3][]$};
\node[nct3] (A) at (0,0) {$\qthree{}[2][]$};
\node[nct3] (B) at (0.7,0.7) {$\qthree{1}[2]$};
\node[nct2] (C) at (1.4,1.4) {$\qthree{0}[1][2]$};
\node[nct2] (D) at (1.4,0) {$\qthree{}[1][]$};
\node[nct3] (E) at (2.1,0.7) {$\qthree{0}[1]$};
\node[nct22] (F) at (2.8,1.4) {$\qthree{-1\;\;\;}[0][1]$};
\node[nct3] (G) at (2.8,0) {$\qthree{}[0][]$};
\node[nct3] (H) at (3.5,0.7) {$\qthree{-1\;\;\;}[0]$};
\node[nct2] (I) at (3.5,-0.7) {$\qthree{-4\;\;\;}[0]$};
\node[nct22] (J) at (4.2,0) {$\begin{smallmatrix} -1 & -4 \\ & 0\;\;\;
\end{smallmatrix}$};
\node[nct3] (K) at (4.9,0.7) {$\qthree{}[-4][]$};
\node[nct3] (L) at (4.9,-0.7) {$\qthree{}[-1][]$};
\node[nct2] (M) at (5.6,1.4) {$\qthree{-5}[-4]$};
\node[nct2] (N) at (5.6,-1.4) {$\qthree{-2}[-1]$};
\node[nct2] (P) at (6.3,0.7) {$\qthree{}[-5][]$};
\node[nct3] (Q) at (6.3,-0.7) {$\qthree{}[-2][]$};
\node[nct2] (S) at (7,-1.4) {$\qthree{-3}[-2]$};
\node[nct2] (U) at (7.7,-0.7) {$\qthree{}[-3][]$};

\draw[->] (V) to (W);
\draw[->] (W) to (A);
\draw[->] (W) to (Z);
\draw[->] (Z) to (B);
\draw[->] (A) to (B);
\draw[->] (B) to (C);
\draw[->] (B) to (D);
\draw[->] (C) to (E);
\draw[->] (D) to (E);
\draw[->] (E) to (F);
\draw[->] (E) to (G);
\draw[->] (F) to (H);
\draw[->] (G) to (H);
\draw[->] (G) to (I);
\draw[->] (H) to (J);
\draw[->] (I) to (J);
\draw[->] (J) to (K);
\draw[->] (J) to (L);
\draw[->] (K) to (M);
\draw[->] (L) to (N);
\draw[->] (M) to (P);
\draw[->] (N) to (Q);
\draw[->] (Q) to (S);
\draw[->] (S) to (U);

\node[nct2] (ZZ) at (-4.2,1.4) {$\qthree{1}[2][3]$};
\node[nct3] (WW) at (-4.9,0.7) {$\qthree{2}[3]$};
\node[nct2] (VV) at (-5.6,0) {$\qthree{}[3][]$};
\node[nct3] (AA) at (-4.2,0) {$\qthree{}[2][]$};
\node[nct3] (BB) at (-3.5,0.7) {$\qthree{1}[2]$};
\node[nct2] (DD) at (-2.8,0) {$\qthree{}[1][]$};
\node[nct2] (CC) at (-5.6,1.4) {$\qthree{2}[3][4]$};
\node[nct3] (EE) at (-6.3,0.7) {$\qthree{3}[4]$};
\node[nct3] (GG) at (-7,0) {$\qthree{}[4][]$};
\node[nct2] (FF) at (-7,1.4) {$\qthree{3}[4][5]$};
\node[nct3] (HH) at (-7.7,0.7) {$\qthree{4}[5]$};
\node[nct2] (II) at (-7.7,-0.7) {$\qthree{4}[8]$};
\node[nct2] (JJ) at (-8.4,0) {$\begin{smallmatrix}  & 4 & \\ 5 && 8
\end{smallmatrix}$};
\node[nct3] (KK) at (-9.1,0.7) {$\qthree{}[8][]$};
\node[nct2] (MM) at (-9.8,1.4) {$\qthree{8}[9]$};
\node[nct2] (PP) at (-10.5,0.7) {$\qthree{}[9][]$};
\node[nct3] (LL) at (-9.1,-0.7) {$\qthree{}[5][]$};
\node[nct2] (NN) at (-9.8,-1.4) {$\qthree{5}[6]$};
\node[nct3] (QQ) at (-10.5,-0.7) {$\qthree{}[6][]$};
\node[nct2] (SS) at (-11.2,-1.4) {$\qthree{6}[7]$};
\node[nct2] (UU) at (-11.9,-0.7) {$\qthree{}[7][]$};

\draw[->] (VV) to (WW);
\draw[->] (WW) to (AA);
\draw[->] (WW) to (ZZ);
\draw[->] (ZZ) to (BB);
\draw[->] (AA) to (BB);
\draw[->] (BB) to (DD);
\draw[->] (CC) to (WW);
\draw[->] (EE) to (CC);
\draw[->] (EE) to (VV);
\draw[->] (GG) to (EE);
\draw[->] (FF) to (EE);
\draw[->] (HH) to (FF);
\draw[->] (HH) to (GG);
\draw[->] (II) to (GG);
\draw[->] (JJ) to (HH);
\draw[->] (JJ) to (II);
\draw[->] (KK) to (JJ);
\draw[->] (MM) to (KK);
\draw[->] (PP) to (MM);
\draw[->] (LL) to (JJ);
\draw[->] (NN) to (LL);
\draw[->] (QQ) to (NN);
\draw[->] (SS) to (QQ);
\draw[->] (UU) to (SS);

\draw[decorate, decoration={zigzag, segment length=+2pt, amplitude=+.75pt,post length=+4pt}, -stealth'] (ZZ.east) -- (Z.west) node[below, midway, scale=0.9] {we glue here};

\draw[loosely dotted] (V.east) -- (A);
\draw[loosely dotted] (W.east) -- (B);
\draw[loosely dotted] (A.east) -- (D);
\draw[loosely dotted] (B.east) -- (E);
\draw[loosely dotted] (D.east) -- (G);
\draw[loosely dotted] (E.east) -- (H);
\draw[loosely dotted] (G.east) -- (J);
\draw[loosely dotted] (H.east) -- (K);
\draw[loosely dotted] (I.east) -- (L);
\draw[loosely dotted] (K.east) -- (P);
\draw[loosely dotted] (L.east) -- (Q);
\draw[loosely dotted] (Q.east) -- (U);

\draw[loosely dotted] (VV.east) -- (AA);
\draw[loosely dotted] (WW.east) -- (BB);
\draw[loosely dotted] (AA.east) -- (DD);
\draw[loosely dotted] (BB.east) -- (EE);
\draw[loosely dotted] (DD.east) -- (GG);
\draw[loosely dotted] (EE.east) -- (HH);
\draw[loosely dotted] (GG.east) -- (JJ);
\draw[loosely dotted] (HH.east) -- (KK);
\draw[loosely dotted] (II.east) -- (LL);
\draw[loosely dotted] (KK.east) -- (PP);
\draw[loosely dotted] (LL.east) -- (QQ);
\draw[loosely dotted] (QQ.east) -- (UU);
\end{tikzpicture}}\]
and the Aus\-lan\-der--Rei\-ten quiver of $\La$ with its $3$-cluster tilting subcategory is
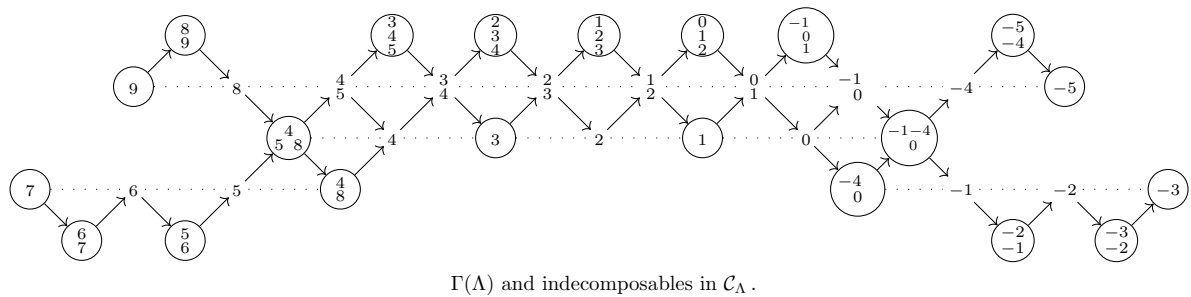
\[\resizebox{\textwidth}{!}{\begin{tikzpicture}
\tikzstyle{nct2}=[circle, minimum width=0.6cm, draw, inner sep=0pt, text centered, scale=0.9]
\tikzstyle{nct22}=[circle, minimum width=0.6cm, draw, inner sep=0pt, text centered, scale=0.8]
\tikzstyle{nct3}=[circle, minimum width=6pt, draw=white, inner sep=0pt, scale=0.9]
\node[scale=0.8] (Name) at (-4.2,-2) {$\Gamma(\La)$ and indecomposables in $\cC_{\La}$\nospacepunct{.}};

\node[nct2] (ZZ) at (-4.2,1.4) {$\qthree{1}[2][3]$};
\node[nct3] (WW) at (-4.9,0.7) {$\qthree{2}[3]$};
\node[nct2] (VV) at (-5.6,0) {$\qthree{}[3][]$};
\node[nct3] (AA) at (-4.2,0) {$\qthree{}[2][]$};
\node[nct3] (BB) at (-3.5,0.7) {$\qthree{1}[2]$};
\node[nct2] (DD) at (-2.8,0) {$\qthree{}[1][]$};
\node[nct2] (CC) at (-5.6,1.4) {$\qthree{2}[3][4]$};
\node[nct3] (EE) at (-6.3,0.7) {$\qthree{3}[4]$};
\node[nct3] (GG) at (-7,0) {$\qthree{}[4][]$};
\node[nct2] (FF) at (-7,1.4) {$\qthree{3}[4][5]$};
\node[nct3] (HH) at (-7.7,0.7) {$\qthree{4}[5]$};
\node[nct2] (II) at (-7.7,-0.7) {$\qthree{4}[8]$};
\node[nct2] (JJ) at (-8.4,0) {$\begin{smallmatrix}  & 4 & \\ 5 && 8
\end{smallmatrix}$};
\node[nct3] (KK) at (-9.1,0.7) {$\qthree{}[8][]$};
\node[nct2] (MM) at (-9.8,1.4) {$\qthree{8}[9]$};
\node[nct2] (PP) at (-10.5,0.7) {$\qthree{}[9][]$};
\node[nct3] (LL) at (-9.1,-0.7) {$\qthree{}[5][]$};
\node[nct2] (NN) at (-9.8,-1.4) {$\qthree{5}[6]$};
\node[nct3] (QQ) at (-10.5,-0.7) {$\qthree{}[6][]$};
\node[nct2] (SS) at (-11.2,-1.4) {$\qthree{6}[7]$};
\node[nct2] (UU) at (-11.9,-0.7) {$\qthree{}[7][]$};
\node[nct2] (C) at (-2.8,1.4) {$\qthree{0}[1][2]$};
\node[nct3] (E) at (-2.1,0.7) {$\qthree{0}[1]$};
\node[nct22] (F) at (-1.4,1.4) {$\qthree{-1\;\;\;}[0][1]$};
\node[nct3] (G) at (-1.4,0) {$\qthree{}[0][]$};
\node[nct3] (H) at (-0.7,0.7) {$\qthree{-1\;\;\;}[0]$};
\node[nct2] (I) at (-0.7,-0.7) {$\qthree{-4\;\;\;}[0]$};
\node[nct22] (J) at (0,0) {$\begin{smallmatrix} -1 & -4 \\ & 0\;\;\;
\end{smallmatrix}$};
\node[nct3] (K) at (0.7,0.7) {$\qthree{}[-4][]$};
\node[nct3] (L) at (0.7,-0.7) {$\qthree{}[-1][]$};
\node[nct2] (M) at (1.4,1.4) {$\qthree{-5}[-4]$};
\node[nct2] (N) at (1.4,-1.4) {$\qthree{-2}[-1]$};
\node[nct2] (P) at (2.1,0.7) {$\qthree{}[-5][]$};
\node[nct3] (Q) at (2.1,-0.7) {$\qthree{}[-2][]$};
\node[nct2] (S) at (2.8,-1.4) {$\qthree{-3}[-2]$};
\node[nct2] (U) at (3.5,-0.7) {$\qthree{}[-3][]$};

\draw[->] (VV) to (WW);
\draw[->] (WW) to (AA);
\draw[->] (WW) to (ZZ);
\draw[->] (ZZ) to (BB);
\draw[->] (AA) to (BB);
\draw[->] (BB) to (DD);
\draw[->] (CC) to (WW);
\draw[->] (EE) to (CC);
\draw[->] (EE) to (VV);
\draw[->] (GG) to (EE);
\draw[->] (FF) to (EE);
\draw[->] (HH) to (FF);
\draw[->] (HH) to (GG);
\draw[->] (II) to (GG);
\draw[->] (JJ) to (HH);
\draw[->] (JJ) to (II);
\draw[->] (KK) to (JJ);
\draw[->] (MM) to (KK);
\draw[->] (PP) to (MM);
\draw[->] (LL) to (JJ);
\draw[->] (NN) to (LL);
\draw[->] (QQ) to (NN);
\draw[->] (SS) to (QQ);
\draw[->] (UU) to (SS);
\draw[->] (BB) to (C);
\draw[->] (DD) to (E);
\draw[->] (C) to (E);
\draw[->] (E) to (F);
\draw[->] (E) to (G);
\draw[->] (F) to (H);
\draw[->] (G) to (H);
\draw[->] (G) to (I);
\draw[->] (H) to (J);
\draw[->] (I) to (J);
\draw[->] (J) to (K);
\draw[->] (J) to (L);
\draw[->] (K) to (M);
\draw[->] (L) to (N);
\draw[->] (M) to (P);
\draw[->] (N) to (Q);
\draw[->] (Q) to (S);
\draw[->] (S) to (U);

\draw[loosely dotted] (BB.east) -- (E);
\draw[loosely dotted] (E.east) -- (H);
\draw[loosely dotted] (H.east) -- (K);
\draw[loosely dotted] (I.east) -- (L);
\draw[loosely dotted] (K.east) -- (P);
\draw[loosely dotted] (G.east) -- (J);
\draw[loosely dotted] (L.east) -- (Q);
\draw[loosely dotted] (Q.east) -- (U);
\draw[loosely dotted] (DD.east) -- (G);

\draw[loosely dotted] (AA.east) -- (DD);
\draw[loosely dotted] (WW.east) -- (BB);
\draw[loosely dotted] (VV.east) -- (AA);
\draw[loosely dotted] (EE.east) -- (WW);
\draw[loosely dotted] (GG.east) -- (VV);
\draw[loosely dotted] (HH.east) -- (EE);
\draw[loosely dotted] (JJ.east) -- (GG);
\draw[loosely dotted] (KK.east) -- (HH);
\draw[loosely dotted] (LL.east) -- (II);
\draw[loosely dotted] (PP.east) -- (KK);
\draw[loosely dotted] (QQ.east) -- (LL);
\draw[loosely dotted] (UU.east) -- (QQ);
\end{tikzpicture}.}\]
\end{example}

\subsection{The case of slice modules}\label{sect:slices}

In this section, we answer Question \ref{Question: can we glue?} positively in the case of $T$ being a \emph{slice module}. We begin with the definition of slice modules for $KA_h$; for the general definition of slice modules we refer to \cite{HR}.

\begin{definition}\label{def:slice}
A set $\{T_1,\dots,T_h\}$ of distinct indecomposable $KA_h$-modules is called a \emph{slice of $KA_h$} if they all have different lengths, and if $l(T_i)=l(T_j)-1$ then either there exists a monomorphism $T_i\hookrightarrow T_j$ or an epimorphism $T_j\twoheadrightarrow T_i$. In this case we call $T=\bigoplus_{i=1}^{h}T_i$ a \emph{slice module} and $\add (T)$ a \emph{slice subcategory}.
\end{definition}

Since the possible lengths of $T_i$ are $1$ to $h$, we can assume without loss of generality that for a slice of $KA_h$, we have $l(T_i)=i$. If we denote the indecomposable $KA_h$-modules by $M(i,j)$ as in (\ref{eq:indecomposables}), it follows that a slice of $KA_h$ is a set of modules $\{M(i_1,1),M(i_2,2),\dots, M(i_h,h)\}$ such that $i_k=i_{k-1}$ or $i_k=i_{k-1}-1$. In particular, $i_h=1$ and $i_{h-1}=1$ or $i_{h-1}=2$.

\begin{definition}\label{def:slice fracture}
Let $\La$ be a rep\-re\-sen\-ta\-tion-di\-rect\-ed algebra and let $T$ be a fracture corresponding to a maximal abutment of $\La$. We will say that $T$ is a \emph{slice fracture} if $T$ viewed as a $KA_h$-module is a slice module.
\end{definition}

Our aim is to answer Question \ref{Question: can we glue?} affirmatively when $T$ is a slice module. Notice that if $T$ is a slice module, then $D(T)$ is also a slice module. Hence by symmetry it is enough to answer Question \ref{Question: can we glue?}(a). The following computational lemma will be used.

\begin{lemma}\emph{\cite[Lemma 4.8]{VAS}}
\label{lemma:taun}
Let $M(i,j)\neq 0$ be a $\La_{m,h}$-module. Then 
\begin{itemize}
\item[(a)] If $M(i,j)$ is nonprojective, we have
\[\tn \left(M(i,j)\right)=\begin{cases} 
M\left(i+j-\frac{n}{2}h-1,h-j\right) &\mbox{if } n \text{ is even,}\\
M\left(i-\frac{n-1}{2}h-1,j\right) &\mbox{if } n \text{ is odd.}
\end{cases}\]
\item[(b)] If $M(i,j)$ is noninjective, we have
\[\tno \left(M(i,j)\right)=\begin{cases} 
M\left(i+j+\frac{n-2}{2}h+1,h-j\right) &\mbox{if } n \text{ is even,}\\
M\left(i+\frac{n-1}{2}h+1,j\right) &\mbox{if } n \text{ is odd.}
\end{cases}\]
\end{itemize}
\end{lemma}

Before we proceed with the main result of this section, let us explain how Lemma \ref{lemma:taun} will be used. By Proposition \ref{prop:abutmentsquiver} there is a unique maximal left abutment of $\La_{m,h}$, namely $P(m-h+1)=M(1,h)$. Moreover, the Aus\-lan\-der--Rei\-ten quiver of $\La_{m,h}$ is a subquiver of $\triangle(m)$ where we remove all the vertices corresponding to indecomposable modules of length at least $h+1$. Described otherwise, it is the same as the quiver $\triangle(h)$ with the addition of more diagonals on the right hand side of the same height $h$. 

Let $T$ be a slice fracture of $P(m-h)$ and let $\hat{T}\oplus P(m-h+1) \cong T$ . Then Lemma \ref{lemma:taun} implies that the action of $\tno$ translates $\hat{T}$ through the Aus\-lan\-der--Rei\-ten quiver of $\Gamma(\La_{m,h})$ without changing its shape. In other words, for $m$ large enough we have the following pictures:
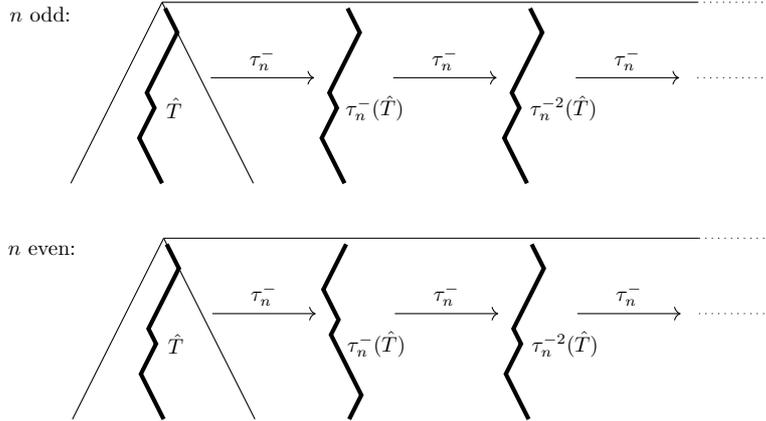
\[\begin{tikzpicture}[scale=0.8, transform shape]

\draw (0,0) -- (1.5,3);
\draw (1.5,3) -- (10.3,3);
\draw (1.5,3) -- (3,0);
\draw[dotted] (10.3,3) -- (11.5,3);

\draw[ultra thick] (1.55,2.9) -- (1.75,2.5) -- (1.25,1.5) -- (1.375,1.25) -- (1.125,0.75) -- (1.5,0);

\draw[ultra thick] (4.55,2.9) -- (4.75,2.5) -- (4.25,1.5) -- (4.375,1.25) -- (4.125,0.75) -- (4.5,0);

\draw[ultra thick] (7.55,2.9) -- (7.75,2.5) -- (7.25,1.5) -- (7.375,1.25) -- (7.125,0.75) -- (7.5,0);

\node (A1) at (1.7,1.25) {$\hat{T}$};
\node (A2) at (5,1.25) {$\tno(\hat{T})$};
\node (A3) at (8.1,1.25) {$\tau_n^{-2}(\hat{T})$};

\draw[->] (2.3,1.75) -- (4,1.75) node[draw=none, midway, above] {$\tno$};
\draw[->] (5.3,1.75) -- (7,1.75) node[draw=none, midway, above] {$\tno$};
\draw[->] (8.3,1.75) -- (10,1.75) node[draw=none, midway, above] {$\tno$};

\draw[dotted] (10.3,1.75) -- (11.5,1.75);

\node(B) at (-0.5,2.8) {$n$ odd:};

\end{tikzpicture}\]

\[\begin{tikzpicture}[scale=0.8, transform shape]
\draw (0,0) -- (1.5,3);
\draw (1.5,3) -- (10.3,3);
\draw (1.5,3) -- (3,0);
\draw[dotted] (10.3,3) -- (11.5,3);

\draw[ultra thick] (1.55,2.9) -- (1.75,2.5) -- (1.25,1.5) -- (1.375,1.25) -- (1.125,0.75) -- (1.5,0);

\begin{scope}[yscale=-1,xscale=1]
\draw[ultra thick] (4.55,0) -- (4.75,-0.4) -- (4.25,-1.4) -- (4.375,-1.65) -- (4.125,-2.15) -- (4.5,-2.9);
\end{scope}

\draw[ultra thick] (7.55,2.9) -- (7.75,2.5) -- (7.25,1.5) -- (7.375,1.25) -- (7.125,0.75) -- (7.5,0);

\node (A1) at (1.7,1.25) {$\hat{T}$};
\node (A2) at (5,1.25) {$\tno(\hat{T})$};
\node (A3) at (8.1,1.25) {$\tau_n^{-2}(\hat{T})$};

\draw[->] (2.3,1.75) -- (4,1.75) node[draw=none, midway, above] {$\tno$};
\draw[->] (5.3,1.75) -- (7,1.75) node[draw=none, midway, above] {$\tno$};
\draw[->] (8.3,1.75) -- (10,1.75) node[draw=none, midway, above] {$\tno$};

\draw[dotted] (10.3,1.75) -- (11.5,1.75);

\node(B) at (-0.5,2.8) {$n$ even:};

\end{tikzpicture}\]

In each of the above pictures the leftmost thick line represents the indecomposable summands of $\hat{T}$ in the foundation of $P(m-h+1)$, the middle thick line represents the indecomposable summands of $\tno(\hat{T})$ and the rightmost thick line represent the indecomposable summands of $\tau_n^{-2}(\hat{T})$. Notice that in the case of $n$ being even $\hat{T}$ is reflected horizontally at every application of $\tno$. Moreover, the module $P(m-h)$, which would be at the top of the slice, is injective and so $\tno(P(m-h+1))=0$. Additionally, the above applications of $\tno$ are invertible by $\tn$. The idea of the proof is that by choosing $m$ correctly we can stop precisely at the point where the thick diagonal aligns with the end of $\La_{m,h}$. Then we can remove these aligned modules from our slice and consider the leftover piece as a slice of smaller height. Finally, an induction on the height of the slice will give us the proof. Let us illustrate with a concrete example.

\begin{example}
Consider the slice module $T$ of $KA_5$ given by the following encircled modules in $\Gamma(KA_5)$:
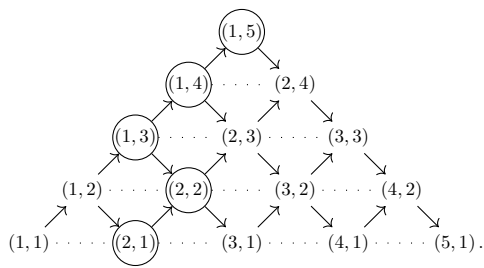
\[\begin{tikzpicture}[scale=0.7, transform shape]
\tikzstyle{nct2}=[circle, minimum width=0.6cm, draw, inner sep=0pt, text centered]

\node (11) at (1,1) {$(1,1)$};
\node[nct2] (21) at (3,1) {$(2,1)$};
\node (31) at (5,1) {$(3,1)$};
\node (41) at (7,1) {$(4,1)$};
\node (51) at (9,1) {$(5,1)$\nospacepunct{.}};

\node (12) at (2,2) {$(1,2)$};
\node[nct2] (22) at (4,2) {$(2,2)$};
\node (32) at (6,2) {$(3,2)$};
\node (42) at (8,2) {$(4,2)$};

\node[nct2] (13) at (3,3) {$(1,3)$};
\node (23) at (5,3) {$(2,3)$};
\node (33) at (7,3) {$(3,3)$};

\node[nct2] (14) at (4,4) {$(1,4)$};
\node (24) at (6,4) {$(2,4)$};

\node[nct2] (15) at (5,5) {$(1,5)$};

\draw[->] (11) -- (12);
\draw[->] (21) -- (22);
\draw[->] (31) -- (32);
\draw[->] (41) -- (42);
\draw[->] (12) -- (13);
\draw[->] (22) -- (23);
\draw[->] (32) -- (33);
\draw[->] (13) -- (14);
\draw[->] (23) -- (24);
\draw[->] (14) -- (15);

\draw[->] (12) -- (21);
\draw[->] (22) -- (31);
\draw[->] (32) -- (41);
\draw[->] (42) -- (51);
\draw[->] (13) -- (22);
\draw[->] (23) -- (32);
\draw[->] (33) -- (42);
\draw[->] (14) -- (23);
\draw[->] (24) -- (33);
\draw[->] (15) -- (24);

\draw[loosely dotted] (11.east) -- (21);
\draw[loosely dotted] (21.east) -- (31);
\draw[loosely dotted] (31.east) -- (41);
\draw[loosely dotted] (41.east) -- (51);

\draw[loosely dotted] (12.east) -- (22);
\draw[loosely dotted] (22.east) -- (32);
\draw[loosely dotted] (32.east) -- (42);

\draw[loosely dotted] (13.east) -- (23);
\draw[loosely dotted] (23.east) -- (33);

\draw[loosely dotted] (14.east) -- (24);
\end{tikzpicture}\]
Assume we want to complete $T$ on the right for $n=4$. Let $ m> 5$ and consider the maximal left $\La_{m,5}$-abutment $P(m-4)=M(1,5)$. By gluing 
\[KA_5 \glue[P(m-4)][I(5)] \La_{m,5}=\La_{m,5}\] 
and assuming $m$ is large enough, Lemma \ref{lemma:taun} gives
\[\tau_4^{-}\left(M(1,4)\right)=M(11,1),\; \tau_4^{-}\left(M(1,3)\right)=M(10,2),\; \tau_4^{-}\left(M(2,2)\right)=M(10,3),\; \tau_4^{-}\left(M(2,1)\right)=M(9,4).\]
Hence for $m=12$ we have the following picture of $\Gamma(\La_{12,5})$:
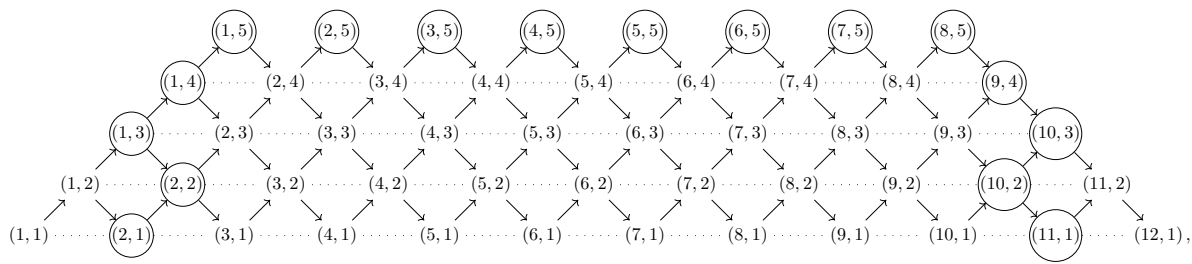
\[\resizebox{\textwidth}{!}{\begin{tikzpicture}
\tikzstyle{nct2}=[circle, minimum width=0.6cm, draw, inner sep=0pt, text centered]
\node (11) at (1,1) {$(1,1)$};
\node[nct2] (21) at (3,1) {$(2,1)$};
\node (31) at (5,1) {$(3,1)$};
\node (41) at (7,1) {$(4,1)$};
\node (51) at (9,1) {$(5,1)$};
\node (61) at (11,1) {$(6,1)$};
\node (71) at (13,1) {$(7,1)$};
\node (81) at (15,1) {$(8,1)$};
\node (91) at (17,1) {$(9,1)$};
\node (101) at (19,1) {$(10,1)$};
\node[nct2] (111) at (21,1) {$(11,1)$};
\node (121) at (23,1) {$(12,1)$\nospacepunct{,}};

\node (12) at (2,2) {$(1,2)$};
\node[nct2] (22) at (4,2) {$(2,2)$};
\node (32) at (6,2) {$(3,2)$};
\node (42) at (8,2) {$(4,2)$};
\node (52) at (10,2) {$(5,2)$};
\node (62) at (12,2) {$(6,2)$};
\node (72) at (14,2) {$(7,2)$};
\node (82) at (16,2) {$(8,2)$};
\node (92) at (18,2) {$(9,2)$};
\node[nct2] (102) at (20,2) {$(10,2)$};
\node (112) at (22,2) {$(11,2)$};

\node[nct2] (13) at (3,3) {$(1,3)$};
\node (23) at (5,3) {$(2,3)$};
\node (33) at (7,3) {$(3,3)$};
\node (43) at (9,3) {$(4,3)$};
\node (53) at (11,3) {$(5,3)$};
\node (63) at (13,3) {$(6,3)$};
\node (73) at (15,3) {$(7,3)$};
\node (83) at (17,3) {$(8,3)$};
\node (93) at (19,3) {$(9,3)$};
\node[nct2] (103) at (21,3) {$(10,3)$};

\node[nct2] (14) at (4,4) {$(1,4)$};
\node (24) at (6,4) {$(2,4)$};
\node (34) at (8,4) {$(3,4)$};
\node (44) at (10,4) {$(4,4)$};
\node (54) at (12,4) {$(5,4)$};
\node (64) at (14,4) {$(6,4)$};
\node (74) at (16,4) {$(7,4)$};
\node (84) at (18,4) {$(8,4)$};
\node[nct2] (94) at (20,4) {$(9,4)$};

\node[nct2] (15) at (5,5) {$(1,5)$};
\node[nct2] (25) at (7,5) {$(2,5)$};
\node[nct2] (35) at (9,5) {$(3,5)$};
\node[nct2] (45) at (11,5) {$(4,5)$};
\node[nct2] (55) at (13,5) {$(5,5)$};
\node[nct2] (65) at (15,5) {$(6,5)$};
\node[nct2] (75) at (17,5) {$(7,5)$};
\node[nct2] (85) at (19,5) {$(8,5)$};

\draw[->] (11) -- (12);
\draw[->] (21) -- (22);
\draw[->] (31) -- (32);
\draw[->] (41) -- (42);
\draw[->] (51) -- (52);
\draw[->] (61) -- (62);
\draw[->] (71) -- (72);
\draw[->] (81) -- (82);
\draw[->] (91) -- (92);
\draw[->] (101) -- (102);
\draw[->] (111) -- (112);

\draw[->] (12) -- (13);
\draw[->] (22) -- (23);
\draw[->] (32) -- (33);
\draw[->] (42) -- (43);
\draw[->] (52) -- (53);
\draw[->] (62) -- (63);
\draw[->] (72) -- (73);
\draw[->] (82) -- (83);
\draw[->] (92) -- (93);
\draw[->] (102) -- (103);

\draw[->] (13) -- (14);
\draw[->] (23) -- (24);
\draw[->] (33) -- (34);
\draw[->] (43) -- (44);
\draw[->] (53) -- (54);
\draw[->] (63) -- (64);
\draw[->] (73) -- (74);
\draw[->] (83) -- (84);
\draw[->] (93) -- (94);

\draw[->] (14) -- (15);
\draw[->] (24) -- (25);
\draw[->] (34) -- (35);
\draw[->] (44) -- (45);
\draw[->] (54) -- (55);
\draw[->] (64) -- (65);
\draw[->] (74) -- (75);
\draw[->] (84) -- (85);

\draw[->] (12) -- (21);
\draw[->] (22) -- (31);
\draw[->] (32) -- (41);
\draw[->] (42) -- (51);
\draw[->] (52) -- (61);
\draw[->] (62) -- (71);
\draw[->] (72) -- (81);
\draw[->] (82) -- (91);
\draw[->] (92) -- (101);
\draw[->] (102) -- (111);
\draw[->] (112) -- (121);

\draw[->] (13) -- (22);
\draw[->] (23) -- (32);
\draw[->] (33) -- (42);
\draw[->] (43) -- (52);
\draw[->] (53) -- (62);
\draw[->] (63) -- (72);
\draw[->] (73) -- (82);
\draw[->] (83) -- (92);
\draw[->] (93) -- (102);
\draw[->] (103) -- (112);

\draw[->] (14) -- (23);
\draw[->] (24) -- (33);
\draw[->] (34) -- (43);
\draw[->] (44) -- (53);
\draw[->] (54) -- (63);
\draw[->] (64) -- (73);
\draw[->] (74) -- (83);
\draw[->] (84) -- (93);
\draw[->] (94) -- (103);

\draw[->] (15) -- (24);
\draw[->] (25) -- (34);
\draw[->] (35) -- (44);
\draw[->] (45) -- (54);
\draw[->] (55) -- (64);
\draw[->] (65) -- (74);
\draw[->] (75) -- (84);
\draw[->] (85) -- (94);

\draw[loosely dotted] (11.east) -- (21);
\draw[loosely dotted] (21.east) -- (31);
\draw[loosely dotted] (31.east) -- (41);
\draw[loosely dotted] (41.east) -- (51);
\draw[loosely dotted] (51.east) -- (61);
\draw[loosely dotted] (61.east) -- (71);
\draw[loosely dotted] (71.east) -- (81);
\draw[loosely dotted] (81.east) -- (91);
\draw[loosely dotted] (91.east) -- (101);
\draw[loosely dotted] (101.east) -- (111);
\draw[loosely dotted] (111.east) -- (121);

\draw[loosely dotted] (12.east) -- (22);
\draw[loosely dotted] (22.east) -- (32);
\draw[loosely dotted] (32.east) -- (42);
\draw[loosely dotted] (42.east) -- (52);
\draw[loosely dotted] (52.east) -- (62);
\draw[loosely dotted] (62.east) -- (72);
\draw[loosely dotted] (72.east) -- (82);
\draw[loosely dotted] (82.east) -- (92);
\draw[loosely dotted] (92.east) -- (102);
\draw[loosely dotted] (102.east) -- (112);

\draw[loosely dotted] (13.east) -- (23);
\draw[loosely dotted] (23.east) -- (33);
\draw[loosely dotted] (33.east) -- (43);
\draw[loosely dotted] (43.east) -- (53);
\draw[loosely dotted] (53.east) -- (63);
\draw[loosely dotted] (63.east) -- (73);
\draw[loosely dotted] (73.east) -- (83);
\draw[loosely dotted] (83.east) -- (93);
\draw[loosely dotted] (93.east) -- (103);

\draw[loosely dotted] (14.east) -- (24);
\draw[loosely dotted] (24.east) -- (34);
\draw[loosely dotted] (34.east) -- (44);
\draw[loosely dotted] (44.east) -- (54);
\draw[loosely dotted] (54.east) -- (64);
\draw[loosely dotted] (64.east) -- (74);
\draw[loosely dotted] (74.east) -- (84);
\draw[loosely dotted] (84.east) -- (94);
\end{tikzpicture}}\]
where the encircled modules form a $(T, \tau_4^-(\hat{T})\oplus M(8,5),4)$-fractured subcategory $\cC$. Notice that the top row of the above Aus\-lan\-der--Rei\-ten quiver has only pro\-jec\-tive-in\-jec\-tive modules and so is included automatically in $\cC$. In this setup the modules $M(8,5)$, $M(9,4)$ and $M(10,3)$ are injective and hence we opt to glue at the one with the smallest height, that is the module $I(3)=M(10,3)$. Then the module
\[T'=M(10,3)\oplus M(10,2)\oplus M(11,1)\] 
corresponds to a slice of $KA_3$. Hence it is enough to show that we can complete this particular slice. After renumbering the vertices and by a similar computation as before we see that the gluing $KA_3 \glue[P(6)][I(3)] \La_{8,3}=\La_{8,3}$ gives the following Aus\-lan\-der--Rei\-ten quiver 
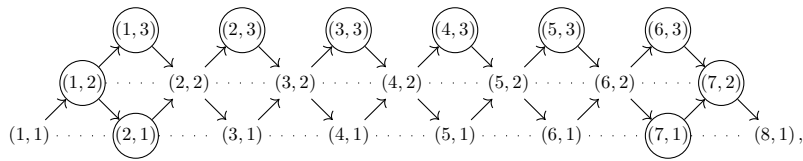
\[\begin{tikzpicture}[scale=0.7, transform shape]
\tikzstyle{nct2}=[circle, minimum width=0.6cm, draw, inner sep=0pt, text centered]
\node (11) at (1,1) {$(1,1)$};
\node[nct2] (21) at (3,1) {$(2,1)$};
\node (31) at (5,1) {$(3,1)$};
\node (41) at (7,1) {$(4,1)$};
\node (51) at (9,1) {$(5,1)$};
\node (61) at (11,1) {$(6,1)$};
\node[nct2] (71) at (13,1) {$(7,1)$};
\node (81) at (15,1) {$(8,1)$\nospacepunct{,}};

\node[nct2] (12) at (2,2) {$(1,2)$};
\node (22) at (4,2) {$(2,2)$};
\node (32) at (6,2) {$(3,2)$};
\node (42) at (8,2) {$(4,2)$};
\node (52) at (10,2) {$(5,2)$};
\node (62) at (12,2) {$(6,2)$};
\node[nct2] (72) at (14,2) {$(7,2)$};

\node[nct2] (13) at (3,3) {$(1,3)$};
\node[nct2] (23) at (5,3) {$(2,3)$};
\node[nct2] (33) at (7,3) {$(3,3)$};
\node[nct2] (43) at (9,3) {$(4,3)$};
\node[nct2] (53) at (11,3) {$(5,3)$};
\node[nct2] (63) at (13,3) {$(6,3)$};

\draw[->] (11) -- (12);
\draw[->] (21) -- (22);
\draw[->] (31) -- (32);
\draw[->] (41) -- (42);
\draw[->] (51) -- (52);
\draw[->] (61) -- (62);
\draw[->] (71) -- (72);

\draw[->] (12) -- (13);
\draw[->] (22) -- (23);
\draw[->] (32) -- (33);
\draw[->] (42) -- (43);
\draw[->] (52) -- (53);
\draw[->] (62) -- (63);

\draw[->] (12) -- (21);
\draw[->] (22) -- (31);
\draw[->] (32) -- (41);
\draw[->] (42) -- (51);
\draw[->] (52) -- (61);
\draw[->] (62) -- (71);
\draw[->] (72) -- (81);

\draw[->] (13) -- (22);
\draw[->] (23) -- (32);
\draw[->] (33) -- (42);
\draw[->] (43) -- (52);
\draw[->] (53) -- (62);
\draw[->] (63) -- (72);

\draw[loosely dotted] (11.east) -- (21);
\draw[loosely dotted] (21.east) -- (31);
\draw[loosely dotted] (31.east) -- (41);
\draw[loosely dotted] (41.east) -- (51);
\draw[loosely dotted] (51.east) -- (61);
\draw[loosely dotted] (61.east) -- (71);
\draw[loosely dotted] (71.east) -- (81);

\draw[loosely dotted] (12.east) -- (22);
\draw[loosely dotted] (22.east) -- (32);
\draw[loosely dotted] (32.east) -- (42);
\draw[loosely dotted] (42.east) -- (52);
\draw[loosely dotted] (52.east) -- (62);
\draw[loosely dotted] (62.east) -- (72);
\end{tikzpicture}\]
where again the encircled modules form a $4$-fractured subcategory. Here we see as before that the module $M(7,2)\oplus M(7,1)$ also corresponds to a slice of $KA_2$. After renumbering the vertices and computing as before we see that for the gluing $KA_2 \glue[P(4)][I(2)] \La_{5,2}$ we get the Aus\-lan\-der--Rei\-ten quiver
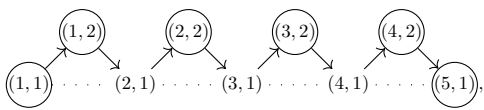
\[\begin{tikzpicture}[scale=0.7, transform shape]
\tikzstyle{nct2}=[circle, minimum width=0.6cm, draw, inner sep=0pt, text centered]
\node[nct2] (11) at (1,1) {$(1,1)$};
\node (21) at (3,1) {$(2,1)$};
\node (31) at (5,1) {$(3,1)$};
\node (41) at (7,1) {$(4,1)$};
\node[nct2] (51) at (9,1) {$(5,1)$\nospacepunct{,}};

\node[nct2] (12) at (2,2) {$(1,2)$};
\node[nct2] (22) at (4,2) {$(2,2)$};
\node[nct2] (32) at (6,2) {$(3,2)$};
\node[nct2] (42) at (8,2) {$(4,2)$};

\draw[->] (11) -- (12);
\draw[->] (21) -- (22);
\draw[->] (31) -- (32);
\draw[->] (41) -- (42);

\draw[->] (12) -- (21);
\draw[->] (22) -- (31);
\draw[->] (32) -- (41);
\draw[->] (42) -- (51);

\draw[loosely dotted] (11.east) -- (21);
\draw[loosely dotted] (21.east) -- (31);
\draw[loosely dotted] (31.east) -- (41);
\draw[loosely dotted] (41.east) -- (51);
\end{tikzpicture}\]
where now the encircled modules form a $4$-cluster tilting subcategory. As a consequence, the algebra
\[ \La_{12,5} \glue[P(6)][I(3)] \La_{8,3} \glue[P(4)][I(2)] \La_{5,2}\]
admits a right $4$-cluster tilting subcategory.
\end{example}

\begin{proposition}
Let $\{M(i_1,1),\dots, M(i_h,h)\}$ be a slice of $KA_h$ and $T=\bigoplus_{k=1}^h M(i_k,k)$. Then we can complete $T$ on the left and on the right.
\end{proposition}

\begin{proof}
Let $n\geq 2$. As mentioned before, by symmetry, it is enough to show that we can complete $T$ on the right. We will use induction on $h$. For $h=1$ we have only one indecomposable $KA_1$-module, say $N$, and so $\add (T)=\add (N)$ is an $n$-cluster tilting subcategory for any $n$. For the induction step, assume that we can complete any slice of $KA_{h-1}$ on the right and we will show that we can complete any slice of $KA_{h}$ on the right.

We consider the two possible cases $i_{h-1}=1$ and $i_{h-1}=2$ separately. For the case $i_{h-1}=2$ consider the slice $\{M(i_1,1),\dots,M(i_{h-1},h-1)\}$ of $KA_{h-1}$. By induction hypothesis there exists a rep\-re\-sen\-ta\-tion-di\-rect\-ed algebra $A$ satisfying the conditions of Question \ref{Question: can we glue?}(b) for this slice. We check that the conditions of Theorem \ref{thrm:fractsubcat} are satisfied for the gluing $A'\coloneqq KA_h \glue[P][M(1,h-1)] A$. By assumption, the algebra $A$ admits a $(T_A^L,\DLab_A,n)$-fractured subcategory $\cC_{A}$, with the only nonprojective fracture of $T^L_{A}$ being $T_A^{(P)}$ and $f_P^{!}(T_A^{(P)})\cong\bigoplus_{k=1}^{h-1}M(i_k,k)$. Notice that $P$ is a maximal left abutment of $A$ of height $h-1$ and so clearly $h\geq \lvl(T_A^{(P)})$. Moreover, the algebra $KA_h$ has a $(T,T,n)$-fractured subcategory $\cC_{KA_h}=\add(T)$ and $M(1,h)$ is a maximal right abutment of $KA_h$ of height $h$. In particular, since $M(1,h)$ is a direct summand of $T$, we have $h\geq \lvl(T)$. Finally, we have
\[f^{!}_P\left(\add(T_A^{(P)}\cap \cF_P\right) = \add\left(\bigoplus_{k=1}^{h-1}M(i_k,k)\right) = g_{M(1,h)}^{\ast}\left(\add T_{KA_h}^{(M(1,h))}\cap \cG_{M(1,h)}\right),\]
and so $\cC_{A'}=\add\{\phi_{\ast}(\cC_A),\psi_{\ast}(T)\}$ is a $(T_{A'}^L,T_{A'}^R,n)$-fractured subcategory. Since the only maximal right abutment of $KA_h$ is $M(1,h)$, it follows that $T_{A'}^R=\DLab_{A'}$. Since the only maximal left abutment of $KA_h$ is $M(1,h)$ and the only nonprojective fracture of $T_{A}^L$ is $T_{A}^{(P)}$, it follows that the only nonprojective fracture of $T_{A'}^L$ is $T_{A'}^{(\psi_{\ast}(M(1,h))}$. By definition, we have $T_{A'}^{(\psi_{\ast}(M(1,h))}=\psi_{\ast}(T)$. Then
\[f^!_{\psi_{\ast}(M(1,h))}\left(T_{A'}^{(\psi_{\ast}(M(1,h))}\right) = T.\]
But this shows that $A'$ completes $T$ on the right, as required.

For the case $i_{h-1}=1$ we consider the cases $n$ being odd and $n$ being even separately. 

For the case $n$ being odd, set $m\coloneqq \frac{n+1}{2}h$. First we glue $KA_h  \glue[P(m-(h-1))][I(h)] \La_{m,h}$. Since this is a trivial gluing as in Example \ref{ex:KAglued}, the resulting algebra is isomorphic to $\La_{m,h}$ again. It is clear that viewing the modules $M(i_k,k)$ as $\La_{m,h}$-modules we have 
\[\psi_{\ast}(M(i_k,k)) \cong M\left(i_k,k\right).\]
Computing $\tno$ by using Lemma \ref{lemma:taun} gives 
\[\tno \left(M\left(i_k,k\right)\right)=M\left(i_k+1+\frac{n-1}{2}h,k\right).\]
In particular we have \[\tno\left(M\left(1,h-1\right)\right)=M\left(2+\frac{n-1}{2}h,h-1\right)=M\left(m+2-h,h-1\right)=I(h-1),\] 
which is a right abutment of $\La_{h+\frac{n-1}{2}h,h}$ of height $h-1$. Moreover, the set $\left\{\tno \left(M(i_k,k\right)\right\}_{k=1}^{h-1}\subseteq \cF_{M(1,h-1)}$ is a fracture which is a slice, viewed as a $KA_{h-1}$-module. By induction hypothesis we can complete this slice on the right. Hence there exists an algebra $A$ and a maximal left abutment $P$ of $A$ such that $A$ admits a $(T_A^{L},\DLab_{A},n)$-fractured subcategory and the only nonprojective fracture of $T_A^{L}$ is $T_{A}^{(P)}$, where $f_P^{!}(T_{A}^{(P)})\cong\bigoplus_{k=1}^{h-1} M(i_k,k)$. It follows that the algebra $A'=\La_{m,h} \glue[P][I(h-1)] A$ completes the fracture $T$ on the right.

Finally, for the case $n$ being even, set $m\coloneqq i_1+\frac{n}{2}h$. A similar computation shows that 
\[\La_{m,h} =KA_h \glue[P(m-(h-1))][I(h)] \La_{m,h}\]
admits a $(T^L,T^R,n)$-fractured subcategory with $\psi^{\ast}(T^L)\cong T$ and where the fracture on the right can be considered as a slice module of $KA_{h-1}$. A similar induction as in the case of $n$ being odd completes the proof.
\end{proof}

\section{Part IV: \texorpdfstring{$(n,d)$}{(n,d)}-rep\-re\-sen\-ta\-tion-fi\-nite algebras}\label{sect:listofnds}

In this section we construct examples of $(n,d)$-rep\-re\-sen\-ta\-tion-fi\-nite algebras. Specifically, if $n$ is odd we will construct an $(n,d)$-rep\-re\-sen\-ta\-tion-fi\-nite algebra for any $d\geq n$ and if $n$ is even we will construct an $(n,d)$-rep\-re\-sen\-ta\-tion-fi\-nite algebra for any $d$ odd or $d\geq 2n$.

In our constructions we again use acyclic Nakayama algebras. Recall that the Aus\-lan\-der--Rei\-ten quiver $\Gamma(KA_h/\cR)$ of a quotient of $KA_h$ by an admissible ideal is a full subquiver of $\triangle(h)$ with the property that if a vertex $(i,j)$ is not in $\Gamma(KA_h/\cR)$, then the vertices $(i-1,j+1)$ and $(i,j+1)$ are not in $\Gamma(KA_h/\cR)$. Recall also that acyclic Nakayama algebras can be classified by \emph{Kupisch series}, first introduced in \cite{KUP}.

\begin{definition}\label{def:Kupisch}
A tuple $(d_1,\dots,d_{m})\in \ZZ^m_{>0}$ is called an \emph{$m$-Kupisch series} if 
\begin{itemize}
    \item $d_m=1$,
    \item $d_i\geq 2$ for all $1\leq i \leq m$ and
    \item $d_{i-1}-1\leq d_i$ for all $2\leq i \leq m$.
\end{itemize}  
\end{definition}

Then the correspondence between $m$-Kupisch series and  acyclic Nakayama algebras is given by
\[\left\{\text{$m$-Kupisch Series } (d_1,\dots,d_{m})\right\}\;\; \overset{1:1}{\longleftrightarrow}\;\; \left\{\text{$KA_m/\cR$ with $\cR=\langle \alpha_i\cdots\alpha_{i+(d_i-1)} \mid 1\leq i \leq m\rangle$}\right\},\]
where in the above description of $\cR$ paths not belonging to $A_m$ should just be ignored. Using this identification, we will identify a Kupisch series with the corresponding acyclic Nakayama algebra.

Moreover, given an $m$-Kupisch series we explain how to describe the Aus\-lan\-der--Rei\-ten quiver of the corresponding acyclic Nakayama algebra. First, to each number $d_i$ we assign the indecomposable modules $\ind(d_i)=\{M(i,1),M(i-1,2),\dots,M(i-d_i+1,d_i)\}$. Then the Aus\-lan\-der--Rei\-ten quiver corresponding to the $m$-Kupisch series $(d_1,\dots,d_m)$ is the full subquiver of $\triangle(m)$ consisting of the vertices $\bigcup_{i=1}^m\text{ind}(d_i)$. 

We will also use the following lemma for computing syzygies, cosyzygies and Aus\-lan\-der--Rei\-ten translations of modules over acyclic Nakayama algebras.

\begin{lemma} 
\label{lemma:nakayamacompute}
Let $\La$ be an acyclic Nakayama algebra and $M(i,j)\neq 0$ be a $\La$-module.
\begin{itemize}
\item[(a)] If $M(i,j)$ is nonprojective, then 
\begin{align*}
    \tau(M(i,j)) &\cong M(i-1,j), \\
    \om (M(i,j)) &\cong M(i+j-u_{i+j},u_{i+j}-j)
\end{align*}
where 
\[u_s=\max\{y\in \{1,...,n\} \mid (x,y)\in \Gamma (\La)_0 \text{ and } x+y=s\}.\]
\item[(b)] If $M(i,j)$ is noninjective, then 
\begin{align*}
    \tau^-(M(i,j)) &\cong M(i+1,j), \\
    \om^- (M(i,j)) &\cong M(i+j,v_{i}-j)
\end{align*}
where 
\[v_s=\max\{y\in \{1,...,n\} \mid (x,y)\in \Gamma (\La)_0 \text{ and } x=s\}.\]
\end{itemize}
\end{lemma}

\begin{proof}
The claims about the Aus\-lan\-der--Rei\-ten translations are \cite[Lemma 4.7]{VAS}. The claim about syzygy in (a) follows by noticing that $M(i+j-u_{i+j},u_{i+j})$ is a projective cover of $M(i,j)$ and the claim about cosyzygy in (b) follows by noticing that $M(i,v_{i})$ is an injective envelope of $M(i,j)$. 
\end{proof}

In the rest of this and the next subsection we will only draw Aus\-lan\-der--Rei\-ten quivers of Nakayama algebras. When drawing such an Aus\-lan\-der--Rei\-ten quiver we will only draw the vertices as the arrows can be inferred by our conventions. We write some vertices in bold to denote that the additive closure of the indecomposable modules corresponding to the bold vertices is an $n$-cluster tilting subcategory for some $n$. Moreover, we will denote by $h^{(k)}$ a sequence $h,h,\dots,h$ where $h$ appears $k$ times in a Kupisch series. We give an example using this notation. 

\begin{example}\label{ex:firstKupisch}
The acyclic Nakayama algebra $\La_{m,h}$ corresponds to the $m$-Kupisch series 
\[(h,\dots,h,h-1,\dots,2,1)\in \ZZ_{>0}^m,\]
or with our notation $\left(h^{(m-h+1)},h-1,\dots,2,1\right)$. For $m=14$ and $h=3$, let $\cC=\add\left(\bigoplus_{i\geq 0}\tau_2^{-i}(\La_{14,3})\right)$. Then the Aus\-lan\-der--Rei\-ten quiver of $\La_{14,3}$ is
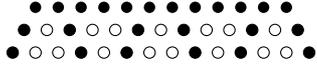
\[\begin{tikzpicture}[scale=1.5, transform shape]
\foreach \x in {0,0.2,...,2.6}
    \mct{\x}{0};

\foreach \x in {0.1,0.3,...,2.5}
    \mct{\x}{0.2};

\foreach \x in {0.2,0.4,...,2.4}
    \nct{\x}{0.4};

\nct{0}{0};
\nct{0.1}{0.2};
\nct{0.6}{0};
\nct{0.5}{0.2};
\nct{1}{0};
\nct{1.1}{0.2};
\nct{1.6}{0};
\nct{1.5}{0.2};
\nct{2}{0};
\nct{2.1}{0.2};
\nct{2.5}{0.2};
\nct{2.6}{0};
\end{tikzpicture}.\]
where the bold vertices correspond to the indecomposable modules in $\cC$. Using Theorem \ref{thrm:char} we can see that $\cC_{\La}$ is a $2$-cluster tilting subcategory.
\end{example}

The algebras $\La_{m,2}$ are of special interest. Specifically, we will use the following Proposition.

\begin{proposition}\label{prop:caseh2}
\begin{itemize}
    \item[(a)] $\gldim(\La_{m,2})=m-1$, the projective dimension of $I(1)$ is $m-1$ and the injective dimension of $P(m)$ is $m-1$,
    \item[(b)] $\La_{m,2}$ admits a unique basic $(m-1)$-cluster tilting module $M$,
    \item[(c)] $M\cong \La_{m,2}\oplus I(1)\cong P(m)\oplus D(\La_{m,2})$.
\end{itemize}
\end{proposition}

\begin{proof}
Part (a) of the Lemma is well-known; for a proof see \cite[Proposition 5.2]{VAS}. Parts (b) and (c) follow immediately as a special case of \cite[Proposition 6.2]{JAS}.  
\end{proof}

The following theorem will be our main tool in constructing examples of $(n,d)$-rep\-re\-sen\-ta\-tion-fi\-nite algebras. 

\begin{theorem}\label{thrm:glueplusn}
\begin{itemize}
\item[(a)] Let $A$ be a strongly $(n,d)$-rep\-re\-sen\-ta\-tion-di\-rect\-ed algebra and assume that there exists a simple projective $A$-module $P$ with injective dimension $d$. Then $\La=\La_{n+1,2} \glue[P][I(1)] A$ is strongly $(n,d+n)$-rep\-re\-sen\-ta\-tion-di\-rect\-ed and there exists a simple projective $\La$-module $P'$ with injective dimension $d+n$.
\item[(b)] Let $B$ be a strongly $(n,d)$-rep\-re\-sen\-ta\-tion-di\-rect\-ed algebra and assume that there exists a simple injective $A$-module $I$ with projective dimension $d$. Then $\La=B \glue[P(n+1)][I] \La_{n+1,2}$ is strongly $(n,d+n)$-rep\-re\-sen\-ta\-tion-di\-rect\-ed and there exists a simple injective $\La$-module $I'$ with projective dimension $d+n$.
\end{itemize}
\end{theorem}

\begin{proof}
We only prove (a); (b) is similar. By Proposition \ref{prop:caseh2}, we have that $\La_{n+1,2}$ is strongly $(n,n)$-rep\-re\-sen\-ta\-tion-di\-rect\-ed. Moreover, $I(1)$ is a simple injective $\La_{n+1,2}$-module since the vertex $1$ is a source. Hence by Corollary \ref{cor:glueatsimple}, it follows that $\La_{n+1,2} \glue[P][I(1)] A$ admits an $n$-cluster tilting subcategory and has global dimension at most $d+n$. To finish the proof it is enough to show that there exists a simple projective $\La$-module $P'$ with injective dimension $d+n$.

Notice that by Corollary \ref{cor:AR glued} we have $\psi_{\ast}(I(1))\cong \phi_{\ast}(P)$ and also that this is the only module in $\psi_{\ast}(\m \La_{n+1,2})\cap \phi_{\ast}(\m A)$. Let $P':=\psi_\ast(P(n+1))$, where $P(n+1)$ is the simple projective $\La_{n+1,2}$-module corresponding to the vertex $n+1$. Since $\dim_{K}\left(P(n+1)\right)=1$ it follows that $\dim_{K}\left(P'\right)=1$ and so $P'$ is simple. Moreover, since $P(n+1)$ is projective, it follows that $P'$ is also projective by Proposition \ref{prop:projinj}.

It is a simple computation to see that $\om^{-n}(P(n+1))\cong I(1)$ (for example, see \cite[Corollary 4.5]{VAS}). Moreover, by Corollary \ref{cor:compute}(b1) we have
\[\om^{-n}(P') = \om^{-n}\psi_{\ast}(P(n+1)) \cong \psi_{\ast}(I(1)).\]
In particular, it follows that
\[\id(P(n+1))=n+\id(\psi_{\ast}(I(1)).\]
Since $\psi_{\ast}(I(1))\cong \phi_{\ast}(P)$ is supported in $A$, it follows from Corollary \ref{cor:compute}(b2) that the injective dimension of $\phi_{\ast}(P)$ is the same as the injective dimension of $P$, which is $d$ by assumption. Hence the injective dimension of $P'$ is $d+n$ which completes the proof.
\end{proof}

In particular we have the following corollary.

\begin{corollary}\label{cor:allnds}
\begin{itemize}
\item[(a)] Let $A$ be a strongly $(n,d)$-rep\-re\-sen\-ta\-tion-di\-rect\-ed algebra and assume that there exists a simple projective $A$-module $P$ with injective dimension $d$. Let $\La$ be the algebra
\[\La = \La_{n+1,2} \glue[P(n+1)][I(1)] \La_{n+1,2} \glue[P(n+1)][I(1)] \cdots \glue[P(n+1)][I(1)] \La_{n+1,2} \glue[P][I(1)]  A, \]
where there are $k-1$ terms $\La_{n+1,2}$ on the right-hand side of the above expression. Then $\La$ is strongly $(n,kn+d)$-rep\-re\-sen\-ta\-tion-di\-rect\-ed.
\item[(b)] Let $B$ be a strongly $(n,d)$-rep\-re\-sen\-ta\-tion-di\-rect\-ed algebra and assume that there exists a simple injective $B$-module $I$ with projective dimension $d$. Let $\La$ be the algebra
\[\La = B \glue[P(n+1)][I] \La_{n+1,2} \glue[P(n+1)][I(1)] \La_{n+1,2} \glue[P(n+1)][I(1)] \cdots \glue[P(n+1)][I(1)] \La_{n+1,2}, \]
where there are $k-1$ terms $\La_{n+1,2}$ on the right-hand side of the above expression. Then $\La$ is strongly $(n,kn+d)$-rep\-re\-sen\-ta\-tion-di\-rect\-ed.
\end{itemize}
\end{corollary}

\begin{proof}
Follows immediately by applying Theorem \ref{thrm:glueplusn} $k$ times.
\end{proof}

As an immediate application of Corollary \ref{cor:allnds} we can now recover \cite[Proposition 6.2(i)]{JAS}:

\begin{example}\label{ex:thecasekn}
It is easy to see (for example using Corollary \ref{cor:AR glued}) that 
\[\La_{kn+1,2}\cong \La_{n+1,2} \glue[P(n+1)][I(1)] \La_{n+1,2} \glue[P(n+1)][I(1)] \cdots \glue[P(n+1)][I(1)] \La_{n+1,2}, \]
where there appear $k$ terms $\La_{n+1,2}$ on the right hand side of the above expression. Hence by Corollary \ref{cor:allnds} it follows that the algebra $\La_{kn+1,2}$ is $(n,kn)$-rep\-re\-sen\-ta\-tion-fi\-nite.
\end{example}

Hence we have examples of $(n,kn)$-rep\-re\-sen\-ta\-tion-fi\-nite algebras for any $k$. Our construction of examples of $(n,d)$-rep\-re\-sen\-ta\-tion-fi\-nite algebras for $d\neq kn$ will follow the same spirit. 

\subsection{The case of \texorpdfstring{$n$}{n} being odd}

Let us start with the case of $n$ being odd. Given $n$, we will construct a strongly $(n,d)$-rep\-re\-sen\-ta\-tion-di\-rect\-ed algebra for any $n\leq d \leq 2n-1$. Moreover, each such algebra will admit a simple projective module of injective dimension $d$. Then by applying Corollary \ref{cor:allnds} we obtain an example of an $(n,d)$-rep\-re\-sen\-ta\-tion-di\-rect\-ed algebra for any $d$. 

Consider the following motivating example.

\begin{example}\label{ex:thecase9}
Let $\La$ be the Nakayama algebra given by the $15$-Kupisch series $\left(2,3^{(11)},2^{(2)},1\right)$ and let $\cC_{\La}=\add\left(\bigoplus_{i\geq 0}\tau_9^{-i}\left(\La\right)\right)$. Then the Aus\-lan\-der--Rei\-ten quiver $\Gamma(\La)$ of $\La$ is
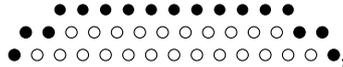
\[\begin{tikzpicture}[scale=1.5, transform shape]
    \foreach \x in {0,0.2,...,2.8}
        \mct{\x}{0};
    \foreach \x in {0.1,0.3,...,2.7}
        \mct{\x}{0.2};
    \foreach \x in {0.4,0.6,...,2.6}
        \nct{\x}{0.4};
    \nct{0}{0};
    \nct{0.1}{0.2};
    \nct{0.3}{0.2};
    \nct{2.5}{0.2};
    \nct{2.7}{0.2};
    \nct{2.8}{0};
\end{tikzpicture},\] 
where the bold vertices denote the indecomposable modules belonging to $\cC_{\La}$. Then using Theorem \ref{thrm:char} we can see that $\cC_{\La}$ is a $9$-cluster tilting subcategory, and a simple calculation shows that 
\[d=\gldim(\La)=\pd(I(1))= 10.\]
Hence $\La$ is a $(9,10)$-rep\-re\-sen\-ta\-tion-fi\-nite algebra. Following the same notation we have in Table \ref{tab:1} a list of $(9,d)$-rep\-re\-sen\-ta\-tion-fi\-nite algebras for different global dimensions $d<18$. Moreover, in each of these examples we have that the $d=\pd(I(1))$. Hence it follows from Corollary \ref{cor:allnds} that there exists a $(9,d)$-rep\-re\-sen\-ta\-tion-fi\-nite algebra for any $d\geq 9$.
\end{example}

\begin{table}[H]\caption{Examples of $(9,d)$--rep\-re\-sen\-ta\-tion-fi\-nite Nakayama algebras for $d<18$}\label{tab:1}
    \centering
    \begin{tabular}{|>{\centering\arraybackslash}m{5.5cm}|>{\centering\arraybackslash}m{8cm}|>{\centering\arraybackslash}m{0.4cm}|}
        \hline 
    $\La$ (as a Kupisch series) & $\Gamma(\La)$ and indecomposables in $\cC_{\La}$  & $d$ \\   
\hline 
    $\left(2^{(2)},3^{(2)},4^{(3)},5^{(13)},4^{(4)},3^{(3)},2^{(3)},1\right)$ 
    & \begin{tikzpicture}[scale=1.3, transform shape]
        \foreach \x in {0,0.2,...,6}
            \mct{\x}{0};
        \foreach \x in {0.1,0.3,...,5.9}
            \mct{\x}{0.2};
        \foreach \x in {0.6,0.8,...,5.4}
            \mct{\x}{0.4};
        \foreach \x in {1.1,1.3,...,4.9}
            \mct{\x}{0.6};
        \foreach \x in {1.8,2.0,...,4.2}
            \nct{\x}{0.8};
        \nct{0}{0};
        \nct{0.1}{0.2};
        \nct{0.3}{0.2};
        \nct{0.5}{0.2};
        \nct{0.6}{0.4};
        \nct{0.8}{0.4};
        \nct{1}{0.4};
        \nct{1.1}{0.6};
        \nct{1.3}{0.6};
        \nct{1.5}{0.6};
        \nct{1.7}{0.6};
        \nct{4.3}{0.6};
        \nct{4.5}{0.6};
        \nct{4.7}{0.6};
        \nct{4.9}{0.6};
        \nct{5}{0.4};
        \nct{5.2}{0.4};
        \nct{5.4}{0.4};
        \nct{5.5}{0.2};
        \nct{5.7}{0.2};
        \nct{5.9}{0.2};
        \nct{6}{0};
        \nct{3}{0};
        
        \node (A) at (3,0.85) {};
    \end{tikzpicture}
    & $11$ \\
\hline    
    $\left(2^{(3)},3^{(8)},2^{(4)},1\right)$
    & \begin{tikzpicture}[scale=1.3, transform shape]
        \foreach \x in {0,0.2,...,3}
            \mct{\x}{0};
        \foreach \x in {0.1,0.3,...,2.9}
            \mct{\x}{0.2};
        \foreach \x in {0.8,1,...,2.2}
            \nct{\x}{0.4};
        \nct{0}{0};
        \nct{0.1}{0.2};
        \nct{0.3}{0.2};
        \nct{0.5}{0.2};
        \nct{0.7}{0.2};
        \nct{2.3}{0.2};
        \nct{2.5}{0.2};
        \nct{2.7}{0.2};
        \nct{2.9}{0.2};
        \nct{3}{0};
                        
        \node (A) at (1.5,0.45) {};
    \end{tikzpicture} 
    & $12$ \\
\hline
    $\left(2^{(4)},3^{(2)},4^{(14)},3^{(3)},2^{(3)},1\right)$
    & \begin{tikzpicture}[scale=1.3, transform shape]
        \foreach \x in {0,0.2,...,5.2}
            \mct{\x}{0};
        \foreach \x in {0.1,0.3,...,5.1}
            \mct{\x}{0.2};
        \foreach \x in {0.6,0.8,...,4.2}
            \mct{\x}{0.4};
        \foreach \x in {1.1,1.3,...,3.7}
            \nct{\x}{0.6};
        \nct{0}{0};
        \nct{0.1}{0.2};
        \nct{0.3}{0.2};
        \nct{0.5}{0.2};
        \nct{0.6}{0.4};
        \nct{0.8}{0.4};
        \nct{1}{0.4};
        \nct{3.8}{0.4};
        \nct{4}{0.4};
        \nct{4.2}{0.4};
        \nct{4.3}{0.2};
        \nct{4.5}{0.2};
        \nct{4.7}{0.2};
        \nct{4.9}{0.2};
        \nct{5.1}{0.2};
        \nct{5.2}{0};
        \nct{2.8}{0};
                        
        \node (A) at (2.6,0.65) {};
    \end{tikzpicture}
    & $13$ \\
\hline 
    $\left(2^{(5)},3^{(5)},2^{(6)},1\right)$
    & \begin{tikzpicture}[scale=1.3, transform shape]
        \foreach \x in {0,0.2,...,3.2}
            \mct{\x}{0};
        \foreach \x in {0.1,0.3,...,3.1}
            \mct{\x}{0.2};
        \foreach \x in {1.2,1.4,...,2}
            \nct{\x}{0.4};
        \nct{0}{0};
        \nct{0.1}{0.2};
        \nct{0.3}{0.2};
        \nct{0.5}{0.2};
        \nct{0.7}{0.2};
        \nct{0.9}{0.2};
        \nct{1.1}{0.2};
        \nct{2.1}{0.2};
        \nct{2.3}{0.2};
        \nct{2.5}{0.2};
        \nct{2.7}{0.2};
        \nct{2.9}{0.2};
        \nct{3.1}{0.2};
        \nct{3.2}{0};
        
        \node (A) at (1.6,0.45) {};
    \end{tikzpicture}
    & $14$ \\
\hline 
    $\left(2^{(6)},3^{(13)},2^{(3)},1\right)$ 
    & \begin{tikzpicture}[scale=1.3, transform shape]
        \foreach \x in {0,0.2,...,4.4}
            \mct{\x}{0};
        \foreach \x in {0.1,0.3,...,4.3}
            \mct{\x}{0.2};
        \foreach \x in {0.6,0.8,...,3}
            \nct{\x}{0.4};
        \nct{0}{0};
        \nct{0.1}{0.2};
        \nct{0.3}{0.2};
        \nct{0.5}{0.2};
        \nct{3.1}{0.2};
        \nct{3.3}{0.2};
        \nct{3.5}{0.2};
        \nct{3.7}{0.2};
        \nct{3.9}{0.2};
        \nct{4.1}{0.2};
        \nct{4.3}{0.2};
        \nct{4.4}{0};
        \nct{2.4}{0};        
        
        \node (A) at (2.2,0.45) {};
    \end{tikzpicture}
    & $15$ \\
\hline 
    $\left(2^{(7)},3^{(2)},2^{(8)},1\right)$
    & \begin{tikzpicture}[scale=1.3, transform shape]
        \foreach \x in {0,0.2,...,3.4}
            \mct{\x}{0};
        \foreach \x in {0.1,0.3,...,3.3}
            \mct{\x}{0.2};
        \foreach \x in {1.6,1.8}
            \nct{\x}{0.4};
        \nct{0}{0};
        \nct{0.1}{0.2};
        \nct{0.3}{0.2};
        \nct{0.5}{0.2};
        \nct{0.7}{0.2};
        \nct{0.9}{0.2};
        \nct{1.1}{0.2};
        \nct{1.3}{0.2};
        \nct{1.5}{0.2};
        \nct{1.9}{0.2};
        \nct{2.1}{0.2};
        \nct{2.3}{0.2};
        \nct{2.5}{0.2};
        \nct{2.7}{0.2};
        \nct{2.9}{0.2};
        \nct{3.1}{0.2};
        \nct{3.3}{0.2};
        \nct{3.4}{0};
        
        \node (A) at (1.7,0.45) {};
    \end{tikzpicture} 
    & $16$ \\
\hline 
    $\left(2^{(8)},3^{(13)},2,1\right)$
    & \begin{tikzpicture}[scale=1.3, transform shape]
        \foreach \x in {0,0.2,...,4.4}
            \mct{\x}{0};
        \foreach \x in {0.1,0.3,...,4.3}
            \mct{\x}{0.2};
        \foreach \x in {0.2,0.4,...,2.6}
            \nct{\x}{0.4};
        \nct{0}{0};
        \nct{0.1}{0.2};
        \nct{2.7}{0.2};
        \nct{2.9}{0.2};
        \nct{3.1}{0.2};
        \nct{3.3}{0.2};
        \nct{3.5}{0.2};
        \nct{3.7}{0.2};
        \nct{3.9}{0.2};
        \nct{4.1}{0.2};
        \nct{4.3}{0.2};
        \nct{4.4}{0};
        \nct{2.6}{0};        
        
        \node (A) at (2.2,0.45) {};
    \end{tikzpicture}
    & $17$ \\
\hline
    \end{tabular}
\end{table}
With Example \ref{ex:thecase9} in mind, we have the following proposition.

\begin{proposition}\label{prop:thecasendodd}
Let $n$ be odd and $n<d<2n$. 
\begin{itemize}
\item[(a)] If $d$ is even, then the acyclic Nakayama algebra $\La$ with Kupisch series 
\[\left(2^{(d-n)},3^{\left(3\left(n-\frac{d}{2}\right)-1\right)},2^{(d-n+1)},1\right)\]
is $(n,d)$-rep\-re\-sen\-ta\-tion-fi\-nite and we have $\pd(I(1))=d$.

\item[(b)] If $d$ is odd and $d\neq 2n-1$, set
\[h=n-\frac{d-1}{2}\; \text{ and } \;s=\left( \frac{d-n}{2}\right)(h+1)+2h.\]
Then the acyclic Nakayama algebra $\La$ with Kupisch series
\[\left( 2^{\left(d-n\right)}, 3^{(2)}, 4^{(3)}, \dots, h^{\left(h-1\right)}, (h+1)^{(s)}, h^{(h)}, (h-1)^{(h-1)}, \dots, 3^{(3)},2^{(2)},2,1\right),\]
is $(n,d)$-rep\-re\-sen\-ta\-tion-fi\-nite and we have $\pd(I(1))=d$. 
    
\item[(c)] If $d=2n-1$, then the acyclic Nakayama algebra $\La = \La_{3\left(\frac{n+1}{2}\right),3} \glue[P(n)][I(2)] \La_{n+1,2}$ with Kupisch series
\[\left(2^{(n-1)}, 3^{\left(3\left(\frac{n+1}{2}\right)-2\right)},2,1\right) \]
is $(n,d)$-rep\-re\-sen\-ta\-tion-fi\-nite and we have $\pd(I(1))=d$.
\end{itemize}
\end{proposition}

\begin{proof}
\item[(a)] The Aus\-lan\-der--Rei\-ten quiver $\Gamma(\La)$ in this case has the form
\[\begin{tikzpicture}
\node (A1) at (0,0) {$Q_1$};
\node (B1) at (0.3,0.3) {$\bullet$};
\node (B2) at (0.6,0) {\textopenbullet};
\node (C1) at (0.9,0.3) {$\bullet$};
\node (C2) at (1.2,0) {\textopenbullet};

\node (D1) at (2.2,0.3) {$\bullet$};
\node (D2) at (2.5,0) {\textopenbullet};
\node (E1) at (2.8,0.3) {$Q_2$};
\node (E2) at (3.1,0) {\textopenbullet};
\node (F1) at (3.1,0.6) {$\bullet$};
\node (F2) at (3.4,0.3) {\textopenbullet};
\node (F3) at (3.7,0) {\textopenbullet};
\node (G1) at (3.7,0.6) {$\bullet$};
\node (G2) at (4,0.3) {\textopenbullet};
\node (G3) at (4.3,0) {\textopenbullet};

\node (H1) at (5,0.6) {$\bullet$};
\node (H2) at (5.3,0.3) {\textopenbullet};
\node (H3) at (5.6,0) {\textopenbullet};
\node (I1) at (5.6,0.6) {$\bullet$};
\node (I2) at (5.9,0.3) {$J_1$};
\node (I3) at (6.2,0) {\textopenbullet};
\node (J1) at (6.5,0.3) {$\bullet$};
\node (J2) at (6.8,0) {\textopenbullet};

\node (K1) at (7.8,0.3) {$\bullet$};
\node (K2) at (8.1,0) {\textopenbullet};
\node (L1) at (8.4,0.3) {$\bullet$};
\node (L2) at (8.7,0) {$J_2$\nospacepunct{,}};

\draw[loosely dotted] (1.4,0.15) -- (2,0.15);
\draw[loosely dotted] (4.35,0.3) -- (4.95,0.3);
\draw[loosely dotted] (7,0.15) -- (7.6,0.15);
\end{tikzpicture}\]
where $Q_1\cong M\left(1,1\right)$, $Q_2\cong M\left(d-n+1,2\right)$, $J_1\cong M\left(2n-\frac{d}{2},2\right)$, $J_2\cong M\left(n+\frac{d}{2}+1,1\right)$ and the bold vertices correspond to the indecomposable pro\-jec\-tive-in\-jec\-tive $\La$-modules. Moreover, the vertices $Q_1$ and $Q_2$ correspond to the indecomposable projective noninjective $\La$-modules and the vertices $J_1$ and $J_2$ correspond to the indecomposable injective nonprojective $\La$-modules. Using Lemma \ref{lemma:nakayamacompute} we compute

\begin{align*} 
\tno (Q_1) &= \tau^- \om^{-(n-1)} (Q_1) \cong \tau^- \om^{-(2n-d-1)}\om^{-(d-n)}\left( M(1,1)\right) \cong \tau^- \om^{-(2n-d-1)}\left(M(d-n+1,1)\right)\\
&\cong \tau^-\left(M\left(2n-\frac{d}{2}-1,2\right)\right) \cong M\left(2n-\frac{d}{2},2\right) 
\cong J_1.
\end{align*}

Similar computations show that 
\[\tno(Q_2) \cong J_2,\;\; \tn(J_1)\cong Q_1,\;\; \tn(J_2)\cong Q_2,\;\; \om^{n+1}(J_1)\cong 0, \om^d(J_2) \cong Q_1.\]
Hence Theorem \ref{thrm:char} implies that 
\[ \cC = \add(\La \oplus \tno \La ) = \add (\La \oplus J_1 \oplus J_2) \]
is an $n$-cluster tilting subcategory. Moreover, since 
\[ \pd(J_1) <d \text{ and } \pd(J_2)=d\]
and $J_1$ and $J_2$ are the only indecomposable injective nonprojective  $\La$-modules, we have $\gldim (\La)=d$. For the final part, we have that $J_2\cong M\left(n+\frac{d}{2}+1,1\right) \cong I(1)$ and hence $\pd(I(1))=d$.

\item[(b)] The two cases are similar; let us only prove the case $h>2$. The Aus\-lan\-der--Rei\-ten quiver $\Gamma(\La)$ in this case has the form
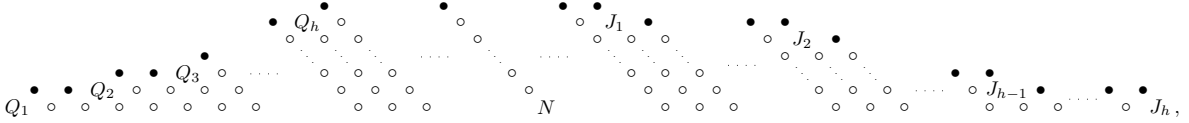
\[\resizebox {\columnwidth} {!} {
\begin{tikzpicture}
\node (Z) at (-0.6,0) {$Q_1$};
\node (A1) at (-0.3,0.3) {$\bullet$};
\node (A2) at (0,0) {\textopenbullet};
\node (B1) at (0.3,0.3) {$\bullet$};
\node (B2) at (0.6,0) {\textopenbullet};
\node (C1) at (0.9,0.3) {$Q_2$};
\node (C2) at (1.2,0) {\textopenbullet};
\node (D1) at (1.2,0.6) {$\bullet$};
\node (D2) at (1.5,0.3) {\textopenbullet};
\node (D3) at (1.8,0) {\textopenbullet};
\node (E1) at (1.8,0.6) {$\bullet$};
\node (E2) at (2.1,0.3) {\textopenbullet};
\node (E3) at (2.4,0) {\textopenbullet};
\node (F1) at (2.4,0.6) {$Q_3$};
\node (F2) at (2.7,0.3) {\textopenbullet};
\node (F3) at (3,0) {\textopenbullet};
\node (G1) at (2.7,0.9) {$\bullet$};
\node (G2) at (3,0.6) {\textopenbullet};
\node (G3) at (3.3,0.3) {\textopenbullet};
\node (G4) at (3.6,0) {\textopenbullet};

\draw[loosely dotted] (3.5,0.6) -- (4,0.6);

\node (H1) at (3.9,1.5) {$\bullet$};
\node (H2) at (4.2,1.2) {\textopenbullet};
\node[scale=0.5] (Ha) at (4.4,1) {$\cdot$};
\node[scale=0.5] (Hb) at (4.5,0.9) {$\cdot$};
\node[scale=0.5] (Hc) at (4.6,0.8) {$\cdot$};
\node (H3) at (4.8,0.6) {\textopenbullet};
\node (H4) at (5.1,0.3) {\textopenbullet};
\node (H5) at (5.4,0) {\textopenbullet};
\node (I1) at (4.5,1.5) {$Q_h$};
\node (I2) at (4.8,1.2) {\textopenbullet};
\node[scale=0.5] (Ia) at (5,1) {$\cdot$};
\node[scale=0.5] (Ib) at (5.1,0.9) {$\cdot$};
\node[scale=0.5] (Ic) at (5.2,0.8) {$\cdot$};
\node (I3) at (5.4,0.6) {\textopenbullet};
\node (I4) at (5.7,0.3) {\textopenbullet};
\node (I5) at (6,0) {\textopenbullet};
\node (J1) at (4.8,1.8) {$\bullet$};
\node (J2) at (5.1,1.5) {\textopenbullet};
\node (J3) at (5.4,1.2) {\textopenbullet};
\node[scale=0.5] (Ja) at (5.6,1) {$\cdot$};
\node[scale=0.5] (Jb) at (5.7,0.9) {$\cdot$};
\node[scale=0.5] (Jc) at (5.8,0.8) {$\cdot$};
\node (J4) at (6,0.6) {\textopenbullet};
\node (J5) at (6.3,0.3) {\textopenbullet};
\node (J6) at (6.6,0) {\textopenbullet};

\draw[loosely dotted] (6.5,0.9) -- (7,0.9);

\node (K1) at (6.9,1.8) {$\bullet$};
\node (K2) at (7.2,1.5) {\textopenbullet};
\node (K3) at (7.5,1.2) {\textopenbullet};
\node[scale=0.5] (Ka) at (7.7,1) {$\cdot$};
\node[scale=0.5] (Kb) at (7.8,0.9) {$\cdot$};
\node[scale=0.5] (Kc) at (7.9,0.8) {$\cdot$};
\node (K4) at (8.1,0.6) {\textopenbullet};
\node (K5) at (8.4,0.3) {\textopenbullet};
\node (K6) at (8.7,0) {$N$};

\draw[loosely dotted] (8.6,0.9) -- (9.1,0.9);

\node (L1) at (9,1.8) {$\bullet$};
\node (L2) at (9.3,1.5) {\textopenbullet};
\node (L3) at (9.6,1.2) {\textopenbullet};
\node[scale=0.5] (La) at (9.8,1) {$\cdot$};
\node[scale=0.5] (Lb) at (9.9,0.9) {$\cdot$};
\node[scale=0.5] (Lc) at (10,0.8) {$\cdot$};
\node (L4) at (10.2,0.6) {\textopenbullet};
\node (L5) at (10.5,0.3) {\textopenbullet};
\node (L6) at (10.8,0) {\textopenbullet};
\node (M1) at (9.6,1.8) {$\bullet$};
\node (M2) at (9.9,1.5) {$J_1$};
\node (M3) at (10.2,1.2) {\textopenbullet};
\node[scale=0.5] (Ma) at (10.4,1) {$\cdot$};
\node[scale=0.5] (Mb) at (10.5,0.9) {$\cdot$};
\node[scale=0.5] (Mc) at (10.6,0.8) {$\cdot$};
\node (M4) at (10.8,0.6) {\textopenbullet};
\node (M5) at (11.1,0.3) {\textopenbullet};
\node (M6) at (11.4,0) {\textopenbullet};
\node (N1) at (10.5,1.5) {$\bullet$};
\node (N2) at (10.8,1.2) {\textopenbullet};
\node[scale=0.5] (Na) at (11,1) {$\cdot$};
\node[scale=0.5] (Nb) at (11.1,0.9) {$\cdot$};
\node[scale=0.5] (Nc) at (11.2,0.8) {$\cdot$};
\node (N3) at (11.4,0.6) {\textopenbullet};
\node (N4) at (11.7,0.3) {\textopenbullet};
\node (N5) at (12,0) {\textopenbullet};

\draw[loosely dotted] (11.9,0.75) -- (12.4,0.75);

\node (O1) at (12.3,1.5) {$\bullet$};
\node (O2) at (12.6,1.2) {\textopenbullet};
\node (O3) at (12.9,0.9) {\textopenbullet};
\node[scale=0.5] (Oa) at (13.1,0.7) {$\cdot$};
\node[scale=0.5] (Ob) at (13.2,0.6) {$\cdot$};
\node[scale=0.5] (Oc) at (13.3,0.5) {$\cdot$};
\node (O4) at (13.5,0.3) {\textopenbullet};
\node (O5) at (13.8,0) {\textopenbullet};
\node (P1) at (12.9,1.5) {$\bullet$};
\node (P2) at (13.2,1.2) {$J_2$};
\node (P3) at (13.5,0.9) {\textopenbullet};
\node[scale=0.5] (Pa) at (13.7,0.7) {$\cdot$};
\node[scale=0.5] (Pb) at (13.8,0.6) {$\cdot$};
\node[scale=0.5] (Pc) at (13.9,0.5) {$\cdot$};
\node (P4) at (14.1,0.3) {\textopenbullet};
\node (P5) at (14.4,0) {\textopenbullet};
\node (Q1) at (13.8,1.2) {$\bullet$};
\node (Q2) at (14.1,0.9) {\textopenbullet};
\node[scale=0.5] (Qa) at (14.3,0.7) {$\cdot$};
\node[scale=0.5] (Qb) at (14.4,0.6) {$\cdot$};
\node[scale=0.5] (Qc) at (14.5,0.5) {$\cdot$};
\node (Q3) at (14.7,0.3) {\textopenbullet};
\node (Q4) at (15,0) {\textopenbullet};

\draw[loosely dotted] (15.2,0.3) -- (15.7,0.3);

\node (R1) at (15.9,0.6) {$\bullet$};
\node (R2) at (16.2,0.3) {\textopenbullet};
\node (R3) at (16.5,0) {\textopenbullet};
\node (Q1) at (16.5,0.6) {$\bullet$};
\node (Q2) at (16.8,0.3) {$J_{h-1}$};
\node (Q3) at (17.1,0) {\textopenbullet};
\node (R1) at (17.4,0.3) {$\bullet$};
\node (R2) at (17.7,0) {\textopenbullet};

\draw[loosely dotted] (17.9,0.15) -- (18.4,0.15);

\node (S1) at (18.6,0.3) {$\bullet$};
\node (S2) at (18.9,0) {\textopenbullet};
\node (T1) at (19.2,0.3) {$\bullet$};
\node (T2) at (19.5,0) {$J_h$\nospacepunct{,}};
\end{tikzpicture}}\]
where 
\begin{itemize}
    \item $Q_1\cong M\left(1,1\right)$,
    \item $Q_i\cong M\left(\frac{i(i-1)}{2}+2,i\right)$, for $2\leq i \leq h$,
    \item $N\cong M\left( \frac{(2n-d+3)(d+1)}{8}+1,1\right)$,
    \item $J_i \cong M\left(\frac{h(h-1)}{2}+2+s+(i-1)h-\frac{(i-2)(i-1)}{2},h-i+1\right)$, for $1\leq i \leq h-1$
    \item $J_h \cong M\left(\frac{1}{4}(n(2n-d+5)+d+5),1\right)$.
\end{itemize}

Moreover the bold vertices correspond to the indecomposable pro\-jec\-tive-in\-jec\-tive $\La$-modules, the vertices $Q_i$ correspond to the indecomposable projective noninjective $\La$-modules and the vertices $J_i$ correspond to the indecomposable injective nonprojective $\La$-modules. Using Lemma \ref{lemma:nakayamacompute} we compute
\[\tno (Q_i) \cong \begin{cases} N &\mbox{ if $i=1$,} \\ J_{i-1} &\mbox{ if $2\leq i \leq h$,}\end{cases} \;\; \tno (N) \cong J_h, \;\;
\tn (J_i) \cong \begin{cases} Q_{i+1} &\mbox{ if $1\leq i \leq h-1$,} \\ N &\mbox{ if $i=h$,}\end{cases}\;\; \tn (N) \cong Q_1.\]
Let us indicatively show that $\tn(J_i)\cong Q_{i+1}$ for $1\leq i \leq h-1$; the other computations are similar. Looking at the Aus\-lan\-der--Rei\-ten quiver $\Gamma(\La)$ it is clear that the projective cover of $J_i$ corresponds to the vertex exactly to the left and above of $J_i$. In other words, we have 
\[u_{\left(\frac{h(h-1)}{2}+2+s+(i-1)h-\frac{(i-2)(i-1)}{2}\right)+(h-i+1)}=h-i+2.\]
Then by Lemma \ref{lemma:nakayamacompute} we have (for simplicity, we write $(x,y)$ instead of $M(x,y)$)
\begin{align*}
    \Omega (J_i) &= \Omega \left(\frac{h(h-1)}{2}+2+s+(i-1)h-\frac{(i-2)(i-1)}{2},h-i+1\right) \\
    &\cong \left(\frac{h(h-1)}{2}+2+s+(i-1)h-\frac{(i-2)(i-1)}{2}+h-i+1-(h-i+2),h-i+2-(h-i+1)\right) \\
    &=\left(\frac{h(h-1)}{2}+1+s+(i-1)h-\frac{(i-2)(i-1)}{2},1\right).
\end{align*}
Assume that $i>1$. Again it is easy to see that \[u_{\left(\frac{h(h-1)}{2}+1+s+(i-1)h-\frac{(i-2)(i-1)}{2}\right)+1}=h-(i-1)+2\]
by noticing that $J_{i-1}$ is a vertex in the Aus\-lan\-der--Rei\-ten quiver $\Gamma(\La)$ and the sum of its coordinates is equal to the sum of the coordinates of $\Omega (J_i)$. Then
\begin{align*}
    \Omega^2 (J_i) &=\Omega\left(\Omega (J_i)\right) \\ 
    &\cong\Omega \left(\frac{h(h-1)}{2}+1+s+(i-1)h-\frac{(i-2)(i-1)}{2},1\right) \\
    &\cong \left(\frac{h(h-1)}{2}+1+s+(i-1)h-\frac{(i-2)(i-1)}{2}+1-\left(h-(i-1)+2\right),h-(i-1)+2-1\right) \\
    &=\left(\frac{h(h-1)}{2}+2+s+(i-2)h-\frac{(i-2)(i-1)}{2}+(i-1)-2,h-(i-1)+1\right) \\
    &=\left(\frac{h(h-1)}{2}+2+s+(i-2)h-\frac{(i-3)(i-2)}{2}-1,h-(i-1)+1\right)\\
    &\cong \tau (J_{i-1}).
\end{align*}
It follows by a simple induction that for every $1\leq i \leq h-1$ and $0\leq k \leq i-1$ we have $\Omega^{2k} (J_i) \cong \tau^k (J_{i-k})$. In particular, for $k=i-1$ we have 
\begin{align*}
    \Omega^{2i-2}(J_i) &\cong \tau^{i-1}(J_1) \\
    &\cong \tau^{i-1}\left(\frac{h(h-1)}{2}+2+s,h\right) \\
    &\cong \left(\frac{h(h-1)}{2}+2+s-(i-1),h\right)\\
    &=\left(\frac{h(h-1)}{2}+3+s-i,h\right).
\end{align*}
Now notice that if $(x,y)$ is a vertex in the Aus\-lan\-der--Rei\-ten quiver $\Gamma(\La)$ such that
\begin{equation}\label{eq:homogeneous coordinates}
    \frac{h(h-1)}{2}+3\leq x+y\leq \frac{h(h-1)}{2}+2+s+h,
\end{equation}
then $(x,y)$ is inside the parallelogram defined by the vertices $\left(\frac{h(h-1)}{2}+2,h+1\right)$, $\left(\frac{h(h-1)}{2}+2+h,1\right)$, $\left(\frac{h(h-1)}{2}+1+s,h+1\right)$, and $\left(\frac{h(h-1)}{2}+1+s+h,1\right)$. It follows that in this case $u_{x+y}=h+1$.
It is easy to check that $\Omega^{2i-2}(J_i)$ satisfies (\ref{eq:homogeneous coordinates}). Then we have
\begin{align*}
    \Omega(\Omega^{2i-2}(J_i)) &\cong \Omega \left(\frac{h(h-1)}{2}+3+s-i,h\right) \\
    &\cong \left(\frac{h(h-1)}{2}+3+s-i+h-(h+1),h+1-h\right) \\
    &= \left(\frac{h(h-1)}{2}+2+s-i,1\right),
\end{align*}
which again satisfies (\ref{eq:homogeneous coordinates}). Hence
\begin{align*}
    \Omega^2(\Omega^{2i-2}(J_i)) &\cong \Omega \left(\frac{h(h-1)}{2}+2+s-i,1\right) \\
    &\cong \left(\frac{h(h-1)}{2}+2+s-i+1-(h+1),h+1-1\right) \\
    &= \left(\frac{h(h-1)}{2}+2+s-i-h,h\right)\\
    &\cong \tau^{h+1}\tau^{i-1}(J_1)
\end{align*}
Since $s=\frac{d-n}{2}(h+1)+2h$, and since $1\leq i \leq h-1$, it follows again by a simple induction that for every $1\leq i \leq h-1$ and $1\leq k\leq \frac{d-n}{2}+1$ we have $\Omega^{2k}\Omega^{2i-2}(J_i) \cong \tau^{k(h+1)}\tau^{i-1}(J_1)$. In particular, for $k=\frac{d-n}{2}+1$ we have
\begin{align*}
    \Omega^{(d-n)+2}\Omega^{2i-2} (J_i) &\cong \tau^{\left(\frac{d-n}{2}+1\right)(h+1)}\tau^{i-1} J_1 \\
    &\cong \tau^{s-h+1}\tau^{i-1} J_1 \\
    &\cong \tau^{s-h+i} \left(\frac{h(h-1)}{2}+2+s,h\right) \\
    &\cong \left(\frac{h(h-1)}{2}+2+s-\left(s-h+i\right),h\right) \\
    &= \left(\frac{h(h-1)}{2}+2+h-i,h\right).
\end{align*}
Rewriting the above, we have shown that for $1\leq i \leq h-1$ we have
\[\Omega^{d-n+2i}(J_i) \cong \left(\frac{h(h-1)}{2}+2+h-i,h\right).\]
Clearly $\Omega^{d-n+2i}(J_i)$ also satisfies (\ref{eq:homogeneous coordinates}) and so
\begin{align*}
    \Omega(\Omega^{d-n+2i}) (J_i) &\cong \Omega \left(\frac{h(h-1)}{2}+2+h-i,h\right) \\
    &\cong \left(\frac{h(h-1)}{2}+2+h-i+h-(h+1),h+1-h\right) \\
    &= \left(\frac{h(h-1)}{2}+1+h-i,1\right).
\end{align*}
We have
\[\frac{(h-1)(h-2)}{2}+2+h \leq\frac{h(h-1)}{2}+2+h-i\leq\frac{h(h-1)}{2}+1+h,\]
that is $\Omega(\Omega^{d-n+2i}) (J_i)$ is inside the parallelogram defined by $\left(\frac{(h-1)(h-2)}{2}+2,h\right)$, $\left(\frac{(h-1)(h-2)}{2}+1+h,1\right)$, $\left(\frac{h(h-1)}{2}+1,h\right)$ and $\left(\frac{h(h-1)}{2}+h,1\right)$. In particular, this implies that $u_{\left(\frac{h(h-1)}{2}+1+h-i\right)+1}=h$. Hence
\begin{align*}
    \Omega^2(\Omega^{d-n+2i}) (J_i) &\cong \Omega \left(\frac{h(h-1)}{2}+1+h-i,1\right) \\
    &\cong \left(\frac{h(h-1)}{2}+1+h-i+1-h,h-1\right) \\
    &=\left(\frac{h(h-1)}{2}+2-i,h-1\right) \\
    &=\left(\frac{(h-1)(h-2)}{2}+2+(h-i-1),h-1\right)\\
    &\cong \tau^{-(h-i-1)}(Q_{h-1}).
\end{align*}
It follows again by a simple induction that for every $1\leq i \leq h-1$ and $0\leq k \leq h-i$ we have 
\[\Omega^{2k}\Omega^{d-n+2i} (J_i)\cong \tau^{-(h-i-k)} (Q_{h-k}).\]
In particular, for $k=h-i-1$ we have 
\[\Omega^{2(h-i-1)}\Omega^{d-n+2i} (J_i) \cong \tau^{-(h-i-(h-i-1))}(Q_{h-(h-i-1)})=\tau^{-}(Q_{i+1})\]
and so
\[\tau \Omega^{2(h-i-1)+d-n+2i} (J_i) \cong \tau\tau^{-}(Q_{i+1}) \cong Q_{i+1}.\]
But 
\[2(h-i-1)+d-n+2i=d-n+2h-2=d-n+2n-(d-1)=n-1,\]
which shows that $\tau_n(J_i)\cong Q_{i+1}$. 

Next, it follows from Theorem \ref{thrm:char} that 
\[\cC = \add(\La \oplus \tno(\La) \oplus \tau_n^{-2}(\La) ) = \add(\La \oplus N \oplus D(\La) )\]
is an $n$-cluster tilting subcategory. For the computation of the global dimension, notice that again using Lemma \ref{lemma:nakayamacompute} as well as the previous computations, for $1\leq i \leq h-1$ we have
\begin{align*} 
\om^{n+1} (J_i) &\cong \om^{2}\om^{n+1}(J_i) \cong \om^{2}\tau^- (Q_{i+1}) \cong \om^{2} \tau^- \left(M\left(\frac{i(i+1)}{2}+2,i+1\right)\right)\\ &\cong \om^2 \left(M\left(\frac{i(i+1)}{2}+3,i+1\right)\right) \cong \om \left(M \left(\frac{i(i+1)}{2}+2,1 \right) \right) \\ &\cong \left(M\left(\frac{i(i-1)}{2}+2,i\right)\right) \cong Q_i.
\end{align*}
A similar computation shows that $\om^{d}(J_h)\cong Q_h$. Since $n$ and $d$ are both odd and since $Q_i$ is projective for any $i$ it follows that
\[\gldim(\La) = \max \{\pd(J_i) \mid 1\leq i \leq h\} = d.\]
Finally, since $I(1) = J_h$, we also have $\pd(I(1))=d$.

\item[(c)] The Aus\-lan\-der--Rei\-ten quiver $\Gamma(\La)$ in this case has the form
\[\begin{tikzpicture}
\node (A1) at (0,0) {$Q_1$};
\node (B1) at (0.3,0.3) {$Q_2$};
\node (B2) at (0.6,0) {\textopenbullet};
\node (C1) at (0.6,0.6) {$\bullet$};
\node (C2) at (0.9,0.3) {\textopenbullet};
\node (C3) at (1.2,0) {\textopenbullet};
\node (D1) at (1.2,0.6) {$\bullet$};
\node (D2) at (1.5,0.3) {\textopenbullet};
\node (D3) at (1.8,0) {\textopenbullet};

\node (E1) at (2.8,0.6) {$\bullet$};
\node (E2) at (3.1,0.3) {\textopenbullet};
\node (E3) at (3.4,0) {\textopenbullet};
\node (F1) at (3.4,0.6) {$\bullet$};
\node (F2) at (3.7,0.3) {\textopenbullet};
\node (F3) at (4,0) {$N$};
\node (G1) at (4,0.6) {$\bullet$};
\node (G2) at (4.3,0.3) {$J_1$};
\node (G3) at (4.6,0) {\textopenbullet};
\node (H1) at (4.9,0.3) {$\bullet$};
\node (H2) at (5.2,0) {\textopenbullet};

\node (H1) at (6.2,0.3) {$\bullet$};
\node (H2) at (6.5,0) {\textopenbullet};
\node (I1) at (6.8,0.3) {$\bullet$};
\node (I2) at (7.1,0) {$J_2$\nospacepunct{,}};

\draw[loosely dotted] (2,0.3) -- (2.6,0.3);
\draw[loosely dotted] (5.4,0.15) -- (6.05,0.15);
\end{tikzpicture}\]
where 
\[Q_1\cong M(1,1),\;\; Q_2\cong M(1,2),\;\; N\cong M\left( \frac{3}{2}(n-1)+2,1 \right),\] 
\[J_1\cong M\left(\frac{3}{2}(n-1)+2,2\right),\;\; J_2\cong M\left(\frac{5}{2}(n-1)+3,1 \right),\]
and the bold vertices correspond to the indecomposable pro\-jec\-tive-in\-jec\-tive $\La$-modules. Moreover the vertices $Q_i$ correspond to the indecomposable projective noninjective $\La$-modules and the vertices $J_i$ correspond to the indecomposable injective nonprojective modules. In particular, the algebra $\La$ can be identified with a gluing $\La_{3\left(\frac{n+1}{2}\right),3} \glue[P(n)][I(2)] \La_{n+1,2}$. We set $A\coloneqq \La_{n+1,2}$ and $B\coloneqq \La_{3\left(\frac{n+1}{2}\right),3}$. Then we can identify the Aus\-lan\-der--Rei\-ten quivers $\Gamma\left(B\right)$ and $\Gamma(A)$ as follows:
\[\begin{tikzpicture}[baseline={(current bounding box.center)}]
\node (A1) at (0,0) {$Q_1$};
\node (B1) at (0.3,0.3) {$Q_2$};
\node (B2) at (0.6,0) {\textopenbullet};
\node (C1) at (0.6,0.6) {$\bullet$};
\node (C2) at (0.9,0.3) {\textopenbullet};
\node (C3) at (1.2,0) {\textopenbullet};
\node (D1) at (1.2,0.6) {$\bullet$};
\node (D2) at (1.5,0.3) {\textopenbullet};
\node (D3) at (1.8,0) {\textopenbullet};

\node (E1) at (2.8,0.6) {$\bullet$};
\node (E2) at (3.1,0.3) {\textopenbullet};
\node (E3) at (3.4,0) {\textopenbullet};
\node (F1) at (3.4,0.6) {$\bullet$};
\node (F2) at (3.7,0.3) {\textopenbullet};
\node (F3) at (4,0) {$N$};
\node (G1) at (4,0.6) {$\bullet$};
\node (G2) at (4.3,0.3) {$J_1$};
\node (G3) at (4.6,0) {\textopenbullet};

\node[scale=0.7] (name) at (2.3,-0.5) {$\Gamma\left(B\right)$};
\draw[loosely dotted] (2,0.3) -- (2.6,0.3);
\end{tikzpicture}
\text{\;\; and \;\;}
\begin{tikzpicture}[baseline={(current bounding box.center)}]

\node (F3) at (4,0) {$N$};
\node (G2) at (4.3,0.3) {$J_1$};
\node (G3) at (4.6,0) {\textopenbullet};
\node (H1) at (4.9,0.3) {$\bullet$};
\node (H2) at (5.2,0) {\textopenbullet};

\node (H1) at (6.2,0.3) {$\bullet$};
\node (H2) at (6.5,0) {\textopenbullet};
\node (I1) at (6.8,0.3) {$\bullet$};
\node (I2) at (7.1,0) {$J_2$\nospacepunct{.}};

\node[scale=0.7] (name) at (6,-0.5) {$\Gamma(A)$};
\draw[loosely dotted] (5.4,0.15) -- (6.05,0.15);
\end{tikzpicture}
\]
Using Lemma \ref{lemma:taun}, we compute in $\m B$:
\[\tno (Q_1) \cong N,\;\; \tno (Q_2)\cong J_1,\;\; \tno (N) \cong J_2,\]
and in $\m A$:
\[\tn (N) \cong Q_1,\;\; \tn (J_1) \cong Q_2,\;\; \tn (J_2)\cong N.\]
It follows that $\cC_A=\add(\La_{n+1,2}\oplus J_2)$ is a $(J_1\oplus N, J_2\oplus I(2),n)$-fractured subcategory of $\m A$ and that $\cC_B=\add(\La_{3\left(\frac{n+1}{2}\right),3}\oplus J_1\oplus N)$ is a $(\La_{3\left(\frac{n+1}{2}\right),3}, I(3)\oplus J_1\oplus N, n)$-fractured subcategory of $\m B$. The gluing is clearly compatible at the fracture as per the requirements of Theorem \ref{thrm:fractsubcat} and hence we have that
\begin{align*} 
\cC &= \add(\La \oplus N\oplus J_1 \oplus J_2)
\end{align*}
is an $n$-cluster tilting subcategory of $\m\La$. Similar computations as above show that in $\m\La$ we have
\[\om^n (J_1) \cong Q_1, \;\; \om^{d} (J_2) \cong Q_2,\]
from which it follows that $\gldim(\La)=d$. Since $J_2=I(1)$, the proof is complete.
\end{proof}

\begin{corollary}\label{cor:oddclustertilting}
Let $n$ be odd and $d\geq n$. There exists an $(n,d)$-rep\-re\-sen\-ta\-tion-fi\-nite algebra $\La$.
\end{corollary}

\begin{proof}
Write $d=qn+d'$ for some $q\geq 1$ and $0\leq d'\leq n-1$. If $d'=0$ then $\La$ exists by Example \ref{ex:thecasekn}. If $0<d'<n$, let $\La'$ be an $(n,d')$-rep\-re\-sen\-ta\-tion-fi\-nite algebra as in Proposition \ref{prop:thecasendodd}. Then $\La'$ satisfies the assumptions of Corollary \ref{cor:allnds}(b) and so there exists an algebra $\La$ which is $(n,qn+d')$-rep\-re\-sen\-ta\-tion-fi\-nite as required.
\end{proof}

\subsection{The case of \texorpdfstring{$n$}{n} being even}

In this case we have the following families of $(n,d)$-rep\-re\-sen\-ta\-tion-fi\-nite algebras. 

\begin{proposition}\label{prop:thecasendeven}
Let $n$ be even and $0< k < n$.
\begin{itemize}

\item[(a)] If $k$ is even, then the acyclic Nakayama algebra $\La$ with Kupisch series
\[\left( 2^{(k)}, 3^{\left(3\frac{n-k}{2}-1\right)}, 2^{(k+1)},1\right)\]
is $(n,n+k)$-rep\-re\-sen\-ta\-tion-fi\-nite and we have $\pd(I(1))=n+k$.

\item[(b)] If $k$ is odd and $k\neq n-1$, then the acyclic Nakayama algebra $\La$ with Kupisch series
\[\left(2^{(k)},3^{\left(3\left(n-\frac{k+1}{2}\right)\right)},2^{(k+1)},1\right)\]
is $(n,2n+k)$-rep\-re\-sen\-ta\-tion-fi\-nite and we have $\pd(I(1))=2n+k$. 

\item[(c)] If $k=n-1$, then the acyclic Nakayama algebra $\La=\La_{\frac{9n}{2},3}$ with Kupisch series
\[\left(3^{\left(\frac{9n}{2}-2 \right)} ,2,1\right)\]
is $(n,2n+k)$-rep\-re\-sen\-ta\-tion-fi\-nite and we have $\pd(I(1))=2n+k$.

\end{itemize}
\end{proposition}

\begin{proof}
The proof is similar to the proof of Proposition \ref{prop:thecasendodd}. Computations are done using Lemma \ref{lemma:nakayamacompute}.
\begin{itemize}

\item[(a)] The Aus\-lan\-der--Rei\-ten quiver $\Gamma(\La)$ in this case has the form
\[\begin{tikzpicture}
\node (A1) at (0,0) {$Q_1$};
\node (B1) at (0.3,0.3) {$\bullet$};
\node (B2) at (0.6,0) {\textopenbullet};
\node (C1) at (0.9,0.3) {$\bullet$};
\node (C2) at (1.2,0) {\textopenbullet};

\node (D1) at (2.2,0.3) {$\bullet$};
\node (D2) at (2.5,0) {\textopenbullet};
\node (E1) at (2.8,0.3) {$Q_2$};
\node (E2) at (3.1,0) {\textopenbullet};
\node (F1) at (3.1,0.6) {$\bullet$};
\node (F2) at (3.4,0.3) {\textopenbullet};
\node (F3) at (3.7,0) {\textopenbullet};
\node (G1) at (3.7,0.6) {$\bullet$};
\node (G2) at (4,0.3) {\textopenbullet};
\node (G3) at (4.3,0) {\textopenbullet};

\node (H1) at (5,0.6) {$\bullet$};
\node (H2) at (5.3,0.3) {\textopenbullet};
\node (H3) at (5.6,0) {\textopenbullet};
\node (I1) at (5.6,0.6) {$\bullet$};
\node (I2) at (5.9,0.3) {$J_1$};
\node (I3) at (6.2,0) {\textopenbullet};
\node (J1) at (6.5,0.3) {$\bullet$};
\node (J2) at (6.8,0) {\textopenbullet};

\node (K1) at (7.8,0.3) {$\bullet$};
\node (K2) at (8.1,0) {\textopenbullet};
\node (L1) at (8.4,0.3) {$\bullet$};
\node (L2) at (8.7,0) {$J_2$\nospacepunct{,}};

\draw[loosely dotted] (1.4,0.15) -- (2,0.15);
\draw[loosely dotted] (4.35,0.3) -- (4.95,0.3);
\draw[loosely dotted] (7,0.15) -- (7.6,0.15);
\end{tikzpicture}\]
where $Q_1\cong M\left(1,1\right)$, $Q_2\cong M\left(k+1,2\right)$, $J_1\cong M\left(\frac{3n-k}{2},2\right)$, $J_2\cong M\left(\frac{3n+k}{2}+1,1\right)$ and the bold vertices correspond to the indecomposable pro\-jec\-tive-in\-jec\-tive $\La$-modules. Moreover, the vertices $Q_1$ and $Q_2$ correspond to the indecomposable projective noninjective $\La$-modules, the vertices $J_1$ and $J_2$ correspond to the indecomposable injective nonprojective $\La$-modules and for $i=1,2$ we have
\[ \tno (Q_i) \cong J_i,\;\; \tn (J_i) \cong Q_i.\]
Hence by Theorem \ref{thrm:char} we have that $\cC=\add(\La\oplus D(\La))$ is an $n$-cluster tilting subcategory. Finally, we have $I(1)=J_2$ and 
\[d=\gldim(\La) = \pd(J_2)=n+k.\]

\item[(b)] The Aus\-lan\-der--Rei\-ten quiver $\Gamma(\La)$ in this case has the form
\[\begin{tikzpicture}
\node (A1) at (0.6,0) {$Q_1$};
\node (B1) at (0.9,0.3) {$\bullet$};
\node (B2) at (1.2,0) {\textopenbullet};

\node (C1) at (2.2,0.3) {$\bullet$};
\node (C2) at (2.5,0) {\textopenbullet};
\node (D1) at (2.8,0.3) {$Q_2$};
\node (D2) at (3.1,0) {\textopenbullet};
\node (E1) at (3.1,0.6) {$\bullet$};
\node (E2) at (3.4,0.3) {\textopenbullet};
\node (E3) at (3.7,0) {\textopenbullet};

\node (F1) at (4.4,0.6) {$\bullet$};
\node (F2) at (4.7,0.3) {\textopenbullet};
\node (F3) at (5,0) {\textopenbullet};
\node (G1) at (5,0.6) {$\bullet$};
\node (G2) at (5.3,0.3) {\textopenbullet};
\node (G3) at (5.6,0) {$N_1$};
\node (H1) at (5.6,0.6) {$\bullet$};
\node (H2) at (5.9,0.3) {\textopenbullet};
\node (H3) at (6.2,0) {\textopenbullet};

\node (I1) at (6.9,0.6) {$\bullet$};
\node (I2) at (7.2,0.3) {\textopenbullet};
\node (I3) at (7.5,0) {\textopenbullet};
\node (J1) at (7.5,0.6) {$\bullet$};
\node (J2) at (7.8,0.3) {\textopenbullet};
\node (J3) at (8.1,0) {$N_2$};
\node (K1) at (8.1,0.6) {$\bullet$};
\node (K2) at (8.4,0.3) {\textopenbullet};
\node (K3) at (8.7,0) {\textopenbullet};

\node (L1) at (9.4,0.6) {$\bullet$};
\node (L2) at (9.7,0.3) {\textopenbullet};
\node (L3) at (10,0) {\textopenbullet};
\node (M1) at (10.3,0.3) {$J_1$};
\node (M2) at (10.6,0) {\textopenbullet};
\node (N1) at (10.9,0.3) {$\bullet$};
\node (N2) at (11.2,0) {\textopenbullet};

\node (O1) at (12.2,0.3) {$\bullet$};
\node (O2) at (12.5,0) {\textopenbullet};
\node (P1) at (12.8,0.3) {$\bullet$};
\node (P2) at (13.1,0) {$J_2$\nospacepunct{,}};

\draw[loosely dotted] (1.4,0.15) -- (2,0.15);
\draw[loosely dotted] (3.75,0.3) -- (4.35,0.3);
\draw[loosely dotted] (6.25,0.3) -- (6.85,0.3);
\draw[loosely dotted] (8.75,0.3) -- (9.35,0.3);
\draw[loosely dotted] (11.4,0.15) -- (12,0.15);
\end{tikzpicture}\]
where 
\[Q_1\cong M\left(1,1\right),\;\; Q_2\cong M\left(k+1,2\right),\;\; N_1\cong M\left(\frac{3n-k+1}{2},1\right),\;\;\]
\[N_2\cong \left(\frac{3n}{2}+k+1,1\right),\;\; J_1\cong M\left(3n-\frac{k+1}{2},2\right),\;\; J_2\cong M\left(3n+\frac{k+1}{2},1\right)\] 
and the bold vertices correspond to the indecomposable pro\-jec\-tive-in\-jec\-tive $\La$-modules. Moreover, the projective noninjective indecomposable modules are $Q_1$ and $Q_2$, the injective nonprojective indecomposable modules are $J_1$ and $J_2$ and for $i=1,2$ we have
\[\tno(Q_i) \cong N_i,\;\; \tno(N_i)\cong J_i,\;\; \tn(N_i)\cong Q_i,\;\; \tn(J_i) \cong N_i.\]
Hence by Theorem \ref{thrm:char} we have that $\cC=\add(\La\oplus D(\La)\oplus N_1\oplus N_2)$ is an $n$-cluster tilting subcategory. Finally, we have $I(1)=J_2$ and 
\[d=\gldim(\La) = \pd(J_2)=2n+k.\]

\item[(c)] The Aus\-lan\-der--Rei\-ten quiver $\Gamma(\La)$ in this case has the form
\[\begin{tikzpicture}
\node (A1) at (0.6,0) {$Q_1$};
\node (B1) at (0.9,0.3) {$Q_2$};
\node (B2) at (1.2,0) {\textopenbullet};
\node (C1) at (1.2,0.6) {$\bullet$};
\node (C2) at (1.5,0.3) {\textopenbullet};
\node (C3) at (1.8,0) {\textopenbullet};

\node (D1) at (2.5,0.6) {$\bullet$};
\node (D2) at (2.8,0.3) {\textopenbullet};
\node (D3) at (3.1,0) {\textopenbullet};
\node (E1) at (3.1,0.6) {$\bullet$};
\node (E2) at (3.4,0.3) {$N_1$};
\node (E3) at (3.7,0) {$N_2$};
\node (F1) at (3.7,0.6) {$\bullet$};
\node (F2) at (4,0.3) {\textopenbullet};
\node (F3) at (4.3,0) {\textopenbullet};

\node (G1) at (5,0.6) {$\bullet$};
\node (G2) at (5.3,0.3) {\textopenbullet};
\node (G3) at (5.6,0) {\textopenbullet};
\node (H1) at (5.6,0.6) {$\bullet$};
\node (H2) at (5.9,0.3) {\textopenbullet};
\node (H3) at (6.2,0) {$K_1$};
\node (I1) at (6.2,0.6) {$\bullet$};
\node (I2) at (6.5,0.3) {$K_2$};
\node (I3) at (6.8,0) {\textopenbullet};
\node (J1) at (6.8,0.6) {$\bullet$};
\node (J2) at (7.1,0.3) {\textopenbullet};
\node (J3) at (7.4,0) {\textopenbullet};

\node (K1) at (8.1,0.6) {$\bullet$};
\node (K2) at (8.4,0.3) {\textopenbullet};
\node (K3) at (8.7,0) {\textopenbullet};
\node (L1) at (8.7,0.6) {$\bullet$};
\node (L2) at (9,0.3) {$J_1$};
\node (L3) at (9.3,0) {$J_2$\nospacepunct{,}};

\draw[loosely dotted] (1.85,0.3) -- (2.45,0.3);
\draw[loosely dotted] (4.35,0.3) -- (4.95,0.3);
\draw[loosely dotted] (7.45,0.3) -- (8.05,0.3);
\end{tikzpicture}\]
where $Q_1\cong M\left(1,1\right)$, $Q_2\cong M\left(1,2\right)$, $N_1\cong M\left(\frac{3n}{2},2\right)$, $N_2\cong \left(\frac{3n}{2}+1,1\right)$, $K_1\cong M\left(3n,1 \right)$, $K_2\cong M\left(3n,2\right)$, $J_1\cong M\left(\frac{9n}{2}-1,2\right)$, $J_2\cong M\left(\frac{9n}{2},1\right)$ and the bold vertices correspond to the indecomposable pro\-jec\-tive-in\-jec\-tive $\La$-modules. Moreover, the vertices $Q_1$ and $Q_2$ correspond to the indecomposable projective noninjective $\La$-modules, the vertices $J_1$ and $J_2$ correspond to the indecomposable injective nonprojective $\La$-modules and for $i=1,2$ we have
\[ \tno(Q_i) \cong N_i,\;\; \tno(N_i)\cong K_i,\;\; \tno(K_i)\cong J_i,\]
\[\tn(J_i) \cong K_i,\;\; \tn(K_i)\cong N_i,\;\; \tn(N_i)\cong Q_i.\]
Hence by Theorem \ref{thrm:char} we have that $\cC=\add(\La\oplus D(\La)\oplus N_1\oplus N_2\oplus K_1\oplus K_2)$ is an $n$-cluster tilting subcategory. Finally, we have $I(1)=J_2$ and 
\[d=\gldim(\La) = \pd(J_2)=3n-1.\]
\end{itemize}
\end{proof}

\begin{corollary}\label{cor:evenclustertilting}
Let $n$ be even and $d\geq 2n$. There exists an $(n,d)$-rep\-re\-sen\-ta\-tion-fi\-nite algebra $\La$.
\end{corollary}

\begin{proof}
Write $d=qn+d'$ for some $q\geq 2$ and $0\leq d'< n$. If $d'=0$ then $\La$ exists by Example \ref{ex:thecasekn}. Assume that $d'>0$. If $d'$ is even, let $\La'$ be an $(n,n+d')$-rep\-re\-sen\-ta\-tion-fi\-nite algebra as in Proposition \ref{prop:thecasendeven}(a). Then $\La'$ satisfies the assumptions of Corollary \ref{cor:allnds}(b) and so there exists an algebra $\La$ which is $(n,n+d'+(q-1)n)=(n,d)$-rep\-re\-sen\-ta\-tion-fi\-nite. If $d$ is odd, let $\La'$ be an $(n,2n+d')$-rep\-re\-sen\-ta\-tion-fi\-nite algebra as in Proposition \ref{prop:thecasendeven}(b) or (c), depending on whether $d'<n-1$ or $d'=n-1$. If $q=2$ then take $\La=\La'$. Otherwise, if $q>2$, since $\La'$ satisfies the assumptions of Corollary \ref{cor:allnds}(b), there exists an algebra $\La$ which is $(n,2n+d'+(q-2)n)=(n,d)$-rep\-re\-sen\-ta\-tion-fi\-nite, as required.
\end{proof}

Let us give an example in this case as well.

\begin{example}\label{ex:thecase6}
Let $n=6$. Using Proposition \ref{prop:thecasendeven} and Example \ref{ex:thecasekn} we have in Table \ref{tab:2} a list of $(6,d)$-rep\-re\-sen\-ta\-tion-fi\-nite algebras $\La$ where the $6$-cluster tilting subcategories are denoted by the bold vertices in the Aus\-lan\-der--Rei\-ten quivers. Using Corollary \ref{cor:allnds} and the list in Table \ref{tab:2}, we obtain a $(6,d)$-rep\-re\-sen\-ta\-tion-fi\-nite algebra for any $d\geq 12$.

\begin{table}[H] \caption{$(6,d)$--rep\-re\-sen\-ta\-tion-fi\-nite Nakayama algebras for $d=6,8,10,13,15,17$}\label{tab:2}  
\begin{tabular}{|>{\centering\arraybackslash}m{5.5cm}|>{\centering\arraybackslash}m{8cm}|>{\centering\arraybackslash}m{0.4cm}|}
\hline 
    $\La$ (as a Kupisch series) & $\Gamma(\La)$ and indecomposables in $\cC_{\La}$  & $d$ \\   
\hline 
    $\left(2^{(6)},1\right)$ 
    & \begin{tikzpicture}[scale=1.3, transform shape]
        \foreach \x in {0,0.2,...,1.2}
            \mct{\x}{0};
        \foreach \x in {0.1,0.3,...,1.1}
            \nct{\x}{0.2};
        \nct{0}{0};
        \nct{1.2}{0};
        \node (A) at (0.6,0.25) {};
    \end{tikzpicture}
    & $6$ \\
\hline    
    $\left(2^{(2)},3^{(15)},2^{(2)},1\right)$
    & \begin{tikzpicture}[scale=1.3, transform shape]
        \foreach \x in {0,0.2,...,3.6}
            \mct{\x}{0};
        \foreach \x in {0.1,0.3,...,3.5}
            \mct{\x}{0.2};
        \foreach \x in {0.4,0.6,...,3.4}
            \nct{\x}{0.4};
        \nct{0}{0};
        \nct{0.1}{0.2};
        \nct{0.3}{0.2};
        \nct{1.6}{0};
        \nct{2}{0};
        \nct{3.3}{0.2};
        \nct{3.5}{0.2};
        \nct{3.6}{0};
                        
        \node (A) at (1.8,0.45) {};
    \end{tikzpicture} 
    & $13$ \\
\hline
    $\left(2^{(2)},3^{(5)},2^{(3)},1\right)$
    & \begin{tikzpicture}[scale=1.3, transform shape]
        \foreach \x in {0,0.2,...,2}
            \mct{\x}{0};
        \foreach \x in {0.1,0.3,...,1.9}
            \mct{\x}{0.2};
        \foreach \x in {0.6,0.8,...,1.4}
            \nct{\x}{0.4};
        \nct{0}{0};
        \nct{0.1}{0.2};
        \nct{0.3}{0.2};
        \nct{0.5}{0.2};
        \nct{1.5}{0.2};
        \nct{1.7}{0.2};
        \nct{1.9}{0.2};
        \nct{2}{0};
                        
        \node (A) at (1,0.45) {};
    \end{tikzpicture}
    & $8$ \\
\hline 
    $\left(2^{(3)},3^{(12)},2^{(4)},1\right)$
    & \begin{tikzpicture}[scale=1.3, transform shape]
        \foreach \x in {0,0.2,...,3.8}
            \mct{\x}{0};
        \foreach \x in {0.1,0.3,...,3.7}
            \mct{\x}{0.2};
        \foreach \x in {0.8,1,...,3}
            \nct{\x}{0.4};
        \nct{0}{0};
        \nct{0.1}{0.2};
        \nct{0.3}{0.2};
        \nct{0.5}{0.2};
        \nct{0.7}{0.2};
        \nct{1.4}{0};
        \nct{2.4}{0};
        \nct{3.1}{0.2};
        \nct{3.3}{0.2};
        \nct{3.5}{0.2};
        \nct{3.7}{0.2};
        \nct{3.8}{0};
        
        \node (A) at (1.9,0.45) {};
    \end{tikzpicture}
    & $15$ \\
\hline 
    $\left(2^{(4)},3^{(2)},2^{(5)},1\right)$ 
    & \begin{tikzpicture}[scale=1.3, transform shape]
        \foreach \x in {0,0.2,...,2.2}
            \mct{\x}{0};
        \foreach \x in {0.1,0.3,...,2.1}
            \mct{\x}{0.2};
        \foreach \x in {1,1.2}
            \nct{\x}{0.4};
        \nct{0}{0};
        \nct{0.1}{0.2};
        \nct{0.3}{0.2};
        \nct{0.5}{0.2};
        \nct{0.7}{0.2};
        \nct{0.9}{0.2};
        \nct{1.3}{0.2};
        \nct{1.5}{0.2};
        \nct{1.7}{0.2};
        \nct{1.9}{0.2};
        \nct{2.1}{0.2};
        \nct{2.2}{0};        
        
        \node (A) at (1.1,0.45) {};
    \end{tikzpicture}
    & $10$ \\
\hline 
    $\left(3^{(25)},2,1\right)$
    & \begin{tikzpicture}[scale=1.3, transform shape]
        \foreach \x in {0,0.2,...,5.2}
            \mct{\x}{0};
        \foreach \x in {0.1,0.3,...,5.1}
            \mct{\x}{0.2};
        \foreach \x in {0.2,0.4,...,5}
            \nct{\x}{0.4};
        \nct{0}{0};
        \nct{0.1}{0.2};
        \nct{1.7}{0.2};
        \nct{1.8}{0};
        \nct{3.4}{0};
        \nct{3.5}{0.2};
        \nct{5.1}{0.2};
        \nct{5.2}{0};
        
        \node (A) at (2.6,0.45) {};
    \end{tikzpicture} 
    & $17$ \\
\hline
    \end{tabular}
\end{table}
\end{example}

\subsection{Main result}

In this short section we summarize the results of section \ref{sect:listofnds} in the following theorem.

\begin{theorem}\label{thrm:thendalgebras}
Let $n$ be a positive integer and $d\geq n$.
\begin{itemize}
\item[(a)] If $n$ is odd, then there exists an $(n,d)$-rep\-re\-sen\-ta\-tion-fi\-nite algebra.
\item[(b)] If $n$ is even, and $d$ is even or $d\geq 2n$, then there exists an $(n,d)$-rep\-re\-sen\-ta\-tion-fi\-nite algebra.
\end{itemize}
\end{theorem}

\begin{proof}
Follows immediately by Corollary \ref{cor:oddclustertilting}, Corollary \ref{cor:evenclustertilting} and Proposition \ref{prop:thecasendeven}.
\end{proof}

\begin{remark}
Let us note that Theorem \ref{thrm:thendalgebras} is not sharp in the sense that there exist alebras that are $(n,d)$-rep\-re\-sen\-ta\-tion-fi\-nite where $n$ is even, while $n<d<2n$ and $d$ is odd. For example, in \cite[Example 3.8]{VAS} it was shown that the path algebra of the quiver with relations
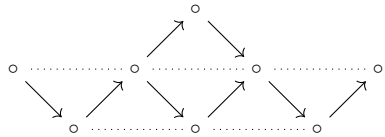
\[\begin{tikzpicture}[scale=0.8, transform shape]
\node (A) at (0,0) {$\circ$};
\node (B) at (-1,1) {$\circ$};
\node (C) at (-2,0) {$\circ$};
\node (D) at (-2,2) {$\circ$};
\node (E) at (-3,1) {$\circ$};
\node (F) at (-4,0) {$\circ$};
\node (G) at (-5,1) {$\circ$};
\node (H) at (1,1) {$\circ$};

\draw[<-] (A) to (B);
\draw[<-] (B) to (D);
\draw[<-] (B) to (C);
\draw[<-] (D) to (E);
\draw[<-] (C) to (E);
\draw[<-] (E) to (F);
\draw[<-] (F) to (G);
\draw[<-] (H) to (A);

\draw[dotted] (A) -- (C);
\draw[dotted] (C) -- (F);
\draw[dotted] (H) -- (B);
\draw[dotted] (B) -- (E);
\draw[dotted] (E) -- (G);
\end{tikzpicture}\]
is $(2,3)$-rep\-re\-sen\-ta\-tion-fi\-nite.
\end{remark}

\section{Summary of notation}\label{sec:summary of notation}

Due to the length and the technical nature of the paper, we include a list of terminology with references to the corresponding numbered definitions as well as a list of symbols with a short description and a reference to where they are first encountered.

\newpage

\subsection*{List of definitions and table of notation}

\begin{itemize}
    \item[] \emph{$(n,d)$--rep\-re\-sen\-ta\-tion-fi\-nite algebra} \dotfill \textbf{Introduction}.
    \item[] \emph{Abutment} \dotfill \textbf{Section \ref{sect:construction}}.
    \item[] \emph{Completion of a fracture} \dotfill \textbf{Section \ref{sect:construction}}.
    \item[] \emph{Gluing of algebras} \dotfill \textbf{Definition \ref{def:gluing}}.
    \item[] \emph{Gluing of fracturings} \dotfill \textbf{Section \ref{sect:construction}}.
    \item[] \emph{Gluing of subcategories} \dotfill \textbf{Definition \ref{def:catglue}}.
    \item[] \emph{Height of an abutment \dotfill} \textbf{Definition \ref{def:abutment}}.
    \item[] \emph{Footing} \dotfill \textbf{Definition \ref{def:footing}}.
    \item[] \emph{Foundation of an abutment} \dotfill \textbf{Proposition \ref{prop:triangles}}.
    \item[] \emph{Fracture} \dotfill \textbf{Definition \ref{def:fracture}}.
    \item[] \emph{Fracturing} \dotfill \textbf{Definition \ref{def:fracturing}}.
    \item[] \emph{Kupisch series} \dotfill \textbf{Definition \ref{def:Kupisch}}.
    \item[] \emph{Left $n$-fractured subcategory} \dotfill \textbf{Definition \ref{def:left and right nfract}}.
    \item[] \emph{Level of a fracture} \dotfill \textbf{Definition \ref{def:fracture}}.
    \item[] \emph{Maximal abutment} \dotfill \textbf{Section \ref{subsec:fractured subcategories}}.
    \item[] \emph{Module supported in an algebra} \dotfill \textbf{Definition \ref{def:supported}}.
    \item[] \emph{$n$-cluster tilting subcategory and module} \dotfill \textbf{Definition \ref{def:n-ct}}.
    \item[] \emph{$n$-fractured subcategory} \dotfill \textbf{Definition \ref{def:nctfract}}.
    \item[] \emph{Realization of an abutment} \dotfill \textbf{Definition \ref{def:abutment}}.
    \item[] \emph{Rep\-re\-sen\-ta\-tion-di\-rect\-ed algebra} \dotfill \textbf{Introduction}.
    \item[] \emph{Right $n$-fractured subcategory} \dotfill \textbf{Definition \ref{def:left and right nfract}}.
    \item[] \emph{Slice fracture} \dotfill \textbf{Definition \ref{def:slice fracture}}.
    \item[] \emph{Slice subcategory and module} \dotfill \textbf{Definition \ref{def:slice}}.
    \item[] \emph{Strongly $(n,d)$-rep\-re\-sen\-ta\-tion-di\-rect\-ed algebra} \dotfill \textbf{Introduction}.
\end{itemize}

\clearpage

\begin{table}[hbt!]
\resizebox{0.89\textwidth}{!}{
\begin{tabular}{|c|p{8.8cm}|p{2.4cm}|}
	\hline
	Symbol & Description & Reference\\ \hline\hline
	\raisebox{-.5\normalbaselineskip}[0pt][0pt]{$P(k)$ (resp. $I(k)$)} & The indecomposable projective (resp. injective) $KQ/\cR$-module corresponding to a vertex $k\in Q_0$. & \raisebox{-.5\normalbaselineskip}[0pt][0pt]{Section~\ref{subsection:conventions}.}\\ \hline
	\raisebox{-.5\normalbaselineskip}[0pt][0pt]{$\overline{\cC}$} & The set of isomorphism classes of indecomposable modules in $\cC\subseteq\m\La$. & \raisebox{-.5\normalbaselineskip}[0pt][0pt]{Section~\ref{subsection:conventions}.}\\ \hline 
	$\abs{\cC}$ & The cardinality of $\overline{\cC}$. &
	Section \ref{subsection:conventions}.\\ \hline
    $\add(\cC)$ & The additive closure of $\cC$. &
	Section~\ref{subsection:conventions}.\\ \hline
    \raisebox{-.5\normalbaselineskip}[0pt][0pt]{$\Sub(\cC)$ (resp. $\Fac(\cC)$)} & The subcategory of $\m\La$ containing all submodules (resp. factor modules) of modules in $\cC$. & 
	\raisebox{-.5\normalbaselineskip}[0pt][0pt]{Section~\ref{subsection:conventions}.}\\ \hline
    \raisebox{-.5\normalbaselineskip}[0pt][0pt]{$D$} & The standard duality $D=\Hom_{K}(-,K)$ between $\m\La$ and $\m\La^{\text{op}}$. & \raisebox{-.5\normalbaselineskip}[0pt][0pt]{Section~\ref{subsection:conventions}.}\\ \hline
    $\Omega(X)$ (resp. $\Omega^{-}(X)$) & The syzygy (resp. cosyzygy) of $X$. & Section~\ref{subsection:conventions}.\\ \hline
    \raisebox{-.5\normalbaselineskip}[0pt][0pt]{$\tau_n(X)$ and $\tau_n^{-}(X)$} & The $n$-Aus\-lan\-der--Rei\-ten translations $\tau_n(X)=\tau\om^{n-1}(X)$ and $\tau_n^{-}(X)=\tau^{-}\om^{-(n-1)}(X)$. & \raisebox{-.5\normalbaselineskip}[0pt][0pt]{Section~\ref{subsection:conventions}.}\\ \hline
    $\Gamma(\La)$ & The Aus\-lan\-der--Rei\-ten quiver of $\La$. & Section~\ref{subsection:conventions}.\\ \hline
    \raisebox{-.5\normalbaselineskip}[0pt][0pt]{$[M]$} & The vertex in $\Gamma(\La)$ corresponding to the indecomposable $\La$-module $M$. & \raisebox{-.5\normalbaselineskip}[0pt][0pt]{Section~\ref{subsection:conventions}.}\\ \hline
    \raisebox{-.5\normalbaselineskip}[0pt][0pt]{$\phi_{\ast}$} & The restriction of scalars functor $\phi_{\ast}:\m\Gamma\to\m\La$ for an algebra morphism $\phi:\La\to \Gamma$. & \raisebox{-.5\normalbaselineskip}[0pt][0pt]{Section~\ref{subsection:conventions}.}\\ \hline
    \raisebox{-.9\normalbaselineskip}[0pt][0pt]{$\phi^{\ast}$} & The left adjoint $\phi^{\ast}(-)=-\otimes_{\La}\Gamma$ to the restriction of scalars functor $\phi_{\ast}:\m\Gamma\to\m\La$ for an algebra morphism $\phi:\La\to \Gamma$. & \raisebox{-.9\normalbaselineskip}[0pt][0pt]{Section~\ref{subsection:conventions}.}\\ \hline
    \raisebox{-.9\normalbaselineskip}[0pt][0pt]{$\phi^{!}$} & The right adjoint $\phi^{!}(-)=\Hom_{\La}(\Gamma,-)$ to the restriction of scalars functor $\phi_{\ast}:\m\Gamma\to\m\La$ for an algebra morphism $\phi:\La\to \Gamma$. & \raisebox{-.9\normalbaselineskip}[0pt][0pt]{Section~\ref{subsection:conventions}.}\\ \hline
    \raisebox{+.2\normalbaselineskip}[0pt][0pt]{$A_h$} & The quiver $1\overset{\alpha_1}{\longrightarrow} 2 \overset{\alpha_2}{\longrightarrow} 3 \longrightarrow\cdots\longrightarrow h-1\overset{\alpha_{h-1}}{\longrightarrow} h$. & \raisebox{+.1\normalbaselineskip}[0pt][0pt]{Section~\ref{subsection:conventions}.}\\ \hline
    \raisebox{-.9\normalbaselineskip}[0pt][0pt]{$M(i,j)$} & The indecomposable representation of a Nakayama algebra $\La=KA_m/\cR$ with support $\{m-(i-1)-(j-1),\dots,m-(i-1)\}$. & \raisebox{-.9\normalbaselineskip}[0pt][0pt]{Section~\ref{subsection:conventions}.}\\ \hline
    $\La_{m,h}$ & The bound quiver algebra $KA_m/\rad(KA_m)^h$. & Section~\ref{subsection:conventions}.\\ \hline
    $\triangle(h)$ & The Aus\-lan\-der--Rei\-ten quiver $\Gamma(KA_h)$. & Section~\ref{subsection:conventions}.\\ \hline
    $\cL=\cB\glue \cA$ & The category $\cL$ is the gluing of its subcategories $\cB$ and $\cA$. & Definition~\ref{def:catglue}.\\ \hline
    \raisebox{-.5\normalbaselineskip}[0pt][0pt]{$(e_i,f_i)_{i=1}^h$ (resp. $(e_i,g_{i-1})_{i=1}^h$)} & The realization of a left abutment $P=e_1\La$ and of a right abutment $I=D(\La e_h)$. & \raisebox{-.5\normalbaselineskip}[0pt][0pt]{Section~\ref{sect:abutments}.}\\ \hline
    $Q_{\La}$ & The ordinary quiver of an algebra $\La$. & Section~\ref{sect:abutments}.\\ \hline
    \raisebox{-.5\normalbaselineskip}[0pt][0pt]{$f_P$ (resp. $g_I$)} & The footing $f_P:\La\twoheadrightarrow KA_h$ for a left abutment $P$ and $g_I:\La\twoheadrightarrow KA_h$ for a right abutment $I$. & \raisebox{-.5\normalbaselineskip}[0pt][0pt]{Definition~\ref{def:footing}.}\\ \hline
    \raisebox{-.5\normalbaselineskip}[0pt][0pt]{$\PD$ (resp. $\DI$)} & The foundation of a left abutment $P$ and of a right abutment $I$. & \raisebox{-.5\normalbaselineskip}[0pt][0pt]{Proposition~\ref{prop:triangles}.}\\ \hline
    \raisebox{-.9\normalbaselineskip}[0pt][0pt]{$\cF_P$ (resp. $\cG_I$)} & The smallest additive subcategory of $\m\La$ containing all modules in the foundation of a left abutment $P$ (resp. a right abutment $I$). & \raisebox{-.9\normalbaselineskip}[0pt][0pt]{Section~\ref{sect:abutments}.}\\ \hline
    \raisebox{-.2\normalbaselineskip}[0pt][0pt]{$B\glue[P][I] A$} & The gluing of $A$ and $B$ along $P$ and $I$. & Definition~\ref{def:gluing}.\\ \hline
    \raisebox{-.5\normalbaselineskip}[0pt][0pt]{$\supp(M)$} & The support of a representation $M$ of a bound quiver algebra. & \raisebox{-.5\normalbaselineskip}[0pt][0pt]{Section~\ref{subsubsection:gluing via pullbacks}.}\\ \hline
    $\cP_{\La}$ (resp. $\cI_{\La}$) & The category of projective (resp. injective) $\La$-modules. & Section~\ref{subsec:fractured subcategories}\\ \hline
    \raisebox{-.5\normalbaselineskip}[0pt][0pt]{$\ABL$ (resp. $\LAB$)} & The additive closure of the category of left (resp. right) abutments of $\La$. & \raisebox{-.5\normalbaselineskip}[0pt][0pt]{Section~\ref{subsec:fractured subcategories}.}\\ \hline
    \raisebox{-.5\normalbaselineskip}[0pt][0pt]{$\MABL$ (resp. $\MLAB$)} & The additive closure of the category of maximal left (resp. right) abutments of $\La$. & \raisebox{-.5\normalbaselineskip}[0pt][0pt]{Section~\ref{subsec:fractured subcategories}.}\\ \hline
    \raisebox{-.5\normalbaselineskip}[0pt][0pt]{$\cC_{\setminus \cV}$} & The additive closure of all indecomposable modules $X\in\cC$ such that $X\not\in\cV$. & \raisebox{-.5\normalbaselineskip}[0pt][0pt]{Section~\ref{subsec:fractured subcategories}.}\\ \hline
    \raisebox{-.5\normalbaselineskip}[0pt][0pt]{$T^{(P)}$ (resp. $T^{(I)}$)} & The fracture corresponding to a maximal left (resp. right) abutment $P$ (resp. $I$). & \raisebox{-.5\normalbaselineskip}[0pt][0pt]{Definition~\ref{def:fracture}.}\\ \hline
    $\lvl(T)$ & The level of a fracture $T$. & Definition~\ref{def:fracture}.\\ \hline
    \raisebox{-.5\normalbaselineskip}[0pt][0pt]{$T^L$ (resp. $T^R$)} & A left (resp. right) fracturing, that is a direct sum of fractures of maximal left (resp. right) abutments. & \raisebox{-.5\normalbaselineskip}[0pt][0pt]{Definition~\ref{def:fracturing}.}\\ \hline
    \raisebox{-.5\normalbaselineskip}[0pt][0pt]{$\Lab$ (resp. $\DLab$)} & A left (resp. right) fracturing which is projective (resp. injective) as a module. & \raisebox{-.5\normalbaselineskip}[0pt][0pt]{Section~\ref{subsec:fractured subcategories}.}\\ \hline
    \raisebox{-.5\normalbaselineskip}[0pt][0pt]{$T_{\ast}^{(P)}$ (resp. $T_{\ast}^{(I)}$)} & The fracture induced by a maximal left (resp. right) abutment on a gluing. & \raisebox{-.5\normalbaselineskip}[0pt][0pt]{Section~\ref{sect:construction}.}\\ \hline
    \raisebox{-.1\normalbaselineskip}[0pt][0pt]{$h^{(k)}$} & A sequence $h,h,\dots,h$ where $h$ appears $k$ times. & Section~\ref{sect:listofnds}.\\ \hline
\end{tabular}
}\end{table}

\clearpage

\noindent \textbf{Acknowledgements.} The author wishes to thank his advisor Martin Herschend for the constant support and suggestions during the preparation of this article. He also wishes to thank Andrea Pasquali for offering helpful suggestions about the manuscript. Finally he wishes to thank an anonymous referee for their careful reading of the article as well as their many helpful and detailed comments and corrections which improved the presentation considerably.

\bibliography{nct}
\bibliographystyle{halpha}

\end{document}